\documentclass[12pt, oldfontcommands]{memoir}
\usepackage[utf8]{inputenc}
\usepackage{hyperref}
\usepackage{mathtools}
\usepackage{amsmath,amsfonts,amssymb,amsthm}
\usepackage{graphicx, color, soul}
\usepackage{MnSymbol}
\usepackage[dvipsnames]{xcolor}
\usepackage[explicit]{titlesec}
\usepackage{psl-cover}
\usepackage{minitoc}
\usepackage{ulem}
\usepackage{tcolorbox}

\graphicspath{{resources/}}

\newtheorem{theorem}{Theorem}
\newtheorem{proposition}[theorem]{Proposition}
\newtheorem{lemma}[theorem]{Lemma}
\newtheorem{corollary}[theorem]{Corollary}

\theoremstyle{remark}
\newtheorem{remark}[theorem]{Remark}

\theoremstyle{definition}
\newtheorem{definition}[theorem]{Definition}

\theoremstyle{example}
\newtheorem{example}[theorem]{Example}

\numberwithin{equation}{section} 

\newcommand{\go}{[G,o]}
\newcommand{\bgo}{[\textbf{G},\textbf{o}]}

\newcommand{\bgfno}{[\textbf{G},f^n(\textbf{o})]}
\newcommand{\bbargo}{[\bar{\textbf{G}},\bar{\textbf{o}}]}
\newcommand{\calp}{\mathcal{P}}
\newcommand{\calg}{\mathcal{G}}
\newcommand{\calpfn}{\mathcal{P}_0^{f,n}}

 \newcommand*{\RNum}[1]{\expandafter\@slowromancap\romannumeral #1@}

\title{Records of stationary processes and unimodular graphs}
\institute{Département d'Informatique de l'ENS}
\doctoralschool{\mbox{Mathématiques Paris Centre}}{386}
\specialty{Mathématiques}
\author{Bharath Roy Choudhury}
\date{18/09/2023}

\jurymember{7}{François Baccelli}{INRIA}{Directeur de thèse}
\jurymember{8}{Bartłomiej (Bartek) Błaszczyszyn}{INRIA}{Co-Directeur de thèse}
\jurymember{3}{Hermann Thorisson}{University of Iceland}{Rapporteur}
\jurymember{2}{Charles Bordenave}{Institut de Mathématiques de Marseille}{Rapporteur}
\jurymember{5}{Sergey Foss}{Heriot-Watt University}{Examinateur}
\jurymember{1}{Pierre Calka}{Université de Rouen Normandie}{Président du jury}
\jurymember{4}{David Dereudre}{Université de Lille}{Examinateur}
\jurymember{6}{Eliza O'Reilly}{Johns Hopkins University}{Examinatrice}

\frabstract{Considérons une règle de navigation définie sur un graphe qui mappe chaque sommet du graphe à un sommet de telle manière que la règle de navigation commute avec chaque automorphisme du graphe. C'est à dire que la règle de navigation appliquée aux sommets reste la même après prise d'un éventuel automorphisme du graphe. Une telle règle de navigation est dite avoir la propriété de covariance.
Cette étude se penche sur un ensemble de telles règles de navigation covariantes, indexées par des graphes localement finis et soumises à une condition de mesurabilité. Cet ensemble de règles est appelé déplacement de sommet. Plus généralement, on peut considérer les déplacements de sommets sur les réseaux, des graphes qui ont des étiquettes sur les arêtes et sur les sommets. Un déplacement de sommet induit une dynamique sur l’espace des réseaux enracinés localement finis. L'objectif central de ce travail réside dans l'étude de la dynamique associée à une règle de navigation spécifique appelée record vertex-shift. Il est défini sur les trajectoires de toute marche aléatoire discrète unidimensionnelle dont les incréments ont une moyenne finie et la marche aléatoire ne peut sauter que d'un pas vers la gauche. De plus, les travaux incluent plusieurs résultats notables concernant les déplacements de sommets d'enregistrement appliqués aux processus à incréments stationnaires. La thèse contient également des résultats sur les déplacements de sommets plus généraux sur les réseaux unimodulaires.
}

\enabstract{Consider a navigation rule defined on a graph that maps every vertex of the graph to a vertex in such a way that the navigation rule commutes with every automorphism of the graph. It is to say that the navigation rule applied to the vertices remains the same after taking any automorphism of the graph. Such a navigation rule is said to have the covariance property.
This study delves into a collection of such covariant navigation rules, indexed by locally finite graphs, and subject to a measurability condition. This ensemble of rules is termed a vertex-shift.  More generally, one can consider vertex-shifts on networks, graphs that have labels on edges and on vertices. A vertex-shift induces a dynamic on the space of locally finite rooted networks. The central focus of this work lies in investigating the dynamic associated to a specific navigation rule called the record vertex-shift. It is defined on the trajectories of any one dimensional discrete random walk whose increments have finite mean and the random walk can jump only one step to the left. Additionally, the work includes several notable results concerning record vertex-shifts applied to processes with stationary increments. The thesis also contains results on more general vertex-shifts on unimodular networks.}
\frkeywords{unimodularité, graphe des records, limites locales, réseaux aléatoires}
\enkeywords{unimodularity, record graph, local limits, random networks}

\pslassetspath{psl-cover/}
 
\begin{document} 
\maketitle
\setcounter{tocdepth}{2}
\dominitoc
\tableofcontents

\newpage 
\chapter*{Acknowledgements}
I would like to express my sincere gratitude to all the people who have supported me during my PhD journey.
Many thanks to the ERC (grant number 788851) for financially supporting this work.
I am profoundly indebted to my advisor, Fran\c{c}ois Baccelli, for the invaluable guidance, encouragement, and mentorship he has provided. His kindness and patience have made this journey not only enjoyable but also rewarding. Under his tutelage, I've learned not only about mathematics but also about life and happiness, gaining insights that extend far beyond the academic realm.

I wholeheartedly thank my co-supervisor, Bartłomiej Błaszczyszyn. He has generously shared his insights and expertise with me and has always encouraged me to improve my work. 
I have greatly benefited from our fruitful discussions and his constructive feedback.

I am thankful to the reviewers of my thesis, Hermann Thorisson, and Charles Bordenave, for their careful reading and thoughtful suggestions. Their suggestions have helped me to improve my work. 
I extend my thanks to the other jury members: Pierre Calka, David Dereudre, Sergey Foss, and Eliza O'Reilly, for agreeing to be part of this important milestone of my academic career and for asking stimulating questions.

I thank my colleague Ali Khezeli, who has consistently been available to validate the technical aspects of the proofs.

My heartfelt appreciation also goes to Michel Davydov, Ke Feng, Sayeh Khaniha, and Pierre Popineau. They not only enriched my experience during my stay at INRIA but also made it enjoyable and memorable through their kindness, humor, and goûter.
I am thankful to my colleagues Christine Fricker, Roman Gambelin, Guodong Sun,  Alessia Rigonat, and Sanket Kalamkar for the informal discussions and encouragement.

I would like to thank Julien Guieu, Marion Cueille, Christine Tanguy from INRIA, and Eric Sinaman from ENS for their help in the administrative work.

I express my gratitude to my wife Sayanika, whose support was indispensable in bringing this work to completion.
Finally, many thanks to all those individuals whom I may not have mentioned explicitly but who have contributed to my academic journey in various ways.
\chapter{Introduction}

Consider a navigation rule defined on a graph that maps every vertex of the graph to a vertex in such a way that the navigation rule commutes with every automorphism of the graph.
It is to say that the navigation rule applied to the vertices remains the same after taking any automorphism of the graph.
Such a navigation rule is said to have the covariance property.
Such covariant navigation rule indexed by locally finite graphs and satisfying a measurability condition is called a vertex-shift.
More generally, one can consider vertex-shifts on networks, graphs that have labels on edges and on vertices.
A vertex-shift induces a dynamic on the space of locally finite rooted networks.
We are interested in a navigation rule, called the record vertex-shift, defined on trajectories of the one dimensional discrete random walk that can jump only one step to the left.
The main part of the thesis is focused on studying the record vertex-shift on random walks.
A few results on record vertex-shift on processes with stationary increments are also derived.
The thesis also contains results on more general vertex-shifts on unimodular networks (see Def. \ref{def_unimodular}).

A random walk on $\mathbb{Z}$ is called skip-free to the left (resp. right) if its i.i.d. increments take values in $\{j \in \mathbb{Z}:j \geq -1\}$(resp. $\{j \in \mathbb{Z}:j \leq 1\}$).

  Consider an i.i.d. sequence of random variables $X= (X_n)_{n \in \mathbb{Z}}$, where $X_n$ takes values in $\{j \in \mathbb{Z}:j \geq -1\}$ and with finite mean for all \(n \in \mathbb{Z}\).
 A two-sided skip-free random walk $(S_n)_{n \in \mathbb{Z}}$ on $\mathbb{Z}$ with increment sequence $X$ is given by: $S_0=0$, $S_n = \sum_{k=0}^{n-1}X_k, \forall n \geq 0$, and $S_n = \sum_{k=n}^{-1}-X_k, \forall n \leq -1$.
 Observe that $(S_n)_{n \geq 0}$ is skip-free to the left whereas $(S_{-n})_{n \geq 0}$ is skip-free to the right.
  Construct the record graph of \(X\) on $\mathbb{Z}$ by drawing a directed edge from each integer $i$ to its record $R(i):=\min\{n>i:S_n \geq S_i\}$ if the minimum exists, otherwise to $i$ itself.
  The record graph is a random directed graph, with \(\mathbb{Z}\) as the vertex set and where the directed edges are random.
  Consider the connected component of $0$ in the constructed record graph and root it at $0$. 
  Denote it by $[T,0]$.
  This work is motivated by the following questions:
  \begin{enumerate}
    \item  What is the distribution of the random rooted graph $[T,0]$?
    \item How does the distribution of $[T,0]$ vary with the mean of the increment $\mathbb{E}[X_{0}]$?
    \item What is the trajectory of the random walk seen from the $n$-th record of $0$ as $n \to \infty$?
    \item Is there a  non-trivial measure-preserving map on the trajectory of the random walk which is compatible with the record map $R$ (see the discussion below for the definition of `compatibility')?
  \end{enumerate} 
  Concerning the second question, we show that the record graph exhibits a phase transition when the mean $\mathbb{E}[X_{0}]$ is varied between $- \infty$ and $\infty$.
  There are three distinct phases, and each phase has a characteristic feature in terms of the classification of unimodular directed trees.  
  The first phase is observed in the region $(-\infty,0)$, in which the record graph has an infinite number of connected components with the component of the origin being a finite unimodular directed tree.
  The second phase occurs when $\mathbb{E}[X_{0}]=0$.
  In this phase, the record graph is a.s. connected, and it is a one-ended unimodular directed tree.
  The third phase, which occurs when $\mathbb{E}[X_{0}] \in (0,\infty)$, has the characteristic feature that the record graph is a.s. a two-ended unimodular directed tree.

  A Family Tree is a rooted directed tree whose out-degree is at most $1$.
  By neglecting the loops, which occur only if $\mathbb{E}[X_{0}]< 0$, we obtain that the connected component of $0$ in the record graph is a random Family Tree $T$ rooted at $0$.
  We address the first question by explicitly computing the distribution of the Family Tree $[T,0]$ in the three regimes, namely when $\mathbb{E}[X_{0}] \in (-\infty,0)$ (Theorem \ref{20230213191656}), $\mathbb{E}[X_{0}]=0$ (Theorem \ref{theorem:R-graph_egwt}), and $\mathbb{E}[X_{0}]\in (0,\infty)$ (Theorem \ref{r_graph_positive_drift_20230203174636}, Theorem \ref{20230226170801}).

  In \cite{baccelliEternalFamilyTrees2018a}, the authors introduced the notion of vertex-shifts on graphs.
  Informally, a vertex-shift is a collection of maps on graphs that are covariant and satisfy a measurability condition.
  Given a vertex-shift $f$, we can associate to each graph $G$ its $f$-graph $G^f$ whose vertices are the same as those of $G$ and with directed edges given by $\{(x,f(x)):x \in V(G), f(x) \not = x\}$.
  The vertex-shift $f$ defines an equivalence relation on $V(G)$ by saying that two vertices $u$ and $v$ are equivalent if there exists some positive integer $n$ such that $f^n(u)=f^n(v)$.
  The equivalence class of \(u \in V(G)\) is called the foil of \(u\) (denoted as foil\((u)\)) and the collection of all foils of \(V(G)\) is called the foliation of \(G\).
  Since the connected component $G_f(u)$ of any vertex $u$ in $G^f$ is the set \(\{v \in V(G): f^n(v)=f^m(u) \text{ for some } n \geq 1, m \geq 1\} \), we have foil\((u)\subseteq V(G_f(u))\).

  The authors showed in \cite{baccelliEternalFamilyTrees2018a} that given a unimodular random rooted graph $[\mathbf{G}, \mathbf{o}]$, the connected component of the root in the $f$-graph $G^f$ is a unimodular Family Tree, and a.s. every component of \(G^f\) belongs to one of the three classes: $\mathcal{F}/\mathcal{F}$, $\mathcal{I}/\mathcal{I}$, and $\mathcal{I}/\mathcal{F}$.
 The class $\mathcal{F}/\mathcal{F}$ corresponds to the component being finite and its foils being finite.
 Similarly, the classes $\mathcal{I}/\mathcal{I}$ and $\mathcal{I}/\mathcal{F}$ correspond to the component being infinite and all foils of the component being infinite, and the component being infinite and all foils of the component being finite, respectively.
  It is easy to see that the record map $R$ is a vertex-shift, which
  implies that the Family Tree $[T,0]$ is unimodular, where \(T\)  is the connected component of \(0\) in the record graph.
  The three phases of the record graph $[T,0]$ for \(\mathbb{E}[X_0]<0, \mathbb{E}[X_0]=0\) and \(\mathbb{E}[X_0]>0\) correspond to these three classes \(\mathcal{F}/\mathcal{F}, \mathcal{I}/\mathcal{I}\) and \(\mathcal{I}/\mathcal{F}\) respectively.
  So, the record vertex-shift is an example for which all the three classes respectively occur depending on the mean $\mathbb{E}[X_{0}]$.

Using the framework of \cite{aldousProcessesUnimodularRandom2007}, we view the increments of the skip-free random walk as marks associated to the edges of $\mathbb{Z}$, seen from the root $0$.
More precisely, we consider the random rooted network $[\mathbb{Z},0,X]$, with $X=(X_n)_{n \in \mathbb{Z}}$ being the sequence of increments and \(X_n\) being the label assigned to the edge \((n,n+1)\), for all  \(n \in \mathbb{Z}\).
It is easy to see that $[\mathbb{Z},0,X]$ is a unimodular network.
By the local view from the $n$-th record ($n \geq 1$), we mean the distribution of $[\mathbb{Z},R^n(0),X]$.
The latter is the same as $[\mathbb{Z},0,R^n(0)X]$, where $R^n(0)X$ is the shifted sequence \( (X_{m+R^n(0)})_{m \in \mathbb{Z}}\).
Using the unimodularity of $[\mathbb{Z},0,X]$, we show that the local view from the $n$-th record is absolutely continuous with respect to $[\mathbb{Z},0,X]$ with Radon-Nikodym derivative $d_n(0)$, which is defined as the number of integers for which $0$ is the $n$-th record.

Further, in connection with the third question, we show that the local view seen from the $n$-th record as $n$ tends to  $\infty$, $\lim_{n \to \infty}[\mathbb{Z},0,R^n(0)X]$, exists for \(\mathbb{E}[X_0]<0, \mathbb{E}[X_0]=0\) and \(\mathbb{E}[X_0]>0\) respectively.
Denote this limit by $[\mathbb{Z},0,Y]$, where $Y=(Y_n)_{n \in \mathbb{Z}}$ is the mark on the edges of $\mathbb{Z}$.
We observe that $Y$ is not an i.i.d. sequence.
In particular, its negative part $(Y_n)_{n<0}$ fails to be an i.i.d. sequence.
However, it has a regenerative structure when $\mathbb{E}[X_{0}]\geq 0$.

As for the fourth question, a map $R'$ on $\mathbb{Z}$ is said to be \emph{foil-preserving} with $R$ if it maps every foil to itself, i.e., for every $i \in \mathbb{Z}$, $R'(i)$ belongs to the foil($i$) (the foil associated to the vertex-shift $R$ and $i$).

When \(\mathbb{E}[X_0] \geq 0\), we define a random map $R_{\perp}$ that takes $i$ to 
\[ R_{\perp}(i)=\inf\{k>i:R^n(i)=R^n(k) \text{ for some } n \geq 1\}.\]
We show that $R_{\perp}$ is foil-preserving with \(R\) and it preserves the distribution of $[\mathbb{Z},0,X]$, i.e., the distribution of \([\mathbb{Z},R_{\perp}(0),X]\) is the same as that of \([\mathbb{Z},0,X]\).

The statistics of records of time series have found applications in finance, hydrology and physics.
See for instance the references in the survey \cite{godrecheRecordStatisticsStrongly2017} for a non-exhaustive list of applications.
These studies of records are focused on statistics such as the distribution of $n$-th record, the age of a record (i.e., the gap between two consecutive records), and the age of the long-lasting record (\cite{wergenRecordsStochasticProcesses2013}, \cite{godrecheRecordStatisticsStrongly2017}). 
When the time series under consideration is an i.i.d. sequence, it is shown in \cite{stepanovCharacterizationTheoremWeak1994} that the mean of the jump between two consecutive record values conditioned on the initial record value could characterize the distribution of the i.i.d. sequence.
In contrast, in the present work, the focus is on the joint local structure of the records of skip-free random walks. 
The structure has a natural family tree flavour with an order suggesting to term it as `genealogy of records'.

The relation between skip-free random walks and Galton-Watson trees was extensively discussed in, e.g., \cite{benniesRandomWalkApproach2000}, \cite{legallRandomTreesApplications2005b}, and \cite{jimpitmanCombinatorialStochasticProcesses2006}.
In fact, the authors of \cite{benniesRandomWalkApproach2000} encode the critical Galton-Watson tree and show that the encoding gives a correspondence between records of the random walk and ancestors of the root in the Galton-Watson tree.
We extend these ideas to define the backward map $\Phi_R$ that takes a record graph and maps it back to the trajectory of a random walk (see Section \ref{subsec_backward_map}, Chapter \ref{chapter_record_V_shift_probability}).
However, we are not restricted to finite Galton-Watson trees here.
To the best of our knowledge this complete description of the record graph of skip-free random walk is new.

In Chapter \ref{chapter_prelim}, we give the background and the statements of the theorems that we use in this work.
We  introduce the canonical probability spaces in which our random objects take values and  introduce the dynamics induced by vertex-shifts.
Our focus is on studying vertex-shifts on unimodular networks.
We state a key theorem that we use in this regard known as Foil classification in unimodular networks Theorem.
We also introduce some operations on unimodular Family Trees.

In Chapter \ref{chapter_record_v_shift}, we introduce the record map with which we obtain the record vertex-shift.
We study its properties on deterministic integer-valued sequences.
We study the properties of the record graph associated to the sequences.
Using these properties, we prove that the record graph of \([\mathbb{Z},0,X]\) exhibits phase transitions, where \(X\) is the i.i.d. sequence of  increments of a skip-free to the left random walk.
We also describe the distribution of component of \(0\) in the record graph of \([\mathbb{Z},0,X]\) in each of these phases.
We define the map \(R_{\perp}\) and prove that it preserves the measure of \([\mathbb{Z},0,X]\) when \(\mathbb{E}[X_0]=0\).

In Chapter \ref{chapter_v_prob}, we focus on general vertex-shifts on unimodular networks and study the existence and uniqueness of vertex-shift probabilities.
We provide sufficient conditions under which a vertex-shift on a random network has a unique vertex-shift probability.
We apply these results to compute the vertex-shift probability of unimodular Family Trees where the vertex-shift is given by the parent vertex-shift.

In Chapter \ref{chapter_record_V_shift_probability}, we apply the results of Chapter \ref{chapter_v_prob} to study the \(R\)-probability of \([\mathbb{Z},0,X]\), where \(R\) is the record vertex-shift and \(X\) is the i.i.d. sequence of increments of a ship-free to the left random walk.
We also provide constructions to generate the \(R\)-probability.

In Chapter \ref{chapter_stationary_seq}, we study the record graph of a stationary and ergodic sequence. 
We show that the component of \(0\) in the record graph exhibits a phase transition at \(0\) when the mean is varied between \(-\infty\) and \(+\infty\) (see Theorem \ref{thm_phase_transition_stationary}).
We provide a sufficient condition under which a unimodular ordered EFT can be represented as the component of \(0\) in the record graph of the network associated to a stationary sequence (Theorem \ref{thm_representation}).

\section*{ Introduction in French}

Considérez une règle de navigation définie sur un graphe qui associe à chaque sommet du graphe un sommet de telle sorte que la règle de navigation commute avec chaque automorphisme du graphe.
Cela signifie que la règle de navigation appliquée aux sommets reste la même après avoir pris n'importe quel automorphisme du graphe.
Une telle règle de navigation est dite avoir la propriété de covariance.
Une telle règle de navigation covariante indexée par des graphes localement finis et satisfaisant une condition de mesurabilité est appelée un décalage de sommet.
Plus généralement, on peut considérer des décalages de sommets sur des réseaux, des graphes qui ont des étiquettes sur les arêtes et sur les sommets.
Un décalage de sommet induit une dynamique sur l'espace des réseaux enracinés localement finis.
Nous nous intéressons à une règle de navigation, appelée décalage de sommet des records, définie sur les trajectoires de la marche aléatoire discrète à une dimension qui ne peut sauter qu'un pas vers la gauche.
La partie principale de la thèse est consacrée à l'étude du décalage de sommet des records sur ces marches aléatoires.
Quelques résultats sur le décalage de sommet des records sur des processus à accroissements stationnaires sont également obtenus.
La thèse contient également des résultats sur des décalages de sommets plus généraux sur des réseaux unimodulaires (voir Def. \ref{def_unimodular}).

Une marche aléatoire sur $\mathbb{Z}$ est dite sans saut vers la gauche (resp. droite) si ses accroissements i.i.d. prennent des valeurs dans $\{j \in \mathbb{Z}:j \geq -1\}$ (resp. $\{j \in \mathbb{Z}:j \leq 1\}$).

Considérons une suite i.i.d. de variables aléatoires $X= (X_n)_{n \in \mathbb{Z}}$, où $X_n$ prend des valeurs dans $\{j \in \mathbb{Z}:j \geq -1\}$ et avec une espérance finie pour tout \(n \in \mathbb{Z}\).
Une marche aléatoire bilatérale sans saut vers la gauche $(S_n)_{n \in \mathbb{Z}}$ sur $\mathbb{Z}$ associée à une suite d'accroissements $X$ est donnée par : $S_0=0$, $S_n = \sum_{k=0}^{n-1}X_k, \forall n \geq 0$, et $S_n = \sum_{k=n}^{-1}-X_k, \forall n \leq -1$.
Remarquez que $(S_n)_{n \geq 0}$ est sans saut vers la gauche tandis que $(S_{-n})_{n \geq 0}$ est sans saut vers la droite.
Construisez le graphe des records de \(X\) sur $\mathbb{Z}$ en traçant une arête orientée de chaque entier $i$ vers son record $R(i):=\min\{n>i:S_n \geq S_i\}$ si le minimum existe, sinon vers $i$ lui-même.
Le graphe des records est un graphe aléatoire orienté, avec \(\mathbb{Z}\) comme ensemble de sommets et où les arêtes orientées sont aléatoires.
Considérez la composante connexe de $0$ dans le graphe des records construit et enracinez-la en $0$.
Notez-la par $[T,0]$.
Ce travail est motivé par les questions suivantes :
\begin{enumerate}
\item  Quelle est la distribution du graphe enraciné aléatoire $[T,0]$ ?
\item Comment varie la distribution de $[T,0]$ avec l'espérance de l'accroissement $\mathbb{E}[X_{0}]$ ?
\item Quelle est la trajectoire de la marche aléatoire vue depuis le $n$-ième record de $0$ lorsque $n \to \infty$ ?
\item Existe-t-il une application non triviale qui préserve la mesure sur la trajectoire de la marche aléatoire qui est compatible avec l'application record $R$ (voir la discussion ci-dessous pour la définition de `compatibilité') ?
\end{enumerate}

En ce qui concerne la deuxième question, nous montrons que le graphe des records présente une transition de phase lorsque la moyenne $\mathbb{E}[X_{0}]$ varie entre $- \infty$ et $\infty$.
Il y a trois phases distinctes, et chaque phase a une caractéristique propre en termes de classification des arbres dirigés unimodulaires.

La première phase est observée dans la région $(-\infty,0)$, dans laquelle le graphe des records a un nombre infini de composantes connexes avec la composante de l'origine étant un arbre dirigé unimodulaire fini.
La deuxième phase se produit lorsque $\mathbb{E}[X_{0}]=0$.
Dans cette phase, le graphe des records est connexe, et c'est un arbre dirigé unimodulaire à une extrémité.
La troisième phase, qui se produit lorsque $\mathbb{E}[X_{0}] \in (0,\infty)$, a la caractéristique que le graphe des records est un arbre dirigé unimodulaire à deux extrémités.

Un arbre généalogique est un arbre dirigé enraciné dont le degré sortant est au plus $1$.
En négligeant les boucles, qui ne se produisent que si $\mathbb{E}[X_{0}]< 0$, nous obtenons que la composante connexe de $0$ dans le graphe des records est un arbre généalogique aléatoire $T$ enraciné en $0$.
Nous abordons la première question en calculant explicitement la distribution de l'arbre généalogique $[T,0]$ dans les trois régimes, à savoir lorsque $\mathbb{E}[X_{0}] \in (-\infty,0)$ (Théorème \ref{20230213191656}), $\mathbb{E}[X_{0}]=0$ (Théorème \ref{theorem:R-graph_egwt}), et $\mathbb{E}[X_{0}]\in (0,\infty)$ (Théorème \ref{r_graph_positive_drift_20230203174636}, Théorème \ref{20230226170801}).

Dans \cite{baccelliEternalFamilyTrees2018a}, les auteurs ont introduit la notion de décalages de sommets sur les graphes.
Informellement, un décalage de sommet est une collection de fonctions sur les graphes qui sont covariantes et satisfont une condition de mesurabilité.
Étant donné un décalage de sommet $f$, nous pouvons associer à chaque graphe $G$ son $f$-graphe $G^f$ dont les sommets sont les mêmes que ceux de $G$ et avec des arêtes dirigées données par $\{(x,f(x)):x \in V(G)\}$.
Le décalage de sommet $f$ définit une relation d'équivalence sur $V(G)$ en disant que deux sommets $u$ et $v$ sont équivalents s'il existe un entier positif $n$ tel que $f^n(u)=f^n(v)$.
La classe d'équivalence de \(u \in V(G)\) est appelée le feuillet de \(u\) (noté feuillet(\(u\))) et l'ensemble de tous les feuillets de \(V(G)\) est appelé la foliation de \(G\).
Comme la composante connexe $G_f(u)$ de tout sommet $u$ dans $G^f$ est l'ensemble \(\{v \in V(G): f^n(v)=f^m(u) \text{ pour certains } n \geq 1, m \geq 1\} \), nous avons feuillet\((u)\subseteq V(G_f(u))\).

Les auteurs ont montré dans \cite{baccelliEternalFamilyTrees2018a} que, étant donné un graphe aléatoire unimodulaire enraciné $[\mathbf{G}, \mathbf{o}]$, c.-à-d. la composante connexe de la racine dans le $f$-graphe $G^f$, est un arbre généalogique unimodulaire appartenant à l'une des trois classes : $\mathcal{F}/\mathcal{F}$, $\mathcal{I}/\mathcal{I}$, et $\mathcal{I}/\mathcal{F}$.
La classe $\mathcal{F}/\mathcal{F}$ correspond au ces où la composante de \(\mathbf{o}\) est finie et ses feuillets sont finis.
De même, les classes $\mathcal{I}/\mathcal{I}$ et $\mathcal{I}/\mathcal{F}$ correspondent au ces où la composante est infinie et tous les feuillets de la composante sont infinis, et où la composante est infinie et tous les feuillets de la composante sont finis, respectivement.
Il est facile de voir que le décalage record $R$ est un décalage de sommet, ce qui
implique que l'arbre généalogique $[T,0]$ est unimodulaire, où \(T\)  est la composante connexe de \(0\) dans le graphe des records.
Les trois phases du graphe des records $[T,0]$ pour \(\mathbb{E}[X_0]<0, \mathbb{E}[X_0]=0\) et \(\mathbb{E}[X_0]>0\) correspondent à ces trois classes \(\mathcal{F}/\mathcal{F}, \mathcal{I}/\mathcal{I}\) et \(\mathcal{I}/\mathcal{F}\) respectivement.
Ainsi, le décalage record est un exemple pour lequel les trois classes se produisent respectivement en fonction de la moyenne $\mathbb{E}[X_{0}]$.

En utilisant le cadre de \cite{aldousProcessesUnimodularRandom2007}, nous considérons les accroissements de la marche aléatoire sans saut comme des marques associées aux arêtes de $\mathbb{Z}$, vues depuis la racine $0$.
Plus précisément, nous considérons le réseau aléatoire enraciné $[\mathbb{Z},0,X]$, avec $X=(X_n)_{n \in \mathbb{Z}}$ étant la suite des accroissements et \(X_n\) étant le label assigné à l'arête \((n,n+1)\), pour tout  \(n \in \mathbb{Z}\).
Il est facile de voir que $[\mathbb{Z},0,X]$ est un réseau unimodulaire.
Par vue locale depuis le $n$-ième record ($n \geq 1$), nous entendons la distribution de $[\mathbb{Z},R^n(0),X]$.
Cette dernière est la même que celle de $[\mathbb{Z},0,R^n(0)X]$, où $R^n(0)X$ est la suite décalée \( (X_{m+R^n(0)})_{m \in \mathbb{Z}}\).
En utilisant l'unimodularité de $[\mathbb{Z},0,X]$, nous montrons que la vue locale depuis le $n$-ième record est absolument continue par rapport à $[\mathbb{Z},0,X]$ avec une dérivée de Radon-Nikodym $d_n(0)$, qui est définie comme le nombre d'entiers pour lesquels $0$ est le $n$-ième record.

De plus, en lien avec la troisième question, nous montrons que la vue locale vue depuis le $n$-ième record lorsque $n$ tend vers  $\infty$, $\lim_{n \to \infty}[\mathbb{Z},0,R^n(0)X]$, existe pour \(\mathbb{E}[X_0]<0, \mathbb{E}[X_0]=0\) et \(\mathbb{E}[X_0]>0\) respectivement.
Notons cette limite par $[\mathbb{Z},0,Y]$, où $Y=(Y_n)_{n \in \mathbb{Z}}$ est la marque sur les arêtes de $\mathbb{Z}$.
Nous observons que $Y$ n'est pas une suite i.i.d.
En particulier, sa partie négative $(Y_n)_{n<0}$ n'est pas une suite i.i.d.
Cependant, elle a une structure régénérative lorsque $\mathbb{E}[X_{0}]\geq 0$.

En ce qui concerne la quatrième question, une fonction $R'$ sur $\mathbb{Z}$ est dite \emph{préservant les feuillets} avec $R$ si elle envoie chaque feuillet sur lui-même, c'est-à-dire que pour tout $i \in \mathbb{Z}$, $R'(i)$ appartient au feuillet($i$) (le feuillet associé au décalage de sommet $R$ et $i$).

Lorsque \(\mathbb{E}[X_0] \geq 0\), nous définissons une fonction aléatoire $R_{\perp}$ qui envoie $i$ en
\[ R_{\perp}(i)=\inf\{k>i:R^n(i)=R^n(k) \text{ pour un certain } n \geq 1\}.\]
Nous montrons que $R_{\perp}$ préserve les feuillets de \(R\) et préserve la distribution de $[\mathbb{Z},0,X]$, c'est-à-dire que la distribution de \([\mathbb{Z},R_{\perp}(0),X]\) est la même que celle de \([\mathbb{Z},0,X]\).

Les statistiques des records des séries temporelles ont trouvé des applications en finance, hydrologie et physique.
Voir par exemple les références dans l'árticle de synthèse \cite{godrecheRecordStatisticsStrongly2017} pour une liste non exhaustive d'applications.
Ces études sur les records sont centrées sur des statistiques telles que la distribution du $n$-ième record, l'âge d'un record (c'est-à-dire l'écart entre deux records consécutifs), et l'âge du record le plus durable (\cite{wergenRecordsStochasticProcesses2013}, \cite{godrecheRecordStatisticsStrongly2017}).
Lorsque la série temporelle considérée est une suite i.i.d., il est montré dans \cite{stepanovCharacterizationTheoremWeak1994} que la moyenne du saut entre deux valeurs consécutives de record conditionnée à la valeur initiale du record peut caractériser la distribution de la suite i.i.d.
En revanche, dans le présent travail, l'accent est mis sur la structure locale conjointe des records des marches aléatoires sans saut vers la gauche.
La structure a une connexion naturelle avec la notion d'arbre généalogique qui justifiè de parler de `généalogie des records'.

La relation entre les marches aléatoires sans saut vers la gauche et les arbres de Galton-Watson a été largement discutée dans, par exemple, \cite{benniesRandomWalkApproach2000}, \cite{legallRandomTreesApplications2005b} et \cite{jimpitmanCombinatorialStochasticProcesses2006}.
En fait, les auteurs de \cite{benniesRandomWalkApproach2000} codent l'arbre de Galton-Watson critique et montrent que le codage donne une correspondance entre les records de la marche aléatoire et les ancêtres de la racine dans l'arbre de Galton-Watson.
Nous étendons ces idées pour définir l'application réciproque $\Phi_R$ qui prend un graphe des records et le renvoie à la trajectoire d'une marche aléatoire (voir Section \ref{subsec_backward_map}, Chapitre \ref{chapter_record_V_shift_probability}).
Cependant, nous ne sommes pas limités aux arbres de Galton-Watson finis ici.
À notre connaissance, cette description complète du graphe des records de la marche aléatoire sans saut est nouvelle.
Nous ne sommes pas limités au cas des marche aléatoires par ailleur.

Dans le chapitre \ref{chapter_prelim}, nous donnons les préliminaires et les énoncés des théorèmes que nous utilisons dans ce travail.
Nous introduisons les espaces probabilisés canoniques dans lesquels nos objets aléatoires prennent leurs valeurs et introduisons la dynamique induite par les décalages de sommets.
Nous nous concentrons sur l'étude des décalages de sommets sur les réseaux unimodulaires.
Nous énonçons un théorème clé que nous utilisons à cet égard, connu sous le nom de théorème de classification des feuillets dans les réseaux unimodulaires.
Nous introduisons également quelques opérations sur les arbres généalogiques unimodulaires.

Dans le chapitre \ref{chapter_record_v_shift}, nous introduisons le décalage de sommet des records.
Nous étudions ses propriétés sur les suites déterministes à valeurs entières.
Nous étudions les propriétés du graphe des records associé aux suites.
En utilisant ces propriétés, nous prouvons que le graphe des records de \([\mathbb{Z},0,X]\) présente des transitions de phase, où \(X\) est la suite i.i.d. des accroissements d'une marche aléatoire sans saut vers la gauche.
Nous décrivons également la distribution de la composante de (0) dans le graphe des records de \([\mathbb{Z},0,X]\) dans chacune de ces phases.
Nous définissons la fonction \(R_{\perp}\) et prouvons qu'elle préserve la loi de \([\mathbb{Z},0,X]\) lorsque \(\mathbb{E}[X_0]=0\).

\hspace{0.5em} Dans le chapitre \ref{chapter_v_prob}, nous nous intéressons aux décalages de sommets généraux sur les réseaux unimodulaires et nous étudions l'existence et l'unicité des probabilités de décalage de sommet.
Nous fournissons des conditions suffisantes sous lesquelles un décalage de sommet sur un réseau aléatoire a une probabilité de décalage de sommet unique.
Nous appliquons ces résultats pour calculer la probabilité de décalage de sommet des arbres généalogiques unimodulaires où le décalage de sommet est donné par le décalage de sommet parent.

Dans le chapitre \ref{chapter_record_V_shift_probability}, nous appliquons les résultats du chapitre \ref{chapter_v_prob} pour étudier la \(R\)-probabilité de \([\mathbb{Z},0,X]\), où \(R\) est le décalage record de sommet et \(X\) est la suite i.i.d. des accroissements d'une marche aléatoire sans saut vers la gauche.
Nous donnons également des constructions pour générer la \(R\)-probabilité.

Dans le chapitre \ref{chapter_stationary_seq}, nous étudions le graphe des records d'une suite stationnaire et ergodique.
Nous montrons que la composante de \(0\) dans le graphe des records présente une transition de phase en \(0\) lorsque la moyenne varie entre \(-\infty\) et \(+\infty\) (voir le théorème \ref{thm_phase_transition_stationary}).
Nous fournissons une condition suffisante sous laquelle un EFT ordonné unimodulaire peut être représenté comme la composante de \(0\) dans le graphe des records du réseau associé à une suite stationnaire (théorème \ref{thm_representation}).

\chapter{Preliminaries} \label{chapter_prelim}

In this chapter, we summarize the framework of networks given in \cite{aldousProcessesUnimodularRandom2007} and in \cite{baccelliEternalFamilyTrees2018a}.

\section{Canonical probability spaces for networks}\label{sec_canonical_prob}

A graph \(G\) with a vertex set \(V=V(G)\) and an (undirected) edge set \(E=E(G)\) is denoted by \(G=(V,E)\).
A \textbf{network} is a graph \(G=(V,E)\) together with a Polish space \(\Xi\) and two maps: one from \(V\) to \(\Xi\) and the other from \(E\) to \(\Xi\).
The space \(\Xi\) is called mark space and the value assigned to a vertex (resp. an edge) is called \textbf{mark of the vertex} (resp. edge).
Each edge \(e=\{x,y\}\) in the network has two marks associated with its endpoints, one assigned to \(\{e,x\}\) and the other assigned to \(\{e,y\}\) .
A directed edge can be obtained from an undirected edge by assigning \(0\) as mark to one of its endpoint and \(1\) as mark to the other, with the understanding that the orientation of the directed edge is from the endpoint with mark \(0\) to the endpoint with mark \(1\).
A graph can be seen as a network with a fixed mark.
In this work, we focus on the networks that are connected and locally finite, i.e., degree of every vertex is finite, and whose mark space is \(\mathbb{Z}\).

An  \textbf{isomorphism} \(\phi\) of networks \(G_1\) and \(G_2\), denoted as \(\phi:G_1 \to G_2\), is a bijective map \(\phi_V:V(G_1) \to V(G_2)\) satisfying the three conditions: (1) for every pair of vertices \(u,v \in V(G_1)\), \((u,v)\) is an edge of \(G_1\) if and only if \((\phi_V(u),\phi_v(v))\) is an edge of \(G_2\), (2) for all \(u \in V(G_1)\), mark of \(u\) is the same as that of \(\phi_V(u)\), (3) for all \(e=(u,v) \in E(G_1)\), mark of \(e\) is same as that of \((\phi_V(u),\phi_V(v))\).
In particular, both \(G_1,G_2\) should have the same mark space.

A \textbf{rooted network} \((G,o)\) is a network \(G\) with a distinguished vertex \(o \in V(G)\).
For any \(r\geq 0\), let \((G,o)_r\) be the subgraph induced by the vertices \(\{u \in V(G):d_{G}(u,o) \leq r\}\), where \(d_{G}\) denotes the graph distance in \(G\), and for any vertex \(v \in V((G,o)_r)\),  let \(((G,o)_r,v)\) be the rooted network obtained by rooting \((G,o)_r\) at \(v\).
The marks of \((G,o)_r\) are induced by the marks of \((G,o)\).
A \textbf{rooted-isomorphism} \(\phi\) of two rooted  networks \((G_1,o_1),(G_2,o_2)\) is an isomorphism  \(\phi:G_1\to G_2\) of networks \(G_1,G_2\) with the additional condition that \(\phi_V(o_1)=\phi_V(o_2)\).
Two rooted networks \((G_1,o_1)\) and \((G_2,o_2)\) are said to be \textbf{equivalent}, denoted as \((G_1,o_1) \sim (G_2,o_2)\), if and only if there is a rooted-isomorphism from one to the other.
An equivalent class of \((G_1,o_1)\) is denoted by \([G_1,o_1]\).
We do not distinguish between a rooted graph and its equivalent class unless explicitly stated.

A doubly rooted network \((G,u,v)\) is a network \(G\) with two distinguished vertices \(u,v \in V(G)\).
For any \(r\geq 0\), let \((G,u,v)_r\) be the subgraph induced by the vertices \(\{w \in V(G): d_{G}(w,u) \wedge d_G(w,v) \leq r\}\) with marks of edges and vertices induced from that of \(G\), where \(d_{G}\) denotes the graph distance in \(G\) and \(\wedge\) denotes the minimum function.
A doubly rooted-isomorphism \(\phi\) of two doubly rooted networks \((G_1,u_1,v_1)\) and \((G_2,u_2,v_2)\) is an isomorphism \(\phi:G_1\to G_2\) of networks \(G_1,G_2\) such that \(\phi_V(u_1)=u_2\) and \(\phi_V(v_1)=v_2\).
Two doubly rooted networks \((G_1,u_1,v_1)\) and \((G_2,u_2,v_2)\) are said to be equivalent, denoted as \((G_1,u_1,v_1) \sim (G_2,u_2,v_2)\), if and only if there exists a doubly rooted-isomorphism between them.

We denote the \textbf{space of equivalence class of rooted locally finite and connected networks} by \(\mathcal{G}_*\) and the \textbf{space of equivalence class of doubly rooted locally finite and connected networks} by \(\mathcal{G}_{**}\).

The space \(\mathcal{G}_*\) can be equipped with a metric in the following way: distance between two rooted networks \((G_1,o_1)\) and \((G_2,o_2)\) is defined as \(2^{-d_{1,2}}\), where \(d_{1,2}\) is the supremum of \(r \geq 0\) such that there exists a network isomorphism \(\phi_r:(G_1,o_1)_r \to (G_2,o_2)_r\) and that for every \(u \in  V(G_1), (v,w) \in E(G_1)\), the distance between marks of \(u\) and of \(\phi(u)\) is at most \(\frac{1}{r}\), and the distance between mark of \((v,w)\) and of \((\phi_r(v),\phi_r(w))\) is at most \(\frac{1}{r}\).
The space \(\mathcal{G}_{**}\) can also be equipped with a metric in the similar way as in \(\mathcal{G}_*\) by defining the distance between two doubly rooted networks \((G_1,u_1,v_1)\) and \((G_2,u_2,v_2)\) as \(2^{-d_{1,2,}}\), where \(d_{1,2}\) is the same as in \(\mathcal{G}_*\) but applied to \((G_1,u_1,v_1)_r\), \((G_2,u_2,v_2)_r\) instead of \((G_1,o_1)_r\), \((G_2,o_2)_r\)  respectively.
The topology generated on \(\mathcal{G}_*\) with the above  metric is called the local topology.

Both the spaces \(\mathcal{G}_*\) and \(\mathcal{G}_{**}\) endowed with their respective metrics are Polish spaces.
Let us denote the space of equivalence classes of graphs by \(\hat{\mathcal{G}}_*\)

In this work, we often encounter with a special networks called Family Trees.
A \textbf{Family Tree} is a locally finite directed tree that has the property that the out-degree of every vertex is at most $1$.
If the out-degree of every vertex is equal to $1$ a.s. then it is called an \textbf{Eternal Family Tree} (EFT).
A vertex $u$ of a Family Tree is called the \textbf{parent} of a vertex $v$ if there is a directed edge from $v$ to $u$. 
In this case, $v$ is called a \textbf{child} of $u$.
 An \textbf{ordered Family Tree} is a Family Tree in which children of every vertex are ordered.
Although, we use similar notation for ordered and unordered, one can distinguish them depending on the context.
A rooted (ordered) Family Tree is a pair $(T,o)$, where $T$ is an (ordered) Family Tree and $o$ is a distinguished vertex of $T$.
The space of equivalence classes of rooted (ordered) Family Trees is denoted by $\mathcal{T}_*$, where the equivalence relation is given by rooted (ordered) isomorphisms, i.e., two rooted Family Trees are equivalent if there is a rooted isomorphism between them that preserves the order.
The space $\mathcal{T}_*$ is Polish with respect to local topology \cite{baccelliEternalFamilyTrees2018a}.
By forgetting the order of an ordered Family Tree, one obtains an unordered Family Tree.
We often call a rooted Family Tree, a Family Tree.

We will also encounter marked Family Trees whose vertices have \(\mathbb{Z}\)-valued marks.
Denote the space of equivalence classes of such (rooted) marked Family Trees by \(\hat{\mathcal{T}}_*\).

A random (rooted) network is a measurable map from a probability space to \(\mathcal{G}_*\).
We use the bold fonts to denote a random network and normal fonts to denote a realization of a  random network.

We use the following theorem to prove the convergence of random networks.
For any rooted network \((G,o)\) and \(r \geq 0\), let \([G,o]_r\) denote the equivalent class \((G,o)_r\).
\begin{theorem}[\cite{nicolascurienRandomGraphs}]
    A sequence of random networks \(([\mathbf{G}_n,\mathbf{o}_n])_{n \geq 1}\) converges in distribution to a random network \([\mathbf{G},\mathbf{o}]\) as \(n \to \infty\) if and only if for every \(r \geq 0\), the sequence of random networks \(([\mathbf{G}_n,\mathbf{o}_n]_r)_{n \geq 1}\) converges in distribution as \(n \to \infty\).
\end{theorem}
\subsection{Relative compactness and tightness} \label{subsec_rel_compact}
We recall the definitions of relative compactness and tightness.
This subsection is concerned only about graphs (not networks).
A subset $\mathcal{A}$ of a topological space is said to be \textbf{relatively compact} if every sequence in $\mathcal{A}$ has a convergent subsequence.
A subset $\mathcal{A}$ of probability measures on a measurable space is said to be \textbf{tight} if for any $\epsilon>0$ there exists a compact set $K$ such that $\forall \calp \in \mathcal{A}$, $\calp[K] \geq 1-\epsilon$.
In particular, if there exists an increasing sequence of compact sets $\{K_m\}_{m \in \mathbb{N}}$ such that
$$\lim_{m \rightarrow \infty}\sup_{\calp \in \mathcal{A}}\calp[K_m^c] =0,$$
then $\mathcal{A}$ is tight.
Prokhorov's theorem states that if $\mathcal{A}$ is tight then $\mathcal{A}$ is relatively compact and if the measurable space is complete and separable, then relative compactness is equivalent to tightness, i.e., $\mathcal{A}$ is relatively compact if and only if $\mathcal{A}$ is tight.

For any $r \ge 0$, we define the restriction map $|_r:\calg_* \rightarrow \calg_*$ taking $\go$ to $\go_r$ by $[G,o]_r = [(G,o)|_r,o]$ for any $[G,o] \in \calg_*$.
The image of this restriction map $\calg_*|_r$ is countable and does not have any accumulation point.
So, every compact subset of $\calg_*|_r$ is finite.
Let $K_m^r$ be those rooted graphs of $\calg_*|_r$ that have at most $m$ vertices, i.e., $K_m^r:= \{\go \in \calg_*|_r \text{ such that } |V(G)|\leq m\}$ which is a finite set.
Then, the sequence $\{K_m^r\}_{m \in \mathbb{N}}$ is an increasing sequence of compact sets.

We state a lemma that gives an equivalent condition for a sequence of random graphs to be tight in terms of their neighborhood restrictions.
\begin{lemma}\label{lm:curien}
  A family of random graphs $\{[\textbf{G}_i,\textbf{o}_i], i \in I\}$ is tight if and only if for all $r \geq 1$, the family $\{[\textbf{G}_i,\textbf{o}_i]_r, i \in I\}$ is tight.
\end{lemma}
\begin{proof}
  Refer \cite{nicolascurienRandomGraphs}.
\end{proof}

Therefore, from the above discussion and Lemma \ref{lm:curien}, we conclude that a family of random graphs $\{[\textbf{G}_i,\textbf{o}_i], i \in I\}$ is relatively compact if and only if for any $r \geq 1$
$$\lim_{m \rightarrow \infty}\sup_{i \in I}\mathbb{P}[[\textbf{G}_i,\textbf{o}_i]_r \not \in K^r_m] = 0.$$
We use this conclusion to prove the relative compactness of vertex-shift-probabilities when the vertex-shift has finite orbits.

\section{Dynamics on networks}
The dynamics we look at are navigation rules known as vertex-shifts on random networks.
The definitions and notations of this section are taken from \cite{baccelliEternalFamilyTrees2018a}.
\begin{definition}
    A \textbf{vertex-shift} $f$ is a collection of maps indexed by networks $f_G:V(G)\rightarrow V(G)$, where $G$ is a locally finite network, satisfying the following properties:
      \begin{enumerate}
      \item  \textit{Covariance:} For any graph isomorphism $\sigma:G \rightarrow H$ of two graphs $G$ and $H$, $f$ satisfies $f_H \circ \sigma = \sigma \circ f_G$,
      \item \textit{Measurability:} The map  $[G,u,v] \mapsto \mathbf{1}_{\{f_G(u)=v\}}$ from the doubly rooted space $\mathcal{G}_{**}$ is measurable. 
    \end{enumerate}
\end{definition}

The first condition allows us to define a (deterministic) dynamic on the space \(\mathcal{G}_*\) in the following way: for each vertex-shift \(f\), define a map \(\theta_f\) by
\begin{align} \label{eq_theta_f}
    \theta_f:\mathcal{G}_* & \longrightarrow \mathcal{G}_* \nonumber \\ 
                                [G,o]& \mapsto    [G,f(o)].
\end{align}
The map \(\theta_f\) is well-defined.

The second condition allows us to define the following functions.
For any non-negative measurable function \(h:\mathcal{G}_{**} \to \mathbb{R}_{\geq 0}\), we define two measurable functions \(h^+:\mathcal{G}_* \to \mathbb{R}_{\geq 0}\) and \(h^-: \mathcal{G}_* \to \mathbb{R}_{\geq 0}\) by 
\begin{align*}
    h^+([G,o]) &= \sum_{u \in V(G)} h([G,o,u]) \mathbf{1}\{f_G(u)=o\}\\
    h^-([G,o])  &= \sum_{u \in V(G)}h([G,u,o])\mathbf{1}\{f_G(o)=u\},
\end{align*}
for all \([G,o] \in \mathcal{G}_*\).

\subsection{Notations}
Given a vertex-shift $f$ and a network $G$, the $f$-graph of $G$ is the network $G^f$ with vertex set $V(G^f):=V(G)$ and directed edges $E(G^f)=\{(u,f_G(u)):u \in V(G)\text{ and } f_G(u)\not = u\}$.
Let \(G^f(u)\) denote the component of \(u\), for any vertex \(u \in V(G)\).
Define the map $\Psi_f:\mathcal{G}_* \rightarrow \mathcal{T}_*$ that maps a rooted network $[G,o]$ to  $[G^f(o),o]$, where \(G^f(o)\) is obtained by discarding all self-loops (if there are any) of the connected component of $o$ in its $f$-graph $G^f$.

For every vertex-shift \(f\) on a network \(G\), the components of the $f$-graph of $G$ are Family Trees \cite{baccelliEternalFamilyTrees2018a}. 
For any integer $n \geq 0$, a vertex of the form \(f^n_G(u)\) is called the \textbf{ancestor of order \(n\) of \(u\)} (also called \(n\)-th ancestor of \(u\)).
We denote the set of \textbf{$n$-th descendants} (also called descendants of order \(n\)) of vertex $u \in V(G)$ by $D_n(u):= \{v \in V(G)\backslash \{u\}: f_G^n(v)=u\}$ and its cardinality by $d_n(u):=\#D_n(u)$, for any \(n \geq 1\).
We follow the convention that \(f^0(u)=D_0(u)=u\).
Denote the set of all descendants of $u$ by $D(u):=\{v \in V(G)\backslash \{u\}: f_G^n(v)=u \text{ for some $n>0$}\}\cup \{u\}$ and its cardinality by $d(u)=\#D(u)$.
For any vertex \(u \in G\), its ancestor of order \(1\) is called the \textbf{parent of \(u\)}, its descendants of order \(1\) are called \textbf{children of \(u\)}, and the set \(\{v \in V(G)\backslash \{u\}:f_G(v)=f_G(u)\}\) is called the \textbf{set of siblings of \(u\)}.

A vertex-shift $f$ defines the following equivalence relation on the vertices of any network $G$:
two vertices $u$ and $v$ of $G$ are equivalent if and only if $f_G^n(u) = f_G^n(v)$ for some $n \geq 0$.
The equivalence class of $u$ is called the \textbf{$f$-foil} of $u$, and the collection of $f$-foils of $G$ is called the \textbf{$f$-foliation} of $G$.

A \textbf{covariant subset} is a collection of subsets \(S_G\) indexed by network \(G\), where \(S_G \subseteq V(G)\) and the collection satisfies the following two conditions, (1) for every network isomorphism \(\alpha:G \to H\), we have \(\alpha(S_G)=S_H\), (2) the map \([G,o] \mapsto \mathbf{1}_{\{o \in S_G\}}\) is measurable.
A \textbf{covariant partition} is a collection of partitions \(\Pi_G\) indexed by network \(G\), where \(\Pi_G\) is a partition of \(V(G)\) and the collection satisfies the following two conditions, (1) for every network isomorphism \(\alpha:G \to H\), we have \(\alpha \circ \Pi_{G}=\Pi_H\), (2) the subset \(\{[G,o,u]:u \in \Pi_G(o)\}\) is a measurable subset.

Examples of covariant partitions can be constructed from any vertex-shift \(f\).
Given \(f\), the collection of \(f\)-foils of \(G\) indexed by \(G\) is a covariant partition.
Similarly, the collection of subsets \(S_G = \{u \in V(G): f_G(u)=u\}\) of \(V(G)\) indexed by network \(G\) is an example of a covariant subset.

\begin{example}
    For every Family Tree \(T\), consider the map \(F_T:V(T) \to V(T)\) that maps every vertex \(u\) to its parent if it exists, otherwise maps to itself.
    The vertex-shift \(F\) is called the \textbf{parent  vertex-shift}.
\end{example}

For every vertex-shift \(f\) on \(G\), the $f$-graph of $G$ is a Family Tree \cite{baccelliEternalFamilyTrees2018a}.
Observe that the terminologies of parent and child are consistent with respect to the parent vertex-shift.
The dynamic naturally induced by the parent vertex-shift is given by \(\theta_F:\mathcal{T}_* \to \mathcal{T}_*\) given by \(\theta_F([T,o]) = [T,F(o)]\).

Given a Family Tree \(T\),  let \(l(\cdot,\cdot)=l_T(\cdot,\cdot):V(T) \times V(T) \to \mathbb{Z}\) be the function defined by the following  equations: for every \(v,w \in V(T)\),
\begin{align*}
    l(v,v)=0,  \quad l(v,F(w))=l(v,w)-1,\\
    l(w,v) = - l(v,w),\\
    l(v,w)+l(w,z) = l(v,z).
\end{align*}
For a rooted Family Tree \([T,o]\), the set \(\{w \in V(T):l(o,w)=n\}\) is  the \textbf{\(n\)-th generation} of \([T,o]\).

\subsection{Royal Line of Succession order}\label{subsec_RLS}
 Let $T$ be an ordered Family Tree, i.e., the children of every vertex are totally ordered, with an order denoted by $<$.
We can extend this order on children to a total order on the vertices of $T$ in the following way.
Let $u,v$ be two vertices of $T$.
We say that $v$ has precedence over $u$, denoted as $v \succ u$ (or $u \prec v$), if either $F^k(u)=v$ for some \(k>0\), or there exists a common ancestor $w= F^m(u)=F^n(v)$ of $u$ and $v$ with $m,n$ being the smallest positive integers with this property, and $F^{m-1}(u)<F^{n-1}(v)$.
In the latter case, both $F^{m-1}(u)$ and $F^{n-1}(v)$ can be compared since they are children of $w$.

\begin{remark} \label{remark:rls_order}
  For two vertices $u,v$ with common ancestor $w$, distinct from $u$ and $v$, $u \prec v$ if and only if $y_u \prec y_v $ for every descendant $y_u$ of $u$ and for every descendant $y_v$ of $v$.
\end{remark}
We call the total order $\prec$ the {\bf Royal Line of Succession (RLS) order} associated with \(<\).
It is also known as {\bf Depth First Search order}.

\section{Unimodular random networks and their properties}

\begin{definition}\label{def_unimodular}
    A random network \([\mathbf{G},\mathbf{o}]\) is said to be \textbf{unimodular} if it satisfies the \textbf{Mass Transport Principle}:
    \begin{equation}
        \mathbb{E}\left[\sum_{u \in V(\mathbf{G})}f([\mathbf{G},\mathbf{o},u])\right] = \mathbb{E}\left[\sum_{u \in V(\mathbf{G})}f([\mathbf{G},u,\mathbf{o}])\right],
    \end{equation}
    for every non-negative measurable function \(f:\mathcal{G}_{**}\to \mathbb{R}_{\geq 0}\).
\end{definition}

The following theorem suggests that there are a wide class of unimodular random networks.
For any finite network \(G\), we can associate a random rooted network \([G,\mathbf{o}]\) by choosing the root uniformly among the vertices of \(G\), i.e., for any measurable set \(A \subset \mathcal{G}_*\), we have
\begin{equation}\label{eq_unif_chosen_root}
    \mathbb{P}[[G,\mathbf{o}] \in A] = \frac{1}{\#V(G)}\sum_{u \in V(G)}\mathbf{1}_A([G,u]). 
\end{equation}
It can be easily shown that \([G,\mathbf{o}]\) is unimodular \cite{nicolascurienRandomGraphs}.
In fact every random finite network can be obtained by uniformly choosing the root \cite{nicolascurienRandomGraphs}.

\begin{theorem}[\cite{aldousProcessesUnimodularRandom2007}]
    Let \([\mathbf{G}_n,\mathbf{o}_n]\) be a sequence of unimodular networks that converge in distribution to a random network \([\mathbf{G},\mathbf{o}]\).
    Then, \([\mathbf{G},\mathbf{o}]\) is unimodular.
\end{theorem}
The above theorem implies that if \((G_n)_{n \geq 1}\) is a sequence of finite deterministic networks such that the sequence of random networks \(([G_n,\mathbf{o}_n])_{n \geq 1}\), obtained by applying Eq. (\ref{eq_unif_chosen_root}), converges in distribution to a random network \([\mathbf{G},\mathbf{o}]\), then \([\mathbf{G},\mathbf{o}]\) is unimodular.
In the above, we can replace deterministic finite network by random finite network.
The network \([\mathbf{G},\mathbf{o}]\) is said to be obtained by taking \textbf{Benjamini-Schramm} limit of \((G_n)_{n \geq 1}\).
    
\begin{definition}[Network associated to a sequence]
    Let \(x=(x_n)_{n \in \mathbb{Z}}\) be a sequence of real numbers.
    The network associated to \(x\) is the network \((\mathbb{Z},x)\) with vertex set \(\mathbb{Z}\), with edge set \(\{(n,n+1):n \in \mathbb{Z}\}\), and with \(x_n\) as mark of the edge \((n,n+1)\) for all \(n \in \mathbb{Z}\).
\end{definition}
See Figure \ref{fig:network_seq} for an illustration.

\begin{figure}[h]
    \begin{center}
      \includegraphics[scale=1]{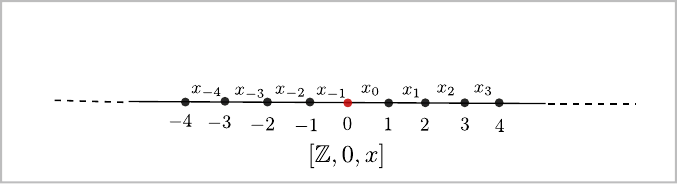}
    \end{center}
     \caption{The network \((\mathbb{Z},x)\) associated to a sequence \(x=(x_n)_{n \in \mathbb{Z}}\). The rooted network \([\mathbb{Z},0,x]\) denotes the equivalence class of the network \((\mathbb{Z},x)\) rooted at \(0\).}
     \label{fig:network_seq}
 \end{figure}

\begin{example}[of a unimodular network]\label{ex_network_sequence}
    Let \(X=(X_n)_{n \in \mathbb{Z}}\) be a stationary sequence.
    The network \([\mathbb{Z},0,X]\) associated to the sequence \(X\) is unimodular.
\end{example}
The network \([\mathbb{Z},0,X]\) in Example \ref{ex_network_sequence} is unimodular because the network can be shown to be the Benjamini-Schramm limit of the sequence \((G_n,X^{(n)})_{n \geq 1}\), where \((G_n,X^{(n)})\) is the network with vertex set \(\{-n,-n+1,\ldots,0,1,\ldots,n\}\), edge set \(\{(k,k+1): -n \leq k \leq n-1\}\) with mark \(X_k\) assigned to the edge \((k,k+1)\) for all \(-n \leq k \leq n-1 \).

\subsection{Vertex-shifts on unimodular networks}
Recall the map \(\theta_f\), in Eq. (\ref{eq_theta_f}), associated to a vertex-shift \(f\).
The following proposition will be used in the later chapters while studying measure preserving vertex-shifts.
\begin{proposition}[Mecke-Thorisson Point Stationarity \cite{baccelliEternalFamilyTrees2018a}]\label{prop_mecke}
    Let \(f\) be  a vertex-shift and \([\mathbf{G},\mathbf{o}]\) be a unimodular measure.
    Then, the map \(\theta_f\) preserves the distribution of \([\mathbf{G},\mathbf{o}]\) if and only if \(f\) is  a.s. bijective.
\end{proposition}

The following two propositions will be used to show in the later chapters that a unimodular ordered Family Tree does not have a smallest nor a largest vertex.

\begin{proposition}[\cite{baccelliEternalFamilyTrees2018a}]\label{prop:injective_bijective}
    Let \([\mathbf{G},\mathbf{o}]\) be unimodular and \(f\) be a vertex-shift.
    \begin{enumerate}
        \item If \(f\) is injective a.s., then \(f\) is bijective a.s.
        \item If \(f\) is surjective a.s., then \(f\)is bijective a.s.
    \end{enumerate}
\end{proposition}

\begin{lemma}[No Infinite/Finite Inclusion \cite{baccelliEternalFamilyTrees2018a}]\label{lemma:no_infinite_finite}
    Let \([\mathbf{G},\mathbf{o}]\) be a unimodular network, \(\mathbf{\mathfrak{S}}\) be a covariant subset and \(\mathbf{\Pi}\) be a covariant partition.
    Almost surely, there is no infinite element \(E\) of \(\mathbf{\Pi}_{\mathbf{G}}\) such that \(E \cap \mathbf{\mathfrak{S}}_{\mathbf{G}}\) is finite and non-empty.
\end{lemma}
The following lemma shows that probability that covariant subset is non-empty is related to the probability that the root belongs to this covariant subset.

\begin{lemma}[\cite{baccelliEternalFamilyTrees2018a}]\label{lemma_non_empty_covariant_set}
    Let \([\mathbf{G},\mathbf{o}]\) be a unimodular network, \(\mathfrak{S}\) be a covariant subset.
    Then, \(\mathbb{P}[\mathfrak{S}_{\mathbf{G}} \text{ is non-empty}]>0\) if and only if \(\mathbb{P}[\mathbf{o} \in \mathfrak{S}_{\mathbf{G}}]>0\).
\end{lemma}
\begin{proof}
    Take \(f([G,u,v]) = \mathbf{1}\{v \in \mathfrak{S}_{G}\}\) and apply the Mass Transport Principle.
\end{proof}

The following theorem gives a classification of the components of \(f\)-graph, for any vertex-shift \(f\) on a unimodular network.

\begin{theorem}[Foil Classification in Unimodular Networks \cite{baccelliEternalFamilyTrees2018a}]
    \label{thm_foil_classification}
    Let $\mathbf{[G, o]}$ be a unimodular network and $f$ be a vertex-shift. Almost surely, every vertex has finite degree in the graph $G^f$. In addition, each component $C$ of $G^f$ has at most two ends and it belongs to one of the following three classes: 
    
    \begin{itemize}
    \item Class $\mathcal{F}/\mathcal{F}$: $C$ and all its foils are finite. If $n = n(C)$ is the number of foils in $C$, then
    \begin{itemize}
    \item $C$ has a unique $f$-cycle and its length is $n$;
    \item $f_{G}^{\infty}(C)$ is the set of vertices of the cycle;
    \item Each foil of $C$ contains exactly one vertex of the cycle.
    \end{itemize}
    \item Class $\mathcal{I}/\mathcal{F}$: $C$ is infinite but all its foils are finite. In this case,
    \begin{itemize}
    \item The (undirected) $f$-graph on $C$ is a tree;
    \item There is a unique bi-infinite $f$-path in $C$, each foil in $C$ contain exactly one vertex of the path, and $f_{G}^{\infty}(C)$ coincides with the set of vertices of the path;
    \item The order of the foils of $C$ is of type $\mathbb{Z}$; that is, there is no youngest foil in $C$.
    \end{itemize}
    \item Class $\mathcal{I}/\mathcal{I}$: $C$ and all foils of $C$ are infinite. In this case,
    \begin{itemize}
    \item  The (undirected) $f$-graph on $C$ is a tree;
    \item $C$ has one end, there is no bi-infinite f-path in $C$, $D(v)$ is finite for every
    vertex $v \in C$, and  $f_{G}^{\infty}(C)=\emptyset$;
    \item The order of the foils of $C$ is of type $\mathbb{N}$ or $\mathbb{Z}$, that is, there may or may
    not be a youngest foil in $C$.
    \end{itemize}
    \end{itemize}
    \end{theorem}
    
We will see examples from each class in the next chapter.
The above theorem when applied to Family Trees gives the following theorem, which will be used in the later chapters.
Recall the notation that \(F\) denotes the parent vertex-shift.
\begin{theorem}[Classification of unimodular Family Trees \cite{baccelliEternalFamilyTrees2018a}]
    A unimodular Family Tree almost surely belongs to one of the following three classes:
    \begin{enumerate}
        \item Class \(\mathcal{F}/\mathcal{F}\): the tree is finite.
        \item Class \(\mathcal{I}/\mathcal{I}\): every generation is infinite, each vertex has finitely many descendants, and there is no bi-infinite path, i.e., the tree is one-ended.
        \item Class \(\mathcal{I}/\mathcal{F}\): every generation is finite and the set of vertices that have infinitely many descendants form a unique bi-infinite \(F\)-path, i.e., the tree has two ends.
    \end{enumerate}
\end{theorem}

The following lemma gives a recipe to construct unimodular Family Trees.
Recall that for a vertex-shift \(f\) and a network \(G\), \(G^f\) denotes the \(f\)-graph of \(G\) and for any vertex \(u \in V(G)\), \(G^f(u)\) denotes the component of \(u\) in \(G^f\).
\begin{lemma}[\cite{baccelliEternalFamilyTrees2018a}]\label{lemma:f_graph_unimodular}
    Let \(f\) be a vertex-shift and \([\mathbf{G}, \mathbf{o}]\) be a random rooted network.
    If \([\mathbf{G}, \mathbf{o}]\) is unimodular then the Family Tree \([\mathbf{G}^f(\mathbf{o}),\mathbf{o}]\) is unimodular.
\end{lemma}

\subsection{Unimodular ordered Eternal Galton-Watson Tree}\label{subsec_EGWT}
We give an example of unimodular Family Tree of class \(\mathcal{I}/\mathcal{I}\) originally introduced in \cite{baccelliEternalFamilyTrees2018a}.

Let $\pi$ be a probability distribution on $\mathbb{N}\cup \{0\}$ with mean $m(\pi)=1$, and $\hat{\pi}$ be its size-biased distribution defined by $\hat{\pi}(k) = k \pi(k), \forall k \geq 0$.

The ordered unimodular Eternal Galton-Watson Tree with offspring distribution $\pi$, $EGWT(\pi)$, is a rooted ordered Eternal Family Tree $[\mathbf{T},\mathbf{o}]$ with the following properties. 
The root $\mathbf{o}$ and all of its descendants reproduce independently with the common offspring distribution \(\pi\).
For all $n \geq 1$, the ancestor $F^n(\mathbf{o})$ of $\mathbf{o}$ reproduces with distribution $\hat{\pi}$. 
The descendants of $F^n(\mathbf{o})$ which are not the descendants of $F^{n-1}(\mathbf{o})$ (and not $F^{n-1}(\mathbf{o})$) reproduce independently with the common offspring distribution \(\pi\), with the convention $F^0(\mathbf{o})=\mathbf{o}$.
All individuals (vertices) reproduce independently.
The order among the children of every vertex is uniform.

In particular, except for the vertices $F^n(\mathbf{o}),\, n\geq 1$, which reproduce independently with distribution $\hat{\pi}$, all the remaining vertices of the tree reproduce independently with distribution $\pi$. See Figure \ref{fi:EGWT} for an illustration. 
It is proved in \cite{baccelliEternalFamilyTrees2018a} that \(EGWT(\pi)\) is unimodular if and only if the mean \(m(\pi)\) is \(1\).
\begin{figure}[htbp] 
  \centering 
  \includegraphics[scale=0.8]{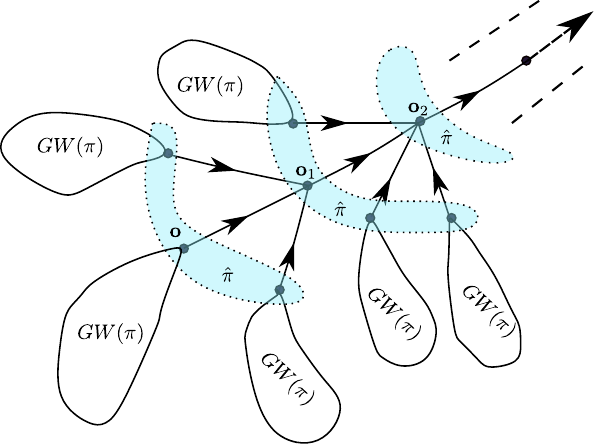}
  \caption{The Eternal Galton-Watson Tree with offspring distribution $\pi$. Notation: $GW(\pi)$ denotes the Galton-Watson Tree, $\hat{\pi}$ is the size-biased distribution of $\pi$. The order among the children of every vertex is from the left to the right.}
  \label{fi:EGWT}
\end{figure}

\begin{remark}\label{remark_supp_EGWT}
 Let $[\mathbf{T},\mathbf{o}]$ be an $EGWT(\pi)$ with mean $m(\pi)=1$. Then, the descendant trees of all vertices, except for the ancestors of the root $\mathbf{o}$, are independent  critical Galton-Watson trees with offspring distribution $\pi$.
 Therefore, the descendant trees are finite.
 Hence, a realization of $EGWT(\pi)$ has the property that the descendant tree of every vertex of it is finite.
 So, it is of class \(\mathcal{I}/\mathcal{I}\).
\end{remark}

The following theorem characterizes a unimodular \(EGWT\).

\begin{theorem}[\cite{baccelliEternalFamilyTrees2018a}]\label{theorem:characterisation_egwt}
    A unimodular Eternal Family Tree \([\mathbf{T},\mathbf{o}]\) is an \(EGWT\) if and only if the number of children of the root \(d_1(\mathbf{o})\) is independent of the non-descendant tree \(D^c(\mathbf{o})\), i.e., the subtree induced by \((V(\mathbf{T})\backslash D(\mathbf{o}))\cup \{\mathbf{o}\}\).
\end{theorem}

\subsection{Unimodular Family Trees of class $\mathcal{I}/\mathcal{F}$}

In the following, we describe a construction given in \cite{baccelliEternalFamilyTrees2018a} used to generate unimodular Family Trees of class \(\mathcal{I}/\mathcal{F}\).

Let $([\mathbf{T}_i,\mathbf{o}_i])_{i \in \mathbb{Z}}$ be a stationary sequence of finite random rooted Family Trees.
Add a directed edge $(\mathbf{o}_i,\mathbf{o}_{i+1})$ for all $i \in \mathbb{Z}$, and denote this new Eternal Family Tree (EFT) by $\mathbf{T}'$ and
let $\mathbf{o}'=\mathbf{o}_0$.
The EFT $[\mathbf{T}',\mathbf{o}']$ is called the{ \bf joining} of $([\mathbf{T}_i,\mathbf{o}_i])_{i \in \mathbb{Z}}$.

In addition, suppose that $\mathbb{E}[\#V(\mathbf{T}_0)]< \infty$.
Choose the root uniformly in $V(\mathbf{T}_0)$ and bias the measure by the mean size of $V(\mathbf{T}_0)$.
Formally, define a measure $\sigma'$ by:
for any measurable set $A$, 
\begin{equation}\label{eq:size_biased_uniform_root}
  \sigma'[A] = \frac{1}{\mathbb{E}[\# V(\mathbf{T}_0)]} \mathbb{E}\displaystyle\left[\sum_{v \in V(\mathbf{T}_0)} \mathbf{1}_{[\mathbf{T}',v]}(A) \right].
\end{equation}
An operator similar to \(\sigma'\) is defined in \cite{aldousAsymptoticFringeDistributions1991}.
An EFT  $[\mathbf{T},\mathbf{o}]$ whose distribution is  $\sigma'$ is called the {\bf typically rooted joining} of $([\mathbf{T}_i,\mathbf{o}_i])_{i \in \mathbb{Z}}$.
Note that the assumption \(\mathbb{E}[\#V(\mathbf{T}_0)]< \infty\) is not needed in the joining operation, however, it is essential in order to define the typically rooted joining operation.
We state the following results taken from \cite{baccelliEternalFamilyTrees2018a}.

\begin{theorem}[\cite{baccelliEternalFamilyTrees2018a}]\label{thm:I_F_unimodularizable}
  Let $[\mathbf{T},\mathbf{o}]$ be the typically rooted joining of a stationary sequence $([\mathbf{T}_i,\mathbf{o}_i])_{i \in \mathbb{Z}}$, where $\mathbb{E}[\# V(\mathbf{T}_0)]< \infty$.
Then, $[\mathbf{T},\mathbf{o}]$ is unimodular.
\end{theorem}
\begin{theorem}[\cite{baccelliEternalFamilyTrees2018a}]\label{thm_eft_I_f_joining}
  Let $[\mathbf{T},\mathbf{o}]$ be a unimodular EFT of class $\mathcal{I}/\mathcal{F}$ a.s., and $[\mathbf{T}',\mathbf{o}']$ be the Family Tree obtained by conditioning $[\mathbf{T},\mathbf{o}]$ on the event that $\mathbf{o}$ belongs to the bi-infinite path of $\mathbf{T}$.
  Then, $[\mathbf{T}',\mathbf{o}']$ is the joining of some stationary sequence of finite Family Trees $([\mathbf{T}_i,\mathbf{o}_i])_{i \in \mathbb{Z}}$ and $[\mathbf{T},\mathbf{o}]$ is the typically rooted joining of  $([\mathbf{T}_i,\mathbf{o}_i])_{i \in \mathbb{Z}}$.
\end{theorem}

\subsection{Typical re-rooting operation}
\begin{definition}
    Let \([\mathbf{T},\mathbf{o}]\) be a random rooted finite Family Tree such that \(\mathbb{E}[\#V(\mathbf{T})]<\infty\).
    A random rooted finite Family Tree \([\mathbf{T}',\mathbf{o}']\) is said to be obtained by \textbf{re-rooting to a typical vertex of} \([\mathbf{T},\mathbf{o}]\) if 
    \begin{equation} \label{eq_typical_rerooting}
        \mathbb{P}[[\mathbf{T}',\mathbf{o}'] \in A] = \frac{1}{\mathbb{E}[\#V(\mathbf{T})]} \mathbb{E}\left[\sum_{u \in V(\mathbf{T})} \mathbf{1}_A([\mathbf{T},u])\right],
    \end{equation}
    for every measurable subset \(A\) of \(\mathcal{T}_*\).
\end{definition}


\chapter{Record vertex-shift for random walk case}\label{chapter_record_v_shift}

In this chapter, we define the record vertex-shift and study its \(f\)-graph called record graph.
The record vertex-shift is defined on the networks of the form \((\mathbb{Z},x)\), where \(x=(x_n)_{n \in \mathbb{Z}}\) is a sequence of integers in \(\{-1,0,1,2,\cdots\}\).
The record vertex-shift is defined by the record map defined as follows: for each integer \(i\), look at the sums \(\sum_{k=i}^{n-1}x_k\), starting from \(i\), and assign \(i\) to the smallest \(n>i\) for which  \(\sum_{k=i}^{n-1}x_k\) is non-negative.
The randomness of the network is introduced when  we take i.i.d. sequence \(X=(X_n)\) with a common distribution on \(\{-1,0,1,\cdots\}\) such that \(0<\mathbb{P}[X_0=-1]<1\) and \(-\infty<\mathbb{E}[X_0]<\infty\).
The \(f\)-graph of the record vertex-shift on \((\mathbb{Z},X)\) is called the record graph.
The properties of the record map and the nature of the random walk influence the distribution of the component of \(0\) of the record graph.
These properties are studied in Section \ref{section_record_map}.
Using these properties, we show, in Section \ref{section_classification}, that the component of \(0\) in the record graph exhibits phase transitions at \(0\) when the mean of increment \(\mathbb{E}[X_0]\) is varied between \(-\infty\) and \(\infty\).
These three phases of the component of \(0\) corresponding to the region \((-\infty,0)\), at \(0\), and the region \((0, \infty)\) are namely Typically rooted Galton-Watson Tree (\(TGWT\)), unimodular Eternal Galton-Watson Tree \(EGWT\) and the unimodularised bi-variate Eternal Kesten Tree \(MEKT\) respectively.
The Family Tree \(EGWT\)  was introduced in \cite{baccelliEternalFamilyTrees2018a}.
We introduce the \(TGWT\) and \(MEKT\) in Subsection \ref{sec_examples}.
All of these Family Trees are unimodular.

{\bf (In French)}
Dans ce chapitre, nous définissons le décalage de sommet des records et étudions son \(f\)-graphe appelé graphe des records.
Le décalage de sommet des records est défini sur les réseaux de la forme \((\mathbb{Z},x)\), où \(x=(x_n)_{n \in \mathbb{Z}}\) est une suite d'entiers dans \(\{-1,0,1,2,\cdots\}\).
Le décalage de sommet des records est défini par la fonction record définie comme suit : pour chaque entier \(i\), on regarde les sommes \(\sum_{k=i}^{n-1}x_k\), en partant de \(i\), et on assigne \(i\) au plus petit \(n>i\) pour lequel  \(\sum_{k=i}^{n-1}x_k\) est non négatif.
L'aléa du réseau est introduit lorsque nous prenons une suite i.i.d. \(X=(X_n)\) avec une loi commune sur \(\{-1,0,1,\cdots\}\) telle que \(0<\mathbb{P}[X_0=-1]<1\) et \(-\infty<\mathbb{E}[X_0]<\infty\).
Le \(f\)-graphe du décalage de sommet des records sur \((\mathbb{Z},X)\) est appelé le graphe des records.
Les propriétés de la fonction record et la nature de la marche aléatoire influencent la distribution de la composante de \(0\) du graphe des record.
Ces propriétés sont étudiées dans la section \ref{section_record_map}.
En utilisant ces propriétés, nous montrons, dans la section \ref{section_classification}, que la composante de \(0\) dans le graphe des records présente des transitions de phase en \(0\) lorsque la moyenne des accroissements \(\mathbb{E}[X_0]\) varie entre \(-\infty\) et \(\infty\).
Ces trois phases de la composante de \(0\) correspondant à la région \((-\infty,0)\), en \(0\), et à la région \((0, \infty)\) sont respectivement l'arbre généalogique typiquement enraciné de Galton-Watson (\(TGWT\)), l'arbre généalogique éternel unimodulaire de Galton-Watson (\(EGWT\)) et l'arbre éternel bivarié de Kesten unimodularisé (\(MEKT\)).
L'arbre généalogique \(EGWT\)  a été introduit dans \cite{baccelliEternalFamilyTrees2018a}.
Nous introduisons le \(TGWT\) et le \(MEKT\) dans la sous-section \ref{sec_examples}.
Tous ces arbres généalogiques sont unimodulaires.

\section{Record map}\label{section_record_map}

\noindent {\bf Notations:}
 Given an integer valued sequence $x=(x_n)_{n \in \mathbb{Z}}$ with \(x_n \geq -1\) for all \(n \in \mathbb{Z}\), let 
\begin{equation}
  y(n,j):= \sum_{l=n}^{j-1}x_l,
\end{equation}
for all integers $n\leq j$, with the convention that $\sum_n^{n-1}\equiv 0$.

For any integer \(i\), the sequence \((-y(n,i))_{n<i}\) is called the sequence of sums in the past of \(i\) seen from \(i\), whereas, the sequence \((y(i,n))_{n\geq i}\) is called the sums in the future of \(i\) seen from \(i\).
Similarly, the sequence \((n,-y(n,i))_{n<i}\) is called the trajectory in past of \(i\) seen from \(i\) and the sequence \((n,y(i,n))_{n\geq i}\) is called the trajectory in the future of \(i\) seen from \(i\).

Let \(s^{(i)}\) denote the sequence \((s^{(i)}_n)_{n\geq 0}\), where \(s^{(i)}_n = y(i-n,i)\) for all \(n\geq 0\).
Let \(s=(s_n)_{n\in \mathbb{Z}}\) be the sequence of locations associated to \(x\), where \(s_0=0\), \(s_n = -y(n,0)\), for all \(n<0\), and \(s_n = y(0,n)\) for all \(n >0\).
In particular, we have \(s^{(0)}_n = -s_{-n}\) for all \(n \geq 0\).

\begin{definition}\label{hyp:increments}
  A sequence of i.i.d. random variables $X=(X_n)_{n \in \mathbb{Z}}$ with $X_n$ taking values in $\mathbb{N}\cup\{-1\}$ for all \(n \in \mathbb{Z}\), such that $\mathbb{P}[X_0=-1]>0$ and $\mathbb{E}[X_0]<\infty$, are called i.i.d. increments of a skip-free to the left random walk.
\end{definition}

We use capital letters to denote random variables.
For instance, the sums associated to \(X=(X_n)_{n \in \mathbb{Z}}\) are denoted by \(S=(S_n)_{n \in \mathbb{Z}}\), and \(S^{(i)}=(S^{(i)}_n)_{n \geq 0}\) respectively.

\begin{remark} \label{remark_one_step_negative}
Since \(x_k \geq -1\) for all \(k \in \mathbb{Z}\), an obvious observation is that if \(y(n,j)=m<0\) for any two integers \(n<j\), then there exist integers \(n=i_{-m+1}<i_{-m}<\cdots<i_2<i_1=j\) such that \(y(i_{l+1},i_l)=-1\) for all \(1 \leq l \leq -m\).
\end{remark}

Define the {\bf record map} $R_{x}:\mathbb{Z} \rightarrow \mathbb{Z}$ to be the function that maps any $n \in \mathbb{Z}$ to 

\begin{equation*}
  R_{x}(n) = \begin{cases}
                                  \inf\{j>n: y(n,j) \geq \max\{y(n,l):n \leq l \leq j-1\}\} \text{ if $\inf$ exists}\\
                                  n \text{ otherwise},
                                   \end{cases}
\end{equation*}
with the convention that $y(n,n)=0$. 
Equivalently, we obtain

\begin{equation*}
  R_{x}(n) = \begin{cases}
                                \inf\{j>n:y(n,j) \geq 0\} \text{ if the infimum exists}\\
                                n \text{ otherwise}.
                                 \end{cases}
\end{equation*}
Denote the $k$-th iterate of the record map by $R^k_x$.

Given a real valued sequence $x=(x_n)_{n \in \mathbb{Z}}$, we consider its associated network $(\mathbb{Z},x)$ with label $x_i$ on the edge $(i,i+1)$, for all $i \in \mathbb{Z}$ (see \cite{aldousProcessesUnimodularRandom2007} for the terminology on networks).
For the random variables $X = (X_n)_{n \in \mathbb{Z}}$, we consider the random rooted network $[\mathbb{Z},0,X]$ associated to \(X\) (with mark $X$ and where the network $(\mathbb{Z},X)$ is rooted at $0$).
Note that $[\mathbb{Z},0,X]$ is unimodular since \(X\) is stationary \cite{aldousProcessesUnimodularRandom2007}.

Consider networks of the form $(\mathbb{Z},i,x)$, where $i \in \mathbb{Z}$ is the root and the mark $x=(x_n)_{n \in \mathbb{Z}}$ is a real-valued sequence which assigns label $x_n$ to edge $(n,n+1)$ for each $n \in \mathbb{Z}$.
We know that any order preserving graph isomorphism of $\mathbb{Z}$ is a shift map $T_i$, for some integer $i$, namely $T_i(n)=n+i,\ n \in  \mathbb{Z}$.
The map \(T_i\)  acts on a sequence \(x=(x_n)_{n \in \mathbb{Z}}\) by \(T_i x = (x_{n-i})_{n \in \mathbb{Z}}\).
So, any network $(\mathbb{Z},i,x)$ is root-isomorphic to $(\mathbb{Z},0,T_{-i}x)$.

Consider the collection of networks $\{(\mathbb{Z},x):x = (x_n)_{n \in \mathbb{Z}}, x_n \in \mathbb{Z}_{\geq -1}\}$.
Then, by taking the record map $R$ on this collection and the identity map on the rest of the networks, we obtain a vertex-shift.
Indeed, $R$ satisfies the first condition, i.e., $R$ is covariant for all $i$ because
\begin{align*}
  R_{T_ix}(k+i) &= \begin{cases} \inf\{ j > k+i: \sum_{l=k+i}^{j-1} x_{l-i} \geq 0\} \text{ if infimum is defined},\\ 
    k+i \text{ otherwise} \end{cases}\\
  &= \begin{cases} i + \inf\{j>k: \sum_{l=k}^{j-1} x_{l}\}  \text{ if infimum is defined} \\
     i+ k \text{ otherwise}\end{cases}\\
  &= i+ R_{x}(k).
\end{align*}

The record map $R$ also satisfies the measurability condition since the map $[\mathbb{Z},i,j,x] \mapsto 1_{R_x(i)=j}$ is a function of $(x_i,x_{i+1},\cdots,x_{j-1})$.

\subsection{Properties of the record map}

Let $x:=(x_n)_{n \in \mathbb{Z}}$ be an arbitrary integer-valued sequence with \(x_n \geq -1, \, \forall n \in \mathbb{Z}\), and let $(\mathbb{Z},x)$ be its associated network with $x_n$ as the mark on the edge $(n,n+1)$ for $n \in \mathbb{Z}$. 
Define the function $L: \mathbb{R}^{\mathbb{Z}} \times \mathbb{Z}\rightarrow \mathbb{Z} \cup \{-\infty\}$ in the following way:
\begin{equation}\label{eq:L_x_defn}
  L_{x}(i) = \begin{cases}
    \inf\{j<i: y(k,i) \geq 0, \forall j \leq k < i\} \text{, if $\inf$ exists;}\\
    i \text{ otherwise}.
    \end{cases}
\end{equation}

The integer \(L_x(i)\) is the largest integer (if it exists) up to which the sums in the past of \(i\) and seen from \(i\) are non-positive, equivalently, \(L_x(i)\) is the largest integer up to which the trajectory in the past of \(i\) seen from \(i\) stays below level \(1\).
Note that $-\infty \leq L_x(i) <i$ if and only if the sum $-y(j,i) \leq 0$ for all integers \(j\) such that $L_x(i) \leq j<i$.
It could happen that \(L_x(i)=i\).
This is possible only if $x_{i-1}=-1$.

For the sequence \(x\) and any integer \(i \in \mathbb{Z}\), the set \(D_1(i):=\{j<i: R_x(j)=i\}\) is called the set of children of \(i\) and \(D(i):=\{j<i: R^n_x(j)=i\) for some \(n>0\}\) is  called the set of descendants of \(i\).
The following lemma shows that the set of descendants of any integer $i$ is the set of integers that lie between $L_x(i)$ and $i$ (including $L_x(i)$ and excluding \(i\)).

\begin{lemma}\label{lemma:descendants}\label{lemma_descendants}
  Let $x = (x_n)_{n \in \mathbb{Z}}$ be an integer-valued sequence with \(x_n \geq -1, \, \forall n \in \mathbb{Z}\), and $R$ be the record vertex-shift on the network \((\mathbb{Z},x)\). 
The set of descendants of any integer $i$ is given by
  \begin{equation*}
    D(i) = \{j \in \mathbb{Z}: L_{x}(i) \leq j <i\}.
  \end{equation*}
\end{lemma}

\begin{proof}
We first prove that $ \{j \in \mathbb{Z}: L_{x}(i) \leq j <i\}  \subset D(i)$.
If $L_x(i)=i$, then there is nothing to prove.
So, assume that $L_{x}(i)<i$ (note that $L_{x}(i)$ can be $-\infty$). 
Consider any integer $j$ such that $L_{x}(i) \leq j < i$.
Since \(y(j,i) \geq 0\), we have,
$$
j< R_{x}(j) := \inf\{k>j: y(j,k) \geq 0\} \leq i.
$$
\noindent By iteratively applying the same argument to \(R^l(j)\) for each \(l>0\), we find the smallest non-negative integer \(m\) such that \(R^m_x(j)=i\).
Such an \(m\) exists as there are only finitely many integers between \(j\) and \(i\).
This implies that $j$ is a descendant of $i$. 
Therefore, $\{j:L_{x}(i) \leq j < i\} \subseteq D(i)$.

We now prove that $D(i) =  \{j \in \mathbb{Z}: L_{x}(i) \leq j <i\}$, i.e., if $-\infty< j<L_x(i)$, then \(j\)  is not a descendant of \(i\).
For any $-\infty<j < L_{x}(i)$, we \emph{claim} that, either $R_{x}(j)< L_{x}(i)$  or $R_{x}(j)>i$.
Assume that the claim is true.
If \(R_x(j)<L_x(i)\), then by applying the claim again to \(R_x(j)\), we obtain that, either \(R^2_x(j)<L_x(i)\) or \(R^2_x(j)>i\).
Iterate this process several times until we find the largest non-negative integer \(n\) such that $R_x^n(j)<L_x(i)$.
Such an \(n\) exists as there are only finitely many integers between \(j\) and \(L_x(i)\).
As \(n\) is the largest integer satisfying this condition, by applying the claim to $R_x^n(j)$, we obtain that $R_x^{n+1}(j)>i$.
This implies that none of the descendants of $R_x^n(j)$ (including $j$) is a descendant of $i$, which completes the proof.

\noindent We now prove the last \emph{claim}.
Let $k := L_{x}(i)-1$ and let \(j\) be as in the claim.
Assume that $i \geq R_{x}(j) \geq L_x(i)$, i.e., $i \geq R_x(j)>k$.
Then, $y(k,i)<0$, by the definition of $L_x(i)$. 
Since $R_x(j)>k$, we have $y(j,k) \leq 0$ (equality holds only if $j=k$).
So, $y(j,i) = y(j,k)+y(k,i)<0$.
Therefore, for  any $L_{x}(i) \leq m\leq i$, $y(j,m) = y(j,i)-y(m,i) <0$, since $y(m,i) \geq 0$.
So, $R_{x}(j)>i$ and the claim is proved.
\end{proof}

\begin{lemma}[Interval property]\label{lemma:record_descendants}
  Let $x = (x_n)_{n \in \mathbb{Z}}$ be an integer-valued sequence with \(x_n \geq -1, \, \forall n \in \mathbb{Z}\), and $R$ be the record vertex-shift on the network \((\mathbb{Z},x)\).
  If $i \in \mathbb{Z}$ satisfies $R_{x}(i)>i$, then the set of descendants of $R_x(i)$ contains $\{j:i \leq j < R_{x}(i)\}$.
  There could be more descendants of $R_x(i)$   that are less than $i$.
\end{lemma}
\begin{proof}
  By Lemma \ref{lemma:descendants}, it is enough to show that $L_x(R_x(i))\leq i$.
  For any integer $m$ such that $i <m < R_x(i)$, since $y(i,m)<0$ and $y(i,R_x(i)) \geq 0$, we have
\[
y(m,R_x(i)) = y(i,R_x(i))- y(i,m) \geq 0.
\]
Thus, $L_{x}(R_{x}(i)) \leq i$.
\end{proof}

\begin{remark}\label{remark_intervalProperty_integerValued}
  Although, Lemma \ref{lemma_descendants} and Lemma \ref{lemma:record_descendants} are proved under the assumption that \(x=(x_n)_{n \in \mathbb{Z}}\) is an integer-valued sequence satisfying the condition \(x_n \geq -1 \ \forall n \in \mathbb{Z}\), the condition does not play any role in the proofs of these two lemmas.
  So the statements are indeed valid for any integer-valued sequence.
\end{remark}

\subsection{Properties of the record map on skip-free to the left random walk}

The following two lemmas focus on computing the probability of certain events of skip-free random walks.
They are used in the later sections for describing distribution of the record graph associated to skip-free random walks.
Although, these computations can be found in \cite{nicolascurienRandomWalksGraphs}, we prove them here for the sake of completeness and notational consistency.

We use the following notation in this and the remaining sections thereafter.
Let $X=(X_n)_{n \in \mathbb{Z}}$ be i.i.d. increments of a skip-free to the left random walk, and $[\mathbb{Z},0,X]$ be its associated network.
Let $(S_n)_{n \in \mathbb{Z}}$ be the random walk associated to the increment sequence $X$, i.e., $S_0=0$, $S_n = S_0+\sum_{i=0}^{n-1}X_i$ for $n \geq 1$ and $S_n = S_0+\sum_{i=n}^{-1}-X_i$ for $n \leq -1$.
Clearly, $(S_n)_{n \geq 0}$ is a skip-free to the left random walk, and $(S_{-n})_{n \geq 0}$ is a skip-free to the right random walk.
For any $j \in \mathbb{Z}$, let 
\begin{equation}\label{eq_eta_j}
  \eta_j:=\begin{cases}
    \inf\{n\geq 0: S_n = j\} \text{ if } S_n=j \text{ for some } n\geq 0\\
    \infty \text{ otherwise.}
  \end{cases}
\end{equation}
For an event $A$, let $\mathbb{P}_k[A]$ denote the probability of the event $A$ when the random walk starts at $k \in \mathbb{Z}$.
Let $\mathbb{P}:=\mathbb{P}_0$, and $c := \mathbb{P}[\eta_{-1} < \infty]$.
 
\begin{lemma}  \label{20230118184651}\label{hitting_time_20230118184651}
    For $0 \leq j \leq k$, $\mathbb{P}[\eta_{j-k}< \infty]= \mathbb{P}_k[\eta_j < \infty]= c^{k-j}$.
\end{lemma}
\begin{proof}
    The first equality (in the statement) is obvious since the probability is invariant to shift of the starting point of the random walk.
    The second equality follows because of the skip free property which prohibits the random walk from taking jumps smaller than $-1$.
    We prove it by induction on $k-j$.
    If $j=k$, then $\mathbb{P}_j[\eta_j<\infty]=1$.
    So, assume that $0 \leq j <k$.
    By induction, we have \(\mathbb{P}_{k'}[\eta_{j'}< \infty] = c^{k'-j'}\) for all $0 \leq j' \leq k'$ such that $k'-j'<k-j$. 
    Then,
    \begin{align*}
        \mathbb{P}_k[\eta_j< \infty] = \mathbb{P}_k[\eta_{k-1}< \infty]\mathbb{P}_{k-1}[\eta_j < \infty] = \mathbb{P}[\eta_{-1}< \infty]c^{k-1-j} = c^{k-j}.
    \end{align*}
    The second equality follows by applying the inductive statement to $k'=k-1$, and $j'=j$.
\end{proof}

Let $\tau:= \inf\{n>0: S_n \geq 0\}$, with \(\tau= \infty\) if \(S_n<0 \quad \forall n>0\), be the weak upper record (epoch), $S_{\tau}$ be the weak upper record height, and $X_{\tau-1}$ be its last increment (both are defined to be arbitrary when \(\tau = \infty\)).

\begin{lemma}[\cite{benniesRandomWalkApproach2000}]\label{20230119141633}
     For all integers $j,k$ such that $0 \leq j \leq k$, 
    \[\mathbb{P}[\tau<\infty,S_{\tau}=j,X_{\tau-1}=k] = \mathbb{P}[X_0=k] c^{k-j}.\]
\end{lemma}
\begin{proof}
    Note the equality of the following events:
    \[\{S_{\tau}=X_{\tau-1},\tau<\infty\} = \{\tau = 1\} = \{S_1 \geq 0\}.\]
    Therefore, for $0 \leq j=k$, we have $\mathbb{P}[S_{\tau}=j,X_{\tau-1}=j,\tau< \infty] = \mathbb{P}[S_1 = j]$.

    So, for $0 \leq j <k$,
    \begin{align*}
        \mathbb{P}[\tau<\infty,S_{\tau}=j,X_{\tau-1}=k] &= \sum_{n=2}^{\infty}\mathbb{P}[\tau=n,S_n=j,X_{n-1}=k]\\
        &=\sum_{n=2}^{\infty} \mathbb{P}[S_1<0,\cdots,S_{n-1}<0,S_n=j,S_n-S_{n-1}=k].
    \end{align*}
The sum in the first equation starts from $n=2$ since we assumed that $j<k$ (\(\{\tau=1\}\) occurs if and only \(j=k\) which is already covered).
By the duality principle (also known as reflection principle), the last expression is equal to
\[\sum_{n=2}^{\infty}\mathbb{P}[S_n-S_{n-1}<0, S_n - S_{n-2}<0, \cdots,S_n-S_1<0, S_n=j,S_n-(S_n-S_1)=k].\]
So, we obtain:
\begin{align*}
  \mathbb{P}[\tau<\infty,S_{\tau}=j,X_{\tau-1}=k]  &= \sum_{n=2}^{\infty}\mathbb{P}[S_n=j, S_{n-1}>j,S_{n-2}>j, \cdots,S_1>j,S_1=k]\\
  &=\sum_{n=2}^{\infty} \mathbb{P}[S_1=k] \mathbb{P}_k[S_1>j,\cdots,S_{n-2}>j,S_{n-1}=j]\\
        &=\mathbb{P}[S_1=k] \sum_{n=2}^{\infty}\mathbb{P}_k[\eta_{j}=n-1]\\
        &= \mathbb{P}[S_1=k] \mathbb{P}_k[0<\eta_j< \infty]\\
        &= \mathbb{P}[S_1=k]c^{k-j} = \mathbb{P}[X_0=k]c^{k-j}.
\end{align*}
The second equation follows by the Markov property.
The second to last equation follows from Lemma \ref{20230118184651} and the fact that $\mathbb{P}[\eta_{j-k}=0]=\delta_j(k)$, where $\delta$ is the Dirac function.
\end{proof}

\begin{remark}\normalfont
    If the random walk has positive drift, i.e., $0< \mathbb{E}{X_0}< \infty$, then $S_n \to \infty$ a.s. as $n \to \infty$.
    So, \(\tau<\infty\) a.s..
    Therefore, $\mathbb{P}[\tau< \infty, S_{\tau}=j, X_{\tau-1}=k]= \mathbb{P}[S_{\tau}=j,X_{\tau-1}=k]$ for $0 \leq j \leq k$.

    \noindent If the random walk has negative drift, i.e., $\mathbb{E}[X_0]<0$, then $S_n \to -\infty$ a.s. as $n \to \infty$.
    Therefore, $c=1$, which implies that $\mathbb{P}[\tau<\infty,S_{\tau}=j,X_{\tau-1}=k] = \mathbb{P}[X_0=k]$, and 
    \[\mathbb{P}[\tau< \infty, X_{\tau-1}=k] = \sum_{j=0}^{k}\mathbb{P}[\tau<\infty,S_{\tau}=j,X_{\tau-1}=k] = (k+1) \mathbb{P}[X_0=k]. \]
  \end{remark}
 Let us assume that \(\mathbb{E}[X_0]>0\).
  The following lemma shows that the random walk conditioned to hit \(-1\) is also a random walk with a different increment distribution.
  Recall that \(\eta_{-1}\) is the first time the random walk \((S_n)_{n \geq 0}\) hits \(-1\) (see Eq. (\ref{eq_eta_j})). 
  Let \(\bar{S}\) has the conditional distribution of \((S_{n\wedge \eta_{-1}})_{n \geq 0}\) given that \(\eta_{-1}< \infty\).
  Let \(p_k:= \mathbb{P}[X_0=k],\, \forall k\geq -1\).
  Observe that \(\sum_{k=-1}^{\infty}p_kc^{k+1} = c\), since the sum on the left can be written as 
  \begin{align*}
    \sum_{k=-1}^{\infty}p_kc^{k+1} &= \sum_{k=-1}^{\infty} \mathbb{P}[X_0=k]c^{k+1}\\
    &= \sum_{k=-1}^{\infty} \mathbb{P}[X_0=k]\mathbb{P}_{k}[\eta_{-1}< \infty]\\
    &= \mathbb{P}_0[\eta_{-1}<\infty]=c.
  \end{align*}
  Therefore, we have 
  \begin{equation}\label{eq_conditioned_valid}
    \sum_{k=-1}^{\infty}p_kc^k=1.
  \end{equation}
 Let \(\hat{X}=(\hat{X}_n)_{n \geq 0}\) be i.i.d. sequence of a skip-free to the left random walk whose common distribution is given by: 
  \begin{equation}\label{eq_conditioned_rw}
    \mathbb{P}[\hat{X}_0=k] = \mathbb{P}[X_0=k]c^k=p_kc^k, \, \forall k \geq -1,
  \end{equation}
  which is a valid distribution by Eq. (\ref{eq_conditioned_valid}).
  Let \((\hat{S}_n)_{n \geq 0}\) be the random walk whose increment sequence is \(\hat{X}\), i.e., \(\hat{S}_n = \sum_{k=0}^{n-1}\hat{X}_k\) for all \(k \geq 0\), and \(\hat{S} = (\hat{S}_{n \wedge \hat{\eta}_{-1}})_{n\geq 0}\) be the stopped random walk.

  \begin{remark}
    In fact, the process \((\hat{S}_n)_{n \geq 0}\) is the Doob \(h\)-transform of \(p\), where \(p\) is the distribution of \(X_0\) and \(h\) is the harmonic function defined on \(\mathbb{Z}_{\geq 0}\) by \(h(i)=\mathbb{P}_i[\eta_{-1}< \infty]\) for all \(i \geq 0\).
  \end{remark}

  \begin{lemma}\label{20230113184856}
   Let \(\mathbb{E}[X_0]>0\).
   Then the process \(\bar{S}\) has the same  distribution as the stopped skip-free to the left random walk \(\hat{S}\).
  \end{lemma}
  \begin{proof}
    Let $(x_0,x_1,\cdots,x_n)$ be the first \(n+1\) increments of a sample of the process with $x_k \geq -1, \forall k \leq n$, $x_1+\cdots+x_m \geq 0, \forall m<n$, and $x_1+\cdots+x_n \geq -1$.
    Then, for all such \(n\) and \(x\),
  
   \begin{align*}
      &\mathbb{P}[S_0=0,S_1=x_1, S_2-S_1=x_2,\cdots,S_{(n \wedge \eta_{-1})}-S_{(n \wedge \eta_{-1}-1)}=x_n, \eta_{-1} \geq n|\eta_{-1}< \infty] \\
      &= \frac{\mathbb{P}[S_1=x_1]\mathbb{P}_{x_1}[S_1 = x_2,S_2-S_1=x_3,\cdots, S_{n-1}-S_{n-2}=x_{n},\eta_{-1}<\infty]}{\mathbb{P}[\eta_{-1}< \infty]}\\
      &= \frac{\left(\prod_{i=1}^n \mathbb{P}[S_1=x_i]\right) \mathbb{P}_{x_1+\cdots+x_n}[\eta_{-1}< \infty]}{\mathbb{P}[\eta_{-1}< \infty]}= \frac{\left(\prod_{i=1}^n \mathbb{P}[S_1=x_i]\right) c^{x_1+x_2+\cdots+x_n+1}}{c}\\
     &= \prod_{i=1}^n (\mathbb{P}[X_0=x_i]c^{x_i})\\
     &=\mathbb{P}[\hat{S}_1=x_1,\hat{S}_2-\hat{S}_1 = x_2, \cdots, \hat{S}_{(n \wedge \eta_{-1})}- \hat{S}_{(n \wedge \eta_{-1}-1)}=x_n, \hat{\eta}_{-1} \geq n].
    \end{align*}
      Thus, the conditioned random walk is a stopped random walk, stopped at the stopping time \(\hat{\eta}_{-1}\), whose increments are distributed as $\mathbb{P}{[S_1=j]}c^j, \forall j \geq -1$.
  \end{proof}
    
  We now show that the random walk \((\hat{S}_n)_{n \geq 0}\) has negative drift.
  
  \begin{lemma}\label{20230116135048}
    Assume \(\mathbb{E}[X_0]>0\). Then the mean of \(\hat{X}_0\) is negative.
  \end{lemma}
  \begin{proof}  
    Consider the function $\phi(x)=\sum_{k=0}^{\infty}(k+1)x^kp_k -1$ on the interval $[0,1]$.
    We have 
    \begin{align*}
      \phi(c) &= \sum_{k=0}^{\infty} (k+1)c^k p_k- 1 = \sum_{k=0}^{\infty} kc^k p_k +\left(\sum_{k=0}^{\infty}c^kp_k \right)- 1\\
      &= \sum_{k=0}^{\infty}kc^kp_k + 1- \frac{p_{-1}}{c} -1= \sum_{k=-1}^{\infty}kc^kp_k= \mathbb{E}[\hat{X}].
    \end{align*}
    The function $\phi$ has the following properties on $[0,1]$:
    \begin{itemize}
      \item $\phi(0)<0$: Since $\phi(0)=p_0-1<0$.
      \item $\phi(1)>0$: Since $\phi(1)= \sum_{k=0}^{\infty}(k+1)p_k-1 = \sum_{k=0}^{\infty}kp_k + 1-p_{-1} - 1 = \mathbb{E}[X]>0$.
      \item $\phi$ is strictly increasing since $\phi'(x)>0, \forall x \in (0,1)$.
    \end{itemize}
    Therefore, there exists a unique $c' \in (0,1)$ such that $\phi(c')=0$.
    The proof is complete if we show that $c<c'$, as this implies that $\phi(c)<\phi(c')=0$ by monotonicity of $\phi$, and the fact that $\phi(c)=\mathbb{E}[\hat{X}]$.
  
    We now show that $c<c'$.
    Consider the function $\psi(x) = \sum_{k=0}^{\infty}x^{k+1}p_k + p_{-1} - x$ on the interval $[0,1]$.
    Observe that $\psi'(x) = \sum_{k=0}^{\infty}(k+1)x^kp_k - 1 = \phi(x)$,
    \[\psi(c)=\sum_{k=0}^{\infty}c^{k+1}p_k+p_{-1}-c = c \left(\sum_{k=0}^{\infty}c^k p_k + \frac{p_{-1}}{c} -1\right)=0,\]
    and $\psi(1) = \sum_{k=0}^{\infty}p_k+p_{-1}-1=0$.
    As $\psi(c)=\psi(1)=0$, there exists a $d \in (c,1)$ such that $\psi'(d)=0$.
    Since $\phi(d)=\psi'(d)=0$, we have $d=c'$.
    Thus, $c<c'$. 
  \end{proof}

  \subsection{Relation between RLS order and the order on integers}
  Let $X = (X_n)_{n \in \mathbb{Z}}$ be the i.i.d. increments  of a skip-free to the left random walk (see Def. \ref{hyp:increments}) with $\mathbb{E}[X_0]=0$ and let $[\mathbb{Z},0,X]$ be its associated network.
  Let $[\mathbf{T},0]$ denote the connected component of $0$ (rooted at \(0\)) in the $R$-graph (see notations of Chapter \ref{chapter_prelim}) of the network $[\mathbb{Z},0,X]$.
  Since the vertices of $\mathbf{T}$ are integers, we can order the children of every vertex using the total order $<$ of $\mathbb{Z}$.
  Then, the RLS order \(\prec\) gives a total order on the vertices of $\mathbf{T}$.
  Thus, we have two total orders $<$ and $\prec$ on the vertices of $\mathbf{T}$.
  The former is deterministic whereas the latter is random (it depends on the realization of $X$). 
  Using the interval property (Lemma \ref{lemma:record_descendants}) of the record vertex-shift, we show in the following lemma that these two orders are in fact equivalent in the following sense:
  
  \begin{lemma} \label{lemma:rls_order}
    Let $x=(x_n)_{n \in \mathbb{Z}}$ be a real-valued sequence, $[\mathbb{Z},0,x]$ be its associated network. Let $i,j$ be two integers having a smallest common ancestor \(w\), with $w=R^m(i)=R^k(j)$, where $m$ and $k$ are the smallest non-negative integers satisfying this property. Then, $i<j$ (with respect to the order on $\mathbb{Z}$) if and only if $i \prec j$.
  \end{lemma}
  \begin{proof}
    ($i<j \implies i \prec j$):
    Since $R^n(i)$ is an increasing sequence for $0<n < m$, there exists a largest $n_0 \leq m$ such that $R^{n_0}(i) \leq j$.
    If \(R^{n_0}(i)=j\), then \(i \prec j\), which is the desired relation.
    So, let us assume that \(R^{n_0}(i)<j\) and \(R^{n_0+1}(i)>j\).
    By Lemma \ref{lemma:record_descendants}, both $R^{n_0}(i)$ and $j$ are descendants of $R^{n_0+1}(i)$, which implies  that $R^{n_0+1}(i)=w$.
    Hence, \(R^{m_0}(j)=R^{n_0+1}(i)=w\) for some \(m_0>0\).
    Since $R^{n_0}(i)<j \leq R^{m_0-1}(j)$, we have $R^{n_0}(i)\prec R^{m_0-1}(j)$, and by Remark \ref{remark:rls_order}, we have $i \prec j$.
  
    ($i\prec j \implies i < j$):
  The case where $j$ is an ancestor of $i$ is trivial.
  So, let us assume that $j$ is not an ancestor of $i$.
  Then we have $k>0$, and $R^{m-1}(i) <  R^{k-1}(j)$.
  Since $R^{m-1} (i)$ is not a descendant of $R^{k-1}(j)$, but $j$ is a descendant of $R^{k-1}(j)$, by Lemma \ref{lemma:descendants}, we have $i<R^{m-1}(i)<L_x(R^{k-1}(j)) \leq j$.
  \end{proof}

\subsection{Relation between offspring count and increment in the record graph}

Let $T$ be an ordered Family Tree, $o \in V(T)$, and $u_n<u_{n-1}<\cdots<u_1$ be the children of $o$ (totally ordered by $<$).
Define the order of any child $u_j$ among its siblings as $j$.
That is, \(u_j\) is the \(j\)-th child of \(\mathbf{o}\).
Recall the notation $y(j,k):=\sum_{l=j}^{k-1}x_l$ for integers $j<k$ and for a given sequence $x=(x_n)_{n \in \mathbb{Z}}$.

The skip-free property of the increments of the random walk gives us the following lemma.

\begin{lemma}\label{lemma_offspring_count}
  Let $x = (x_k)_{k \in \mathbb{Z}}$ be a sequence where $x_k \in \mathbb{N}\cup\{0,-1\}$.
  Let $[\mathbb{Z},0,x]$ be the associated network of \(x\), and $T$ be the $R$-graph of \([\mathbb{Z},0,x]\).
  Let $i$ be an integer such that $L_x(i)>-\infty$.
  Then,
  \begin{enumerate}
    \item The number of children $d_1(i)$ of $i$ in $T$ is given by $d_1(i) = x_{i-1}+1$.
    \item Let $d_1(i)=n>0$ and $i_n<i_{n-1}<\cdots<i_1$ be the positions on $\mathbb{Z}$ of the \(n\) children of $i$. Then, the position $i_m$ of the $m$-th child satisfies the relation $m = x_{i-1}+1-y(i_m,i)$.
    \item Let $d_1(i) = n>0$, and $m \in \{1,2,\cdots,n\}$. Then, an integer $i'<i$ is the $m$-th child of $i$ if and only if $i'$ is the largest integer among $\{L_x(i), \cdots,i-1\}$ that satisfies $y(i',i) = x_{i-1}+1-m$.
  \end{enumerate} 
\end{lemma}
\begin{proof}
  If $x_{i-1}=-1$, then $L_x(i) = i$. So, by Lemma \ref{lemma:descendants}, $i$ has no descendant.
  Then, \((1.)\) holds as \(d_1(i) = x_{i-1}+1=0\), whereas the other statements are empty.

  So, let us assume that $x_{i-1}\geq 0$.
  Clearly $i-1$ is then a child of $i$ as \(i\) is the record of \(i-1\).
  Therefore, \(d_1(i)\geq 1\).
  We show below that the sum of increments between any two successive children is $-1$.

  Let $i_n<i_{n-1}<\cdots <i_1$ be the positions on $\mathbb{Z}$ of the children of $i$, with $i_1=i-1$.
Consider a child $i_m$ with $m<n$.
If $i_{m+1} = i_m-1$, then $x_{i_m-1}=-1$. 
Therefore, $y(i_{m+1},i_m)=-1$. 
Consider now the case where $i_{m+1}<i_m-1$, and let $i'$ be an integer with $i_{m+1}<i'<i_{m}$.
  As $i'$ is not a child of $i$, and $i_m$ is a child of $i$, by Lemma \ref{lemma:record_descendants}, we have $R_x(i')\leq i_m$, and there exists a smallest $k>0$ such that $R_x^k(i')=i_m$.
  Therefore, $y(i',i_m) = \sum_{l=0}^{k-1} y(R_x^l(i'),R_x^{l+1}(i')) \geq 0$. 

  In particular, $y(i_{m+1}+1,i_m) \geq 0$. 
But, $y(i_{m+1},i_m)<0$, because the record of $i_{m+1}$ is \(i\) (not $i_m$).
Therefore, the only possibility is that $x_{i_{m+1}}=-1$.
This also implies that $y(i_{m+1}+1,i_m)=0$.
Indeed, if $y(i_{m+1}+1,i_m)\geq 1$, then $y(i_{m+1},i_m)=x_{i_{m+1}}+y(i_{m+1}+1,i_m)=-1+y(i_{m+1}+1,i_m) \geq 0$, which gives a contradiction that \(i_m\) is the record of \(i_{m+1}\).
Hence,  $y(i_{m+1},i_m)=x_{i_{m+1}}+y(i_{m+1}+1,i_m)=-1$.
Thus, we have shown that the sum of increments between two successive children is \(-1\). We now use this to prove the three statements of the lemma.

The proof of the \emph{second statement} follows because, for any child $i_m$ of $i$,  we have 
\begin{equation} \label{eqn:cond_child}
  y(i_m,i)=y(i_m,i_1)+x_{i-1}=y(i_m,i_{m-1})+\cdots+y(i_2,i_1)+x_{i-1}=-(m-1)+x_{i-1}.
\end{equation}
We now prove the third statement.
Since \(i_m\) is a child of \(i\), \(y(i_m,i)\geq 0,\, \forall 1 \leq m \leq n\).
Further, for any $i_m<i' \leq i-1$, we have $y(i',i) = y(i',i_j)+y(i_j,i)$, with $i_j\geq i'$ being the unique smallest integer which is a child of $i$.
Such an \(i_j\) always exists since \(i-1\) is a child of \(i\).
Moreover, we have \(j<m\).
Since \(y(i',i_j) \geq 0\), we have
\[y(i',i)=y(i',i_j)+x_{i-1}+1-j \geq x_{i-1}+1-j > x_{i-1}+1-m,\]
where we used Eq. (\ref{eqn:cond_child}) to get the first equality.
This proves the \emph{forward implication of the third statement}: if $h$ is the $m$-th child of $i$ then $h<i$ is the largest integer that satisfies $y(h,i)=x_{i-1}+1-m$.

As for the \emph{backward implication of the third statement}, if $i'<i$ is the largest integer satisfying $y(i',i)=x_{i-1}+1-m \geq 0$ for some $m \in \{1,2, \cdots, n\}$, then $y(i',j)= y(i',i)-y(j,i)<0$, for all $i'<j \leq i-1$.
So, $i'$ is a child of $i$. 
But, by Eq. (\ref{eqn:cond_child}), there is only one child that satisfies this condition, namely, the $m$-th child of $i$.
This proves the third statement.

We now show that $y(L_x(i),i)=0$.
Indeed, from the definition of $L_x(i)$, it follows that $y(L_{x}(i)-1, i) = x_{L_x(i)-1}+y(L_x(i),i)<0$.
But $y(L_x(i),i) \geq 0$, implying that $x_{L_{x}(i)-1}=-1$ (the only possibility for \(x_{L_{x}(i)-1}\)).
Hence, $y(L_x(i),i)=0$.
This implies that $i_n = L_x(i)$ since no integer smaller than \(L_x(i)\) can be a descendant of \(i\) by Lemma \ref{lemma:descendants}.
To prove the first statement, observe that 
\begin{equation*}
  y(i_n,i) = x_{i-1}+\sum_{k=1}^{n-1} y(i_k,i_{k+1}) = x_{i-1}-(n-1).
\end{equation*}
Therefore, $y(i_n,i)=0=x_{i-1}+1-n$ which gives $d_1(i) = n =x_{i-1}+1$, proving the \emph{first statement}.
\end{proof}

\begin{remark}\label{remark_smallest_child_type_negative}
  From the last part of the proof of Lemma \ref{lemma_offspring_count}, it follows that, if \(L_x(i)>-\infty\), then \(L_x(i)\) is the smallest child of \(i\).
\end{remark}

\subsection{The type function and its properties}\label{subsection_typeFunction}
Let $x=(x_n)_{n \in \mathbb{Z}}$ be an integer-valued sequence with $x_n \geq -1$ for all $n \in \mathbb{Z}$, \([\mathbb{Z},0,x]\) be the rooted network associated to \(x\), and \(T\) be the record graph of the network \((\mathbb{Z},x)\).
For any integer \(i\), denote the number of children of \(i\) in \(T\) by \(d_1(i,T)\).

The {\bf type function} associated to \(x\) is the map $t:\mathbb{Z} \to \mathbb{Z}$ that assigns to each integer $i \in \mathbb{Z}$ the integer $t(i)\in \{-1,0,1,\cdots\}$ defined as 
\begin{equation}
  t(i):= t_x(i)= \inf\{n \geq -1: y(m,i) = n \text{ for some } m<i\}.
\end{equation}
The integer \(t(i)\) is called the {\bf type} of \(i\).
Note that the type function is a deterministic function of \(x\).
The type \(t_x(i)\) of an integer \(i\) can also be written 
\begin{equation} \label{eq_type_inf}
  t_x(i)  = \inf\{y(m,i)\vee -1: m<i\},
\end{equation}
where \(\vee\) is the max function.

The type function is analogous to the Loynes construction \cite{baccelliElementsQueueingTheory2003}.

We show, in the later sections, that when \(\mathbb{E}[X_0]>0\) the record graph together with the type function encodes the trajectory of the random walk.
This allows one to move back and forth between a trajectory of the random walk and its record graph together with the type function.

\begin{remark} \label{remark_type_well_defined}
  Let \(X\) be the i.i.d. increments of a skip-free to the left random walk.
  Then, for any integer \(i \in \mathbb{Z}\), the sequence $S^{(i)} = (S^{(i)}_m)_{m \geq 0}:=(\sum_{k=i-m}^{i-1}X_k)_{m \geq 0}$ is a skip-free to the left random walk.
  So, it can only take negative steps of size \(1\) at a time (i.e., $S^{(i)}_m <0$ for some $m \geq 1$ if and only if $S^{(i)}_{j} = -1$ for some $j\leq m$).
  The type of vertex $i$ is $-1$ if and only if $S^{(i)}_m=-1$ for some $m \geq 1$, and it is equal to $n \geq 0$ if and only if \(\min_{m \geq 1}\{S^{(i)}_m\} = n\), i.e., the random walk $(S^{(i)}_m)_{m \geq 1}$ reaches $n$ but stays above \(n-1\).
  In particular, $t(i)$ is well-defined for all $i \in \mathbb{Z}$.
\end{remark}
Define the function $l_x:\mathbb{Z}\to \mathbb{Z}$ that maps an integer $i$ to
\begin{equation}\label{eqn:def_l_x}
  l(i):=l_x(i)= \sup\{m<i:y(m,i)=t_x(i)\}.
\end{equation} 

Recall the function \(L\) defined in Eq. (\ref{eq:L_x_defn}).
This function is related to the function \(l\) in the following way.
First note that for every \(i \in \mathbb{Z}\),  we have \(-\infty<l(i)<i\) since \(t_x(i)=\min\{y(m,i)\vee -1: m<i\}\) and \(l(i)\) is the largest \(m\) for which \(y(m,i)\) attains \(t_x(i)\), when \(m\) is varied in \(\mathbb{Z}_{<i}\). 
 If \(t_x(i)=-1\), then \(l(i)=L(i)-1\).
 For any integer \(i\), \(t_x(i) \geq 0\) if and only if \(L(i)=-\infty\).

\begin{lemma}\label{defn_l_20230404185835}\label{smallest_child_20230404185835}
  For any integer \(i\), we have the following dichotomy:
  \begin{itemize}
    \item if \(t_x(i)=-1\), then \(L(i)\) is the smallest child of \(i\) in \(T\),
    \item if \(t_x(i)\geq 0\), then \(l(i)\) is the smallest child of \(i\) in \(T\).
  \end{itemize}
\end{lemma}

\begin{proof}
  If \(t_x(i)=-1\), then \(L(i)>-\infty\).
  So, by Remark \ref{remark_smallest_child_type_negative}, \(L(i)\) is the smallest child of \(i\) in \(T\).

  So, let us assume that \(t_x(i)\geq 0\).
  In this case, $i$ is the record of $l(i)$, i.e., $i = R(l(i))$.
  This follows because, for any $l(i)<m<i$, we have $y(m,i)>t_x(i)$ by the definition of $t_x(i)$ and $l(i)$.
  Since $y(l(i),i)=t_x(i)$, we have $y(l(i),m) = y(l(i),i)-y(m,i)<0$.
  Thus, $R(l(i))=i$.
  Further, if $t_x(i)\geq 0$, then $y(m,l(i)) \geq 0$ for all $m<l(i)$.
  Indeed, for \(m<l(i)\), since $y(m,i) \geq t_x(i)$, we have $y(m,l(i)) = y(m,i)-y(l(i),i) \geq t_x(i)-t_x(i)=0$.
  So, $R(m) \leq l(i)$ for all $m<l(i)$.
  This implies that $l(i)$ is the smallest among the children of $i$ as none of the integers smaller than $l(i)$ are children of $i$.
\end{proof}

The following lemma describes the relation between the number of children of any integer \(i\), its associated increment \(x_i\) and its type \(t_x(i)\) when \(t_x(i) \geq 0\).
When \(t_x(i)=-1\), the type does not play any role in determining this relation.

\begin{lemma}\label{20230405112810}
  Let $i \in \mathbb{Z}$ be an integer.
  We have the following dichotomy:
  \begin{itemize}
    \item If \(t_x(i)=-1\), then the number of children of $i$ in $T$ is given by \(d_1(i,T)=x_{i-1}+1\).
    \item If \(t_x(i) \geq 0\), then the number of children of $i$ in $T$ is given by $d_1(i,T) = x_{i-1}+1-t_x(i)$.
  \end{itemize}
\end{lemma}

\begin{figure}[h]
    \begin{center}
      \includegraphics[scale=0.92]{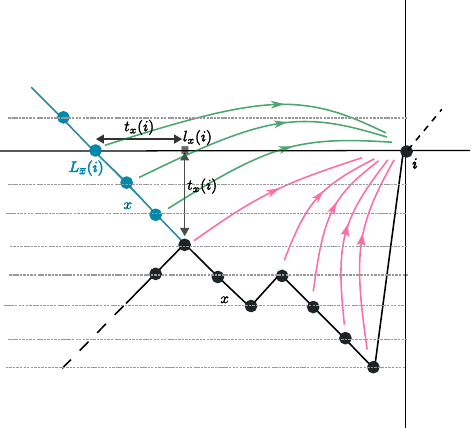}
    \end{center}
    \caption{The trajectories of $x=(x_n)_{n \in \mathbb{Z}}$ (in black) and $\bar{x}=(\bar{x}_n)_{n \in \mathbb{Z}}$ (in blue). The children of $i$ in $T$ are drawn in red, and the additional children of $i$ in $\bar{T}$ are drawn in green. See Lemma \ref{20230405112810}.}
    \label{fig:relation_children_type_20230405112810}
\end{figure}
\begin{proof}
  If \(t_x(i)=-1\), then \(L_x(i) > - \infty\) (see the paragraph above Lemma \ref{defn_l_20230404185835}), which is the condition needed to apply Lemma \ref{lemma_offspring_count}.
  The first part of the statement follows from part 1 of Lemma \ref{lemma_offspring_count}.

  We now prove the second part of the statement.
  To follow the proof, see Fig. \ref{fig:relation_children_type_20230405112810}. Define a new sequence $\bar{x} = (\bar{x}_n)_{n \in \mathbb{Z}}$ by
  \[\bar{x}_n = \begin{cases}
    -1, \forall n<l_x(i)\\
    x_n, \forall n \geq l_x(i). 
  \end{cases}\]
  Let $\bar{y}(j,k)=\sum_{l=j}^{k-1}\bar{x}_l$ for all integers $j<k$.
  Then, $y(j,k)=\bar{y}(j,k)$ for all $ l_x(i) \leq j < k$.
  Thus, we have 
  \begin{equation}\label{eq:undisturbed_x_20230405112810}
      \bar{y}(k,i) = y(k,i) \geq t_x(i)>0 \quad \forall l_x(i)\leq k < i,
  \end{equation}
  and 
  \[\bar{y}(k,i) = k- l_x(i)+t_x(i) \quad \forall k<l_x(i).\]
  Therefore, (see Eq.(\ref{eq:L_x_defn}) for the definition of \(L\)) 
  \begin{equation}\label{eq:rel_l}
    L_{\bar{x}}(i) = l_x(i) - t_x(i)> - \infty.
  \end{equation}
  Let $\bar{T}$ denote the record graph of the network $(\mathbb{Z},\bar{x})$, $D_1(i,\bar{T})$ denote the set of children of $i$ in $\bar{T}$ and let $d_1(i,\bar{T}) = \#D_1(i,\bar{T})$.
  Then, by the first part of Lemma \ref{lemma_offspring_count}, we obtain that $d_1(i,\bar{T})=\bar{x}_{i-1}+1$.
  Since $l_x(i)<i$, we have $x_{i-1} = \bar{x}_{i-1}$ which implies that $d_1(i,\bar{T}) = x_{i-1}+1$.

  Note that $\bar{y}(l_x(i),i)=t_x(i)$ and $\bar{y}(j,i)>t_x(i)$ for all $l_x(i)<j<i$.
  Therefore, by the third part of Lemma \ref{lemma_offspring_count}, $l_x(i)$ is the $x_{i-1}+1-t_x(i)$-th child of $i$ in $\bar{T}$.
  Since $\bar{x}_j=-1$ for all $j<l_x(i)$, by (3) of Lemma \ref{lemma_offspring_count}, we obtain that
  \[\{j \in \mathbb{Z}: L_{\bar{x}}(i) \leq j < l_x(i)\} \subseteq D_1(i,\bar{T}).  \]

  By Eq. (\ref{eq:rel_l}), $\#\{j \in \mathbb{Z}:L_{\bar{x}}(i)  \leq j < l_x(i)\}=l_x(i)-L_{\bar{x}}(i)=t_x(i)$.
  Observe that $D_1(i,\bar{T})\cap \{j \in \mathbb{Z}: l_x(i)\leq j < i\} = D_1(i,T)$ by Eq. (\ref{eq:undisturbed_x_20230405112810}).
  Finally, Lemma \ref{defn_l_20230404185835} implies that $l_x(i)$ is the smallest child of $i$ in $T$.
  Thus, we have
  \begin{align*}
      D_1(i,T) &= D_1(i,\bar{T})\backslash \{j \in \mathbb{Z}: L_{\bar{x}} \leq j<l_x(i)\},
  \end{align*}
  which gives
  \[d_1(i,T)=d_1(i,\bar{T})- t_x(i)=x_{i-1}+1-t_x(i).\]
\end{proof}

\section{Classification theorem and phase transitions} \label{section_classification}
Recall the foil classification theorem \ref{thm_foil_classification}.
It states that a.s. any unimodular Family Tree belongs to either class \(\mathcal{F}/\mathcal{F}\) or class $\mathcal{I}/\mathcal{F}$ or class $\mathcal{I}/\mathcal{I}$.
These three classes are characterized by the following distinctive properties.
Every Family Tree of class \(\mathcal{F}/\mathcal{F}\) has finitely many vertices.
Every Family Tree of class $\mathcal{I}/\mathcal{I}$ is eternal, one-ended and all of its vertices have finitely many descendants. 
In contrast, every Family Tree of class $\mathcal{I}/\mathcal{F}$ is eternal, two ended, and all the vertices that lie on its unique bi-infinite path have infinitely many descendants and all the remaining vertices have finitely many descendants.
In addition, if a unimodular Family Tree is of class \(\mathcal{I}/\mathcal{F}\) a.s., then it is obtained by the operation called typically rooted joining of a stationary sequence of finite Family Trees.

In Subsection \ref{sec_examples}, we will use the typically rooted joining construction to define a unimodular Family Tree of  class \(\mathcal{I}/\mathcal{F}\).

We first describe the class (according to the classification theorem) of the connected component of $0$ in  the \(R\)-graph of \([\mathbb{Z},0,X]\), where \(X=(X_n)_{n \in \mathbb{Z}}\) are the i.i.d. increments of a skip-free to the left random walk (see Def. \ref{hyp:increments}).
Let \((S_n)_{n \in \mathbb{Z}}\) be the two-sided random walk associated to \(X\), where \(S_0=0,\, S_n = \sum_{k=0}^{n-1}X_k\) for \(n>0\), and \(S_n= \sum_{k=n}^{-1}-X_k\) for \(n<0\).
Let \(\mathbb{Z}^R\)  be the record graph of \((\mathbb{Z},X)\) and \(\mathbf{T} = \mathbb{Z}^R(0)\) be the connected component of \(0\) in \(\mathbb{Z}^R\).
Since \([\mathbb{Z},0,X]\) is unimodular, \([\mathbf{T},0]\) is a unimodular Family Tree (see Lemma \ref{lemma:f_graph_unimodular} taken from \cite{baccelliEternalFamilyTrees2018a}).

The phase transitions occur not only for the i.i.d. sequence of random variables that have the support in \(\mathbb{Z}_{\geq -1}\), but also for more general stationary and ergodic integer-valued sequence of random variables (See Theorem \ref{thm_phase_transition_stationary}).
In the i.i.d. case, the component of \(0\) in the record graph exhibits three phases (\(\mathcal{F}/\mathcal{F},\mathcal{I}/\mathcal{I}\) and \(\mathcal{I}/\mathcal{F}\) resp.)  corresponding to the conditions \(\mathbb{E}[X_0]<0,\mathbb{E}[X_0]=0\) and \(\mathbb{E}[X_0]>0\) respectively.
But for more general stationary and ergodic case, one can construct examples for which there are only two phases: \(\mathcal{F}/\mathcal{F}\) when \(\mathbb{E}[X_0]<0\) and \(\mathcal{I}/\mathcal{F}\) when \(\mathbb{E}[X_0]\geq 0\) (see Example \ref{example_i_f_mean_0} and Theorem \ref{thm_phase_transition_stationary}).

\begin{proposition}
    If \(\mathbb{E}[X_0]<0\), then the unimodular Family Tree \([\mathbf{T},0]\) is of class \(\mathcal{F}/\mathcal{F}\).
  \end{proposition}
  \begin{proof}
    Since \(\mathbb{E}[X_0]<0\), the random walk \((S_n)_{n \geq 0}\) drifts to \(-\infty\), i.e., \(S_n \to-\infty\) a.s..
    So, a.s., \(0\) has finitely many ancestors in \(\mathbf{T}\).
    However, the random walk \((S_{-n})_{n \geq 0}\) drifts to \(\infty\).
    So, \(L_X(0)>-\infty\).
    Therefore, by Lemma \ref{lemma:descendants}, \(0\) has finitely many descendants.
    Thus, \([\mathbf{T},0]\) is of class \(\mathcal{F}/\mathcal{F}\).
  \end{proof}
  
  \begin{proposition}\label{proposition:r-graph_is_eft}
    If \(\mathbb{E}[X_0]=0\), then the record graph \(\mathbb{Z}^R\) is connected and the unimodular Family Tree \([\mathbf{T},0]\) is of class $\mathcal{I}/\mathcal{I}$. 
  \end{proposition}
  
  \begin{proof}
    Since $\mathbb{E}[|X_n|]< \infty$, by the Chung-Fuchs theorem \cite{kallenbergFoundationsModernProbability2021}, the random walk starting at $0$ with increments $(X_n)_{n \geq i}$ is recurrent for any $i \in \mathbb{Z}$.
  This implies that for any $i \in \mathbb{Z}$, almost surely, the sequence $(R^n_X(i))_{n \in \mathbb{N}}$ (of higher order record epochs of \(i\)) is strictly increasing, i.e., $R^n_X(i) \not = R^n_X(i), \forall n \geq 1$.
  
  Let $A_i$ be the event that the integer $i$ is in the connected component of $0$ in the $R$-graph.
  We claim that  the event $A = \bigcap_{i \in \mathbb{Z}}A_i$ occurs with probability $1$.
  
  Let $i<0$ (the proof for $i>0$ is obtained by interchanging $i$ and $0$).
  Since the sequence $(R^n_X(i))_{n \in \mathbb{Z}}$ is strictly increasing a.s., there exists a smallest (random) integer $k>0$ such that $R_X^{k-1}(i)<0 \leq R^k_X(i)$ a.s..
  By Lemma \ref{lemma:record_descendants} (applied to $R_X^{k-1}(i)$), $0$ is a descendant of $R^k_X(i)$ a.s..
  Thus, a.s., $0$ and $i$ are in the same connected component.
  The claim follows since $i$ is arbitrary.
  
  The above argument also implies that, a.s., \([\mathbf{T},0]\) is an EFT, and Lemma \ref{lemma:f_graph_unimodular} implies that it is unimodular.
  We prove that \([\mathbf{T},0]\) is of class \(\mathcal{I}/\mathcal{I}\) by
  showing that the descendant tree of every vertex of \(\mathbf{T}\) is finite.
  For any integer $i$, the random walk starting at $0$ with increments $(-X_{(i-1-n)})_{n \geq 0}$ is recurrent by the Chung-Fuchs theorem.
  Therefore, $L_X(i)>- \infty$ a.s., implying that the set of descendants of $i$ is a.s. finite by Part 1 of Lemma \ref{lemma_offspring_count}.
  This completes the proof.
  \end{proof}
  
  \begin{proposition}\label{remark_pos_drift_i_f}
    If \(0<\mathbb{E}[X_0]<\infty\), then the record graph \(\mathbb{Z}^R\) is connected and the unimodular Family Tree \([\mathbf{T},0]\) is of class \(\mathcal{I}/\mathcal{F}\).
  \end{proposition}
  \begin{proof}
   We show that there exists a vertex of \(\mathbf{T}\) that has infinitely many descendants, which implies that \([\mathbf{T},0]\) is of class \(\mathcal{I}/\mathcal{F}\).
    Since the random walk has positive drift (\(0<\mathbb{E}[X_0]<\infty\)), the positively indexed part of the random walk \((S_n)_{n \geq 0}\) drifts to \(+\infty\) a.s..
    So, a.s., the \(n\)-th record epoch of \(0\) is always larger than the \((n-1)\)-th record epoch of \(0\) for all \(n \geq 1\), i.e., \(R^n(0)>R^{n-1}(0),\,  \forall n \geq 1\).
    This implies that \(0\) has an infinite number of ancestors.
    The negatively indexed part \((S_{-n})_{n \geq 0}\) drifts to \(- \infty\) a.s..
    In particular, a.s., the random variable \(m_{-1}:= \sup_{n \geq 1}\{S_{-n}\}\) is finite and the random non-positive integer \(u^* := -\min\{n \geq 1: S_{-n}=m_{-1}\}\) exists.
    Since \(R(u^*) \geq 0\), there exists the smallest ancestor \(u \in \mathbb{Z}\) of \(0\) such that \(R(u^*) = u\).
    Since \(S_{n} \leq m_{-1}\) for all \(n \leq u^*\), \(u^*\) has infinitely many descendants by Lemma \ref{lemma:descendants}.
    Therefore, by the interval property (Lemma \ref{lemma:record_descendants}), it follows that the record graph \(\mathbb{Z}^R\) is connected.
    Further, by the classification theorem of unimodular Family Trees \cite{baccelliEternalFamilyTrees2018a}, \([\mathbf{T},0]\) is of class \(\mathcal{I}/\mathcal{F}\).
  \end{proof}
  
  \subsection{Instances of unimodular Family Trees}\label{sec_examples}

In this subsection, we describe three instances of unimodular Family Trees, one from each of the classes \(\mathcal{F}/\mathcal{F}, \mathcal{I}/\mathcal{F}\) and \(\mathcal{I}/\mathcal{I}\), which will be used in  the next sections.
The first instance is the Eternal Galton-Watson tree with offspring distribution \(\pi\).
It was studied in \cite{baccelliEternalFamilyTrees2018a}, in which it is
shown that the tree is unimodular if and only if the mean of \(\pi\) is \(1\).
It is of class \(\mathcal{I}/\mathcal{I}\) when \(\pi\) is non degenerate.
We introduce two new unimodular Family Trees, namely, the typically rooted Galton-Watson tree; which is of class \(\mathcal{F}/\mathcal{F}\); and the unimodularised marked ECS ordered bi-variate Eternal Kesten tree, which is of class \(\mathcal{I}/\mathcal{F}\).
The former is shown to  be obtained from Galton-Watson Tree by re-rooting the root to a typically chosen vertex of the Galton-Watson Tree.

\subsubsection{Typically rooted Galton-Watson Tree (TGWT)}

Let $\pi$ be a probability distribution on $\{0,1,2,3,\cdots\}$ such that its mean $0< m(\pi)<1$ and \(\hat{\pi}\) be the size-biased distribution of \(\pi\), given by \(\hat{\pi}(k) = \frac{k \pi(k)}{m(\pi)}\) for all \(k \in \{0,1,2,\cdots\}\).
The {\bf Typically rooted Galton-Watson Tree} (\(TGWT(\pi)\)) with offspring distribution $\pi$ is a Family Tree $[\mathbf{T},\mathbf{o}]$ defined in the following way.
\begin{itemize}
    \item Add a parent to the root $\mathbf{o}$ with probability $m(\pi)$. Independently iterate the same to the parent of $\mathbf{o}$ and to all of its ancestors. So, the number of ancestors of the root $\mathbf{o}$ has a geometric distribution with the success probability $1-m(\pi)$.
    \item Let $Z$ be a random variable that has distribution $\hat{\pi}$ (the size-biasing of $\pi$). To each of the ancestors of $\mathbf{o}$ attach independently children with distribution the same as that of $Z-1$.
    \item To each of these new children and to the root $\mathbf{o}$, attach independently ordered Galton-Watson trees with offspring distribution $\pi$.
\end{itemize}
See Figure \ref{figure_sgwt_2.31} for an illustration.

\begin{remark}\normalfont
    The \(TGWT(\pi)\) is a finite unimodular ordered Family Tree (see Proposition \ref{20230305162953}).
\end{remark}

\begin{figure}[h]
\begin{center}
      \includegraphics[scale=0.8]{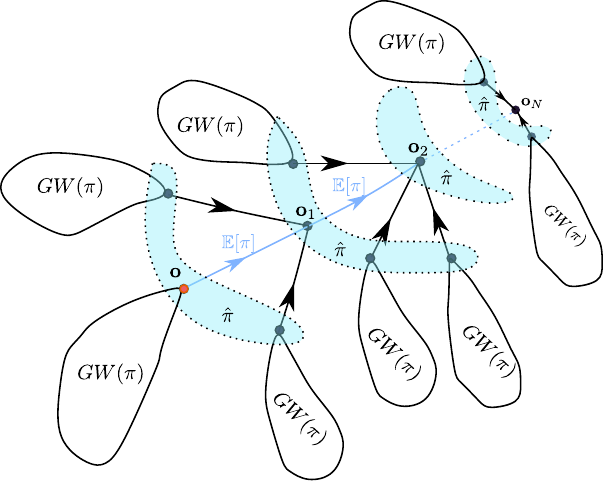}    
\end{center}
    \caption{Typically rooted Galton-Watson Tree with offspring distribution $\pi$ (\(\mathbf{o}\) is the root), \(\mathbb{E}[\pi]\) is the probability that the root \(\mathbf{o}\) has a parent, \(\hat{\pi}\) is the offspring distribution of the parent of \(\mathbf{o}\), \(N\) is a geometric random variable with the success probability \(1-\mathbb{E}[\pi]\).}
    \label{figure_sgwt_2.31}
\end{figure}

The nomenclature ``typically rooted'' in the typically rooted Galton-Watson Tree suggests that it is obtained by re-rooting to a typical vertex of a Galton-Watson tree.
It is indeed true, we prove this fact in Proposition \ref{prop_tgwt_is_typical}.
We also give a characterizing condition for \(TGWT\) in Proposition \ref{20230305162953}.

\subsubsection{Bi-variate Eternal Kesten Tree}\label{subsec_bi_variate_ekt}

In the following, we give an instance of a unimodular tree of class \(\mathcal{I}/\mathcal{F}\) using the typically rooted joining operation.
We first introduce a Family Tree which we call it the bi-variate Eternal Kesten Tree \(EKT(\alpha,\beta)\) with offspring distributions \(\alpha\) and \(\beta\). 
It is unordered, unlabelled and parametrized by \(\alpha\) and \(\beta\).
Using this Family Tree, we introduce a class of marked and ordered Family Trees namely marked ECS ordered bi-variate Eternal Kesten Trees.
None of these trees are unimodular, but we can obtain a unimodularised version from them using an operation called typically rooted joining.
In order to perform this operation, it is essential that the mean of offspring distribution \(\beta\) satisfies \(m(\beta)<1\).
But, to define \(EKT(\alpha,\beta)\), it is sufficient that the mean of the offspring distributions satisfy \(m(\alpha)<\infty\) and \(m(\beta)\leq 1\).

Let $\alpha, \beta$ be two probability distributions on $\{0,1,2,\cdots\}$ such that their means satisfy $m(\alpha)<\infty$ and $m(\beta) \leq 1$.
A {\bf bi-variate Eternal Kesten Tree} with offspring distributions $\alpha,\beta$, denoted as $EKT(\alpha,\beta)$, is a random Family Tree $[\mathbf{T}',\mathbf{o}']$ consisting of a unique bi-infinite $F$-path $(\mathbf{o}_n)_{n \in \mathbb{Z}}$, where $\mathbf{o}_0=\mathbf{o}'$, with the following property:
 the sequence of Family Trees $([D(\mathbf{o}_n)\backslash D(\mathbf{o}_{n-1}), \mathbf{o}_n])_{n \in \mathbb{Z}}$ is i.i.d. with the following common distribution: the offspring distribution of the root is \(\alpha\) and the descendant trees of the children of the root are i.i.d. Galton-Watson trees with offspring distribution $\beta$ (denoted as \(GW(\beta)\)).
The Family Trees \(\{[D(\mathbf{o}_n)\backslash D(\mathbf{o}_{n-1}), \mathbf{o}_n]:n \in \mathbb{Z}\}\) are called the bushes of \(EKT(\alpha,\beta)\) and the Family Tree \\ \([D(\mathbf{o}_0)\backslash D(\mathbf{o}_{-1}), \mathbf{o}_0]\) is called the {\bf bush of the root}.

Note that $EKT(\alpha,\beta)$ is the joining of the i.i.d. sequence $([\mathbf{T}_i,\mathbf{o}_i])_{i \in \mathbb{Z}}$, where $T_i = D(\mathbf{o}_i)\backslash D(\mathbf{o}_{i-1})$.
Keep in mind that $d_1(\mathbf{o}_i,\mathbf{T}_i)= d_1(\mathbf{o}_i,\mathbf{T}')-1$, where \(d_1(\mathbf{o}_i,\mathbf{T}_i)\) denotes the offspring count of \(\mathbf{o}_i\) in \(\mathbf{T}_i\).

In addition, suppose that \(m(\beta)<1\).
Then, \(GW(\beta)\) is subcritical, which implies that \(\mathbb{E}[\#V(\mathbf{T}_0)]< \infty\).
So, in this case, we can take the typically rooted joining of the i.i.d. Family Trees \(([\mathbf{T}_i,\mathbf{o}_i])_{i \in \mathbb{Z}}\).
Let \([\mathbf{T},\mathbf{o}]\) be the typically rooted joining of \(([\mathbf{T}_i,\mathbf{o}_i])_{i \in \mathbb{Z}}\) (as in Eq. (\ref{eq_typical_rerooting})).
Then, by Theorem \ref{thm:I_F_unimodularizable}, the EFT \([\mathbf{T},\mathbf{o}]\) is unimodular.
We call the distribution of \([\mathbf{T},\mathbf{o}]\) the unimodularised bi-variate \(EKT(\alpha,\beta)\).
The unimodularised \(EKT(\alpha,\beta)\) is of class \(\mathcal{I}/\mathcal{F}\).

An \textbf{Every Child Succeeding (ECS) order} on \([\mathbf{T}',\mathbf{o}']\) is obtained by declaring that $\mathbf{o}_n$ is the smallest among $D_1(\mathbf{o}_{n+1})$ and using the uniform order on the remaining children for all $n \in \mathbb{Z}$.
The name ``ECS'' order comes from the fact that the succession line (see Def. \ref{defn_succession_line} ) starting from the root reaches all the vertices on the bi-infinite path
(i.e., the succession line is an order preserving bijection from \(\mathbb{Z}\) to the vertices of the tree).
The unimodularised ECS ordered \(EKT(\alpha,\beta)\) is an example of a unimodular ordered EFT of class \(\mathcal{I}/\mathcal{F}\).  
Later, we show that the unimodularised ECS ordered \(EKT(\alpha,\beta)\) is the record graph of the network \([\mathbb{Z},0,X]\), where \(X\) is as in Def. \ref{hyp:increments} and satisfies \(\mathbb{E}[X_0]>0\).

We focus on the ordered Family Trees whose vertices have marks that take values in \(\mathbb{Z}\), as will be shown in the later sections that they encode information about the trajectory of a skip-free to the left random walk when the mean of the random walk is positive.
Let \(\hat{\mathcal{T}_*}\) denote the space of equivalence classes of rooted ordered Family Trees whose mark function assigns \(\mathbb{Z}\)-valued labels to its vertices.
The space \(\hat{\mathcal{T}_*}\) is Polish \cite{aldousProcessesUnimodularRandom2007}.
We give an instance of a unimodular marked ordered Family Tree of class \(\mathcal{I}/\mathcal{F}\).
It is important to note that the distribution of the marks need not be independent of the tree structure.
For instance, marks can be a deterministic function of the structure of a tree.

Let \(\alpha, \beta\) be probability distributions on \(\{0,1,2,\cdots\}\) such that their means satisfy \(m(\alpha)<\infty\) and $m(\beta) < 1$.
A {\bf bi-variate ECS ordered marked Eternal Kesten Tree} \(MEKT(\alpha,\beta)\) with parameters \(\alpha, \beta\) is a random ordered marked rooted Family Tree \([\mathbf{T}',\mathbf{o}',Z']\)  that has the following properties:
\begin{enumerate}
  \item The Family Tree \([\mathbf{T}',\mathbf{o}']\) obtained by forgetting the marks is the ECS ordered \(EKT(\alpha,\beta)\).
  Let \(([\mathbf{T}_i,\mathbf{o}_i])_{i \in\mathbb{Z}}\) be the i.i.d. sequence of Family Trees whose joining is \([\mathbf{T}',\mathbf{o}']\) (with \(\mathbf{o}'=\mathbf{o}_0\)).

  \item The sequence of marked Family Trees \(([\mathbf{T}_i,\mathbf{o}_i,Z_i])_{i \in \mathbb{Z}}\) is i.i.d., where \(Z_i\) is the mark function on \(\mathbf{T}_i\) obtained by restricting \(Z'\) to \(\mathbf{T}_i\).
\end{enumerate}

\begin{figure}[h]
  \begin{center}
        \includegraphics[scale=1.5]{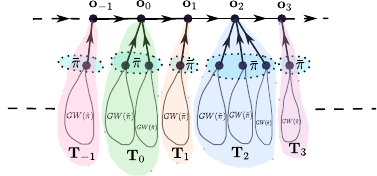}    
  \end{center}
      \caption{Bi-variate ECS ordered EKT with distributions $\bar{\pi},\tilde{\pi}$ (\(\mathbf{o}_0\) is the root), \(\bar{\pi}\) is the offspring distribution of \(\mathbf{o}_0\) in \(\mathbf{T}_0\), \(\tilde{\pi}\) is the offspring distribution of every child of \(\mathbf{o}_0\) in \(\mathbf{T}_0\), see Eq. (\ref{eq:pi_defn}) for the expressions of \(\bar{\pi},\tilde{\pi}\).}
      \label{fig_bi_variate_EKT}
  \end{figure}

\begin{remark}\label{remark_joining_ekt}
  Because of the independence condition involved in the second part of the construction of \(MEKT(\alpha,\beta)\), it follows that $EKT(\alpha,\beta)$ is the ECS ordered joining of an i.i.d. sequence of finite marked Family Trees $([\mathbf{T}_i,\mathbf{o}_i,Z_i])_{i \in \mathbb{Z}}$, where the distribution of $[\mathbf{T}_0,\mathbf{o}_0]$ is as follows: the root $\mathbf{o}_0$ does not have parent, 
  the offspring distribution of $\mathbf{o}_0$ is $\alpha$, 
  the descendant trees of the children of $\mathbf{o}_0$ are independent subcritical ordered Galton-Watson trees with offspring distribution $\beta$.
\end{remark}

It follows, from the second part of the construction of \(MEKT(\alpha,\beta)\), that the sequence of random variables \((Z_i(\mathbf{o}_i))_{i \in \mathbb{Z}}\) is i.i.d.
(see the Figure \ref{fig_bi_variate_EKT}).
Note that \(Z_i\) need not be independent of \([\mathbf{T}_i,\mathbf{o}_i]\) for any \(i \in \mathbb{Z}\).

\begin{remark}\label{remark_characterization_bi_ecs_EKT}
    It is also clear from the construction of \(MEKT(\alpha,\beta)\) that the distribution of \( [\mathbf{T}_0,\mathbf{o}_0,Z_0]\) characterizes \(MEKT(\alpha,\beta)\).
\end{remark}
\begin{remark}
Let \([\mathbf{T}',\mathbf{o}',Z']\) be an \(MEKT(\alpha,\beta)\) with \(\alpha,\beta\) same as the above, and let \([\mathbf{T},\mathbf{o},Z]\) be the random ordered marked Family Tree obtained by taking \(\sigma'\) on \([\mathbf{T}',\mathbf{o}',Z']\), i.e., for any measurable subset \(A\) of \(\hat{\mathcal{T}}_*\),
\begin{equation}\label{eq_unimodularised_ekt}
  \sigma'[A] = \frac{1}{\mathbb{E}[\# V(\mathbf{T}_0)]} \mathbb{E}\displaystyle\left[\sum_{v \in V(\mathbf{T}_0)} \mathbf{1}_{[\mathbf{T}',v,Z']}(A) \right].
\end{equation}
Then, by Theorem \ref{thm:I_F_unimodularizable}, the marked Family Tree \([\mathbf{T},\mathbf{o},Z]\) is unimodular.
\end{remark}

The distribution of the marked Family Tree \([\mathbf{T},\mathbf{o},Z]\) is called the {\bf unimodularised} \(MEKT(\alpha,\beta)\).

\subsection{Distribution of the record graph in the three phases}

An \(n+1\)-tuple of integers of the form \((0,s_1,\cdots,s_{n-1},1)\) with \(s_j \leq 0\) for all \( 1 \leq j \leq n-1 \) is called an {\bf excursion set}.
Consider the skip-free to the right random walk $(S_{-n})_{n \geq 0}$ associated to the i.i.d. sequence \(X=(X_n)_{n \in \mathbb{Z}}\).
The following proposition describes the relation between an excursion set of this random walk and the descendant tree of $0$.
Let $\delta=\max\{n<0: S_n\geq 1\} = \max\{n<0: S_n = 1\}$ (by the skip-free property).

\begin{proposition}\label{20230128172431}
  Assume $\mathbb{E}[X_{0}]\leq 0$.
  Consider the rooted network $[\mathbb{Z},0,X]$. Then, the descendant tree of the root $0$ in the connected component of $0$ of the $R$-graph of $[\mathbb{Z},0,X]$ is a sample of the ordered Galton-Watson tree with offspring distribution $\pi \overset{\mathcal{D}}{=} X_{0}+1$.
\end{proposition}
\begin{proof}
Since the random walk $(S_{-n})_{n \geq 0}$ drifts to $\infty$, we have $\delta>-\infty$.
It follows from the definition of $L_X(0)$ (see Eq. (\ref{eq:L_x_defn})) that $L_X(0)= \delta-1$ since $y(\delta,0) = -S_{\delta} = -1$ and for any $\delta-1\leq j<0$, $y(j,0)=-S_j>-1$.
Thus, $L_X(0)> - \infty$ a.s., so that the condition required to apply Lemma \ref{lemma_offspring_count} is satisfied.
The first part of Lemma \ref{lemma_offspring_count} implies that $d_1(0)=X_{-1}+1 \overset{\mathcal{D}}{=} \pi$.

We now show that, conditioned on the number of children of $0$, the subtrees of these children are jointly independent Galton-Watson trees with offspring distribution $\pi$.

For each $m \geq 1$, let $i_m = \max\{n<0:y(i_m,0)=X_{-1}+1-m\}$ and let \(i_0=0\).
By the third part of Lemma \ref{lemma_offspring_count}, $i_1>i_2>\cdots>i_{X_{-1}+1}$ are the children of $0$, and $i_{X_{-1}+1}=L_X(0)$.
The fact that, for each $m\geq 1$, $i_m$ is a stopping time for the random walk $(S_{-n})_{n \geq 0}$, and the third part of Lemma \ref{lemma_offspring_count}, together imply the following:
the excursions $((S_{i_k}-S_{i_{k}},S_{i_k-1}-S_{i_k},\cdots,S_{i_{k+1}}-S_{i_{k}}))_{k=1}^{X_{-1}}$, with the notation $S_{i_0}=0$, are independent and identically distributed as $(S_0,S_{-1},\cdots,S_{\delta})$.

By Lemma \ref{lemma:descendants}, for any $1 \leq k \leq X_{-1}+1$, the descendant tree of $i_k$ is completely determined by 
\[(S_{i_k}-S_{i_{k-1}},S_{i_k-1}-S_{i_{k-1}},\cdots,S_{i_{k+1}}-S_{i_{k-1}}) \overset{\mathcal{D}}{=}(S_0,S_{-1},\cdots,S_{\delta}). \]
This completes the proof.
\end{proof}

\begin{remark}
  Consider the i.i.d. sequence \(\hat{X}=(\hat{X}_n)_{n \in \mathbb{Z}}\) with the distribution of \(\hat{X}_0\) given by Eq. (\ref{eq_conditioned_rw}).
Since \(\mathbb{E}[\hat{X}_0]<0\) by Lemma \ref{20230116135048}, the assumption of Proposition \ref{20230128172431} is satisfied.
This implies that the descendant tree of \(0\) in the record graph of \([\mathbb{Z},0,\hat{X}]\) is a Galton-Watson tree with offspring distribution the same as that of \(\hat{X}_0+1\).
This will be used in the last part of Step 2 of Theorem \ref{r_graph_positive_drift_20230203174636}.
\end{remark}

\subsubsection{Negative mean}

\begin{theorem}\label{20230213191656}
  Let $[T,o]$ be the connected component of $0$ in the $R$-graph of $[\mathbb{Z},0,X]$ and rooted at $o=0$.
  If \(\mathbb{E}[X_0]<0\), then $[T,o]$ is the Typically rooted Galton-Watson Tree with offspring distribution $\pi$, where \(\pi\) is the distribution of \(X_0+1\).
\end{theorem}
\begin{proof}
  We prove the theorem in several steps.

  In \emph{Step 1}, we show that the descendant tree of the root $o$ is a Galton-Watson Tree with offspring distribution $\pi$.
  In \emph{Step 2}, we compute the probability that $o$ has a parent.
  One way to compute it is by using Kemperman's formula (see \cite[Proposition 3.7]{bhattacharyaRandomWalkBrownian2021}).
  Alternatively, one could use the unimodularity of $[T,o]$.
  We use the latter in this step.
  \emph{Step 3} involves computing the offspring distribution of $F(o)$ conditioned that $o$ has a parent.
  In \emph{Step 4}, we give the distribution of the order of $o$ among its siblings conditioned on the fact that $o$ has a parent.
  In \emph{Step 5}, we show that the descendant trees of the siblings of $o$ conditioned on the fact that $o$ has a parent are independent Galton-Watson trees with offspring distribution $\pi$.
  
  The final step (\emph{Step 6}) involves showing that the offspring distribution of $F^2(o)$ conditioned on the fact that $o$ has an ancestor of order \(2\) is the same as the offspring distribution of $F(o)$ conditioned on the fact that $o$ has a parent.
  Similarly, we show that the descendant trees of siblings of $F(o)$ and its order among its siblings conditioned on the fact that $o$ has an ancestor of order \(2\) have the same distribution as that of $o$ conditioned on the fact that $o$ has a parent.

  \noindent \textbf{Step 1:} 
  Since the random walk $(S_{-n})_{n \geq 0}$ has positive drift, $-\infty<L(0)$.
  By Proposition \ref{20230128172431}, the descendant tree of $o$ is GW($\pi$).

  \noindent \textbf{Step 2:}
  Since $[T,o]$ is the connected component of $0$ in the $R$-graph of a unimodular network $[\mathbb{Z},0,X]$, it is also unimodular by Lemma \ref{lemma:f_graph_unimodular}.
  Therefore, 
  \begin{align*}
      \mathbb{P}[o \text{ has a parent}] &= \mathbb{E}\left[\sum_{u \in V(T)}\mathbf{1}\{F(o)=u\} \right]=\mathbb{E}\left[\sum_{u \in V(T)}\mathbf{1}\{F(u)=o\} \right]\\
      &= \mathbb{E}[d_1(o)]=m(\pi)=\mathbb{E}[X_{0}]+1,
  \end{align*}
  where \(m(\pi)\) is the mean of \(\pi\).
In particular, 
\[\mathbb{P}[S_n<0\quad \forall n>0] = \mathbb{P}[o \text{ does not have parent}] = -\mathbb{E}[X_{0}].\]

\noindent \textbf{Step 3:} 
Note that $o$ has a parent if and only if $\tau<\infty$ (see the paragraph above Lemma \ref{20230119141633} for the definition of \(\tau\)).
So, $\mathbb{P}[\tau< \infty] = m(\pi)$.
Since $(S_{-n})_{n \geq 0}$ drifts to $+\infty$, part (1) of Lemma \ref{lemma_offspring_count} implies the equality of the two events $\{d_1(F(o))=n,\tau<\infty\} = \{X_{\tau-1}=n-1,\tau< \infty\}$, for any $n>0$.
Therefore, for any $n>0$,
\begin{align}
  \mathbb{P}[d_1(F(o))=n,\tau<\infty] &= \sum_{j=0}^{n-1}\mathbb{P}[\tau<\infty,X_{\tau-1}=n-1,S_{\tau}=j] \notag\\
  &= \sum_{j=0}^{n-1}\mathbb{P}[S_1=n-1]c^{j-n+1} =n\pi(n). \label{eq_1_2.25} 
\end{align} 
Eq. (\ref{eq_1_2.25}) follows from the previous equation by Lemma \ref{20230119141633}.
The last equation follows since $c=\mathbb{P}[S_n=-1 \text{ for some } n>0]=1$ as the random walk $(S_n)_{n \geq 0}$ drifts to $-\infty$.
Therefore,
\[\mathbb{P}[d_1(F(o))=n|\tau< \infty] = \frac{n\pi(n)}{m(\pi)}= \hat{\pi}(n),\]
for all $n>0$.

\noindent \textbf{Step 4:}
Let $c_k(u)$ denote $k$-th child of $u$ for any vertex $u$ and positive integer $k$.
Note that the equality of the two events $\{\tau< \infty,c_k(F(o))=o,d_1(o)=n\}= \{\tau<\infty, X_{\tau-1}=n-1,S_{\tau}=n-k\}$ for any $0<k\leq n$ follows from part (3) of Lemma \ref{lemma_offspring_count} (with \(i_m\) as \(0\) and \(i\) as \(F(0)\)).
Therefore, using Lemma \ref{20230119141633} and $c=1$, we obtain for any $0<k \leq n$,
\[\mathbb{P}[\tau< \infty,c_k(F(o))=o,d_1(o)=n] = \mathbb{P}[X_0 = n-1].\]
Thus, for any $0<k \leq n$
\[\mathbb{P}[c_k(F(o))=o|d_1(o)=n,\tau<\infty] = \frac{\mathbb{P}[X_0 = n-1]}{n \mathbb{P}[X_0 = n-1]}=\frac{1}{n},\]
which is the uniform order among the siblings of $o$.

\noindent \textbf{Step 5:}
On the event $\{\tau<\infty, d_1(F(o))=n, c_k(F(o))=o\}$, let $i_n<i_{n-1}<\cdots<i_1$ be the positions of the children of $F(o)$.
Then, conditioned on $\{\tau<\infty, d_1(F(o))=n, c_k(F(o))=o\}$, for each $1\leq j \leq n-1$, the part of the random walk $(0, S_{i_j-1}-S_{i_j},S_{i_j-2}-S_{i_j},\cdots, S_{i_{j+1}}-S_{i_j})$ is an excursion set by Lemma \ref{lemma_offspring_count} (part 3).
Further, these excursion sets are independent of one another because the times $i_j$ (for $1 \leq j \leq n$) are stopping times.
Indeed, for all $1 \leq j \leq n-1$, $S_{i_{j+1}}-S_{i_j}=1$ and $S_l<S_{i_{j+1}}$ for all $i_{j+1}<l \leq i_j$.
Therefore, for each $1 \leq j \leq n-1$, $(0, S_{i_j-1}-S_{i_j},S_{i_j-2}-S_{i_j},\cdots, S_{i_{j+1}}-S_{i_j})$ is the skip-free to the right random walk $(S_{-n})_{n \geq 0}$ conditioned on $\eta_1< \infty$, and stopped at $\eta_{1}$, where $\eta_1 = \min\{n>0: S_{-n}=1\}$.
But $(S_{-n})_{n \geq 0}$ drifts to $+\infty$.
So, \(\eta_1<\infty\) a.s..
Hence, $(0, S_{i_j-1}-S_{i_j},S_{i_j-2}-S_{i_j},\cdots, S_{i_{j-1}}-S_{i_j})$ has the same distribution as $(0,S_{-1},\cdots,S_{\eta_1})$.
Therefore, by Proposition \ref{20230128172431}, for each $1 \leq j \leq n$, the descendant trees are independent $GW(\pi)$.

\noindent \textbf{Step 6:}
Observe that the distribution of $(S_n-S_{\tau})_{n \geq \tau}$ conditioned on $\tau< \infty$ is the same as that of $(S_n)_{n \geq 0}$ by the strong Markov property.
Therefore, 
\[\mathbb{P}[F(o) \text{ has a parent }|\{o \text{ has a parent}\}] = \mathbb{P}[o \text{ has a parent}] = m(\pi).\]
Since $(S_{-n})_{n \geq 0}$ drifts to $+ \infty$ a.s., it reaches $S_{R^2(0)+1}$ a.s..
Therefore, by part (1) of Lemma \ref{lemma_offspring_count}, the distribution of $d_1(F^2(o))$ conditioned on the event $\{o \text{ has a grandparent}\}$ is equal to the distribution of $X_{R^2(0)-1}+1$ conditioned on $\{R^2(0)>R(0)\}$ (i.e., $o$ has a grandparent) which has the same distribution as that of $X_{R(0)-1}+1$ conditioned on $R(0)>0$ (i.e., \(X_{\tau-1}\) conditioned on \(\tau< \infty\)).
In particular, the distribution of $[D(F^2(o))\backslash D(F(o)),o]$ conditioned on $\{o \text{ has a grandparent}\}$ is the same as that of $[D(F(o))\backslash D(o),o]$ conditioned on $\{o \text{ has a parent}\}$.
By induction, it follows that \\ $[D(F^n(o))\backslash D(F^{n-1}(o)),o]$ conditioned on $\{o \text{ has }n- \text{th ancestor}\}$ has the same distribution as that of $[D(F(o))\backslash D(o),o]$ conditioned on \\ $\{o \text{ has a parent}\}$.
\end{proof}

We now show that the Typically rooted Galton-Watson Tree is obtained by re-rooting to a typical vertex of a Galton-Watson Tree.
Let us continue to use the notations in Theorem \ref{20230213191656} and let \(\{F(\mathbf{o})=\mathbf{o}\}\) denote the event that \(\mathbf{o}\) does not have a parent.
Let \([\mathbf{T}',\mathbf{o}']\) be the Family Tree obtained from \([\mathbf{T},\mathbf{o}]\) by conditioning on the event that \(\mathbf{o}\) does not have a parent.
We first show that the distribution of  \([\mathbf{T}',\mathbf{o}']\) is \(GW(\pi)\).
Since \(m(\pi)<1\), then \(GW(\pi)\) has finite mean size.
So, we could apply Eq. \ref{eq_typical_rerooting} to \(GW(\pi)\).
We then prove that \([\mathbf{T},\mathbf{o}]\) is obtained from \([\mathbf{T}',\mathbf{o}']\) by re-rooting to a typical vertex of \(\mathbf{T}'\).

\begin{lemma}\label{lemma_conditioned_is_GW}
    With the notations as in the above paragraph, the distribution of \([\mathbf{T}',\mathbf{o}']\) is \(GW(\pi)\).
\end{lemma}
\begin{proof}
    Note that the events \(\{F(\mathbf{o})=\mathbf{o}\}\) and \(\{S_n<0 \, \forall n>0\}\) are one and the same because the latter event is the same as \(\{R(0)=0\}\).
    The random variables \(\{S_n:n \leq 0\}\) are independent of the event \(\{S_n<0 \, \forall n>0\}\), and the descendant tree of \(0\) is a measurable function of \(\{S_n:n \leq 0\}\).
    This implies that the distribution of \([D(\mathbf{o}),\mathbf{o}]\) conditioned on the event \(\{F(\mathbf{o})=\mathbf{o}\}\) is the same as \([[D(\mathbf{o}),\mathbf{o}]]\) (the unconditioned Family Tree).
    Therefore,
    \begin{equation*}
        [\mathbf{T}',\mathbf{o}']= [D(\mathbf{o}),\mathbf{o}]|\{F(\mathbf{o})=\mathbf{o}\}  \overset{\mathcal{D}}{=} [D(\mathbf{o}),\mathbf{o}].
    \end{equation*}

    By Proposition \ref{20230128172431}, the distribution of \([D(\mathbf{o}),\mathbf{o}]\) is \(GW(\pi)\).
\end{proof}

\begin{proposition}\label{prop_tgwt_is_typical}
Let \(\alpha\) be a probability measure on \(\{0,1,2,\cdots\}\) such that its mean \(m(\alpha)<1\).
Then, the \(TGWT(\alpha)\) is obtained by re-rooting to a typical vertex of \(GW(\alpha)\).
\end{proposition}
\begin{proof}
    Take an i.i.d. sequence of random variables \(X' = (X'_n)_{n \in \mathbb{Z}}\) such that their common distribution is given by \(\mathbb{P}[X'_0=j]=\alpha(j+1)\) for all \(j \in \{-1,0,1,2,\cdots\}\).
    Let \([\mathbf{T}, \mathbf{o}]\) be the component of \(0\) in the record graph of \([\mathbb{Z},0,X']\) and let \([\mathbf{T}',\mathbf{o}']\) be the Family Tree obtained from \([\mathbf{T},\mathbf{o}]\) by conditioning on the event that \(\mathbf{o}\) does not have a parent.
    By Theorem \ref{20230213191656}, the distribution of \([\mathbf{T}, \mathbf{o}]\) is \(TGWT(\alpha)\), which is unimodular; and by Lemma \ref{lemma_conditioned_is_GW}, the distribution of \([\mathbf{T}',\mathbf{o}']\) is \(GW(\alpha)\), which has finite mean size.
    We show that \([\mathbf{T}, \mathbf{o}]\) is obtained from \([\mathbf{T}',\mathbf{o}']\) by re-rooting to a typical vertex of \(\mathbf{T}'\).

    For any measurable subset \(A\) of \(\mathcal{T}_*\), consider the function \(g_A:\mathcal{T}_{**} \to \mathbb{R}_{\geq 0}\) given by \(g_A([T,u,v]) = \mathbf{1}_A([T,u]) \mathbf{1}\{F(v)=v\}\).
    Then, for any measurable set \(A\), we have,
    \begin{equation*}
        \mathbb{E}\left[\sum_{u \in V(\mathbf{T})}g_A([\mathbf{T},\mathbf{o},u])\right]= \mathbb{E}[\mathbf{1}_A([\mathbf{T},\mathbf{o}])] = \mathbb{P}[[\mathbf{T},\mathbf{o}] \in A],
    \end{equation*}
since there is only one vertex that has no parent.
Similarly,
\begin{align*}
    \mathbb{E}\left[\sum_{u \in V(\mathbf{T})}g_A([\mathbf{T},u,\mathbf{o}])\right] &= \mathbb{E}\left[\mathbf{1}\{F(\mathbf{o})=\mathbf{o}\} \sum_{u \in V(\mathbf{T})}\mathbf{1}_A([\mathbf{T},u])\right]\\
    &= \mathbb{E}\left[ \sum_{u \in V(\mathbf{T}')}\mathbf{1}_A([\mathbf{T}',u])\right] \mathbb{P}[F(\mathbf{o})=\mathbf{o}].
\end{align*}
The last step follows since \([\mathbf{T}',\mathbf{o}']\) is the Family Tree obtained from \([\mathbf{T},\mathbf{o}]\) by conditioning  on the event \(\{F(\mathbf{o})=\mathbf{o}\}\).

By the unimodularity of \([\mathbf{T},\mathbf{o}]\), we have 
\begin{equation*}
    \mathbb{P}[[\mathbf{T},\mathbf{o}] \in A] = \mathbb{E}\left[ \sum_{u \in V(\mathbf{T}')}\mathbf{1}_A([\mathbf{T}',u])\right] \mathbb{P}[F(\mathbf{o})=\mathbf{o}].
\end{equation*}

By taking \(A=\mathcal{T}_*\), we obtain \( \mathbb{P}[F(\mathbf{o})=\mathbf{o}] \mathbb{E}[\#V(\mathbf{T})]=1\).
Thus, for any measurable set \(A\), we have 
\[\mathbb{P}[[\mathbf{T},\mathbf{o}] \in A] = \frac{1}{\mathbb{E}[\#V(\mathbf{T})]}\mathbb{E}\left[ \sum_{u \in V(\mathbf{T}')}\mathbf{1}_A([\mathbf{T}',u])\right].\]
\end{proof}

We now give a characterizing condition for \(TGWT\).
The proof of this proposition depends on the following lemma which says that in order to completely describe a unimodular Family Tree, it is sufficient to describe the distribution of the descendant tree of the root.

\begin{lemma}\label{20230303153138}
    Let $[\mathbf{T},\mathbf{o}]$ be a unimodular Family Tree.
    Then, $[\mathbf{T},\mathbf{o}]$ is completely characterized by $[D(\mathbf{o}),\mathbf{o}]$.
\end{lemma}
\begin{proof}
    Observe that $[D(F^n(\mathbf{o})),\mathbf{o}]$ converges weakly to $[\mathbf{T},\mathbf{o}]$ as $n \to \infty$.
    Indeed, $[D(F^n(\mathbf{o})),\mathbf{o}]_r \overset{\mathcal{D}}{=} [\mathbf{T},\mathbf{o}]_r$ for all $n>r$.

   We have, for any measurable set $A$,
    \begin{align*}
        \mathbb{P}[[D(F^n(\mathbf{o})),\mathbf{o}] \in A] &= \mathbb{E}\left[\sum_{u \in V(\mathbf{T})} \mathbf{1}_A [D(u),\mathbf{o}] \mathbf{1}\{u = F^n(\mathbf{o})\} \right].
    \end{align*}
    By unimodularity, the equation on the right-hand side is equal to
       \[ \mathbb{E}\left[\sum_{u \in V(\mathbf{T})} \mathbf{1}_A [D(\mathbf{o}),u] \mathbf{1}\{\mathbf{o} = F^n(u)\} \right].\]
      Therefore, we have
     \[ \mathbb{P}[[D(F^n(\mathbf{o})),\mathbf{o}] \in A]= \mathbb{E}\left[\sum_{u \in D_n(\mathbf{o})} \mathbf{1}_A[D(\mathbf{o}),u]\right].\]
\end{proof}

Lemma \ref{20230303153138} is closely related to \cite[Proposition 10]{aldousAsymptoticFringeDistributions1991}.
In the language of this work, the latter proposition proves that if the descendant tree \([D(\mathbf{o}),\mathbf{o}]\) of a rooted Family Tree \([\mathbf{T},\mathbf{o}]\) is a fringe distribution, then \([\mathbf{T},\mathbf{o}]\) is completely described by \([D(\mathbf{o}),\mathbf{o}]\).
Its connection to Lemma \ref{20230303153138} follows from the observation: for any unimodular tree \([\mathbf{T},\mathbf{o}]\) of class \(\mathcal{I}/\mathcal{I}\), its descendant tree \([D(\mathbf{o}),\mathbf{o}]\) is a fringe distribution.
However, Lemma \ref{20230303153138} is also applicable to the unimodular Family trees of class \(\mathcal{I}/\mathcal{F}\).
For more details on this connection, see \cite[Bibliographical Comments (Section 6.3)]{baccelliEternalFamilyTrees2018a}.
 
The following proposition and its proof are analogous to the characterization of Eternal Galton-Watson Tree of \cite{baccelliEternalFamilyTrees2018a}.
However, the following proposition differs from the former as its statement is about finite Family Trees, whereas the former is about EFTs.

In the following proposition, only the random objects are denoted in bold letters.
We use the notations of Section \ref{sec_canonical_prob},  Chapter \ref{chapter_prelim}.
For a (deterministic) Family Tree $[T,o]$, the non-descendant tree of the root $o$ is the Family Tree $[D^c(o),o]$, where $D^c(o)$ is the subtree induced by $(V(T)\backslash D(o))\cup \{o\}$.
For any vertex $u$ of $T$, let $c_j(u)$ denote the $j$-th child of $u$.

 \begin{proposition}[Characterization of $TGWT$]\label{20230305162953}
    A random \emph{finite} Family Tree $[\mathbf{T},\mathbf{o}]$ is a Typically rooted Galton-Watson Tree ($TGWT$) if and only if
    \begin{enumerate}
        \item it is unimodular, and
        \item the number of children of the root $d_1(\mathbf{o})$ is independent of the non-descendant tree of the root $\mathbf{o}$.
    \end{enumerate} 
 \end{proposition}
\begin{proof}
    Let $[\mathbf{T},\mathbf{o}]$ be a random Family Tree whose distribution is $TGWT(\pi)$, where $\pi$ is a probability measure on $\{0,1,2,3,\cdots\}$ with $m(\pi)<1$.
    Consider an i.i.d. sequence $X=(X_n)_{n \in \mathbb{Z}}$ of random variables whose common distribution is given by $\mathbb{P}[X_0=n] = \pi(n+1)$ for all $n \in \{-1,0,1,2,\cdots\}$, and let $[\mathbb{Z},0,X]$ be its network.
    Since $\mathbb{E}[X_0]<0$, by Theorem \ref{20230213191656}, the connected component of $0$ in the $R$-graph of $[\mathbb{Z},0,X]$ is $TGWT(\pi)$.
    Therefore, $[\mathbf{T},\mathbf{o}]$ is unimodular by Lemma \ref{lemma:f_graph_unimodular}.
    The second condition is satisfied by $TGWT$, which follows from its construction.

    We now show that if a random finite Family Tree $[\mathbf{T},\mathbf{o}]$ satisfies the above conditions $1$ and $2$, then it is a $TGWT$.

    Let $E$ denote the event that $o$ has a parent.
    For a positive integer $k$, let $A = (A';A_1,A_2,\cdots,A_k)$ be an event of the form
    \begin{equation}\label{eq_event_comp_child}
        d_1(o)=k, D^c(o) \in A', D(c_1(o)) \in A_1,\cdots, D(c_k(o)) \in A_k.
    \end{equation}
    By Lemma \ref{20230303153138}, any unimodular Family Tree is characterized by the descendant tree of its root.
    So, it is sufficient to prove that
    \begin{equation} \label{eq:1_2.35}
        \mathbb{P}[A] = \mathbb{P}[d_1(\mathbf{o})=k] \mathbb{P}[D^c(\mathbf{o}) \in A'] \left(\prod_{i=1}^{k}\mathbb{P}[D(\mathbf{o}) \in A_i]\right),
    \end{equation}
    for any such event \(A\) of the form given by Eq (\ref{eq_event_comp_child}).
    Assume further that the events $A_1,A_2,\cdots,A_k$ depend only up to $n$-th generation of $o$, which is the union of \(D_1(o)\), \(D_2(o)\), \(\cdots\), \(D_n(o)\).
    It suffices to show Eq. (\ref{eq:1_2.35}) for such events because of the local topology.

    We prove the result by induction on $n$.
    For $n=0$, we have $\mathbb{P}[A]= \mathbb{P}[d_1(\mathbf{o})=k]\mathbb{P}[D^c(\mathbf{o}) \in A']$ by condition $2$ of the hypothesis.
    So, assume that the result is true for $n-1$.   
    Let $n \geq 1$ and $1 \leq j \leq k$.
    For any Family Tree $[T,o]$, let $h_j([T,o]):= \mathbf{1}_A[T,F(o)] \mathbf{1}\{o = c_j(F(o))\}$.
    Then, for any $1 \leq j \leq k$, we have
  \begin{equation} \label{eq:2_2.35}
      \mathbf{1}_A[T,o] = \sum_{u \in D_1(o)} \mathbf{1}_A[T,F(u)] \mathbf{1}\{u=c_j(o)\} = \sum_{u \in D_1(o)}h_j([T,u]).
  \end{equation}
  Note that, since \(k \geq 1\), the event \(\{[T,o] \in A\}\) is a subset of the event \(\{D_1(o)\not = \emptyset\}\)
  Applying Eq. (\ref{eq:2_2.35}) for $j=1$, we obtain
  \begin{align}
    \mathbb{P}[A] &= \mathbb{E}\left[\sum_{u \in V(\mathbf{T})} \mathbf{1}_A[\mathbf{T},F(u)] \mathbf{1}\{u = c_1(\mathbf{o})\} \mathbf{1}\{u \in D_1(\mathbf{o})\}\right] \notag\\
    &= \mathbb{E}\left[\sum_{u \in V(\mathbf{T})}\mathbf{1}_A[\mathbf{T},F(\mathbf{o})] \mathbf{1}\{\mathbf{o} = c_1(u)\} \mathbf{1}\{\mathbf{o} \in D_1(u)\}\right] \label{eq:3_2.35}\\
    &= \mathbb{E}\left[\mathbf{1}_A[\mathbf{T},F(\mathbf{o})] \mathbf{1}\{ \mathbf{o} \in E\}  \mathbf{1}\{\mathbf{o} = c_1(F(\mathbf{o}))\} \right] \notag\\
    &= \mathbb{P}[\mathbf{o} \in E, \mathbf{o} = c_1(F(\mathbf{o})),[\mathbf{T},F(\mathbf{o})] \in A ]\label{eq:4_2.35}\\
    &= \mathbb{P}[D(\mathbf{o}) \in A_1, \mathbf{o} \in E, \mathbf{o} = c_1(F(\mathbf{o})),[\mathbf{T},F(\mathbf{o})] \in A(A';\mathcal{T}_*,A_2,\cdots,A_k) ] \notag.
  \end{align}
  In the above, Eq. (\ref{eq:3_2.35}) follows from unimodularity.
  Note that $D(o) \in A_1$ depends on one generation less that of $[T,F(o)] \in A$.
  Therefore, by induction, we obtain,
  \begin{align*}
    \mathbb{P}&[A]\\
    &= \mathbb{P}[D(\mathbf{o}) \in A_1] \mathbb{P}[\mathbf{o} \in E, \mathbf{o} = c_1(F(\mathbf{o})),[\mathbf{T},F(\mathbf{o})] \in A(A';\mathcal{T}_*,A_2,\cdots,A_k)]\\
    &= \mathbb{P}[D(\mathbf{o}) \in A_1] \mathbb{P}[A(A';\mathcal{T}_*,A_2,\cdots,A_k)].
  \end{align*}
  The last equation is obtained by Eq. (\ref{eq:4_2.35}).

  Applying Eq.(\ref{eq:2_2.35}) with $j=2$ to $A(A';\mathcal{T}_*,A_2,\cdots,A_k)$, we obtain
  \begin{align*}
    \mathbb{P}[A&(A';\mathcal{T}_*,A_2,\cdots,A_k)]\\
     &= \mathbb{P}[D(\mathbf{o}) \in A_2]\mathbb{P}[\mathbf{o} \in E, \mathbf{o} = c_1(F(\mathbf{o})), [\mathbf{T},F(\mathbf{o})] \in A(A';\mathcal{T}_*,\mathcal{T}_*,\cdots,A_k)]\\
    &=\mathbb{P}[D(\mathbf{o}) \in A_2]\mathbb{P}[A(A';\mathcal{T}_*,\mathcal{T}_*,\cdots,A_k)].
  \end{align*}
 By iterating the same procedure for $j=3,4,\cdots$, we obtain,
 \begin{align*}
    \mathbb{P}[A] &= \left(\prod_{i=1}^{k} \mathbb{P}[D(\mathbf{o}) \in A_i]\right) \mathbb{P}[d_1(\mathbf{o})=k, D^c(\mathbf{o}) \in A']\\
    &= \left(\prod_{i=1}^{k} \mathbb{P}[D(\mathbf{o}) \in A_i]\right) \mathbb{P}[d_1(\mathbf{o})=k] \mathbb{P}[ D^c(\mathbf{o}) \in A'].
 \end{align*}
Thus, we showed Eq. (\ref{eq:1_2.35}), which completes the proof. 
 \end{proof}

 \subsubsection{Zero mean}\label{subsec:r_graph_egwt}
 \begin{theorem} \label{theorem:R-graph_egwt}
    Let $X=(X_n)_{n \in \mathbb{Z}}$ be a sequence  of random variables as in Def. \ref{hyp:increments} with $\mathbb{E}[X_0]=0$, $[\mathbb{Z},0,X]$ be its associated network, and $[\mathbf{T},0]$ be the connected component of $0$ in the $R$-graph of $[\mathbb{Z},0,X]$.
    Then, $[\mathbf{T},0]$ is the ordered Eternal Galton-Watson Tree $EGWT(\pi)$ with offspring distribution $\pi \stackrel{\mathcal{D}}{=}X_{0}+1$.
  \end{theorem}
  \begin{proof}
    We apply Theorem \ref{theorem:characterisation_egwt}.
    We show that \([\mathbf{T},0]\) satisfies the second condition of Theorem \ref{theorem:characterisation_egwt}.
    Observe that $L_X(0)$ is a stopping time for the random walk with increments $(X_{-n})_{n\geq 1}$ and \(-\infty<L_X(0)\) a.s..
    By the strong Markov property, the random variables $Y=(X_{-n})_{n > -L_X(0)}$ are independent of  $Z=(X_{-n})_{n=1}^{n=L_X(0)}$.
    Let $(\mathbf{T}',0)$ denote the non-descendant tree of \(o\), i.e., the tree generated by $(\mathbf{T}\backslash D(0))\cup \{0\}$.
    By Lemma \ref{lemma:descendants}, the subtree $(D(0),0)$ is a function of $Z$ and the subtree $(\mathbf{T}',0)$ is a function of $Y \cup (X_n)_{n \geq 0}$.
    So, the unordered tree $(D(0),0)$ is independent of the unordered tree $(\mathbf{T}',0)$.
    The second part of Lemma \ref{lemma_offspring_count} implies that, conditionally on the event that $0$ has children, the order of a child is independent of $Y \cup (X_n)_{n \geq 0}$, whereas the order of any vertex in $(\mathbf{T}',0)$ is a function of $Y \cup (X_n)_{n \geq 0}$.
    Thus, the ordered subtree $(D(0),0)$ is independent of the ordered subtree $(\mathbf{T}',0)$.
    So, the sufficient condition of Theorem \ref{theorem:characterisation_egwt} is satisfied by \([\mathbf{T},0]\).
    
    The first statement of Lemma \ref{lemma_offspring_count} implies that $d_1(0) = X_{-1}+1$, and thus completes the proof.
  \end{proof}

\subsubsection{Positive mean}\label{subsec_positive_mean}

Let $X=(X_n)_{n \in \mathbb{Z}}$ be the i.i.d. increments of a skip-free random walk with mean $0<\mathbb{E}[X_{0}]< \infty$, and $[\mathbb{Z},0,X]$ be its associated network.

Let $\eta_{-1} = \inf\{n> 0: S_n = -1\}$ if the infimum exists and \(\infty\) if \(S_n > -1 \, \forall n \geq 1\), and let $c = \mathbb{P}[\eta_{-1}< \infty]$.
Let $[\mathbf{T},\mathbf{o}]$ be the connected component of \(0\) in the $R$-graph of $[\mathbb{Z},0,X]$, with root $0 = \mathbf{o}$.

Let $\tilde{\pi},\bar{\pi}$ be the probability distributions on \(\{0,1,2,\cdots\}\) given by:
\begin{align} \label{eq:pi_defn}
 \tilde{\pi}(k) &= c^{k-1} \mathbb{P}[X_{0}=k-1],\\ \nonumber
 \bar{\pi}(k)&= \mathbb{P}[X_{0}\geq k] c^{k}.
\end{align}

Let $[\mathbf{T}_0,\mathbf{o}_0]$ be a Family Tree whose distribution is given by the following: the offspring distribution of $\mathbf{o}_0$ is \(\bar{\pi}\).
The descendant trees of the children of $\mathbf{o}_0$ are independent Galton-Watson trees with offspring distribution $\tilde{\pi}$, and they are independent of the offspring distribution of \(\mathbf{o}_0\).
Let \(([\mathbf{T}_i,\mathbf{o}_i])_{i \in \mathbb{Z}}\) be an i.i.d. sequence of Family Trees.
Recall that the distribution of the Family Tree obtained from the typically rooted joining of \(([\mathbf{T}_i,\mathbf{o}_i])_{i \in \mathbb{Z}}\) is the unimodularised \(EKT(\bar{\pi},\tilde{\pi})\).

\begin{theorem}\label{r_graph_positive_drift_20230203174636}
   If  \(0<\mathbb{E}[X_{0}]< \infty\), then the distribution of $[\mathbf{T}, \mathbf{o}]$ is the unimodularised \(EKT(\bar{\pi},\tilde{\pi})\).
\end{theorem}

\begin{proof}
   Since the random walk has a finite positive drift, the unimodular EFT $[\mathbf{T},\mathbf{o}]$ is of class $\mathcal{I}/\mathcal{F}$, see Theorem \ref{remark_pos_drift_i_f}.
   Let $\mathbf{o} \in \neswarrow$ denote the event that $\mathbf{o}$ belongs to the bi-infinite path of $\mathbf{T}$.
   Let $[\mathbf{T}',\mathbf{o}']$ denote the Family Tree obtained from $[\mathbf{T},\mathbf{o}]$ by conditioning on the event $\mathbf{o} \in \neswarrow$, denoted as \([\mathbf{T},\mathbf{o}]|\mathbf{o} \in \neswarrow\).
   By Theorem \ref{thm_eft_I_f_joining}, $[\mathbf{T}',\mathbf{o}']$ is the joining of some stationary sequence of Family Trees $([\mathbf{T}'_i,\mathbf{o}'_i])_{i \in \mathbb{Z}}$.
   The proof is complete once we prove that this sequence is i.i.d., and $[\mathbf{T}'_1,\mathbf{o}'_1] \overset{\mathcal{D}}{=} [\mathbf{T}_0,\mathbf{o}_0]$.
   
   \noindent \emph{Step 1: $([\mathbf{T}'_i,\mathbf{o}'_i])_{i \in \mathbb{Z}}$ is an i.i.d. sequence:}
   Since the increments have positive mean, $S_n \to +\infty$ as $n \to + \infty$ and $S_n \to - \infty$ as $n \to -\infty$.
   By Lemma \ref{lemma:descendants}, $\mathbf{o} \in \neswarrow$ if and only if \(L_X(0) = \infty\), and the latter condition holds  if and only if $S_n \leq 0, \forall n \leq -1$.

   For any $u \in \mathbb{Z}$, let $\mathfrak{S}(u)$ denote the (random) subtree of the descendants of $u$ (including \(u\)) in $\mathbf{T}$. 
   The distribution of $([\mathbf{T}'_i,\mathbf{o}'_i])_{i \geq 1}$ is the same as that of 
   \[([\mathfrak{S}(F^n(\mathbf{o}))\backslash \mathfrak{S}(F^{n-1}(\mathbf{o})),F^n(\mathbf{o})])_{n \geq 1}|\mathbf{o} \in \neswarrow\]
   (distribution obtained by conditioning on \(\mathbf{o} \in \neswarrow\)), where \(F\) is the parent vertex-shift on \(\mathbf{T}\) with the convention that $F^0(\mathbf{o})= \mathbf{o}$ (and \(F^n(\mathbf{o})\) is the parent of order \(n\) of \(\mathbf{o}\)).
   On the event $\mathbf{o} \in \neswarrow$, each subtree $\mathfrak{S}(F^n(\mathbf{o}))\backslash \mathfrak{S}(F^{n-1}(\mathbf{o}))$ is completely determined by the part of the random walk
   \[(S_{R^{n-1}(0)+1}-S_{R^{n-1}(0)}, \ldots, S_{R^{n}(0)}-S_{R^{n-1}(0)}),\]
   for each $n \geq 1$ (see Figure \ref{fig_I_f_trajectory} for an illustration).
The distribution of the sequence 
   \[((S_{R^{n-1}(0)+1}-S_{R^{n-1}(0)}, \ldots, S_{R^{n}(0)}-S_{R^{n-1}(0)}))_{n \geq 1}\]conditioned on $\{S_n \leq 0: n \leq -1\}$ is the same as that of the unconditioned sequence (the same sequence without the condition) because the sequence is independent of the event on which it is conditioned.
   Moreover, it is an i.i.d. sequence by the strong Markov property of the random walk.
   This implies that $([\mathbf{T}'_i,\mathbf{o}'_i])_{i \geq 1}$ is an i.i.d. sequence.
   \begin{figure}[h]
    \begin{center}
          \includegraphics[scale=0.7]{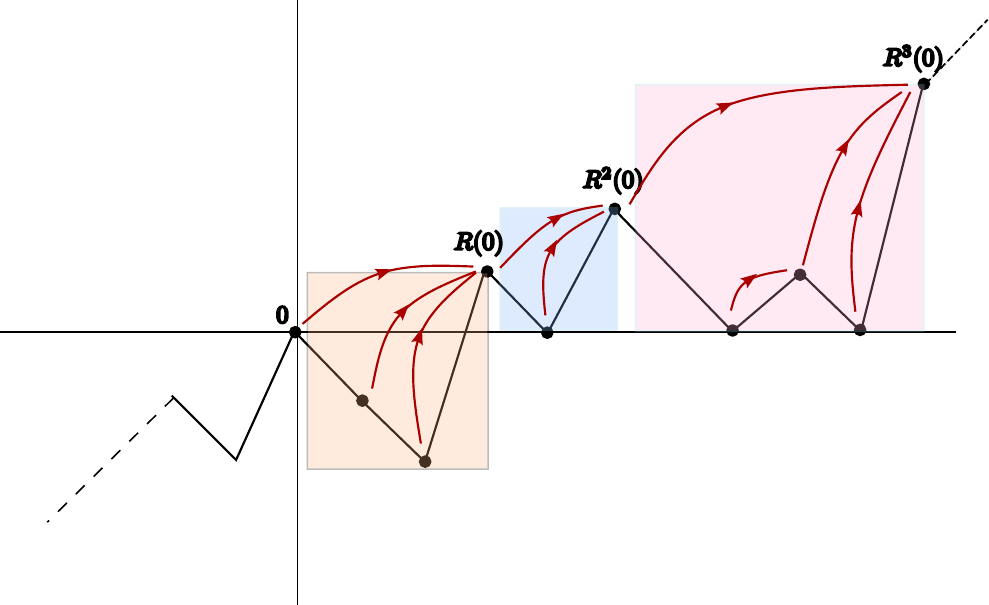}    
    \end{center}
        \caption{Illustration of the subtrees of ancestors of \(0\) on a trajectory conditioned on the event that \(S_n \leq 0\) for all \(n \leq 0\) when \(\mathbb{E}[X_0]>0\). Each box represents the subtrees \(\mathbf{T}_1,\mathbf{T}_2, \mathbf{T}_3, \cdots\) (see Step 1 of Theorem \ref{r_graph_positive_drift_20230203174636}).}  
        \label{fig_I_f_trajectory}
    \end{figure}

   Further, the stationarity of $([\mathbf{T}_i',\mathbf{o}_i'])_{i \in \mathbb{Z}}$ implies that $([\mathbf{T}_i',\mathbf{o}_i'])_{i \in \mathbb{Z}}$  is an i.i.d. sequence (the index set here is \(\mathbb{Z}\), whereas in the previous sentence it is \(i \geq 1\)).

   \noindent \emph{Step 2: \([\mathbf{T}'_1, \mathbf{o}'_1] \overset{\mathcal{D}}{=} [\mathbf{T}_0,\mathbf{o}_0]\):}
   We first show that the offspring distribution of $\mathbf{o}'_1$ is \(\bar{\pi}\).
   Let $\tau = \inf\{n>0: S_n \geq 0\}$ and \(\infty\) if \(S_n<0,\, \forall n>0\), and $X_{\tau-1}$, $S_{\tau}$ be the increment and location of the random walk at this time.
   Since the increments have positive mean, $\tau< \infty$ a.s., and therefore \(R(0)=\tau\) a.s. (if \(\tau=\infty\) then \(R(0)\) is defined to be \(0\)). 

   We also have the equality of the following events,
   \[\{d_1(F(\mathbf{o}))=1\} = \{\tau = 1\} = \{X_{\tau-1} = S_{\tau}\} = \{S_1 \geq 0\}.\] 
   Therefore,
   \[\mathbb{P}[d_1(F(\mathbf{o}))=1, \mathbf{o} \in \neswarrow] = \sum_{j=0}^{\infty}(\mathbb{P}[X_0 = j] \mathbb{P}[ \mathbf{o} \in \neswarrow]) = \mathbb{P}[X_0 \geq 0] \mathbb{P}[\mathbf{o} \in \neswarrow].\]
   We now study the event \(d_1(F(\mathbf{o}))=k\) with \(k>1\) (since \(0\) is the child of \(\tau\), the event \(\mathbb{P}[d_1(\tau) = 0]=0\)).
   Observe that, on the event $\mathbf{o} \in \neswarrow$, no negative integer is a child of $F(\mathbf{o})$.
   To see this, it is sufficient to show that \(L_X(0)=- \infty\). Because, the latter condition implies that all the non-negative integers are the descendants of \(0\) by Lemma \ref{lemma:descendants}.
   So, none of them are the children of \(\tau\) (see Figure \ref{fig_I_f_trajectory}).
   The latter condition is equivalent to \(S_n \leq 0,\, \forall n<0\), which is further equivalent to \(\mathbf{o} \in \neswarrow\).
   Thus, the observation follows. 
   This observation implies that \(0\) is the last child of \(\tau\).
   So, \(l_X(\tau) = 0\) and \(t_X(\tau) = S_{\tau} \geq 0\).
   Using Lemma \ref{20230405112810}, we obtain the relation \(d_1(\tau)=X_{\tau-1}+1-S_{\tau}\).
   Use this relation to get
   \[\mathbb{P}[d_1(F(\mathbf{o}))=k, \mathbf{o} \in  \neswarrow] = \sum_{j=k-1}^{\infty}\mathbb{P}[X_{\tau-1}=j, S_{\tau}= j-(k-1)] \mathbb{P}[\mathbf{o} \in \neswarrow].\]
   We apply Lemma \ref{20230119141633} to the first term inside the sum to get 
   \[\mathbb{P}[X_{\tau-1}=j, S_{\tau}= j-(k-1)] = \mathbb{P}[X_0 = j]c^{k-1},  0<k-1 \leq j.\]
   Using this, we obtain
   \begin{align*}
       \mathbb{P}[d_1(F(\mathbf{o}))=k, \mathbf{o} \in  \neswarrow] &= \sum_{j=k-1}^{\infty}\mathbb{P}[X_0=j]c^{k-1}\mathbb{P}[\mathbf{o} \in \neswarrow]\\
       &= \mathbb{P}[X_0 \geq k-1] c^{k-1} \mathbb{P}[\mathbf{o} \in \neswarrow].
   \end{align*}
   Thus, the offspring distribution of \(\mathbf{o}_1'\) in the tree \(T_1'\) is given by: for any $k \geq 1$,
   \begin{align*}
    \mathbb{P}[d_1(\mathbf{o}_1',T_1')=k-1]=\mathbb{P}[d_1(\mathbf{o}_1',T') = k] &= \mathbf{P}[d_1(F(\mathbf{o}))=k|\mathbf{o} \in \neswarrow]\\ 
    &= \mathbb{P}[X_0 \geq k-1]c^{k-1},
   \end{align*}
   which proves the first result.

   We now show that conditioned on $\{d_1(\mathbf{o}_1',T')=k\}$, the descendant subtrees of the children of $\mathbf{o}_1'$ are independent Galton-Watson trees with offspring distribution $\tilde{\pi}$.

   On $\{d_1(F(\mathbf{o}))=k\}$, let $0=i_k<i_{k-1}< \cdots<i_1$ be the children of $F(\mathbf{o})$.
   Observe that, on the event $\{d_1(F(\mathbf{o}))=k\}$, for each $1 \leq j \leq k$, the part of the random walk $(0, S_{i_j-1}-S_{i_j},S_{i_j-2}-S_{i_j},\ldots, S_{i_{(j+1)}}-S_{i_j})$ is an excursion set because, by the third part of Lemma \ref{lemma_offspring_count}, $S_{i_{(j+1)}}-S_{i_{j}}=1$ and $S_l \leq S_{i_{(j+1)}}$ for all $i_{(j+1)}<l \leq i_j$.
   Further, these excursion sets are mutually independent because the times $i_j$, for $1 \leq j \leq k$, are stopping times.
   Therefore, for each $j$ chosen above, $(0, S_{i_j-1}-S_{i_j},S_{i_j-2}-S_{i_j},\cdots, S_{i_{(j+1)}}-S_{i_j})$ obtained by conditioning on \(\{d_1(F(\mathbf{o}))=k\}\) is the conditioned skip-free to the right random walk of $(S_{-n})_{n \geq 0}$ conditioned on $\eta_{1}< \infty$ and stopped at $\eta_1$, where $\eta_1 = \min\{n >0: S_{-n}=1\}$ and \(-\infty\) if \(S_{-n}<1\) for all \(n>0\).
   By Lemma \ref{20230113184856} (and applying the same to $(S_n)_{n \geq 0}$ and $\eta_{-1}$), the conditioned random walk is an unconditioned random walk $(\hat{S}_n)_{n \geq 0}$ whose increments $(\hat{X}_n)_{n \geq 1}$ have distribution $\mathbb{P}[\hat{X}_1=k] = \mathbb{P}[X_1=k]c^{k}$, for all $k \geq -1$.
   Moreover, by Lemma \ref{20230116135048}, the random walk $(\hat{S}_n)_{n \geq 0}$ has negative drift.
   Therefore, by Proposition \ref{20230128172431}, the descendant subtree of the children of \(\mathbf{o}_1'\) is a Galton-Watson tree with offspring distribution $\tilde{\pi}$.
   In particular, the mean \(m(\tilde{\pi}) < 1\) by Lemma \ref{20230116135048}.
   So, \(GW(\tilde{\pi})\) is a.s. finite.

\end{proof}

We now focus on the distribution of the type function associated to a trajectory of the random walk and on the order of the record graph when \(\mathbb{E}[X_0]>0\).
We enrich the record graph by assigning the type function as the mark function to the record graph.
It will be shown later that the knowledge of the type function and of the record graph of a trajectory is sufficient to construct the negative part of the trajectory.
The following theorem shows that the distribution of the record graph with type function as its mark function is a bi-variate ECS ordered marked Eternal Kesten Tree.
The latter is characterized by the joint distribution of the mark function restricted to the bush of the root and the distribution of the (unlabelled) bush of the root  by Remark \ref{remark_characterization_bi_ecs_EKT} and the discussion above it.

\begin{theorem}\label{20230226170801}
  Let \(0<\mathbb{E}[X_0]<\infty\).
  Let \([\mathbf{T},\mathbf{o},t]\) be the component of \(0\) in the record graph of \([\mathbb{Z},0,X]\) with \(t\), the type function associated to \(X\), as the mark function.
  Let \([\mathbf{T}',\mathbf{o}',t']\) be the marked Family Tree obtained from \([\mathbf{T},\mathbf{o},t]\) by conditioning on \(\mathbf{o} \in \neswarrow\).
  Then, the distribution of \([\mathbf{T}',\mathbf{o}',t']\) is \(MEKT(\bar{\pi},\tilde{\pi})\) where the marked bush \([\hat{\mathbf{T}}_0,\hat{\mathbf{o}}_0,m_0]\) of the root has the following distribution: the offspring distribution of \(\hat{\mathbf{o}}_0\) is \(\bar{\pi}\); the descendant trees of the children of \(\hat{\mathbf{o}}_0\) are i.i.d. \(GW(\tilde{\pi})\) and they are independent of \(d_1(\hat{\mathbf{o}}_0)\) (see Eq. (\ref{eq:pi_defn}) for the definitions of \(\bar{\pi}\) and \(\tilde{\pi}\)); the mark of every vertex other than \(\hat{\mathbf{o}}_0\) is \(-1\); the mark of \(\hat{\mathbf{o}}_0\) is independent of the descendant subtrees of children of \(\hat{\mathbf{o}}_0\), i.e., 
  \[\mathbb{P}[m_0(\hat{\mathbf{o}}_0)|[\hat{\mathbf{T}}_0, \hat{\mathbf{o}}_0]] = \mathbf{P}[m_0(\hat{\mathbf{o}}_0)|d_1(\hat{\mathbf{o}}_0)];\]
   and the joint distribution of the mark of \(\hat{\mathbf{o}}_0\) and \(d_1(\hat{\mathbf{o}}_0)\) is given by 
  \begin{equation} \label{eqn_joint_dist_type_offsprings}
    \mathbb{P}[m_0(\hat{\mathbf{o}}_0)=n, d_1(\hat{\mathbf{o}}_0)=k] = \mathbb{P}[X_0=n+k]c^k,\,  \forall n \geq 0, k \geq 0.
  \end{equation}
\end{theorem}

\begin{proof}
Forgetting the marks and the order of vertices of $\mathbf{T}'$, it follows from Theorem \ref{r_graph_positive_drift_20230203174636} that $[\mathbf{T}',\mathbf{o}']$ is the joining of the i.i.d. sequence $([\mathbf{T}_i,\mathbf{o}_i])_{i \in \mathbb{Z}}$, where $[\mathbf{T}_0,\mathbf{o}_0]$ has the same distribution as that of $[\hat{\mathbf{T}}_0,\hat{\mathbf{o}}_0]$, the random Family Tree obtained after forgetting the marks of $V(\hat{\mathbf{T}}_0)$.
Let \(t_i\) be the restriction of the type function \(t'\) to \(V(\mathbf{T}_i)\) for all \(i \in \mathbb{Z}\).

Observe that an integer \(j\) belongs to the bi-infinite path of \(\mathbf{T}\) if and only if \(t(j) \geq 0\).
This follows because, for any \(j \in \mathbb{Z}\), its type \(t(j)\) is non-negative  if and only if \(y(n,j) \geq 0\) for all \(n<j\).
The latter condition is satisfied if and only if \(L(j)=-\infty\).
By Lemma \ref{lemma:descendants}, \(L(j)=-\infty\) if and only if \(j\) has infinitely many descendants.
In view of this, it is sufficient to focus on the type of vertices that lie on the bi-infinite path since the remaining vertices have fixed type \(-1\).

\noindent By the above argument, if the sequence of random variables  \((t(\mathbf{o}_i), d_1(\mathbf{o}_i))_{i \in \mathbb{Z}}\) is i.i.d., then the sequence \(([\mathbf{T}_i,\mathbf{o}_i,t_i])_{i \in \mathbb{Z}}\) is i.i.d..
So, the statement of the theorem is proved once we show that
\begin{itemize}
    \item  (Step 1:) the order of $\mathbf{o}_{i}$ is the smallest among the children of $\mathbf{o}_{i+1}$, and
    \item (Step 2:) the sequence \(((t(\mathbf{o}_i), d_1(\mathbf{o}_i))_{i \in \mathbb{Z}})\) is i.i.d. with the common distribution given by Eq. (\ref{eqn_joint_dist_type_offsprings}).
  \end{itemize}

 Let $\mathbf{o} \in \neswarrow$ denote the event that the root $\mathbf{o}$ belongs to the bi-infinite path of $\mathbf{T}$.   

 \noindent \emph{Step 1:}
 For any $i \in \mathbb{Z}$ with \(t(i) \geq 0\), let $l(i) :=l_X(i)= \max\{m<i: y(m,i)=t(i)\}$ and $l^n(i):=l^{n-1}(l(i))$ for all $n>1$.
 By Lemma \ref{smallest_child_20230404185835}, it follows that $l(i)$ is the smallest among the children of $i$.

 On the event $\{S_n\leq 0 \quad \forall n<0\}$ (which is equivalent to $\mathbf{o} \in \neswarrow$), we have $t(0)\geq 0$.
 Since $S_m \leq S_{l^n(0)}$ for all $m \leq l^n(0), n \geq 1$, we have $t(l^n(0))\geq 0$ for all $n \geq 1$.
 So, by the above discussion, on the event $\mathbf{o} \in \neswarrow$, we get that $R(l^n(0))=l^{n-1}(0)$ for all $n \geq 1$ (with the notation $l^0(0)=0$).
 Moreover, on the same event, the vertices of the bi-infinite $F$-path of $\mathbf{T}$ are $\{l^n(0): n \geq 1\} \cup \{0\} \cup \{R^m(0): m \geq 1\}$. 
 Since, on the event $\mathbf{o} \in \neswarrow$, for any $m \geq 1$, $S_{R^m(0)}-S_n > S_{R^m(0)}-S_{R^{m-1}(0)} $ for all $R^{m-1}(0) < n < R^m(0)$, we have $l(R^m(0))=R^{m-1}(0)$ (with the notation that $R^0(0)=0$).
 Thus, $\mathbf{o}_i$ has the smallest order among the children of $\mathbf{o}_{i+1}$ for all $i \in \mathbb{Z}$.
 Since the order of $(\mathbf{o}_i)_{i\in \mathbb{Z}}$ characterizes the order on $[\mathbf{T}', \mathbf{o}']$, this completes Step 1.

 \noindent \emph{Step 2:}
  Firstly, we compute the joint distribution of \((t(\mathbf{o}_0),d_1(\mathbf{o}_0))\).
 Note that \(d_1(\mathbf{o}_0) = d_1(\mathbf{o}')-1=(d_1(\mathbf{o})-1)|\mathbf{o} \in \neswarrow\). 
 By Lemma \ref{20230405112810}, we have, for $n \geq 0, k \geq 0$,
 \begin{align*}
   \mathbb{P}[\mathbf{o} \in \neswarrow, t(\mathbf{o})=n, d_1(\mathbf{o})=k+1] &=  \mathbb{P}[\mathbf{o} \in \neswarrow, X_{-1}=k+n, t(\mathbf{o})=n].
 \end{align*}
 The right-hand term in the above is equal to
 \[ \mathbb{P}[X_{-1}=k+n] \mathbb{P}_{-(k+n)}[S_m = -n \text{ for some } m \leq 0] \mathbb{P}_{-n}[S_m \leq -n \  \forall m \leq 0],\]
 where the second term in the above product is the probability that the skip-free to the right random walk starting from $-(n+k)$ hits $-n$, whereas the third term is the probability that the same random walk starting from $-n$ never hits $-n+1$.
 Lemma \ref{hitting_time_20230118184651} implies that the second term is equal to $c^k$ and the third term is equal to $(1-c)$.
 Therefore, for $n \geq 0$,
 \[\mathbb{P}[\mathbf{o} \in \neswarrow,t(\mathbf{o})=n, d_1(\mathbf{o})=k+1] = \mathbb{P}[X_{-1}=n+k]c^k(1-c). \]
 Since \(\mathbb{P}[\mathbf{o}\in \neswarrow] = 1-c\), we obtain, for $n \geq 0$,
 \begin{align*}
    \mathbb{P}[t(\mathbf{o}_0)=n, d_1(\mathbf{o}_0)=k]&=\mathbb{P}[t(\mathbf{o}')=n, d_1(\mathbf{o}')=k+1]\\
    &=\mathbb{P}[t(\mathbf{o})=n, d_1(\mathbf{o})=k+1|\mathbf{o} \in \neswarrow] \\
    &=\mathbb{P}[X_{-1}=n+k]c^k.
 \end{align*}

 \begin{figure}[h]
  \begin{center}
        \includegraphics[scale=0.75]{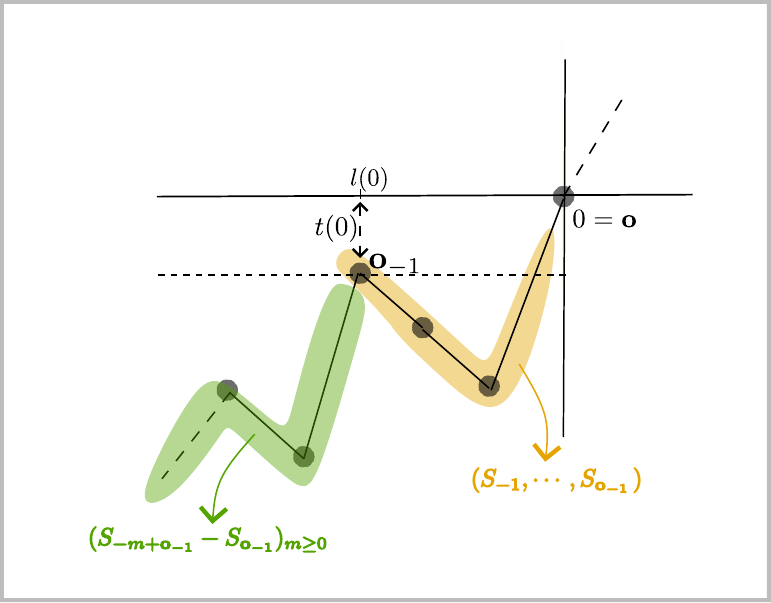}    
  \end{center}
      \caption{The conditioned part of the trajectory. An illustration showing that the distribution of \((S_{-m+\mathbf{o}_{-1}}-S_{\mathbf{o}_{-1}})_{m \geq 0}|\{\mathbf{o} \in \neswarrow, t(\mathbf{o})=n\}\) is the same as that of \((S_{-m})_{m \geq 0}|\{\mathbf{o}\in \neswarrow\}\) and independent of \((S_{-1},\cdots,S_{\mathbf{o}_{-1}})|\{\mathbf{o}  \in \neswarrow\}\).}
      \label{fig_cond_rw_is_cond}
  \end{figure}

 Secondly, observe that the distribution of $(S_{-m+\mathbf{o}_{-1}}-S_{\mathbf{o}_{-1}})_{m \geq 0}$ conditioned on $\{\mathbf{o} \in \neswarrow, t(\mathbf{o})=n\}$ is the same as that of $(S_{-m})_{m \geq 0}$ conditioned on $\mathbf{o} \in \neswarrow$. 
 This follows because
 \[(S_{-m+\mathbf{o}_{-1}}-S_{\mathbf{o}_{-1}})_{m \geq 0} \cap \{\mathbf{o} \in \neswarrow,t(\mathbf{o})=n\}\]
 is a conditioned skip-free to the right random walk starting at $0$ conditioned on the event that the random walk stays below $1$ (see Figure \ref{fig_cond_rw_is_cond}).
 Further, this conditioned random walk is independent of $(S_0,S_{-1}, \cdots,S_{\mathbf{o}_{-1}})|\mathbf{o} \in \neswarrow$ (the process conditioned on the event $\mathbf{o} \in \neswarrow$).
 The joint-distribution of \((t(\mathbf{o}_0), d_1(\mathbf{o}_0))\), however, is a function of \((S_0,S_{-1}, \cdots,S_{\mathbf{o}_{-1}})\)
 Thus, by iteratively applying the above argument to $(S_{-m+\mathbf{o}_{-1}}-S_{\mathbf{o}_{-1}})_{m \geq 0}$ conditioned on $\{\mathbf{o} \in \neswarrow,t(\mathbf{o})=n\}$, we obtain that the sequence of random variables $((t(\mathbf{o}_k),d_1(\mathbf{o}_k)))_{k \leq 0}$ is i.i.d.

 Thirdly, the distribution of the sequence of random variables\\ $((t(\mathbf{o}_n),d_1(\mathbf{o}_n)))_{n \geq 1}$ (conditioned on \(\mathbf{o}\in \neswarrow\)) is same as that of 
 \[((S_{R^n(0)}-S_{R^{n-1}(0)},X_{(R^n(0)-1)}-(S_{R^n(0)}-S_{R^{n-1}(0)})))_{n \geq 1}|\{\mathbf{o}\in \neswarrow\}\]
 (conditioned on the event $\mathbf{o} \in \neswarrow$), with the notation that $R^0(0)=0$.
 Note that the latter process has the same distribution as that of unconditioned process $((S_{R^n(0)}-S_{R^{n-1}(0)},X_{(R^n(0)-1)}-(S_{R^n(0)}-S_{R^{n-1}(0)})))_{n \geq 1}$ as it is independent of the event ${\mathbf{o} \in \neswarrow}$.
 But $((S_{R^n(0)}-S_{R^{n-1}(0)},X_{(R^n(0)-1)}-(S_{R^n(0)}-S_{R^{n-1}(0)})))_{n \geq 1}$ is i.i.d. by the strong Markov property (since $R(0)$ is a stopping time).
 Thus,  $((t(\mathbf{o}_n),d_1(\mathbf{o}_n)))_{n \geq 1}$ is an i.i.d. sequence.

 The proof is complete once we show that 
 \[\mathbb{P}[t(\mathbf{o}_1)=n,d_1(\mathbf{o}_1)=k] = \mathbb{P}[t(\mathbf{o}_0)=n,d_1(\mathbf{o}_0)=k],\]
for all $n\geq 0,\, k \geq 0$.
For $n \geq 0,\, k \geq 0$,
 \begin{align*}
   \mathbb{P}[t(\mathbf{o}_1)=n,d_1(\mathbf{o}_1)=k] &= \mathbb{P}[S_{R(0)}=n,X_{(R(0)-1)}-S_{R(0)}=k|\mathbf{o}\in \neswarrow]\\
   &=\mathbb{P}[S_{R(0)}=n,X_{(R(0)-1)}-S_{R(0)}=k] \text{ (by independence)}\\
   &= \mathbb{P}[S_{R(0)}=n,X_{R(0)-1}=n+k]\\
   &=\mathbb{P}[X_0=n+k]c^k= \mathbb{P}[t(\mathbf{o}_0)=n,d_1(\mathbf{o}_0)=k].\\
\end{align*}
We used Lemma \ref{20230119141633} in the last equation.
This concludes the proof.
\end{proof}

\section{Record foil structure and a measure preserving map}

In this section, we define the vertex-shift \(R_{\perp}\) on \([\mathbb{Z},0,X]\), where \(X\) is the i.i.d. sequence, such that its induced map \(\theta_{R_{\perp}}\) preserves the distribution of \([\mathbb{Z},0,X]\).
This question is of great interest in the field of stochastic geometry, see \cite{thorissonCouplingStationarityRegeneration2000} for more details.
\vskip3mm

\noindent \textbf{R-orthogonal map ($R_{\perp}$):}
Let $[\mathbb{Z},0,X]$ be a network, where $X=(X_n)_{n \in \mathbb{Z}}$ as in Def. \ref{hyp:increments} and satisfies $\mathbb{E}[X_0]=0$.
Consider the $R$-foliation of $(\mathbb{Z},X)$ (see Figure \ref{fi:foils} for an illustration).
Observe that the $R$-foliation is the same as the $F$-foliation of $\Psi_R([\mathbb{Z},0,X])$, where $F$ is the Parent vertex-shift and $\Psi_R([\mathbb{Z},0,X])$ is the connected component of $0$ in the $R$-graph of $[\mathbb{Z},0,X]$.
The order of $\mathbb{Z}$ naturally induces an order on an $R$-foil as the latter is a subset of the former.
This order is the same as the RLS order by Lemma \ref{lemma:rls_order}.
This motivates us to define a new map $R_{\perp}$ on the vertices of $\mathbb{Z}$ in the following way:
for an integer $u$, let $R_{\perp}(u)$ be the smallest integer of the $R$-foil of $u$ which is  larger than $u$ (if no such integer exists then define $R_{\perp}(u)=u$).
Since $R$ is a vertex-shift, $R_{\perp}$ is a vertex-shift.

On the random network $(\mathbb{Z},X)$, the map $R_{\perp}$ is the same as the following map, 
\[i \mapsto \begin{cases}
  \inf\{k>i: R_X^n(i) = R_X^n(k) \text{ for some } n\geq 1 \} \text{ if infimum exists}\\
  i \text{ otherwise},
\end{cases}\]
for all $i \in \mathbb{Z}$.
In Proposition \ref{prop:R_orthogonal_measure_preserving}, we show that $R_{\perp}$ is measure preserving.

\begin{figure}[!h]
  \centering 
  \includegraphics[scale=1]{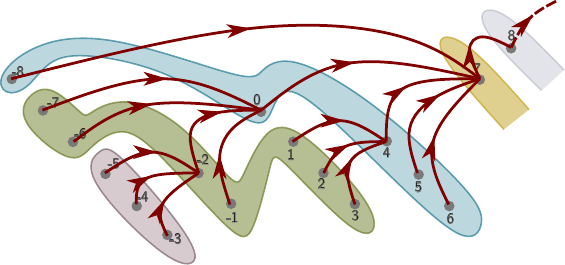}
  \includegraphics[scale=1]{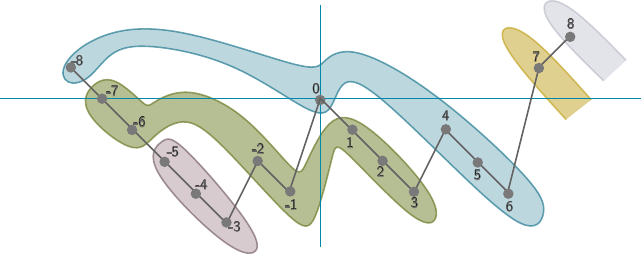}
  \caption{Illustration of $R$-foils in $R$-graph (above) and in trajectory (below) when the mean \(\mathbb{E}[X_0]=0\). Vertices that belong to the same $R$-foil are grouped together.}
  \label{fi:foils}
\end{figure}

\begin{proposition}\label{prop:R_orthogonal_measure_preserving}
  Let $X = (X_n)_{n \in \mathbb{Z}}$ be the sequence of random variables as in Def. \ref{hyp:increments} and satisfies $\mathbb{E}[X_0]=0$. 
  The map $R_{\perp}$ preserves the distribution of $X$, i.e., $[\mathbb{Z},R_{\perp}(0),X]  \overset{\mathcal{D}}{=} [\mathbb{Z},0,X]$. 
\end{proposition}
\begin{proof}
We show that $R_{\perp}$ is bijective and use the fact that bijective maps preserve every unimodular measure (in our case, it is the distribution of $[\mathbb{Z},0,X]$). 
  Firstly, we show that a.s. $R_{\perp}(i) \not = i$ for all $i \in \mathbb{Z}$.
   By Proposition \ref{proposition:r-graph_is_eft}, the component of \(0\) in the record graph of $[\mathbb{Z},0,X]$ is of class $\mathcal{I}/\mathcal{I}$.
  Therefore, all $R$-foils are infinite.
  In order to show that a.s., \(R_{\perp}(i)>i\) for all \(i \in \mathbb{Z}\), it is enough to show that a.s., there does not exist an $R$-foil with a largest element (i.e., for each integer \(i\), there exists an integer \(j\) such that \(j>i\) and \(j\) belongs to the \(R\)-foil of \(i\)).
  There can be at most one largest element in an \(R\)-foil as the foil is totally ordered.
  Let $S_{T}$ be the set of integers that are the largest in their $R$-foils.
  Then, $S$ is a covariant subset.
  The $R$-foliation $\Pi_X$ of $[\mathbb{Z},0,X]$ is a covariant partition.
  Since every foil $E \in \Pi$ is infinite, and $E \cap S$ has at most one element, by No Infinite/Finite inclusion Lemma \ref{lemma:no_infinite_finite}, $E \cap S$ is empty.
  In particular, $R$-foils do not have any largest element a.s..
  
  Secondly, since each $R$-foil is totally ordered as it is a subset of $\mathbb{Z}$, the map $R_{\perp}$ is injective.
  Then, it is bijective almost surely by Proposition \ref{prop:injective_bijective}.
  Therefore, by Mecke-Thorisson Point Stationarity (Prop. \ref{prop_mecke}) \(R_{\perp}\) preserves the distribution of \([\mathbb{Z},0,X]\).
  \end{proof}
  
\begin{remark}\label{remark:foils}
  In the proof of Proposition \ref{prop:R_orthogonal_measure_preserving}, we showed that, almost surely, $R$-foils of $[\mathbb{Z},0,X]$ do not have the largest vertex (according to RLS order).
  Similarly, by taking $S_{T}$ to be set of vertices that are smallest in their $R$-foils, it can be shown that almost surely, none of the $R$-foils has the smallest vertex.
\end{remark}

Note that it is not obvious to see that \(R_{\perp}\) is a measure preserving map without using the framework of \(R\)-foil structure and Mecke-Thorisson point Stationarity.


\chapter{Vertex-shift probabilities}\label{chapter_v_prob}

In this chapter, we consider vertex-shifts on random rooted networks and study the existence and uniqueness of certain limits induced by them on the random networks.
These limits are called vertex-shift probabilities.
In Section \ref{section_vertex_probs}, we introduce them and study their properties.
Results analogous to those in Section \ref{sec_fin_orbit_rel_cpt}-\ref{sec_periodicity} for point-shifts can be found in \cite{baccelliPointmapprobabilitiesPointProcess}.
Given a random network and a vertex-shift \(f\), one obtains the \(f\)-graph of the random network.
The component of the root in the \(f\)-graph is a random rooted Family Tree.
 One may study the existence and uniqueness of vertex-shift probabilities of this Family Tree under parent vertex-shift and may wonder if this study helps to understand the vertex-shift probabilities of the random network.
 This question is addressed by Lemma \ref{thm_forward_backward}, which plays a role in the next  chapter.
 Under different circumstances, discussed in the later sections, existence of vertex-shift probabilities can be shown.
The last section focuses on computing the \(F\)-probabilities of examples discussed in the previous chapter.
Some of these examples will be used in the next chapter to show the existence of a unique vertex-shift probability induced by the record vertex-shift on the network \([\mathbb{Z},0,X]\) when \(X\) is an i.i.d. sequence as in Def. \ref{hyp:increments} and satisfy \(\mathbb{E}[X_0]\geq 0\).

{\bf (In French)} Dans ce chapitre, nous considérons les décalages de sommets sur les réseaux aléatoires enracinés et nous étudions l'existence et l'unicité de certaines limites induites par eux sur les réseaux aléatoires.
Ces limites sont appelées probabilités de décalage de sommet.
Dans la section \ref{section_vertex_probs}, nous les introduisons et étudions leurs propriétés.
Des résultats analogues à ceux de la section \ref{sec_fin_orbit_rel_cpt}-\ref{sec_periodicity} pour les décalages de points peuvent être trouvés dans \cite{baccelliPointmapprobabilitiesPointProcess}.
Étant donné un réseau aléatoire et un décalage de sommet \(f\), on obtient le \(f\)-graphe du réseau aléatoire.
La composante de la racine dans le \(f\)-graphe est un arbre généalogique aléatoire enraciné.
On peut étudier l'existence et l'unicité des probabilités de décalage de sommet de cet arbre généalogique associè au le décalage de sommet des parents et se demander si cette étude aide à comprendre les probabilités de décalage de sommet du réseau aléatoire.
Cette question est abordée par le lemme \ref{thm_forward_backward}, qui joue un rôle dans le chapitre suivant.
Dans différentes circonstances, discutées dans les sections suivantes, l'existence des probabilités de décalage de sommet peut être montrée.
La dernière section se concentre sur le calcul des \(F\)-probabilités des exemples discutés dans le chapitre précédent.
Certains de ces exemples seront utilisés dans le chapitre suivant pour montrer l'existence d'une probabilité unique de décalage de sommet induite par le décalage de sommet des records sur le réseau \([\mathbb{Z},0,X]\) lorsque \(X\) est une suite i.i.d. comme dans la définition \ref{hyp:increments} et qui satisfait la condition \(\mathbb{E}[X_0]\geq 0\).

\section{Vertex-shift probabilities}\label{section_vertex_probs}
Let \textbf{[G, o]}  be a random network, i.e., \textbf{[G, o]} is a measurable map from a probability space $(\Omega,\mathcal{F},\mathbb{P})$ to the space of connected and locally finite rooted networks $\mathcal{G}_*$.
Let $\mathcal{P}_0$ be the probability measure on $\mathcal{G}_*$ induced by the random network \textbf{[G,o]}.
Let $f$ be a vertex-shift.
The map f induces a map $\theta_f:\mathcal{G}_* \rightarrow \mathcal{G}_*$, defined by $\theta_f([G,o]) = [G,f(o)]$.
This determines an action of $\mathbb{N}$ on $\mathcal{G}_*$ defined by the iterates of $\theta_f$, i.e., for each \(n \geq 1\), we have the map \(\theta_f^n:\mathcal{G}_* \to \mathcal{G}_*\) defined by \(\theta^n_f([G,o]) = [G,f^n(o)]\), for all \([G,o] \in \mathcal{G}_*\).
The push-forward measures of $\mathcal{P}_0$ under the iterates of $\theta_f$ are called the \textbf{vertex-shift probabilities} of \([\textbf{G},\textbf{o}]\) under \(f\).

Let \(\mathcal{P}_n\) be the push-forward measure of \(\mathcal{P}_0\) under the \(n\)-th iterate \(\theta_f^n\) of \(\theta_f\),i.e., $\mathcal{P}_n = (\theta_f^n)_* \mathcal{P}_0$.
Therefore, for any measurable subset $A \subset \mathcal{G}_*$, we have
\begin{align*}
\mathcal{P}_n[A] &= (\theta_f^n)_* \mathcal{P}_0[A]= \mathcal{P}_0[\{[G,o]: [G,f^n(o)] \in A\}]=  \mathbb{P}[[\textbf{G},f^n(\textbf{o})] \in A].
\end{align*}

Given a random network $\mathcal{P}_0$, we are interested to study the conditions under which the weak limit points of the sequence $\mathcal{A}_{\mathcal{P}_0}=\{\mathcal{P}_0,\mathcal{P}^{f,1}_0,\mathcal{P}^{f,2}_0,...\}$ exist.
Any probability measure which is a weak limit of $\mathcal{A}_{\mathcal{P}_0}$ is called an \textbf{$f$-probability} of the random network $\bgo$.
If the sequence $(\mathcal{P}^{f,n}_0)_{n \geq 0}$ (with the understanding that $\mathcal{P}_0 = \mathcal{P}^{f,0}_0$) converges weakly to a probability measure $\calp_0^f$ then, we call $\calp_0^f$  \textbf{the $f$-probability} of $\calp_0$.
Note that a probability measure $\mathcal{P}$ on $\mathcal{G}_*$ is a weak limit point of $\mathcal{A}_{\mathcal{P}_0}$ if and only if there exists a subsequence $\{\mathcal{P}_0^{f,n_k}\}$ such that
$$\int_{[G,u] \in \mathcal{G_*}} h([G,u]) \ \mathcal{P}_0^{f,n_k}(d[G,u]) \longrightarrow \int_{[G,u] \in \mathcal{G_*}} h([G,u]) \ \mathcal{P}(d[G,u])$$
for all continuous and bounded functions $h$ on $\mathcal{G}_*$.

In ergodic theory, one studies measure preserving transformations.
Seen from this perspective, Mecke's Point Stationarity (\cite{baccelliEternalFamilyTrees2018a}) tells that a vertex-shift is unimodular measure preserving if and only if the vertex-shift is bijective almost surely.
\begin{lemma} \label{lemma_abs_cont}
    Let $[\mathbf{G},\mathbf{o}]$ be a unimodular network with distribution $\mathcal{P}$ and $f$ be a vertex-shift. 
  Then, the distribution $\mathcal{P}^{f,n}_0$ of $[\mathbf{G},f^n(\mathbf{o})]$ is absolutely continuous with respect to the initial distribution $\mathcal{P}_0$ with the Radon-Nikodym derivative
    \begin{equation*}
      \frac{d\mathcal{P}^{f,n}_0}{d \mathcal{P}}[G,o] = d_n(o).
    \end{equation*}
  \end{lemma}
  \begin{proof}
    Let $h$ be a non-negative function on $\mathcal{G}_*$. 
  Then,
  \begin{align*}
    \mathbb{E}[h([\mathbf{G},f^n(\mathbf{o})])] &=  \mathbb{E}\left[\sum_{u \in V(\mathbf{G})}h([\mathbf{G},u])\mathbf{1}_{u = f^n(\mathbf{o})}\right]\\
    &= \mathbb{E}\left[\sum_{u \in V(\mathbf{G})}h([\mathbf{G},\mathbf{o}])\mathbf{1}_{\mathbf{o}= f^n(u)}\right] \text{ (by unimodularity)}\\
    &= \mathbb{E}[h([\mathbf{G},\mathbf{o}])d_n(\mathbf{o})].\\
  \end{align*}
  \end{proof}
  
  Observe that, $d_1(\mathbf{o})=1$ a.s. if and only if $f$ is bijective a.s..
  Combining this observation with the above lemma, we obtain that a vertex-shift \(f\) on a unimodular network $[\mathbf{G},\mathbf{o}]$ is bijective a.s. if and only if it is measure preserving (i.e., \([\mathbf{G},f(\mathbf{o})]\overset{\mathcal{D}}{=}[\mathbf{G},\mathbf{o}]\)).
  This is Mecke-Thorisson point stationarity (see Prop. \ref{prop_mecke}).

Lemma \ref{lemma_abs_cont} implies that the probability measure \(\mathcal{P}^{f,n}_0\) is supported on the networks that have \(n\)-the descendants.
The following proposition compares this measure with the measure obtained from \(\mathcal{P}_0\) by conditioning on the event that the root has \(n\)-th descendants.

\begin{proposition}\label{prop:conditional}
  Let $\bgo$ be a unimodular network and $\calp_0$ be the probability measure induced by $\bgo$ on $\calg_*$.
  Let $f$ be a vertex-shift and $\{\calpfn\}_{n \in \mathbb{N}}$ be the vertex-shift probabilities of \(\mathcal{P}_0\) under $f$.
  Then, for any measurable subset $A$ of $\calg_*$,

  $$\calpfn[A] = \mathcal{E}_0 \left[ \frac{d_n(o)}{\mathcal{E}_0[d_n(o) |d_n(o)>0]} \mathbf{1}_A \ | \ d_n(o)>0\right],$$
  where $\mathcal{E}_0$ is the expectation with respect to $\calp_0$.
\end{proposition}
\begin{proof}
  By Lemma \ref{lemma_abs_cont},
  \begin{align*}
    \calpfn[A] &= \mathcal{E}_0[ d_n(o) \mathbf{1}_A]
= \mathcal{E}_0[d_n(o)\mathbf{1}_A\mathbf{1}_{d_n(o)>0}].
  \end{align*}
  In particular, for \(A=\mathcal{G}_*\), we obtain
  \begin{align*}
    \calpfn[\calg_*] =\mathcal{E}_0[d_n(o)\mathbf{1}_{d_n(o)>0}].
  \end{align*}
  Since  $\calpfn[\calg_*] = 1$, we obtain,
  \begin{equation*}
    1 = \mathcal{E}_0[d_n(o)\mathbf{1}_{d_n(o)>0}]=\mathcal{E}_0[d_n(o)|d_n(o)>0] \calp_0[d_n(o)>0].
  \end{equation*}
  From this, we obtain,
  \begin{align*}
    \calp_0[d_n(o)>0]&=\frac{1}{\mathcal{E}_0[d_n(o)|d_n(o)>0]}.
  \end{align*}
  Therefore,
  \begin{align*}
    \calpfn[A] &=  \mathcal{E}_0[d_n(o)\mathbf{1}_A\mathbf{1}_{d_n(o)>0}]
               = \mathcal{E}_0[d_n(o)\mathbf{1}_A|d_n(o)>0]\calp_0[d_n(o)>0]\\
               &= \mathcal{E}_0 \left[ \frac{d_n(o)}{\mathcal{E}_0[d_n(o) |d_n(o)>0]} \mathbf{1}_A \ | \ d_n(o)>0\right].
  \end{align*}
\end{proof}
From the above proposition, $\calpfn$ can be seen to be absolutely continuous with respect to the conditional probability $\calp_0[. | d_n(o)>0]$.

\begin{corollary}
  Consider the conditions mentioned in Proposition \ref{prop:conditional}.
  If the $f$-probability $\calp^f_0$ of $\calp_0$ exists, then the following hold:
  \begin{enumerate}
  \item For any bounded and continuous function $h:\calg_* \rightarrow \mathbb{R}$,
    $$\int_{\calg_*}h(\go) \ d\calp^f_0 = \lim_{n \rightarrow \infty}\int_{\calg_*}h(\go)d_n(o) \ d \calp_0.$$
  \item For any measurable subset $A$ of $\calg_*$ for which $\calpfn[A]$ tends to $\calp^f_0[A]$ as $n$ tends to infinity, 
    $$ \calp^f_0[A] = \lim_{n \rightarrow \infty}\mathcal{E}_0 \left[ \frac{d_n(o)}{\mathcal{E}_0[d_n(o) |d_n(o)>0]} \mathbf{1}_A \ | \ d_n(o)>0\right].$$
  \end{enumerate}
\end{corollary}
\begin{proof}
  Follows from the definition of the $f$-probability of $\calp_0$.
\end{proof}

\subsection{Forward Backward Lemma} 

Consider a vertex-shift $f$ and a unimodular random network $[\mathbf{G},\mathbf{o}]$.
Recall that the map $\theta_f:\mathcal{G}_* \rightarrow \mathcal{G}_*$ takes $[G,o]$ to $[G,f(o)]$, and the forward map $\Psi_f:\mathcal{G}_* \rightarrow \mathcal{T}_*$ takes $[G,o]$ to $[G^f(o),o]$, where $G^f(o)$ is the connected component of $o$ (without any loops) in the $f$-graph $G^f$.

We know that $\Psi_f([\mathbf{G},\mathbf{o}])$ is a unimodular Family Tree (Lemma \ref{lemma:f_graph_unimodular}).
Consider the parent vertex-shift $F$ on this Family Tree.
If \(\Psi_f([\mathbf{G},\mathbf{o}])\) admits an $F$-probability, does it imply the existence and uniqueness of the $f$-probability?
The following lemma gives sufficient conditions, under which we obtain a unique $f$-probability from the $F$-probability.

\begin{lemma}[Forward Backward (FB)] \label{thm_forward_backward}
  Let $\mathcal{P}$ be a unimodular measure on $\mathcal{G}_*$, and $f$ be a vertex-shift. 
  Let $\mathcal{P}_n := \mathcal{P} \circ \theta_f^{-n} , n \geq 1$, be the push-forward measure of $\mathcal{P}$ under the $n$-th iterate of the vertex-shift map $\theta_f$.
  Let $\mathcal{Q}:= \mathcal{P} \circ \Psi_f^{-1}$ and $\forall n \geq 1$, let $\mathcal{Q}_n :=  \mathcal{Q} \circ \theta_F^{-n}$ be the push-forward measure of $\mathcal{Q}$ under the $n$-th iterate of the Parent vertex-shift map $\theta_F$.

  Suppose there exists a map $\Phi_f:\mathcal{T}_* \rightarrow  \mathcal{G}_*$ and a probability measure $\mathcal{Q}_F$ on $\mathcal{G}_*'$ such that
  \begin{enumerate}
    \item $\Phi_f \circ \Psi_f = I$, $\mathcal{P}$-a.s., where $Id$ is the identity map.
    \item $\mathcal{Q}_F$ is the $F$-probability of $\mathcal{Q}$, i.e., $\mathcal{Q}_n \rightarrow \mathcal{Q}_F$ weakly as $n \rightarrow \infty$.
    \item $\Phi_f$ is continuous $\mathcal{Q}_F$-a.s..
  \end{enumerate}
  Then, $\mathcal{Q}_F \circ \Phi_f^{-1}$ is the $f$-probability of $\mathcal{P}$.
\end{lemma}

\begin{proof}
  Observe that $\Psi_f \circ \theta_f = \theta_F \circ \Psi_f$. Indeed, for any $[G,o]\in \mathcal{G}_*$,
  \begin{align*}
    \Psi_f \circ \theta_f([G,o]) &= \Psi_f([G,f(o)])= [G^f(f(o)),f(o)]\\ 
    &= [G^f(o),f(o)]= [G^f(o),F(o)]= \theta_F \circ \Psi_f([G,o]).
  \end{align*} 
  The third equality in the above equation follows from the fact that the connected component of $o$ in $G^f$ is the same as that of $f(o)$.
  Using this inductively we obtain, for any $n \geq 1$, 
  \begin{equation*}
    \mathcal{Q}_n = \mathcal{Q} \circ \theta_F^{-n} = \mathcal{P}\circ \Psi_f^{-1} \circ \theta_F^{-n} = \mathcal{P}\circ \theta_f^{-n} \circ \Psi_f^{-1} = \mathcal{P}_n \circ \Psi_f^{-1}.
  \end{equation*}

We know that $\mathcal{Q}$ is unimodular because $\mathcal{P}$ is unimodular by Lemma \ref{lemma:f_graph_unimodular}.
By Lemma \ref{lemma_abs_cont}, $\mathcal{Q}_n$ (resp. $\mathcal{P}_n$) is absolutely continuous with respect to $\mathcal{Q}$ (resp. $\mathcal{P}$) for all $n \geq 1$.
So, it follows from the first condition that $\Phi_f \circ \Psi_f = I$ $\mathcal{P}_n$-a.s..
Using this, we obtain
\[\mathcal{P}_n = \mathcal{P}_n \circ (\Phi_f \circ \Psi_f)^{-1} =\mathcal{P}_n \circ \Psi_f^{-1} \circ \Phi_f^{-1}= \mathcal{Q}_n \circ \Phi_f^{-1}.\]

Since $\mathcal{Q}_n$ weakly converges to $\mathcal{Q}_F$ and $\Phi_f$ is continuous $\mathcal{Q}_F$-a.s., by the Mann-Wald theorem \cite{bhattacharyaRandomWalkBrownian2021}, $\mathcal{Q}_n \circ \Phi_f^{-1}$ converges weakly to $\mathcal{Q}_F \circ \Phi_f^{-1}$ as $n \rightarrow \infty$.
Thus, $\lim_{ n\rightarrow \infty}\mathcal{P}_n = \mathcal{Q}_F \circ \Phi_f^{-1}$.
\end{proof}

In general, the existence of the map $\Phi_f$ is not guaranteed since the $f$-graph may not contain the needed information about the structure of the network.
For instance, consider the identity vertex-shift.
 In this case, the connected components of its $f$-graph are just singletons.
 So, no information can be obtained from the \(f\)-graph.
We may however attach more information of a network to its $f$-graph as marks.
In the case of the record vertex-shift $R$ on the random network \([\mathbb{Z},0,X]\) (where \(X\) is the sequence of i.i.d. increments of a skip-free to the left random walk), it suffices to consider the order of $\mathbb{Z}$ when \(\mathbb{E}[X_0]=0\) and both the order and  the type function \(t\) when \(\mathbb{E}[X_0]>0\) as we shall see in the next chapter.
Note that, in FB Lemma \ref{thm_forward_backward}, the backward map \(\Phi_f\) may depend on the initial distribution \(\mathcal{P}\).

\section{Finite orbits and relative compactness}\label{sec_fin_orbit_rel_cpt}

We use the notations of Subsection \ref{subsec_rel_compact}.
In this section, we are only focused on unlabelled graphs.
Recall that \(\hat{\mathcal{G}}_*\) denotes the space of rooted graphs.

\begin{definition}
  A vertex-shift $f$ is said to have \textbf{finite orbits} on a random graph $\bgo$ if the set $\{[G,f^n(o)], n \in \mathbb{N}\}$ is finite for $-\calp_0$-a.s. $\go \in \hat{\calg}_*$, where \(\mathcal{P}_0\) is the probability measure induced by \([\mathbf{G},\mathbf{o}]\).
\end{definition}

\begin{proposition}
  Let $\bgo$ be a random graph which induces the probability measure $\calp_0$ on $\hat{\calg}_*$, $f$ be a vertex-shift and $\mathcal{A}_{\calp_0} = \{\calpfn, n \in \mathbb{N}\}$ be the vertex-shift probabilities induced by $f$. If $f$ has finite orbits on $\bgo$ then $\mathcal{A}_{\calp_0}$ is relatively compact.
\end{proposition}
\begin{proof}
  Let $r\geq 1$ be fixed. We show that,
  \begin{equation}\label{eq:rel_comp}
    \lim_{m \rightarrow \infty}\sup_{n \in \mathbb{N}}\mathbb{P}[\bgfno_r \not \in K_m^r] = 0.
  \end{equation}
  For any fixed \(m \geq 1\) and for all \(n \geq 1\), the event $\{\bgfno_r \in \hat{\calg}_* \backslash K_m^r\}$ is a subset of $\{orb(\bgo)_r \cap (\hat{\calg}_* \backslash K_m^r) \not = \emptyset\}$, where $orb(\bgo)_r = \{\bgfno_r, n \in \mathbb{N}\}$.
  This implies that
  $$\sup_{n \in \mathbb{N}}\mathbb{P}[\bgfno_r \in \hat{\calg}_*\backslash K_m^r] \leq \mathbb{P}[orb(\bgo)_r \cap (\hat{\calg}_* \backslash K_m^r) \not = \emptyset].$$
  Since $K_m^r$ is an increasing sequence in $m$, $\hat{\calg}_*\backslash K_m^r$ is a decreasing sequence of sets, which implies that the events $\{orb(\bgo)_r \cap (\hat{\calg}_* \backslash K_m^r) \not = \emptyset\}$ are decreasing in $m$.
  Therefore,
  \begin{align*}
    \lim_{m \rightarrow \infty}\sup_{n \in \mathbb{N}}\mathbb{P}&[\bgfno_r \in \hat{\calg}_*\backslash K_m^r] \\
    &\leq \lim_{m \rightarrow \infty}\mathbb{P}[orb(\bgo)_r \cap (\hat{\calg}_* \backslash K_m^r) \not = \emptyset]\\
    &=\mathbb{P}\left[\bigcap_{m \in \mathbb{N}} \{orb(\bgo)_r \cap (\hat{\calg}_* \backslash K_m^r) \not = \emptyset \}\right].
  \end{align*}
  But $orb(\bgo)_r$ is a.s. finite. So, it can have non-empty intersection with $\hat{\calg}_* \backslash K_m^r$ only for finitely many $m$.
  Therefore,
  $$\mathbb{P}\left[\bigcap_{m \in \mathbb{N}} \{orb(\bgo)_r \cap (\hat{\calg}_* \backslash K_m^r) \not = \emptyset \}\right]=0,$$
  which gives the desired equality in Equation (\ref{eq:rel_comp}).
\end{proof}

\section{Evaporation}

The $n$-th vertex-shift-probability measure $\calp_0^{f,n}$ of a probability measure $\calp_0$ is supported on the rooted network $[G,o]$ whose root $o$ has $n$-th descendants, i.e., $d_n(o)>0$.
This leads us to ask, ``are all $f$-probabilities, assuming they exist, supported only on the rooted networks whose root has infinitely many descendants, i.e., $d_n(o)>0, \forall n \in \mathbb{N}$?".
The answer is affirmative under the sufficient conditions mentioned in Lemma \ref{lm:evap_lim} or in Lemma \ref{lm:evap_lim2}.

Let us denote $I:= \{\go \in \calg_* \text{ such that } d_n(o)>0 \ \forall n \geq 1\}$, the set of rooted networks whose root has infinitely many descendants.
Let $A_m = \{\go \in \calg_* \text{ such that } d_m(o) = 0\}$ and $J = \cup_{m \geq 1}A_m$. Then, \(J\) is the set of rooted networks whose root has finitely many descendants, and we have
$$J^c = \cap_{m \geq 1}A_m^c = \cap_{m \geq 1}\{ \go \in \calg_* : d_m(o)>0\} = I.$$

\begin{definition}
  We say that a random network $\bgo$ \textbf{evaporates} on a vertex-shift $f$ if $\mathbb{P}[\bgo \in J]=1$.
\end{definition}

\begin{lemma}\label{lm:evap_lim}
  Let \(f\) be a vertex-shift and \([\mathbf{G},\mathbf{o}]\) be a unimodular random network.
  If \(Q\) is an \(f\)-probability of \([\mathbf{G},\mathbf{o}]\), then \(Q[I]=1\).
\end{lemma}
\begin{proof}
  For any $n \in \mathbb{N}$, let us denote $B_n = \{\go \in \calg_* \text{ such that } d_n(o)>0\}$.
  Then $B_n$ is a non-increasing sequence of subsets whose intersection is $I$, i.e., $I = \cap_n B_n$ and $B_n \searrow I$.
  Therefore,
  $$\mathcal{Q}[I] = \lim_{n \rightarrow \infty} \mathcal{Q}[B_n].$$
  Since $\mathcal{Q}$ is a $f$-probability, there exists a subsequence $\{\calp^{f,n_k}_0\}$ of the vertex-shift probabilities $\{\calpfn\}$ that converges weakly to $ \mathcal{Q}$.
  Equivalently, for any continuity set $B$, $\calp^{f,n_k}_0[B] \rightarrow \mathcal{Q}[B]$ as $k$ tends to infinity.

  For all \(n\geq 1\), the map $\go \mapsto \mathbf{1}_{d_n(o)>0}$ is continuous on \(\mathcal{G}_*\) since for any sequence \(([G_m,o_m])_{m \geq 1}\) that converges to a network \([G',o']\)  as \(m \to \infty\), the sequence \((d_n(o_m))_{m \geq 1}\) is eventually a constant sequence,i.e., there exists \(k \geq 1\) such that \(d_n(o_m) =d_n(o')\) for all \(m \geq k\).
  Therefore, $B_n$ is a $\mathcal{Q}$-continuity set, i.e., $\calp^{f,n_k}_0[B_n] \rightarrow \mathcal{Q}[B_n]$ as $k$ tends to infinity.
  But, notice that $\calp^{f,n_k}_0[B_{n_k}] = 1$, as it is only supported on the roots that have $n_k$-th descendants, and $B_{n_k} \subseteq B_n$ for all $n_k \geq n$.
  So, $\calp^{f,n_k}_0[B_n] = 1$ for all $n_k \geq n$ which implies that $ \mathcal{Q}[B_n]=1$.
  This gives the desired result that $ \mathcal{Q}[I] = 1$.

\end{proof}

\begin{lemma}\label{lm:evap_lim2}
   Let $f$ be a vertex-shift and $\mathcal{Q}$ be a probability measure on $\calg_*$ which satisfies  Mecke's invariant measure equation, i.e., $(\theta_f)^*\mathcal{Q} = \mathcal{Q}$. Then $\mathcal{Q}[I] = 1$.
 \end{lemma}

 \begin{proof}
   We show that \(Q[J]=0\).
   For any $m \in \mathbb{N}$, let us denote $A_m = \{\go \in \calg_* :d_m(o)=0\}$.
   Since $(\theta_f)_* \mathcal{Q} = \mathcal{Q}$, we have $(\theta^m_f)_* \mathcal{Q} =\mathcal{Q}$ for any $m \in \mathbb{N}$. Therefore,
   \begin{align*}
     \mathcal{Q}[A_m] &= (\theta^m_f)_* \mathcal{Q}[A_m]=\mathcal{Q}[(\theta_f^{m})^{-1}(A_m)]=\mathcal{Q}\left[\go \in \calg_* : [G,f^m(o)] \in A_m \right]\\
                    &=0 \text{ (since no element of $A_m$ has $m$-th descendant).}
   \end{align*}
   Since $\mathcal{Q}[A_m]=0$ for all $m \in \mathbb{N}$, we have
   $$\mathcal{Q}[J] = \mathcal{Q}\left[\bigcup_{m \in \mathbb{N}}A_m\right] = 0.$$
  Therefore, $\mathcal{Q}[I] = 1$.
\end{proof}

 \begin{corollary}\label{cor:evap_sing}
   Let $\bgo$ be a random network which evaporates under the action of a vertex-shift $f$. Let $\calp_0^f$ be a $f$-probability of $\calp_0$ (the distribution of $\bgo$) which satisfies Mecke's invariant measure equation, i.e., $(\theta_f)^* \calp_0^f = \calp_0^f$. Then, $\calp_0^f$ is singular with respect to the distribution $\calp_0$ of $\bgo$.
 \end{corollary}
 \begin{proof}
    Since $\bgo$ evaporates under $f$, $\calp_0[J]=1$.
   By Lemma \ref{lm:evap_lim2}, $\calp_0^f[J] = 0$.
   Therefore, $\calp_0^f$ is singular with respect to $\calp_o$.
 \end{proof}

 \begin{corollary}\label{cor_f_prob_sing}
   Let $\bgo$ be  a random network which evaporates under the action of a vertex-shift $f$. Assume that the vertex-shift probabilities converge to a unique $f$-probability $\calp_0^f$ and $\theta_f$ is continuous at $\calp_0^f$. Then, $\calp_0^f$ is singular with respect to the distribution $\calp_0$ of $\bgo$.
 \end{corollary}

 \begin{proof}
   Let $\{\calp_0^{f,n}\}$ be the vertex-shift-probabilities of $\calp_0$.
   Since, for all \(n\geq 1\), $(\theta_f)_*\calp_0^{f,n} = \calp_0^{f,n+1}$, the sequence \(((\theta_f)_*\calp_0^{f,n})_n\) converges to $\calp_0^f$.
   Since $\theta_f$ is continuous at $\calp_o^f$, we have $(\theta_f)_*\calp_0^f = \calp_0^f$.
   The statement follows from Corollary \ref{cor:evap_sing}.
 \end{proof}

\section{Periodicity}\label{sec_periodicity}

 In the previous Section, we saw conditions under which an $f$-probability is singular with respect to the initial distribution.
In this section, we show that if the vertex-shift is periodic then $f$-probabilities are absolutely continuous with respect to the initial distribution (see Theorem \ref{thm:1_periodic}, Remark \ref{rmk:p_periodic}).
\begin{definition}
  A vertex-shift $f$ is said to be \textbf{periodic} if for all $G$ and $o \in V(G)$, there exists positive integers $k=k(o)$ and $p=p(o)$ such that
  $$\forall n\geq k,\ f_G^{n+p}(o)=f^n(o).$$
\end{definition}
\begin{definition}
  A vertex-shift $f$ is said to be \textbf{$p$-periodic} if $p$ does not depend on $\go$.
\end{definition}
A vertex \(o \in V(G)\) is called a \textbf{trap} of \(G\) with respect to a vertex-shift $f$ if $f_G(o)=o$.
For the identity vertex-shift, every vertex of every network is a trap.
Note that, for any vertex-shift \(f\), the collection of set of traps of networks \(S_{G}:= \{u\in V(G):f_G(u)=u\}\) is a covariant subset.

\begin{lemma}
  If a vertex-shift $f$ is $1$-periodic a.s. on a random network $\bgo$ then the limit $\lim_{n \rightarrow \infty}\bgfno$ exists and is supported on the traps, i.e., \(\mathbb{P}[\lim_{n \to \infty}[\mathbf{G},f^n(\mathbf{o})] \in B]=1\), where \(B := \{[G,o] \in \mathcal{G}_*: o \text{ is a trap}\}\).
\end{lemma}
\begin{proof}
  Since $f$ is $1$-periodic a.s., for $\calp_0$-a.s. $\go \in \calg_*$ there exists $k = k(o)$ such that
  $$\forall n\geq k, f_G^{n+1}(o)=f_G^n(o)=f_G^k(o).$$
Therefore, \(f_G^k(o)\) is a trap. Hence, $\mathbb{P}[\lim_{n \rightarrow \infty} \bgfno \in B]=1$.
\end{proof}

In view of the above lemma, for a random network $\bgo$ and a $1$-periodic vertex-shift $f$, let $\bbargo = \lim_{n \rightarrow \infty}\bgfno$ and $\calp_0^{\infty}$ be the distribution of $\bbargo$.
\begin{theorem}\label{thm:1_periodic}
  Let $\bgo$ be a random network, \(\mathcal{P}_0\) be its distribution and $f$ be an a.s. $1$-periodic vertex-shift on $\bgo$.
  Then the $f$-probability $\calp_0^f$ of $\calp_0$ exists and $\calp_0^f$ is the distribution of $\bbargo$.

  In addition, if $\bgo$ is unimodular then the $f$-probability $\calp_0^f$ is absolutely continuous with respect to the initial distribution $\calp_0$.
  The limit function $d_{\infty}(o)= \lim_{n \rightarrow \infty}d_n(o)$ exists and it is finite for $\calp_o^f$-a.s. $\go$.
  The Radon-Nikodym derivative of $\calp_0^f$ with respect to $\calp_0^f$ is given by
  $$\frac{d \calp_0^f}{d \calp_0}\go = d_{\infty}(o) \mathbf{1}\{\text{o is a trap}\}.$$
\end{theorem}
\begin{proof}
  We prove that $\calp_0^{\infty}$ (the distribution of $\bbargo$) is the $f$-probability of $\calp_0$.
  For any bounded and continuous function $h:\calg_* \rightarrow \mathbb{R}$,
  \begin{align*}
    \int_{\calg_*}h \go \ d \calp_0^{\infty} &= \mathbb{E}[h \bbargo]=\mathbb{E}[h(\lim_{n \rightarrow \infty}\bgfno)]\\
   &=\mathbb{E}[\lim_{n \rightarrow \infty}h \bgfno]  \text{ (by continuity of $h$)} \\
   &= \lim_{n \rightarrow \infty}\mathbb{E}[h(\bgfno)].
  \end{align*}
  In the last step of the above equation, we used the boundedness of $h$ and the dominated convergence theorem.
  Thus, $\calp_0^{\infty}$ is the $f$-probability of $\calp_0$.

  We now assume that \([\mathbf{G},\mathbf{o}]\) is unimodular.

  \noindent
  \textit{The limit function $d_{\infty}$:}

  The sequence of functions $(d_n(o))_{n \in \mathbb{N}}$ is non-decreasing on the set of traps due to the following reason. 
  Let \(o \in V(G)\) be a trap, and $u \in D_n(o)$ for some $n \in \mathbb{N}$. Then, since $f_G^{n+1}(u) = f_G^n(u)=o$, we have $D_n(o) \subseteq D_{n+1}(o)$ and $d_n(o) \leq d_{n+1}(o)$.

  Since the random network $\bbargo$ is supported on the traps, the random variable
  $$d_{\infty}(\bbargo) = \lim_{n \rightarrow \infty}d_n(\bbargo)$$
  exists a.s.

  As for the finiteness of the random variable $d_{\infty}(\bbargo)$, let
	\[\mathcal{A} := \{[G,o] \in \mathcal{G}_*: f_G^{n+1}(o)=f_G^n(o) \text{ for some $n= n(o) \in \mathbb{N}$ } \}.\]
  Then, for any $[G,o] \in \mathcal{A}$, the root \(o\) has finitely many ancestors in the $f$-graph.
  Therefore, by the unimodular classification theorem (\cite{baccelliEternalFamilyTrees2018a}), the connected component \(C(o)\) of \(o\) is of type $\mathcal{F}/\mathcal{F}$ as it is the only case in which vertices have finitely many ancestors.
  This implies that $C(o)$ is finite.
  But the connected component of $f^{n(o)}(o)$ in the $f$-graph is also $C(o)$ and $f^{n(o)}(o)$ is a trap.
Thus, $d_{\infty}(\lim_{m \to \infty}[G,f^m(o)]) = \#D(f^{n(o)}(o))=\#C(o)$ is finite.

	\noindent
  \textit{Radon-Nikodym derivation of $\calp_0^f$ with respect to $\calp_0$:}

  Let $h:\calg_* \rightarrow \mathbb{R}_{\geq 0}$ be any measurable function. Then,
  \begin{align}
    \int_{\calg_*} &h(\go)\ d\calp_0^f = \mathbb{E}[h(\bbargo)]  = \mathbb{E}\left[ \sum_{u \in V(\bar{\textbf{G}})}h([\bar{\textbf{G}},u])\mathbf{1}\{\bar{\textbf{o}} = u\}\right] \nonumber \\
  &=\mathbb{E}\left[ \sum_{u \in V(\textbf{G})}h([\textbf{G},u])\mathbf{1}\{f^{n(\textbf{o})}(\textbf{o}) = u\}\mathbf{1}\{f^{n(\textbf{o})}(\textbf{o}) \text{ is a trap}\}\right] \nonumber \\
  &=\mathbb{E}\left[ \sum_{u \in V(\textbf{G})}h([\textbf{G},\textbf{o}])\mathbf{1}\{f^{n(u)}(u)=\textbf{o}\}\mathbf{1}\{f^{n(u)}(u) \text{ is a trap}\}\right] \label{eq_trap_sum}\\
    &=\mathbb{E}\left[h(\bgo)\ \#D(\textbf{o})\mathbf{1}\{\textbf{o} \text{ is a trap}\}\right] \nonumber\\
    &=\mathbb{E}\left[h(\bgo)\lim_{m \rightarrow \infty}d_m(\textbf{o})\mathbf{1}\{\textbf{o} \text{ is a trap}\}\right] \nonumber \\
   &=\mathbb{E}[h(\bgo)\ d_{\infty}(\textbf{o})\mathbf{1}\{ \textbf{o}\text{ is a trap}\}].\nonumber
  \end{align}
  We used unimodularity to get Eq. (\ref{eq_trap_sum}).
  Thus, we obtain the Radon-Nikodym derivative of $\calp_0^f$ with respect to $\calp_0$,
  $$\frac{d \calp_0^f}{d \calp_0}\go = d_{\infty}(o) \mathbf{1}\{\text{o is a trap}\}.$$
\end{proof}

\begin{remark}
  From the proof of Theorem \ref{thm:1_periodic}, it follows that \(d_{\infty}(\mathbf{o}) = \#D(\mathbf{o})\).
\end{remark}

\begin{remark}\label{rmk:p_periodic}
  Let \(p\) be a positive integer.
  If a vertex-shift $f$ is a.s. $p$-periodic on a random network $\bgo$ then for $\calp_0$-a.s. $\go \in \calg_*$, there exists a $k=k(o)\in \mathbb{N}$ such that $f_G^{n+p}(o) = f_G^n(o)$ for all $n \geq k$.
  In particular, $f_G^{np+i+p}(o) = f_G^{np+i}(o)$ for all $n \geq k$ and for any $i=0,1, \cdots, p-1$.
  So, the vertex-shift $f^p$ is a.s. $1$-periodic on the random networks $\{\bgo, [\textbf{G},f(\textbf{o})], \cdots, [\textbf{G},f^{p-1}(\textbf{o})]\}$.
  By Theorem \ref{thm:1_periodic}, for each \(i=0,1,\cdots,p-1\), \([\mathbf{G},f^i(\mathbf{o})]\) has unique \(f^p\)-probability.
  So, \([\mathbf{G},\mathbf{o}]\) can have at most \(p\) \(f\)-probabilities.
  In addition, if \([\mathbf{G},\mathbf{o}]\) is unimodular, then all of its \(f\)-probabilities are absolutely continuous with respect to \(\mathcal{P}_0\) by Lemma \ref{lemma_abs_cont} and by Theorem \ref{thm:1_periodic}.
 \end{remark}

 \section{Examples and their F-probabilities}

 In this section, we compute the \(F\)-probabilities of some Family Trees.
A few remarks are given in the respective subsections about the application of results of the previous section to the examples studied here.
Among the examples, the Canopy tree and the unimodular Eternal Galton-Watson Tree are taken from \cite{baccelliEternalFamilyTrees2018a}.
We have already introduced two new unimodular Family Trees namely, the typically rooted Galton-Watson Tree and the unimodularised bi-variate Eternal Kesten Tree.
The \(F\)-probability of the latter is shown to be the bi-variate Eternal Kesten Tree in Proposition \ref{20230314185620}.
We introduce Eternal Kesten Tree, a special instance of bi-variate Eternal Kesten Tree, and prove that it is the \(F\)-probability of the Eternal Galton-Watson Tree (in Proposition \ref{prop:F_prob_EGWT_EKT}).
We also introduce the size-biased Galton-Watson Tree, and show that it is the \(F\)-probability of the typically rooted Galton-Watson Tree.

The results from this section together with FB Lemma \ref{thm_forward_backward} will be used in the next chapter to compute the \(R\)-probability.

 \subsection{Canopy tree}
The Canopy tree with the number of children d, $C_d$ is an Eternal Family Tree with vertices partitioned into infinitely many levels $L_0,L_1,L_2,\cdots$ such that, for any $i> 0$, every vertex $v$ in $L_i$ has $d$ children in $L_{i-1}$, and a unique parent in $L_{i+1}$; the vertices of \(L_0\) do not have any children but each vertex of \(L_0\) has a unique parent in \(L_1\).
The tree is deterministic, but the root is chosen randomly to obtain a random rooted Family Tree \([C_d,\mathbf{o}]\).

Since, for any two vertices \(u,v\) of the same level \(L_i\) (for any \(i \geq 0\)), the equivalent class of \((C_d,u)\) and \((C_d,v)\) are one and the same, the set of probabilities \(\{p_i = \mathbb{P}[\textbf{o} \in L_i]:i \geq 1\}\) fully characterizes the distribution of \([C_d,\mathbf{o}]\).
Thus, in order to describe the distribution of the random root $\textbf{o}$ of $C_d$, it is enough to define $\mathbb{P}[\textbf{o} \in L_i]$.
Let
$$p_i = \mathbb{P}[\textbf{o} \in L_i], \ i \geq 0.$$

Let $T_{d+1}$ be the directed tree whose vertices have one out-degree and $d$ in-degree, i.e., $T_{d+1}$ is the $d+1$ regular tree with one selected end.
Consider the parent vertex-shift on \(C_d\).
In Proposition \ref{prop_f_prob_canopy}, we show that the randomly rooted Canopy tree \([C_d,\mathbf{o}]\) has a unique \(F\)-probability \(T_{d+1}\).

Recall the notations of Section \ref{sec_canonical_prob} of Chapter \ref{chapter_prelim}. 

\begin{lemma}
  Let $u$ be a vertex of the canopy tree $C_d$ that belongs to level $L_i$.
  Then
  $$((C_d,u)_r,u) \sim ((T_{d+1},v)_r,v) \text{ for all } 0 \leq r \leq i.$$
\end{lemma}

\begin{proof}
  We prove the lemma by induction on $r$.
  For $r=0$, $((C_d,u)_r,u) \sim ((T_{d+1},v)_r,v)$.
  Let us assume that the statement is true for $r-1 \geq 0$ and let $\rho: (C_d,u)_{r-1} \rightarrow (T_{d+1},v)_{r-1}$ be a network isomorphism such that $\rho(u)=v$.
  To show $((C_d,u)_r,u) \sim ((T_{d+1},v)_r,v)$, it is enough to extend the isomorphism $\rho$ to every vertex in $ (C_d,u)_r \backslash (C_d,u)_{r-1}$.

  If $w \in (C_d,u)_r \backslash (C_d,u)_{r-1}$ then $w$ can only be one of the three cases: an ancestor of $u$, or a descendant of $u$, or a descendant of an ancestor of $u$.
  \begin{itemize}
  \item If $w$ is an ancestor of $u$ then $w = F^r(u)$ and $F^{r-1}(u) \in (C_d,u)_{r-1}$.
    We extend $\rho$ to $w$ by assigning $\rho(w)$ to the parent of $\rho(F^{r-1}(u))$ in $T_{d+1}$.

  \item If $w$ is a descendant either of $u$ or, of an ancestor of $u$ then $F(w) \in (C_d,u)_{r-1}$.
    As $i \geq r$, the vertex $F(w) \not \in L_0$.
    Therefore, $F(w)$ has $d$-children and we can obtain a bijection between the children of $F(w)$ and the children of $\rho(F(w))$.
    We extend $\rho$ to the children of $F(w)$ by using this bijection.
  \end{itemize}
\end{proof}

Let $[C_d,u_i]$ be the equivalence class of any $u_i$ vertex in $L_i$.
By the above lemma, the distance between $[C_D,u_i]$ and $[T_{d+1},v]$ in $\mathcal{T}_*$ is bounded above by \(\frac{1}{2^i}\).
Thus, $[C_d,u_i] \rightarrow [T_{d+1},v]$ in $\mathcal{G}_*$ as $i$ tends to infinity.
We now show that the \(F\)-probability of a randomly rooted Canopy tree \([C_d,\mathbf{o}]\) is \([T_{d+1},v]\), for any vertex \(v\)  of \(T_{d+1}\).

\begin{proposition}\label{prop_f_prob_canopy}
  Let $[C_d,\textbf{o}]$ be the canopy tree with $\textbf{o}$ the root \(\mathbf{o}\) chosen randomly
  and $[C_d,\textbf{o}_i] \in \mathcal{T}_*$ be the canopy tree with random root $\textbf{o}_i$, where $\textbf{o}_i = F^i(\textbf{o}), \ i \geq 0 $.
  Then $[C_d,\textbf{o}_i]$ converges in distribution to $[T_{d+1},v]$ as \(i \to \infty\), for any vertex \(v\)  of \(T_{d+1}\).
\end{proposition}

\begin{proof}
  Let g be a continuous bounded function on $\mathcal{G}_*$, $u_k \in L_k$  be a vertex of $C_d$, $g_k:=g([C_d,u_k])$, and $a := g([T_{d+1},v])$.
  As $[C_d,u_k] \rightarrow [T_d,v]$ in $\mathcal{G}_*$ when $k$ tends to infinity, and by the continuity of the function $g$, $g_k$ converges to $a$.

  The probability, $\mathbb{P}[\textbf{o}_i \in L_k] = \mathbb{P}[F^i(\textbf{o}) \in L_k] = \mathbb{P}[\textbf{o} \in L_{k-i}]$ is given by
\[\mathbb{P}[\textbf{o}_i \in L_k] = \begin{cases}
      0 & k-i<0 \\
      p_{k-i} & k-i \geq 0.
   \end{cases}
\] 
and  we have
\begin{align*}
  \mathbb{E}[g([C_d,\textbf{o}_i])] &= \sum_{k=0}^{\infty}g([C_d,u_k]) \mathbb{P}[\textbf{o}_i \in L_k]=\sum_{k=i}^{\infty}g_k p_{k-i}.
\end{align*}

Since $g_i$ converges to $a$ as $i$ tends to infinity, given $\epsilon >0 $ there exists $N$ such that for all $i \geq N$, $|g_i-a|< \epsilon$.
Therefore, for such $i$,

\begin{align*}
 \left|\mathbb{E}[g([C_d,\textbf{o}_i])] - g([T_{d+1},v])\right| &= \left| \sum_{k=i}^{\infty}g_k p_{k-i} -a\right|= \left| p_0(g_i-a)+p_1(g_{i+1}-a)+\cdots \right|\\
 &\leq p_0|g_i-a|+p_1|g_{i+1}-a|+\cdots+p_k|g_{i+k}-a|+ \cdots\\
& < \epsilon \left(\sum_{k=0}^{\infty}p_k\right)< \epsilon.
\end{align*}
\end{proof}

\begin{remark}
  Observe that the map \(\theta_F\) induced by the parent vertex-shift \(F\) is continuous on \(\mathcal{T}_*\).
  It is also clear that the randomly rooted Canopy tree \([C_d,\mathbf{o}]\) evaporates under \(F\) as the root \(\mathbf{o}\) has a.s. finitely many descendants.
  Therefore, by Corollary \ref{cor_f_prob_sing}, the \(F\)-probability of \([C_d,\mathbf{o}]\) is singular with respect to the distribution of \([C_d,\mathbf{o}]\).
  This is indeed true since the directed regular tree \([T_{d+1},u]\) has infinitely many descendants.
\end{remark}

\subsection{Typically rooted Galton-Watson Tree (TGWT)}

Let \(\pi\) be a probability distribution on \(\{0,1,2,\cdots\}\) with mean \(m(\pi)<1\) and \([\mathbf{T}',\mathbf{o}']\) be a random Family Tree distributed as \(GW(\pi)\).
The \textbf{Size-Biased Galton-Watson Tree} with offspring distribution \(\pi\), denoted as \(SBGW(\pi)\) is a random Family Tree \([\bar{\mathbf{T}},\bar{\mathbf{o}}]\) whose distribution is given by:
\begin{equation}
  \mathbb{P}[[\bar{\mathbf{T}},\bar{\mathbf{o}}] \in A] = \frac{1}{\mathbb{E}[\#V(\mathbf{T}')]}\mathbb{E}[\#V(\mathbf{T}') \mathbf{1}_A([\mathbf{T}',\mathbf{o}'])],
\end{equation}
for every measurable set \(A\).
Note that \(\mathbb{E}[\#V(\mathbf{T})]\) is finite since \(m(\pi)<1\).

We show that the \(F\)-probability of \(TGWT(\pi)\) is \(SBGW(\pi)\) using Theorem \ref{thm:1_periodic}.

\begin{proposition}
  Let \(m(\pi)<1\).
  Then, the \(F\)-probability of \(TGWT(\pi)\) is \(SBGW(\pi)\).
\end{proposition}
\begin{proof}
  Let \([\mathbf{T},\mathbf{o}]\) be a random Family Tree distributed as \(TGWT(\pi)\) and let \([\mathbf{T}',\mathbf{o}']\) be a random Family Tree distributed as \(GW(\pi)\).
 Since, the root \(\mathbf{o}\) has a.s. finitely many ancestors, \([\mathbf{T},\mathbf{o}]\) is a.s. \(1\)-periodic with respect to the parent vertex-shift \(F\).
 By Theorem \ref{thm:1_periodic}, the \(F\)-probability of \([\mathbf{T},\mathbf{o}]\) exists.
 Let \([\bar{\mathbf{T}},\bar{\mathbf{o}}]\) be the \(F\)-probability of \([\mathbf{T},\mathbf{o}]\).
 Since \(TGWT(\pi)\) is unimodular (see Proposition \ref{20230305162953}), by Theorem \ref{thm:1_periodic}, we have
 \[\mathbb{P}[[\bar{\mathbf{T}},\bar{\mathbf{o}}] \in A] = \mathbb{E}[\#V(\mathbf{T})\mathbf{1}\{\mathbf{o} \text{ is a trap }\}\mathbf{1}_A([\mathbf{T},\mathbf{o}])],\]
 for every measurable set \(A\).
 By Proposition \ref{prop_tgwt_is_typical}, we have 
 \[\mathbb{P}[[\mathbf{T},\mathbf{o}] \in A] =\frac{1}{ \mathbb{E}[\#V(\mathbf{T}')]}\mathbb{E}\left[\sum_{u \in V(\mathbf{T}')}\mathbf{1}_A([\mathbf{T}',u])\right],\]
 for every measurable set \(A\).
 Combining the above two equations, we obtain
 \begin{align*}
  \mathbb{P}[[\bar{\mathbf{T}},\bar{\mathbf{o}}] \in A] &= \frac{1}{ \mathbb{E}[\#V(\mathbf{T}')]}\mathbb{E}\left[\#V(\mathbf{T}')\sum_{u \in V(\mathbf{T}')}\mathbf{1}\{\mathbf{u} \text{ is a trap }\}\mathbf{1}_A([\mathbf{T}',u])\right]\\
  &= \frac{1}{ \mathbb{E}[\#V(\mathbf{T}')]} \mathbb{E}\left[\#V(\mathbf{T}')\mathbf{1}_A([\mathbf{T}',\mathbf{o}'])\right],
 \end{align*}
 for every measurable set \(A\).
 The last step follows since the root is the only vertex that does not have a parent and thus the statement is proved.
\end{proof}

\subsection{Unimodular Eternal Galton-Watson Tree}\label{subsec:F_prob_EGWT}

 In the following, we introduce a new EFT called the Eternal Kesten Tree (\(EKT\)) and show that it is the \(F\)-probability of the unimodular Eternal Galton-Watson Tree.
The \(EKT\) is an eternal  version of Kesten's tree which was defined in \cite{abrahamIntroductionGaltonWatsonTrees2015}.

Let $[\mathbf{T},\mathbf{o}]$ be a rooted  Family Tree. 
A vertex of $\mathbf{T}$ is called {\bf common($\pi$)} if its offspring distribution is $\pi$ and all of its descendants reproduce independently with offspring distribution $\pi$.
A vertex of $\mathbf{T}$ is {\bf special($\pi$)} if it satisfies the following three properties: (1) its offspring distribution is $\hat{\pi}$ (the size-biased distribution of \(\pi\)), (2) all of its descendants reproduce independently, (3) for all \(n \geq 1\), exactly one of its \(n\)-th descendants reproduces with distribution $\hat{\pi}$ and the rest (if any) are common($\pi$).

\noindent \textbf{The ordered Eternal Kesten Tree:} 
Let \(\pi\) be a probability distribution on \(\{0,1,2,\cdots\}\) with mean \(m(\pi)=1\).
The ordered Eternal Kesten Tree with offspring distribution $\pi$, denoted as $EKT(\pi)$, is defined as follows: it is a rooted Eternal Family Tree $[\mathbf{S},\mathbf{u}]$ containing two types of individuals (vertices) namely common($\pi$) and special($\pi$).
 It has a unique bi-infinite $F$-path $(\mathbf{o}_n)_{n \in \mathbb{Z}}$ of special(\(\pi\)) individuals with $\mathbf{o}_0 = \mathbf{u}$, i.e., $\forall n \in \mathbb{Z}$, $\mathbf{o}_n$ is a parent of $\mathbf{o}_{n-1}$ and \(\mathbf{o}_{n-1}\) is the special(\(\pi\)) child of $\mathbf{o}_n$.
All the remaining vertices are common($\pi$).
The order among the children of every vertex is uniform.
All individuals reproduce independently.

The relation between the Eternal Kesten Tree and the bi-variate Eternal Kesten Tree (defined in Subsection \ref{subsec_bi_variate_ekt}) is as follows: Let \(\pi\) be a probability distribution on \(\{0,1,2,\cdots\}\) with mean \(m(\pi)=1\), \(\hat{\pi}\) be the size-biased distribution of \(\pi\), and \(\alpha\) be the distribution given by \(\alpha(k)=\hat{\pi}(k+1)\) for all \(k \geq 0\).
Then, the unordered \(EKT(\pi)\) is the same as the unordered \(EKT(\alpha,\pi)\), and the ordered \(EKT(\pi)\) is obtained by taking the uniform order on \(EKT(\alpha,\pi)\).

\begin{figure}[htbp]
  \centering 
  \includegraphics[scale=0.8]{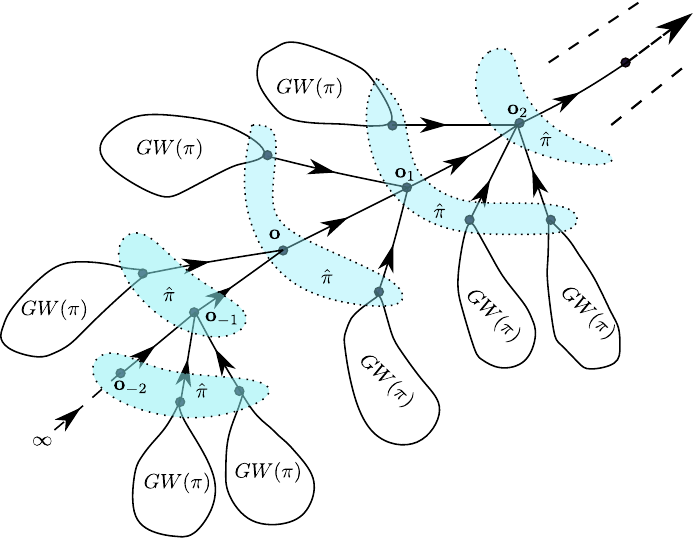}
  \caption{The Eternal Kesten tree with distribution $\pi$. Notation: $GW(\pi)$ denotes Galton-Watson tree; $\hat{\pi}$ is the size-biased distribution of $\pi$.}
  \label{fi:EKT}
\end{figure}

\begin{remark} \label{remark_supp_EKT}
For all \(n \in \mathbb{Z}\), the descendant trees of siblings of $\mathbf{o}_{n}$ are independent critical Galton-Watson trees with offspring distribution $\pi$, and thus they are all finite trees.
 Therefore, every realization of $EKT(\pi)$ has a unique bi-infinite $F$-path passing through the root. See Figure \ref{fi:EKT} for an illustration.
\end{remark}

We follow the notations of Section \ref{sec_canonical_prob}, Chapter \ref{chapter_prelim}.

\begin{proposition} \label{prop:F_prob_EGWT_EKT}
  The $F$-probability of the unimodular (ordered) Eternal Galton-Watson Tree $[\mathbf{T},\mathbf{o}]$ with offspring distribution $\pi$ is the (ordered) Eternal Kesten Tree $[\mathbf{S},\mathbf{u}]$ with distribution $EKT(\pi)$.
\end{proposition}

\begin{proof}
  Let $r>0$, and $\mathbf{u}_{-r},\mathbf{u}_{-r+1},\cdots,\mathbf{u}_{-1}$ be the special descendants of $\mathbf{u}_0=\mathbf{u}$ in \(\mathbf{S}\).
Note that $F^n(\mathbf{o})$ has at least one \(k\)-th descendant, for all $1 \leq k \leq n$.
  For all $n>r$, the distributions of $[\mathbf{T},F^n(\mathbf{o})]_r$ and $[\mathbf{S},\mathbf{u}]_r$ are the same.
  Indeed, the offspring distribution of $F^n(\mathbf{o})$ is the same as that of $\mathbf{u}$ which is $\hat{\pi}$.
The special child of $\mathbf{u}$ has the same offspring distribution $\hat{\pi}$ as that of $F^{n-1}(\mathbf{o})$ and the remaining other children of \(u\) follow the offspring distribution $\pi$.
By induction, $F^{n-r+1}(\mathbf{o})$ has the same distribution as that of the special $n-r+1$-th descendant of $\mathbf{u}$. 
The same is true for the remaining vertices of $[\mathbf{T},F^n(\mathbf{o})]_r$ and $[\mathbf{S},\mathbf{u}]_r$.
Therefore, 
$$
[\mathbf{S},\mathbf{u}] \stackrel{\mathcal{D}}{=} \lim_{n \rightarrow \infty} [\mathbf{T},F^n(\mathbf{o})] .
$$

In addition, if $[\mathbf{T},\mathbf{o}]$ is an ordered $EGWT(\pi)$ and $[\mathbf{S},\mathbf{u}]$ is an ordered $EKT(\pi)$, then the order of $F^i(\mathbf{o})$ is uniform among the children of $F^{i+1}(\mathbf{o})$ for $n-r+1 \leq i \leq n-1$, which has the same distribution as the order of $u_{i-n}$ among the children of $u_{i-n+1}$.
The same is true for the order of the remaining vertices of \([\mathbf{T},F^n(\mathbf{o})]_r\) and \([\mathbf{S},u]_r\).
Thus, the \(F\)-probability of the unimodular ordered \(EGWT(\pi)\) is the ordered \(EKT(\pi)\).
\end{proof}

\begin{remark}
Consider a realization of $EGWT(\pi)$, where $\pi$ has mean $m(\pi)=1$ and satisfies $\pi(1)<1$.
Then, its root has finitely many descendants, i.e., \(EGWT(\pi)\) evaporates under the parent vertex-shift.
In contrast, the root of a realization of $EKT(\pi)$ has infinitely many descendants.
Therefore, $EKT(\pi)$ is \emph{singular} with respect to $EGWT(\pi)$.
This is another instance where the \(F\)-probability is singular with respect to the initial distribution.
This fact can also be inferred from Lemma \ref{lm:evap_lim2} and Corollary \ref{cor_f_prob_sing}, since the parent vertex-shift is continuous on \(\mathcal{T}_*\).
\end{remark}

\subsection{ Unimodularised bi-variate marked Eternal Kesten Tree}\label{subsec_f_prob_ecs_ekt}

Let $\alpha,\beta$ be probability distributions on $\{0,1,2,\cdots\}$ such that their means satisfy $m(\alpha)<\infty$ and $m(\beta) < 1$.
Let $[\mathbf{T}',\mathbf{o}',Z']$ be a random marked Family Tree whose distribution is $MEKT(\alpha,\beta)$, where \(Z'\) is a mark. 
Let \([\mathbf{T}',\mathbf{o}',Z']\) be the joining of $([\mathbf{T}_i,\mathbf{o}_i,Z_i])_{i \in \mathbb{Z}}$, the i.i.d. sequence of random marked Family Trees defined in Remark \ref{remark_joining_ekt}.

Let \([\mathbf{T},\mathbf{o},Z]\) denote the size-biased joining of $([\mathbf{T}_i,\mathbf{o}_i,Z_i])_{i \in \mathbb{Z}}$.
Recall that the distribution of $[\mathbf{T},\mathbf{o},Z]$ is given by:
\begin{equation}\label{eq_size_biased}
  \mathbb{P}[[\mathbf{T},\mathbf{o},Z] \in A] = \frac{\mathbb{E}[\sum_{u \in V(T_0)}\mathbf{1}_A([\mathbf{T}',u,Z'])]}{\mathbb{E}[\#V(\mathbf{T}_0)]},
\end{equation}
for any measurable set $A$.
Observe that $\mathbb{E}[\#V(\mathbf{T}_0)]< \infty$ because $\beta <1$,  and the descendant trees of the children of $\mathbf{o}_0$ (in \(\mathbf{T}_0\)) are subcritical Galton-Watson trees.

Recall that $\hat{\mathcal{T}}_*$ denotes the space of rooted marked Family Trees.
For any $[T,u,z] \in \hat{\mathcal{T}}_*$ that has a unique bi-infinite $F$-path $\neswarrow(T)$, for any \(w \in V(T)\), let
\[M(w):= \inf\{n \geq 0: F^n(w) \in \neswarrow(T)\},\]
with the convention that $F^0(w)=w$ and let  $w^*=F^{M(w)}(w)$ be the smallest ancestor of $w$ that lies on the unique bi-infinite $F$-path of $T$.
Let $[T,u,z]_r$ denote the subtree induced by the vertices of $T$ that are at a graph distance of at most $r$ from \(u\) and with the mark function restricted to these vertices.

\begin{proposition}\label{20230314185620}
    The $F$-probability of the unimodularized \(MEKT(\alpha,\beta)\) is \(MEKT(\alpha,\beta)\).
\end{proposition}
 \begin{proof}
 Fix $r>0$.
 Let $A$ be a measurable subset of $\hat{\mathcal{T}}_*$, and let $n>r$.
 Then,
 \begin{align*}
    \mathbb{P}[[\mathbf{T},F^n(\mathbf{o}),Z]_r \in A] =& \mathbb{P}\left[[\mathbf{T},F^n(\mathbf{o}),Z]_r \in A,M(\mathbf{o})<n-r\right]\\
    &+\mathbb{P}\left[[\mathbf{T},F^n(\mathbf{o}),Z]_r \in A,M(\mathbf{o})\geq n-r\right].
 \end{align*}
 The second sum is bounded by $\mathbb{P}[M(\mathbf{o}) \geq n-r]$.
 As $n \to \infty$, $\mathbb{P}[M(\mathbf{o}) \geq n-r] \to 0$ since $M(\mathbf{o})$ is a.s. finite, which follows from the finiteness of the tree $\mathbf{T}_0$.
  Therefore, the second sum tends to $0$ as $n \to \infty$.

  Consider the first sum.
  We have,
  \begin{align*}
    \mathbb{P}[\mathbf{T},F^n(\mathbf{o}),Z]_r \in A,&M(\mathbf{o})<n-r] \\
    &= \sum_{k=0}^{n-r-1}\mathbb{P}[[\mathbf{T},F^n(\mathbf{o}),Z]_r \in A, M(\mathbf{o})=k]\\
    &=\sum_{k=0}^{n-r-1}\mathbb{P}[[\mathbf{T},F^{n-k}(\mathbf{o}^*),Z]_r \in A, M(\mathbf{o})=k].
  \end{align*}
  The last equation follows by the definition of \(\mathbf{o}^*: = F^{M(\mathbf{o})}(\mathbf{o})\).
  By the definition of $[\mathbf{T},\mathbf{o},Z]$ (see Eq. (\ref{eq_size_biased})), the last equation is equal to 
  \begin{equation}\label{eqn:1_2.38}
    \sum_{k=0}^{n-r-1}\frac{\mathbb{E}\left[\sum_{u \in V(\mathbf{T}_0)}\mathbf{1}_A([\mathbf{T}',F^{n-k}(u^*),Z']_r) \mathbf{1}\{M(u)=k\}\right]}{\mathbb{E}[\#V(\mathbf{T}_0)]}.
  \end{equation}
  On the tree $\mathbf{T}'$, we have $F^{n-k}(u^*)=\mathbf{o}_{n-k}$, for all $u \in V(\mathbf{T}_0)$ and for all $0 \leq k \leq n-r-1$.
  Since \(([\mathbf{T}_i,\mathbf{o}_i,Z_i])_{i \in \mathbb{Z}}\) is an i.i.d. sequence, the distribution of $[\mathbf{T}',\mathbf{o}_{n-k},Z']_r$ is independent of $[\mathbf{T}_0,\mathbf{o}_0,Z_0]$ for all $0 \leq k \leq n-r-1$.
  By applying this fact to Eq. (\ref{eqn:1_2.38}), it is equal to 
  \begin{align*}
    &\sum_{k=0}^{n-r-1}\frac{\mathbb{E}\left[\sum_{u \in V(\mathbf{T}_0)}\mathbf{1}_A([\mathbf{T}',\mathbf{o}_{n-k},Z']_r) \mathbf{1}\{M(u)=k\}\right]}{\mathbb{E}[\#V(\mathbf{T}_0)]}\\
    &=\sum_{k=0}^{n-r-1} \frac{\mathbb{P}[[\mathbf{T}',\mathbf{o}_{n-k},Z']_r \in A]  \mathbb{E}\left[\sum_{u \in V(\mathbf{T}_0)} \mathbf{1}\{M(u)=k\}\right]}{\mathbb{E}[\#V(\mathbf{T}_0)]}.
  \end{align*}
  The distribution of \([\mathbf{T}',\mathbf{o}_{n-k},Z']_r\) is identical to that of \([\mathbf{T}',\mathbf{o}',Z']_r\) for all \(n>2r+k\), by the construction of \(MEKT(\alpha,\beta)\).
  Using this and applying Eq. (\ref{eq_size_biased}) to the last equation, we get
  \begin{align*}
    \mathbb{P}\left[[\mathbf{T},F^n(\mathbf{o}),Z]_r \in A,M(\mathbf{o})<n-r\right]&=\mathbb{P}[[\mathbf{T}',\mathbf{o}',Z']_r \in A]  \sum_{k=0}^{n-r-1}\mathbb{P}[M(\mathbf{o})=k]\\
    &=\mathbb{P}[[\mathbf{T}',\mathbf{o}',Z']_r \in A] \mathbb{P}[M(\mathbf{o})<n-r].
  \end{align*}
  As $n \to \infty$, $\mathbb{P}[M(\mathbf{o})<n-r]\to 1$.
Therefore,
  \[\mathbb{P}[[\mathbf{T},F^n(\mathbf{o}),Z]_r \in A] \xrightarrow[n \to \infty]{} \mathbb{P}[[\mathbf{T}',\mathbf{o}',Z']_r \in A].\]
  Since this is true for any $r\geq 1$, we have 
  \[[\mathbf{T},F^n(\mathbf{o}),Z] \xrightarrow[n \to \infty]{} [\mathbf{T}',\mathbf{o}',Z'].\]
 \end{proof}

\begin{remark}
From the construction of \(MEKT(\alpha,\beta)\), it follows that both  \([\mathbf{T}',F(\mathbf{o}'),Z']\) and \([\mathbf{T}',\mathbf{o}',Z']\) have the same distribution.
So, the marked EFT \([\mathbf{T}',\mathbf{o}',Z']\)  satisfies the Mecke's invariant measure equation.
Therefore, by Lemma \ref{lm:evap_lim2}, \(\mathbf{o}'\) has a.s. infinitely many descendants, which can also be seen directly from the construction of \(MEKT(\alpha,\beta)\).
\end{remark}


\chapter{Record vertex-shift probability for the random walk case}\label{chapter_record_V_shift_probability}

In this chapter, we focus again on the record vertex-shift on the network \([\mathbb{Z},0,X]\), where  \(X\) is an i.i.d. sequence as in Def. \ref{hyp:increments} and satisfies \(\mathbb{E}[X_0]\geq 0\).
We show the existence of a unique \(R\)-probability (i.e., a vertex-shift probability under record vertex-shift) when \(\mathbb{E}[X_0]\geq 0\).
Although one could show the same when \(\mathbb{E}[X_0]< 0\), we did not consider it since the component of \(0\) can only give partial information of the trajectory of random walk associated to \(X\).

We use the FB Lemma \ref{thm_forward_backward} to prove the existence and uniqueness of the  \(R\)-probability.
We define two maps namely the backward map \(\Phi_R\) and the extended backward map \(\hat{\Phi}_R\) that satisfy the hypotheses of the FB Lemma when \(\mathbb{E}[X_0]=0\) and \(\mathbb{E}[X_0]>0\) respectively.
This is done in several sections.
We also give constructions for the \(R\)-probability in these two cases and prove that the constructions indeed generate the \(R\)-probability.

{\bf (In French)} Dans ce chapitre, nous nous intéressons à nouveau au décalage de sommet des records sur le réseau \([\mathbb{Z},0,X]\), où  \(X\) est une suite i.i.d. comme dans la définition \ref{hyp:increments} et qui satisfait \(\mathbb{E}[X_0]\geq 0\).
Nous montrons l'existence d'une \(R\)-probabilité unique (c'est-à-dire une probabilité de décalage de sommet associée au décalage de sommet des records) lorsque \(\mathbb{E}[X_0]\geq 0\).
Bien que l'on puisse montrer la même chose lorsque \(\mathbb{E}[X_0]< 0\), nous ne l'avons pas considéré car la composante de \(0\) ne peut donner qu'une information partielle sur la trajectoire de la marche aléatoire associée à \(X\).

Nous utilisons le lemme FB \ref{thm_forward_backward} pour prouver l'existence et l'unicité de la \(R\)-probabilité.
Nous définissons deux fonctions nommées la fonction réciproque \(\Phi_R\) et la fonction réciproque étendue \(\hat{\Phi}_R\) qui satisfont les hypothèses du lemme FB lorsque \(\mathbb{E}[X_0]=0\) et \(\mathbb{E}[X_0]>0\) respectivement.
Ceci est fait dans plusieurs sections.
Nous donnons également des constructions pour la \(R\)-probabilité dans ces deux cas et prouvons que les constructions génèrent bien la \(R\)-probabilité.

\section{The backward map of the record vertex-shift}\label{subsec_backward_map}

In this section, we define a backward map $\Phi_R:\mathcal{T}_*' \to \mathcal{G}_*$, where $\mathcal{T}_*'$ is the subset of $\mathcal{T}_*$ consisting of all ordered Family Trees that have at most one bi-infinite $F$-path.
Recall that the forward map \(\Psi_R: \mathcal{G}_{*} \to \mathcal{T}_*'\) maps a network of the form \([\mathbb{Z},0,x]\) to \([\mathbb{Z}^R(0),0]\), where \(x\) is a sample of the i.i.d. sequence \(X\) with \(\mathbb{E}[X_0]=0\), and \(\mathbb{Z}^R(0)\) is the connected component of \(0\) of the record graph \(\mathbb{Z}^R\) of \([\mathbb{Z},0,x]\).
For any ordered Family Tree that has at most one bi-infinite path, using the depth-first search order and the tree structure, we can define a sequence in such a way that the record graph associated to the sequence is the Family Tree.
This map essentially depends on a sequence of vertices of the Family Tree called succession line, which can be seen as a factor map from \(\mathbb{Z}\) to the vertices of a unimodular Family Tree (\cite{ferrariPoissonTreesSuccession2004a}).

This map extends the usual depth-first search encoding of finite ordered trees, defined in \cite[Chapter~5]{jimpitmanCombinatorialStochasticProcesses2006}, \cite{legallRandomTreesApplications2005b}, to infinite ordered Family Trees that have at most one bi-infinite path.

In the next sections,we show that the map \(\Phi_R\) satisfies the three conditions of the FB Lemma \ref{thm_forward_backward} when \(\mathbb{E}[X_0]=0\).
However, this maps fails to be bijective when \(\mathbb{E}[X_0]>0\).
So, in Section \ref{subsec_reverse_map_I_F}, we extend this map to the space of marked Family Trees in order to obtain a bijective map.

In Subsection \ref{subsec:r_graph_egwt} of Chapter \ref{chapter_record_v_shift}, we show that the $R$-graph of \([\mathbb{Z},0,X]\) when $\mathbb{E}[X_0]=0$ is a unimodular ordered Eternal Galton-Watson Tree $EGWT(\pi)$ with offspring distribution $\pi$, where $\pi$ is the distribution of $X_{0}+1$.
The parent vertex-shift probability of $EGWT(\pi)$ is shown to be  the Eternal Kesten Tree $EKT(\pi)$ (see Subsection \ref{subsec:F_prob_EGWT} for the definition of \(EKT(\pi)\)).
In contrast to \(EGWT(\pi)\), when \(\pi(1)<1\), a realization of $EKT(\pi)$ has a unique bi-infinite path.
So, the map $\Phi_R$ should be defined such that it is continuous on the support of $EKT(\pi)$ in order to satisfy the hypotheses of FB Lemma \ref{thm_forward_backward}.
Therefore, our first goal is to define $\Phi_R$ on ordered Family Trees that have at most one bi-infinite path.
We achieve this goal using the properties of the RLS order.
Then, we show that the defined map $\Phi_R$ is continuous $EKT(\pi)$-a.s..

Let $T$ be an ordered Family Tree.
Consider the RLS order $\prec$ on the vertices of $T$ and recall that for any two vertices \(u\) and \(v\), we say \(u\) succeeds \(v\) if \(u \prec v\) (equivalently \(v\) precedes \(u\)).
\begin{definition}
  For any vertex $u \in T$, let $B(u):=\{v \in T: v \prec u\}$ and define $b(u) = \max\{v \in B(u)\}$ if it exists, where $\max$ is taken with respect to $\prec$.
Similarly, let $A(u):= \{v \in T: v \succ u\}$ and  $a(u)= \min\{u \in A(u)\}$ if it exists, where $\min$ is taken with respect to $\prec$.
\end{definition}

According to the RLS order, $b(u)$ has the highest precedence over all those vertices that succeed $u$, and $a(u)$ has the smallest precedence among those vertices that precede $u$.

\begin{lemma}\label{lemma:b_exists}
  Let $T$ be an ordered Family Tree and $u \in T$.
  If $u$ is not the smallest vertex, then $b(u)$ exists.
\end{lemma}
\begin{proof}
  By hypothesis, $B(u)$ is not empty.
  So, either of the following two cases occur:
  \begin{itemize}
    \item If $u$ has children, then since it could only have finitely many children, $b(u)$ is the largest(eldest) child of $u$.
    \item If $u$ does not have any child, let $w=F^m(u)$ (for some $m \geq 1$) be the smallest ancestor of $u$ that has a child smaller than $F^{m-1}(u)$, with the notation $F^0(u)=u$. Such an $m$ exists because $B(u)$ is non-empty.
    Then, $b(u)$ is the largest among those siblings of $F^{m-1}(u)$ that are smaller than $F^{m-1}(u)$.
  \end{itemize}
\end{proof}

\begin{lemma} \label{lemma:rls_order_previous_element_exists}
  Let $T$ be an ordered Family Tree, $u$ be a non-maximal vertex of $T$.
 If every sibling of $u$ that precedes $u$ has finitely many descendants, then $a(u)$ exists.
\end{lemma}
\begin{proof}
  If $u$ does not have a sibling that precedes $u$, then $a(u)=F(u)$ (by hypothesis $F(u) \not = u$ in this case).
  Otherwise, let $w$ be the smallest sibling of $u$ larger than $u$.
  It has finitely many descendants by hypothesis.
  Then $a(u)$ is the minimal vertex of $D(w)$ (recall that $D(w) = \{v \in V(T)\backslash \{w\}: F^m(v)=w \text{ for some } m>0\} \cup \{w\}$).
\end{proof}

\begin{definition}\label{defn_succession_line}
    Consider an ordered rooted Eternal Family Tree $(T,o)$ that has at most one bi-infinite $F$-path passing through $o$.
    With the RLS order, we call \textbf{succession line} of $(T,o)$ the sequence $U((T,o))=(u_n)_{n \in \mathbb{Z}}$ iteratively given by $u_0=o$ and,
    \begin{align} \label{eq:succ_line}
      u_n &= \begin{cases}
        b(u_{n+1}) \text{ if it exists}\\
        u_{n+1} \text{ otherwise}
      \end{cases}
      \text{ for $n<0$},\\ \nonumber
      u_n &= \begin{cases}
        a(u_{n-1}) \text{ if it exists}\\
        u_{n-1} \text{ otherwise}
      \end{cases}
      \text{ for $n>0$}.
    \end{align}
    
\end{definition}

Observe that this is the unique sequence that satisfies $u_0=o$, and $u_{n-1}$ is the immediate successor of $u_{n}$, when they are distinct, for all $n \in \mathbb{Z}$.
With the above preparation, we are ready to define the backward map which plays a crucial role in the rest of the chapter.

\begin{definition}
  Define the {\bf backward map of the record vertex-shift} $\Phi_R:\mathcal{T}_*' \to \mathcal{G}_*$ as the function that maps $[T,o] \in \mathcal{T}_*'$ to 
  \begin{equation} \label{eq:reverse_map}
    \Phi_R([T,o])=[\mathbb{Z},0,x((T,o))],
  \end{equation}
  where $x((T,o)) = (x_n)_{n \in \mathbb{Z}}$ is given by $x_n = d_1(u_{n+1})-1$ and $(u_n)_{n \in \mathbb{Z}}=U((T,o))$.
\end{definition}

Note that \(\Phi_R\) is well-defined: if \((T,o) \sim (T',o')\) be two isomorphic ordered Family Trees and \((u_n)_{n \in \mathbb{Z}}\) and \((u'_n)_{n \in \mathbb{Z}}\) be their respective succession lines passing through \(o,o'\) respectively, then \(d_1(u_n)-1 = d_1(u_n')-1\) for all \(n \in \mathbb{Z}\), which implies that \(x((T,o))=x((T',o'))\).
In general the succession line of a rooted Family Tree \([T,o]\) passing through \(o\) may not contain all the vertices of \(T\).
But, if \([T,o]\) is a sample of component of \(0\) in the record graph of \([\mathbb{Z},0,X]\) with \(\mathbb{E}[X_0]\geq 0\), then we show that the succession line of \([T,o]\) contains every vertex of \(T\) (see Lemma \ref{lemma_succ_line_bijective} and Remark \ref{remark_succession_line_is_integers} ).

\begin{lemma}\label{lemma_succ_line_bijective}
  Let \(X=(X_n)_{n \in \mathbb{Z}}\) be the i.i.d. increments of a skip-free to the left random walk with \(\mathbb{E}[X_0] \geq 0\) and \([\mathbb{Z},0,X]\) be the network associated to \(X\).
  Let $[\mathbf{T},0]$ be the component of \(0\) in the record graph of \([\mathbb{Z},0,X]\) rooted at \(0\).
  Then, a.s., the succession line $U((\mathbf{T},\mathbf{o}))$ consists of distinct entries.
\end{lemma}

\begin{proof}
  Observe that the set of minimal vertices of $\mathbf{T}$ is a covariant subset.
  It can only be either empty or a singleton, because $\mathbf{T}$ is totally ordered by the RLS order.
  But $[\mathbf{T},\mathbf{o}]$ is unimodular (Lemma \ref{lemma:f_graph_unimodular}), and \(\#V(\mathbf{T})\) is infinite since \(\mathbb{E}[X_0]\geq 0\), in fact \(V(\mathbf{T}) = \mathbb{Z}\) since the record graph is connected when \(\mathbb{E}[X_0]\geq 0\) (see Proposition \ref{proposition:r-graph_is_eft} and Proposition \ref{remark_pos_drift_i_f}).
   Therefore, by No Infinite/Finite inclusion Lemma \ref{lemma:no_infinite_finite}, a.s., the minimal vertex does not exist.
  Therefore, by Lemma \ref{lemma:b_exists}, almost surely, the successor of $u$, $b(u)$ exists for all $u \in V(\mathbf{T})$.
  Since $a(b(u))=u$ for all $u \in V(\mathbf{T})$, \(a(v)\) exists for all \(v\) in the range of the map \(b\).
  But $b$ is injective a.s. and hence surjective a.s. (by the unimodularity of $[\mathbf{T},\mathbf{o}]$).
  Therefore, the range of $b$ is $V(\mathbf{T})$ a.s. and hence, a.s., $a(u)$ exists for all $u \in V(\mathbf{T})$.
  Thus, the map $n \mapsto u_n$ from $\mathbb{Z}$ to $V(\mathbf{T})$ is injective, which implies that all the elements of the sequence \(U((\mathbf{T},\mathbf{o}))\) are distinct.
\end{proof}

\begin{remark}\label{remark_succession_line_is_integers}
  Observe that \(U(\Psi_R([\mathbb{Z},0,X])) = (u_n)_{n \in \mathbb{Z}} = (n)_{n \in \mathbb{Z}}\), when the i.i.d. sequence \(X=(X_n)_{n \in \mathbb{Z}}\) satisfies \(\mathbb{E}[X_0] \geq 0\).
  This follows by combining Lemma \ref{lemma:rls_order} with the facts that the \(R\)-graph \(\mathbb{Z}^R\) of \([\mathbb{Z},0,X]\) is a.s. connected, and the map \(n \mapsto u_n\) from \(\mathbb{Z}\) to \(V(\mathbb{Z}^R) = \mathbb{Z}\) is injective.
  Indeed, by definition, \(B(0) = \{i \in \mathbb{Z}: i \prec 0\}\).
  But, by Lemma \ref{lemma:rls_order}, \(i \in B(0) \iff i<0\).
  Therefore, \(B(0) = \{i \in \mathbb{Z}:i<0\}\).
  Hence, \(b(0)=-1\) and \(u_{-1}=-1\).
  By induction, assume \(u_{-n} = -n\).
  Consider \(B(u_{-n}) = \{i \in \mathbb{Z}:i \prec -n\} = \{i \in \mathbb{Z}:i<-n\}\).
  Therefore, \(u_{-(n+1)} = -(n+1)\).
  Similarly, we could show that \(a(u_n) = n, \forall n \geq 0\).
\end{remark}

\begin{lemma}
  Let \([\mathbf{T}',\mathbf{o}']\) be a random ordered Family Tree whose distribution is \(EKT(\pi)\), where $\pi$ is a probability distribution on $\{0,1,2,\cdots\}$ with mean $m(\pi)=1$.
  Then, the succession line $U((\mathbf{T}',\mathbf{o}'))$ consists of distinct entries.
\end{lemma}
\begin{proof}
  Let $[T',o']$ be a realization of $EKT(\pi)$.
  Since $T'$ has a unique bi-infinite $F$-path passing through the root $o$, by Lemma \ref{lemma:b_exists}, $b(u)$ exists for all $u \in V(T')$.
  If $u$ satisfies the property that every sibling of $u$ that precedes $u$ has finitely many descendants, then, by Lemma \ref{lemma:rls_order_previous_element_exists}, $a(u)$ exists.
  In this case, \(a(u)\) also satisfies the same property.
  Since the root $o$ satisfies this property, the map $n \mapsto u_n$ from $\mathbb{Z}$ to $V(T')$ is injective.
  Thus, all the elements of \(U(\mathbf{T}',\mathbf{o}')\) are distinct.
\end{proof}

The following remark is used in the later sections to describe the \(R\)-probability from \(EKT\).
\begin{remark}
  Let $[T',o']$ be a realization of $EKT(\pi)$ and $\neswarrow([T',o'])$ denote its bi-infinite $F$-path.
  Since $o' \in \neswarrow([T',o'])$, we could define the following.
  The {\bf positive part} of $[T',o']$ is the rooted Family Tree $[T'^{+},o']$, where $T'^{+}$ is the subtree of $T'$ induced by the vertices $\{u \in V(T'): u \succ v \text{ for some } v \in \neswarrow([T',o']) \}$.
  Observe that the set of elements of $U((T',o'))$ is equal to $V(T')$.
\end{remark}

\section{Bijectivity and continuity of the backward map}\label{subsec_backward_cont_bij}

We first show that $\Psi_R \circ \Phi_R([T,o]) = [T,o]$ a.s. for every realization \([T,o]\) of the ordered $EGWT(\pi)$, where the mean \(m(\pi)=1\), using the following lemma.

\begin{lemma} \label{lemma:legall_sums}
  Let $T$ be a finite ordered tree and $v_1 \prec v_2 \prec \cdots \prec v_n$ be the vertices of $T$ ordered according to the RLS order. Then,
  \begin{enumerate}
    \item $\sum_{i=1}^n (d_1(v_i)-1) = -1$.
    \item $\sum_{i = 1}^k (d_1(v_i)-1) < 0$,  for all $1<k \leq n$.
  \end{enumerate}
\end{lemma}

\begin{proof}
The first equation follows from the fact that for a finite tree $T$,
\begin{equation*}
  \sum_{u \in V(T)} d_1(u)= \#V(T) - 1.
\end{equation*}

We prove the second equation by induction on $k$.
If $k=1$ then $v_1$ is a leaf and hence $d_1(v_1)-1 = -1$.
Now, assume that the statement is true for $1 \leq m \leq k$, where \(k>1\).
Observe that by the RLS order, either $v_{k+1}$ is a leaf or $v_{k+1}$ is the parent of $v_k$. 
In the latter case, there exists $1 \leq j \leq k$ such that $v_j<v_{j+1} < \cdots < v_k$ (written according to the RLS order) are all the descendants of $v_{k+1}$.
If $v_{k+1}$ is a leaf then, by induction, 
\begin{align*}
  \sum_{i=1}^{k+1} (d_1(v_i)-1) &= \sum_{i=1}^k (d_1(v_i)-1) + d_1(v_{k+1})-1 <-1.
\end{align*}

The only other possibility is that $v_{k+1}$ is the parent of $v_k$.
 Then, 
\begin{align*}
  \sum_{i=1}^{k+1} (d_1(v_i)-1) &= \sum_{i = 1}^{j-1}(d_1(v_i) - 1) +  \sum_{i=j}^{k+1}(d_1(v_i)-1) \leq 0 +( -1).
\end{align*}
The last step follows from the fact that the first sum on the right-hand side is at most $0$ (by induction), while the second sum is on the descendant tree of $v_{k+1}$ which is equal to $-1$ by the first result of the lemma.
\end{proof}

\begin{remark}\label{remark:legall_sum}
  Let  $v_1 \prec v_2 \prec \cdots \prec v_n$ be the vertices of a finite ordered tree, ordered according to RLS order.
  Then using Lemma \ref{lemma:legall_sums} we obtain, for any $1<k< n$,
  \begin{align*}
    \sum_{i=k+1}^n (d_1(v_i)-1) &= \sum_{i=1}^n (d_1(v_i)-1) - \sum_{i=1}^k (d_1(v_i)-1 ) \geq -1 - (-1) \geq 0. 
  \end{align*}
\end{remark}

\begin{proposition}\label{prop:reverse_R-graph_identity}
  Almost surely, every realisation $[T,o]$ of the ordered \\ $EGWT(\pi)$, where the mean \(m(\pi)=1\), satisfies $\Psi_R \circ \Phi_R ([T,o]) = [T,o]$.
\end{proposition}
\begin{proof}
  Let \((u_n)_{n \in \mathbb{Z}} = U((T,o))\) be the succession line of \((T,o)\) and $[\mathbb{Z},0,x] = \Phi_R([T,o])$, where $x=(x_n)_{n \in \mathbb{Z}}$ is such that $x_n = d_1(u_n)-1$ is the mark of the edge $(n,n+1)$, for all $n \in \mathbb{Z}$.

  First, observe that the map \(\alpha:\mathbb{Z}\to V(T)\) defined by \(\alpha(n)=u_n\), for all \(n \in \mathbb{Z}\) is bijective.
  This follows because, every vertex \(v\) of \(T\) has finitely many descendants.
  So, there are only finitely many vertices between \(o\) and \(v\) according to the RLS order, and hence \(u_n=v\) for some \(n \in \mathbb{Z}\).
  Therefore, \(\alpha\) is surjective.
  The injectivity of the map \(\alpha\) follows from Lemma \ref{lemma_succ_line_bijective} because \([T,o]\) is a realisation of the component of \(0\) in the record graph of \([\mathbb{Z},0,X]\) when \(\mathbb{E}[X_0]=0\). 
  We now prove that the bijective map \(\alpha:\mathbb{Z} \to V(T)\) defined by \(\alpha(n)= u_n\), for all \(n \in \mathbb{Z}\), induces a rooted network isomorphism \(\alpha\) from \(\Psi_R([\mathbb{Z},0,x])\) to \((T,o)\).
  Note that, if \(u_j\) is a parent of \(u_i\) in \(T\) for some integers \(i\) and \(j\) then \(j>i\), which follows because \((u_n)_{ \in \mathbb{Z}}\) is the succession line of \((T,o)\).
  In view of this, to prove that \(\alpha\) induces a rooted network isomorphism, it is enough to show that for all $i \in \mathbb{Z}$, if \(u_j\) is the parent of \(u_i\) in \(T\) for some  \(j>i\), then \(j=R_x(i)\) (since every vertex has at most one parent).

 Let \(i \in \mathbb{Z}\) and \(u_j\) be the parent of \(u_i\) in \(T\) for some integer $j>i$.
  If $j=i+1$, then $x_i = d_1(u_{i+1})-1 \geq 0$, because $u_i$ is a child of $u_{i+1}$.
  Hence, $R_x(i)=i+1$.

 So, let us assume that $j>i+1$.
 This implies that either $u_{i+1}$ is a leaf and a descendant of a sibling of $u_i$, or \(u_{i+1}\) is a leaf and a sibling of \(u_i\).
 The descendant tree of \(u_j\) is a finite ordered tree, as $[T,o]$ is a realization of $EGWT(\pi)$.
  Let $u_{i_1} \prec u_{i_2} \prec \cdots  \prec u_{i_n}$ be all the elder siblings of $u_i$, and $T_1,T_2,\cdots,T_n$ be their descendant trees respectively.
 For any $i<k<j$, let $l(k)$ be the smallest element of the set $\{1,2,\cdots,n\}$ such that $\{u_{i+1},u_{i+2},\cdots,u_k\} \subset V(T_1) \cup V(T_2) \cup \cdots V(T_{l(k)})$, and $u_{k_1}$ be the smallest vertex of $T_{l(k)}$.

 Then, the sum $y(i,k) = \sum_{m=i}^{k-1} x_m= \sum_{m=i+1}^{k}(d_1(u_m)-1)$ can be written as 
 \begin{align*}
   y(i,k) &= \left(\sum_{u \in V(T_1)}d_1(u)-1 \right)+ \cdots + \left(\sum_{u \in V(T_{l(k)-1})}d_1(u) - 1 \right) + \sum_{m=k_1}^k (d_1(u_m)-1)\\
   &< (-1) + \cdots + (-1) + \sum_{m=k_1}^k (d_1(u_m)-1)< 0,
 \end{align*}
 with the notation that \(V(T_0)\) is just an empty set.
 The last steps follow from the first and the second statements of Lemma \ref{lemma:legall_sums}. 

  On the other hand, the sum $y(i,j) = \sum_{m=i+1}^j (d_1(u_m)-1)$ can be written as
  \begin{align*}
    y(i,j) &= \left(\sum_{u \in V(T_1)}d_1(u)-1 \right)+ \cdots + \left(\sum_{u \in V(T_n)}d_1(u) - 1 \right) + d_1(u_j)-1\\
    &= -n + d_1(u_j) - 1 \geq 0.
  \end{align*}
  The above steps follow from the first part of Lemma \ref{lemma:legall_sums}, and by the assumption that $u_j$ has at least $n+1$ children, namely $u_i,u_{i_1},\cdots,u_{i_n}$.
  Thus, $R_x(i)=j$, completing the proof.
\end{proof}

The following proposition shows that the backward map \(\Phi_R\) satisfies the first condition of the FB lemma (Lemma \ref{thm_forward_backward}).

\begin{proposition}\label{prop:R-graph_reverse_identity}
  Let $X=(X_n)_{n \in \mathbb{Z}}$ be as in Def. \ref{hyp:increments} and satisfies $\mathbb{E}[X_0]=0$. 
  Let $x = (x_n)_{n \in \mathbb{Z}}$ be a realization of $X$, and $[\mathbb{Z},0,x]$ be its associated network.
  Then, $\Phi_R \circ \Psi_R ([\mathbb{Z},0,x]) = [\mathbb{Z},0,x]$.
\end{proposition}
\begin{proof}
  Let \((u_n)_{n \in \mathbb{Z}} = U(\Psi_R([\mathbb{Z},0,x]))\) be the succession line of \(\Psi_R([\mathbb{Z},0,x])\).
  By Remark \ref{remark_succession_line_is_integers}, \(u_n=n\) for all \(n \in \mathbb{Z}\).
  Let $[\mathbb{Z},0,y]:= \Phi_R \circ \Psi_R ([\mathbb{Z},0,x])$, where $y=(y_n)_{n \in \mathbb{Z}}$ with $y_n := d_1(u_{n+1})-1 = d_1(n+1)-1$ for all $n \in \mathbb{Z}$.
  From the first statement of Lemma \ref{lemma_offspring_count}, $x_n = d_1(n+1)-1$, for all $n \in \mathbb{Z}$.
  Therefore, $x_n=y_n$ for all $n \in \mathbb{Z}$.
\end{proof}

We now show that the map $\Phi_R$ is continuous on the set of ordered family trees that have at most one bi-infinite $F$-path.

\begin{lemma}\label{lemma:reverse_map_continuity}
  Let $[T,o]$ be an ordered EFT that has at most one bi-infinite $F$-path, and such that the bi-infinite path passes through $o$. 
If \(([T_n,o_n])_{n \geq 1}\) is any sequence of rooted Family Trees such that $[T_n,o_n] \rightarrow [T,o]$ as $n \to \infty$, then $\Phi_R([T_n,o_n])\rightarrow \Phi_R([T,o])$  as $n \to \infty$ (both convergences are in local sense).
\end{lemma}
\begin{proof}
  For $r>0$, denote $x_r((T,o)) = (x_{-r},x_{-r+1},\cdots,x_0,\cdots,x_{r-1},x_r)$.
  We need to show that given $r>0$, there exists $N_0$ such that for all $n>N_0$, $x_r((T_n,o_n))=x_r((T,o))$.
  Since $(T_n,o_n)\rightarrow (T,o)$, it is enough to show that, given $r>0$, there exists $R>0$ such that $x_r((T,o)_R) = x_r(T,o)$, where $(T,o)_R$ is the restriction of $T$ to the graph ball of radius $R$ centered at $o$.

 Let $(o_{-r-1},o_{-r},\cdots,o,o_1,\cdots,o_{r+1})$ be the $r+1$ restriction of the $F$-bi-infinite path of $(T,o)$ passing through $o$.
 If such a path does not exist (take for instance, a sample of the non-trivial unimodular EGWT), then we can find the largest $k<0$ such that $o_{l}=o_{k}$ for all $-r-1 \leq l \leq k$.

 Let $h_i$ be the largest graph distance from $o_i$ to the descendants of $o_i$ which are not the descendants of $o_{i-1}$ for $-r-1 \leq i \leq r+1$.
 Then, $h_i$ is finite for all $i$ since the descendants of $o_i$ that are not the descendants of \(o_{i-1}\) are finite for all \(i\) (otherwise this contradicts the fact that $T$ has at most one bi-infinite $F$-path).
 Take $R= \max\{h_i: -r-1 \leq i \leq r+1\}$.
Then, we have $U((T,o)_R) = U((T,o))$.
Thus, $x_r((T,o)_R)=x_r((T,o))$.
\end{proof}

\section{ The extended backward map of record vertex-shift} \label{subsec_reverse_map_I_F}

In this section, we define a map \(\hat{\Phi}_R\) on the space of marked Family Trees and show that it satisfies all the conditions of FB Lemma \ref{thm_forward_backward}.
The map \(\hat{\Phi}_R\) will be used later to show that the \(R\)-probability of \([\mathbb{Z},0,X]\) exists.
We call the map \(\hat{\Phi}_R\) {\bf the extended backward map}.

Let $\mathcal{T}_*'$ (resp. $\hat{\mathcal{T}_*}'$) be the space of rooted ordered Family Trees (resp. space of rooted marked ordered Family Trees whose marks take values in \(\mathbb{Z}\)) that have at most one bi-infinite $F$-path.
 
We have defined earlier in Eq. (\ref{eq:succ_line}) the succession line of a rooted ordered Family Tree $[T,o]$ that has at most one bi-infinite $F$-path, and using this we defined in Eq. (\ref{eq:reverse_map}) the backward map $\Phi_R: \mathcal{T}_*' \rightarrow \mathcal{G}_*$.

In order to use the FB lemma (Lemma \ref{thm_forward_backward}), we need this backward map to satisfy

\begin{equation}\label{eq:1_2.39}
    \Phi_R \circ \Psi_R = I \ \mathcal{P} \text{-a.s.},
\end{equation}
 where $\mathcal{P}$ is the distribution of $[\mathbb{Z},0,X]$, where $X=(X_n)_{n \in \mathbb{Z}}$ satisfies $\mathbb{E}[X_0] \geq 0$.

When $\mathbb{E}[X_0]>0$, the random walk $(S_{-n})_{n \geq 0}$ starting at $0$ drifts to $-\infty$ (where $S_{-n} = \sum_{k=-n}^{-1}-X_k$ for all $n> 0$ and \(S_0=0\)).
In this case, $\Phi_R$ does not satisfy Eq. (\ref{eq:1_2.39}) since $\Psi_R$ fails to be injective.
For instance, consider the set \(A = \{[\mathbb{Z},0,x=(x_n)_{n \in \mathbb{Z}}]: x_{-1}=3 \text{ and }s_{-n} \to -\infty \text{ as } n \to \infty \}\).
The set \(A\) has positive measure since we assumed that \(\mathbb{E}[X_0]>0\).
For each \([\mathbb{Z},0,x] \in A\), define a new element \([\mathbb{Z},0,\tilde{x}]\), where the sequence $\tilde{x}=(\tilde{x}_n)_{n \in \mathbb{Z}}$ is given by the following: $\tilde{x}_n = x_n$ for all $n \not = -1$ and $\tilde{x}_{-1} = 4$.
In comparison to the sums \((s_n)_{n \in \mathbb{Z}}\) associated to \(x\), the sums $(\tilde{s}_n)_{n \in \mathbb{Z}}$ associated to the sequence $\tilde{x}$ are shifted by $-1$ for all $n \leq -1$.
Both the networks \([\mathbb{Z},0,x]\) and \([\mathbb{Z},0,\tilde{x}]\) have the same record graphs.
So, $\Psi_R$ is not injective with positive probability, see Figure \ref{fig_example_non_injective}.
\begin{figure}[h]
    \centering 
    \includegraphics[scale=1]{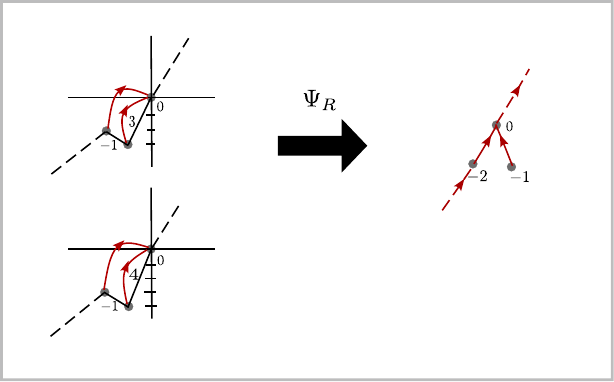}
    \caption{Example showing that when \(\mathbb{E}[X_0]>0\), two samples of \([\mathbb{Z},0,X]\) can have the same record graph.}
    \label{fig_example_non_injective}
  \end{figure}

To obtain the injectivity, we enrich the record graph \([\mathbb{Z},0,X]\) with the type function of \(X\) as a mark function on the vertices of the record graph.
This gives the forward map \(\hat{\Psi}_R:\mathcal{G}_* \to \hat{\mathcal{T}}_*'\) given by \(\hat{\Psi}_R([\mathbb{Z},0,x])=[T,0,t]\), where \([T,0] = \Psi_R([\mathbb{Z},0,x])\) is the component of \(0\) in the record graph of \([\mathbb{Z},0,x]\), \(t\) is the type function associated to \(x\), and \(x =(x_n)_{n \in \mathbb{Z}}\) is a sample of the i.i.d. sequence \(X=(X_n)_{n \in \mathbb{{Z}}}\) when \(0<\mathbb{E}[X_0]<\infty\).
We define a backward map $\hat{\Phi}_R$ (associated to \(X\)) on $\hat{\mathcal{T}_*}'$ such that it satisfies Eq. (\ref{eq:1_2.39}) with $\hat{\Psi}_R$ as the forward map.
\begin{definition}
  Let $[T,o,M]$ be a rooted marked Family Tree that has at most one bi-infinite $F$-path and \(M:V(T) \to \mathbb{Z}\) be a marked function such that \(M(u) \geq -1\) for all \(u \in V(T)\).
  Let $u((T,o,M)) = (u_n)_{n \in \mathbb{Z}}$ be its succession line with $u_0=o$.
  Define 
  \begin{equation}
    \hat{\Phi}_R([T,o,M])=[\mathbb{Z},0,x],
  \end{equation}
  where $x=(x_n)_{n \in \mathbb{Z}}$ is given by $x_n = (d_1(u_{n+1})-1) \mathbf{1}\{u_{n+1} \not \in \neswarrow\}+ (d_1(u_{n+1})+M(u_{n+1})-1) \mathbf{1}\{u_{n+1} \in \neswarrow\}$ for all $n \in \mathbb{Z}$.
\end{definition}
The map \(\hat{\Phi}_R\) is defined keeping in mind the relation between increment, offspring count and type as given in Lemma \ref{20230405112810}.

\section{Bijectivity and continuity of the extended backward map}\label{subsec_ext_backward_cont_bij}

Let \(\hat{\Psi}_R:\mathcal{G}_* \to \hat{\mathcal{T}}_*'\) be the forward map given by \(\hat{\Psi}_R([\mathbb{Z},0,x])=[T,0,t]\), where \([T,0] = \Psi_R([\mathbb{Z},0,x])\) is the component of \(0\) in the record graph of \([\mathbb{Z},0,x]\), \(t\) is the type function associated to \(x\) (see subsection \ref{subsection_typeFunction} for the definition of the type function), and \(x =(x_n)_{n \in \mathbb{Z}}\) is a sample of the i.i.d. sequence \(X=(X_n)_{n \in \mathbb{{Z}}}\) when \(0<\mathbb{E}[X_0]<\infty\).

\begin{proposition}\label{20230407161152}
  Let $X=(X_n)_{n \in \mathbb{Z}}$ be an i.i.d. sequence as in Def. \ref{hyp:increments} with \(0<\mathbb{E}[X_0]<\infty\), and $\mathcal{P}$ be the distribution of $[\mathbb{Z},0,X]$ on $\mathcal{G}_*$.
  Then, $\hat{\Phi}_R \circ \hat{\Psi}_R = I \ \mathcal{P}-a.s.$  
\end{proposition}

\begin{proof}
  Let $x=(x_n)_{n \in \mathbb{Z}}$ be a sample of $X$, $[\mathbb{Z},0,x]$ be its network, and let $[T,0,t]=\hat{\Psi}_R([\mathbb{Z},0,x])$ be the marked random rooted tree, where the mark function is give by the type function \(t\) associated to \(x\).
We know that $[T,0,t]$ has at most one bi-infinite path and $T$ is eternal.
So, we obtain its succession $u([T,0,t])=(u_n)_{n \in \mathbb{Z}}$.
Note that $u_n=n$ for all $n \in \mathbb{Z}$ by Remark \ref{remark_succession_line_is_integers}.
Let $[\mathbb{Z},0,y] = \hat{\Phi}_R([T,0,t])$, where $y=(y_n)_{n \in \mathbb{Z}}$ is the sequence given by $y_n = (d_1(u_{n+1})-1) \mathbf{1}\{u_{n+1} \not \in \neswarrow\}+ (d_1(u_{n+1})+t_x(u_{n+1})-1) \mathbf{1}\{u_{n+1} \in \neswarrow\}$ for all $n \in \mathbb{Z}$.
By Lemma \ref{20230405112810} and the fact that for any \(n \in \mathbb{Z}\), \(u_n \in \neswarrow\) if and only if \(t(u_n)>-1\), it follows that $x_n = (d_1(n+1)-1) \mathbf{1}\{t_x(n+1)=-1\}+ (d_1(n+1)+t(n+1)-1) \mathbf{1}\{t(n+1) \geq 0\}=y_n$ for all $n \in \mathbb{Z}$.
Thus, $[\mathbb{Z},0,x]=[\mathbb{Z},0,y]$.
\end{proof}

\begin{lemma} \label{20230414163734}
  Let $[T,o,m] \in \hat{\mathcal{T}_*}'$.
  If a sequence \(([T_n,o_n,m_n])_{n \geq 1}\) in \(\hat{\mathcal{T}_*}'\) converges to \([T,o,m]\)  as \(n \to \infty\), then \(\hat{\Phi}_R([T_n,o_n,m_n])\) converges to \(\hat{\Phi}_R([T,o,t])\) as \(n \to \infty\).
\end{lemma}
\begin{proof}
  Since the marks take values in \(\mathbb{Z}\), which is a discrete set, the  proof of the statement is the same as in Lemma \ref{lemma:reverse_map_continuity}, but the underlined space is \(\hat{\mathcal{T}_*}'\) instead of \(\mathcal{T}_*'\).
\end{proof}

\section{ Existence of R-probability for non-negative mean}
Let \(X=(X_n)_{n \in  \mathbb{Z}}\) be the i.i.d. sequence of increments of a skip-free to the left random walk and let \([\mathbb{Z},0,X]\) be its associated network.
Let \(\pi\) be the distribution of \(X_0+1\).
\begin{theorem}\label{theorem:r-probability_exists}
If \(\mathbb{E}[X_0]=0\), then the $R$-probability of $[\mathbb{Z},0,X]$ exists.
  Its distribution is equal to the push-forward of $EKT(\pi)$ under the map $\Phi_R$, where $\pi$ is the probability distribution of $X_0+1$.
\end{theorem}
\begin{proof}
Let $\mathcal{P}$ be the distribution of $[\mathbb{Z},0,X]$, and $\mathcal{Q}:= \mathcal{P} \circ \Psi_R^{-1}$.
We show that $\mathcal{P}$ and $\Phi_R$ satisfy all the hypotheses of FB Lemma \ref{thm_forward_backward}.
Theorem \ref{theorem:R-graph_egwt} shows that $\mathcal{Q}=EGWT(\pi)$, where  the offspring distribution $\pi \overset{\mathcal{D}}{=}X_0+1$.
Proposition \ref{prop:R-graph_reverse_identity} implies the first condition of Lemma \ref{thm_forward_backward}.
In Lemma \ref{prop:F_prob_EGWT_EKT}, we showed that the $F$-probability $\mathcal{Q}_F$ of $EGWT(\pi)$ exists and it is equal to $EKT(\pi)$, satisfying the second condition of Lemma \ref{thm_forward_backward}.
The third condition of Lemma \ref{thm_forward_backward} is shown to be satisfied by $\Phi_R$ in Lemma \ref{lemma:reverse_map_continuity}.
Thus, the $R$-probability of $[\mathbb{Z},0,X]$ exists, and its distribution is $\mathcal{Q}_F \circ \Phi_R^{-1}$.
\end{proof}

Assume \(0<\mathbb{E}[X_0]<\infty\).
Let \(\bar{\pi}\) and \(\tilde{\pi}\) be the probability distributions defined in Eq. (\ref{eq:pi_defn}).
Consider the bi-variate ECS ordered marked Eternal Kesten Tree \(MEKT(\bar{\pi},\tilde{\pi})\) whose marked bush has the distribution given in Theorem \ref{20230226170801}.

\begin{theorem}\label{20230414180529}
  If \(0<\mathbb{E}[X_{0}]< \infty\), then the $R$-probability of \([\mathbb{Z},0,X]\) exists, and its distribution is \(MEKT(\bar{\pi},\tilde{\pi})\circ \hat{\Phi}_R^{-1}\), the push-forward measure of \(MEKT(\bar{\pi},\tilde{\pi})\) under the extended backward map \(\hat{\Phi}_R\).
\end{theorem}

\begin{proof}
  Let \([\mathbf{T},0,t] = \hat{\Psi}_R([\mathbb{Z},0,X])\), where \(t\)  is the type function associated to \(X\), and \(\mathcal{P}\) be the distribution of \([\mathbb{Z},0,X]\).
  We apply FB Lemma \ref{thm_forward_backward}.
  By Proposition \ref{20230407161152}, the extended map \(\hat{\Phi}_R \circ \hat{\Psi}_R = Id \ \mathcal{P}-a.s.\).
 By Theorem \ref{20230226170801}, \([\mathbf{T},0,t]\) is the unimodularised \(MEKT(\bar{\pi},\tilde{\pi})\) (in the sense of Eq. (\ref{eq:size_biased_uniform_root})).
  By Proposition \ref{20230314185620}, \(MEKT(\bar{\pi},\tilde{\pi})\) is the \(F\)-probability of the unimodularised \(MEKT(\bar{\pi},\tilde{\pi})\).
  By Lemma \ref{20230414163734}, the extended backward map \(\hat{\Phi}_R\) is continuous.

  So, all the three conditions of the FB lemma (Lemma \ref{thm_forward_backward}) are satisfied by the extended backward map \(\hat{\Phi}_R\) and the probability measure \(\mathcal{P}\).
  Thus, $R$-probability \([\mathbb{Z},0,X]\) exists, and its distribution is equal to \(MEKT(\bar{\pi},\tilde{\pi}) \circ \hat{\Phi}_R^{-1}\).
\end{proof}

\begin{remark}\label{remark_rProbability_conditioned}
  When \(0<\mathbb{E}[X_{0}]< \infty\), \(MEKT(\bar{\pi},\tilde{\pi})\) is the distribution of the record graph of the conditioned network \([\mathbb{Z},0,X]\) conditioned on \(\{S_n \leq 0, \forall n \leq 0\}\) which is equivalent to \(\{0 \in \neswarrow\}\)(i.e., \(0\) belongs to the bi-infinite path of the record graph) by Theorem \ref{20230226170801}.
  This implies that the \(R\)-probability of \([\mathbb{Z},0,X]\) is the conditioned network \([\mathbb{Z},0,X]\) conditioned on \(\{S_n \leq 0, \forall n \leq 0\}\).
\end{remark}

\section{Construction of the R-probability for zero mean}

We have seen in the previous sections that the \(R\)-probability of \([\mathbb{Z},0,X]\) exists when \(\mathbb{E}[X_0]=0\).
We now provide two constructions for the \(R\)-probability and prove in Proposition \ref{prop:R_prob_i_i} that the constructions indeed generate the \(R\)-probability when \(\mathbb{E}[X_0]=0\).
The second construction is useful to prove several properties of the \(R\)-probability, see Section \ref{subsec_prop_r_prob_i_i}.

\noindent \textbf{ Construction 1:} We now construct a sequence of random variables $Y=(Y_n)_{n \in \mathbb{Z}}$ and show, in Proposition \ref{prop:R_prob_i_i}, that its associated network $[\mathbb{Z},0,Y]$ is the $R$-probability of $[\mathbb{Z},0,X]$ when $\mathbb{E}[X_0]=0$. 
In particular, we show that $[\mathbb{Z},0,Y] = \Phi_R([\mathbf{T},\mathbf{o}])$, where $[\mathbf{T},\mathbf{o}]$ is distributed as $EKT(\pi)$ with \(\pi\) being the distribution of \(X_0+1\).

Define the non-negatively indexed subsequence $(Y_n)_{n \geq 0}$ to be i.i.d. random variables with $Y_n \stackrel{\mathcal{D}}{=}X_n, n \geq 0$.
The construction for the negatively indexed subsequence $(Y_n)_{n<0}$ is given in the following steps.
Recall that the size-biased distribution $\hat{\pi}$ of $\pi$ is given by $\hat{\pi}(k)=k\pi(k), k \in \{0,1,2,\cdots\}$.
In particular, \(\hat{\pi}(\{1,2,3,\cdots\})=1\).

  \noindent Let $\lambda$ be the distribution of $X_0$.
  \begin{enumerate}
    \item Let $l = -1$.
    \item Generate a random variable $Z$ with distribution $\hat{\pi}$ and assign $Y_l:=(Z-1)$.
    \item Let $K$ be an integer chosen uniformly among $\{0,1,\cdots,Y_l\}$.
    \item Generate an i.i.d. sequence of random variables $(Z_n)_{n<0}$ with distribution $\lambda$.
    \item Let $\tau = \sup\{n<0: -Y_{l}-\sum_{m=n}^{-1}Z_m= -K\}$.  Assign $Y_{n+l} := Z_n$ for $\tau \leq n < 0$.
    \item Assign $l := l+\tau-1$.
    \item Repeat the procedure from step 2.
  \end{enumerate}

Let $\tau_n,l_n$ denote $\tau$ and $l$ of the $n$-th ($n \geq 1$) iteration in the above construction, with the notation that $\tau_0=0,l_0=0$.
Then, for all $n \geq 2$, the following equality of joint distribution holds because of the iterative nature of the construction: 
\begin{equation} \label{eq:regeneration}
  (Y_{\tau_n+l_n},Y_{\tau_n+l_n+1},\cdots,Y_{\tau_{n-1}+l_{n-1}-1}) \stackrel{\mathcal{D}}{=}(Y_{\tau_{n-1}+l_{n-1}},Y_{\tau_{n-1}+l_{n-1}+1},\cdots,Y_{\tau_{n-2}+l_{n-2}-1}).
\end{equation}

\begin{figure}[htbp]
  \centering 
  \includegraphics[scale=0.8]{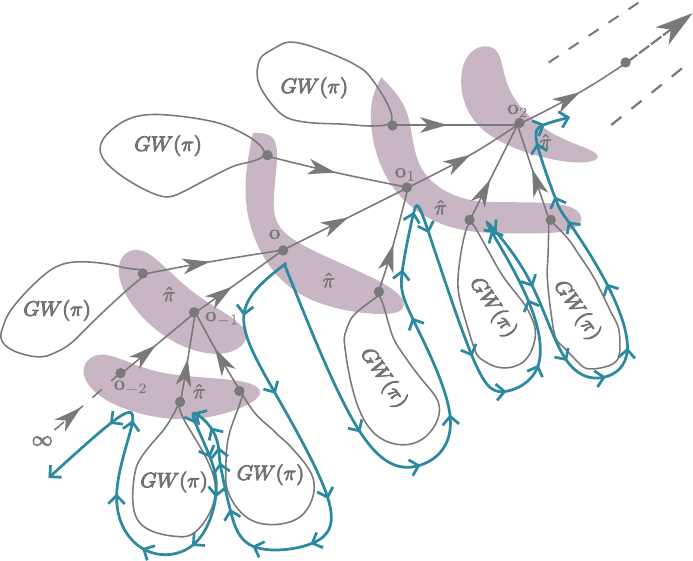}
  \caption{Distribution of $EKT(\pi)$ and its succession line starting from the root \(\mathbf{o}\) (drawn in blue). }
  \label{fi:EKT_succession_line}
\end{figure}
\begin{figure}[htbp]
  \centering 
  \includegraphics[scale=0.8]{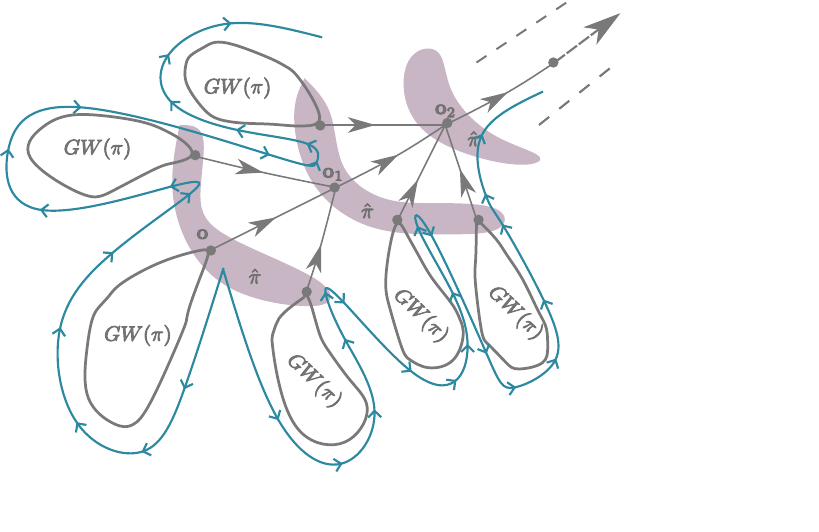}
  \caption{Distribution of $EGWT(\pi)$ and its succession line starting from the root \(\mathbf{o}\) (drawn in blue). }
  \label{fi:EGWT_succession_line}
\end{figure}

\begin{figure}[htbp]
  \centering 
  \includegraphics[scale=0.6]{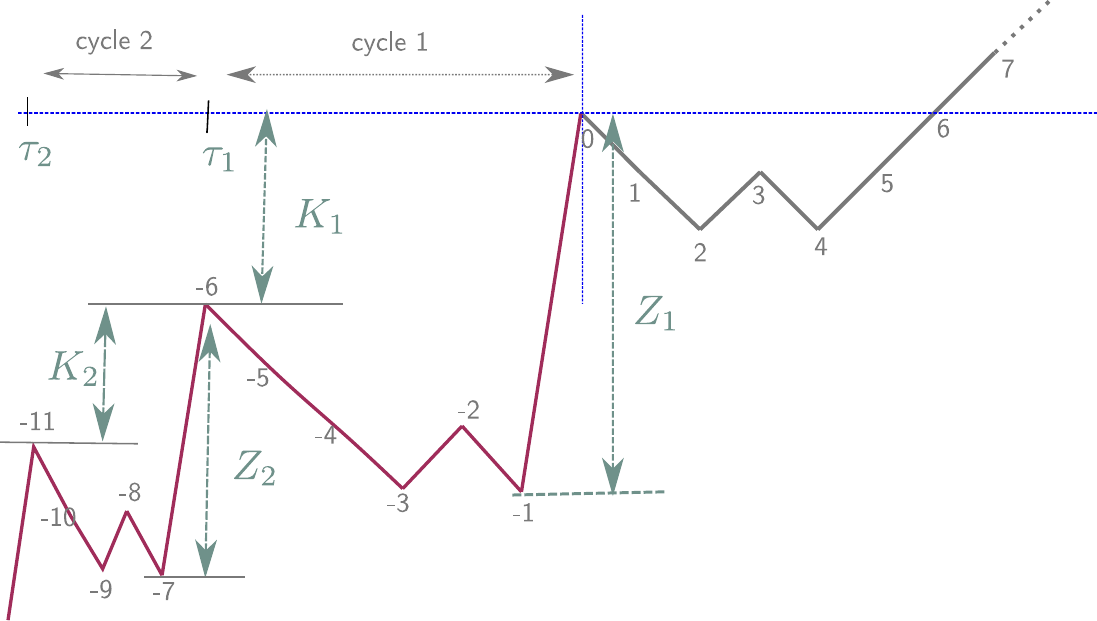}
  \caption{Sample trajectory of the sums of the $R$-probability sequence when \(\mathbb{E}[X_0]=0\) (the $R$-probability sequence $Y=(Y_n)_{n \in \mathbb{Z}}$ is the sequence of increments of this trajectory).}
  \label{fi:Rprobability}
\end{figure}

\begin{proposition}\label{prop:R_prob_i_i}
  Let $[\mathbb{Z},0,Y]$ be the network associated to the sequence of random variables $Y=(Y_n)_{n \in \mathbb{Z}}$ obtained from Construction 1.
  Then $[\mathbb{Z},0,Y]$ is the $R$-probability of $[\mathbb{Z},0,X]$, where $X=(X_n)_{n \in \mathbb{Z}}$ is as in  Def. \ref{hyp:increments} with $\mathbb{E}[X_0]=0$.
\end{proposition}
\begin{proof}
  Let $\pi$ be the distribution of $X_0+1$, $[\mathbf{T},\mathbf{o}]$ and $[\mathbf{T}',\mathbf{o}']$ be random rooted trees with distributions $EKT(\pi)$ (see Fig. \ref{fi:EKT_succession_line}) and $EGWT(\pi)$ (see Fig. \ref{fi:EGWT_succession_line}) respectively.
  We show that \([\mathbb{Z},0,Y] = \Phi_R([\mathbf{T},\mathbf{o}])\).
  Let $(u_n)_{n \in \mathbb{Z}}$ and $(w_n)_{n \in \mathbb{Z}}$ be the succession lines of $[\mathbf{T},\mathbf{o}]$ and $[\mathbf{T}',\mathbf{o}']$ respectively, with $u_0 = \mathbf{o}$ and $w_0=\mathbf{o}'$.
  It follows from the constructions of $EGWT(\pi)$ and $EKT(\pi)$ that the distribution of the sequence of offspring counts of $(u_n)_{n \geq 1}$ is the same as that of $(w_n)_{n \geq 1}$.
  Since the map $\Phi_R$ is bijective $EGWT(\pi)$-a.s., by Proposition \ref{prop:R-graph_reverse_identity} we have $(d_1(w_n)-1)_{n \geq 1} \overset{\mathcal{D}}{=}(X_{n-1})_{n \geq 1}$.
  Therefore, $(d_1(u_n)-1)_{n \geq 1} \overset{\mathcal{D}}{=}(d_1(w_n)-1)_{n \geq 1} \overset{\mathcal{D}}{=}(X_{n-1})_{n \geq 1} \overset{\mathcal{D}}{=}(Y_{n-1})_{n \geq 1}$.

  We now show that $(Y_{-n})_{n \geq 1} \overset{\mathcal{D}}{=}(d_1(u_{-n+1})-1)_{n \geq 1}$.
  Consider the random variable $\tau$ defined in the first iteration of Construction 1.
  Let $N$ be a positive random variable such that $u_{-N}$ is the unique child of $\mathbf{o}$ that belongs to the bi-infinite path of $\mathbf{T}$, i.e., $u_{-N} = \mathbf{o}_{-1}$.
  By the construction of $EKT(\pi)$, $(d_1(u_{-N-n}))_{n \geq 0} \overset{\mathcal{D}}{=}(d_1(u_{-n}))_{n \geq 0}$.
  Therefore, by Eq. (\ref{eq:regeneration}), it is enough to show that 
  \begin{equation} \label{eq_tau}
    \tau \overset{\mathcal{D}}{=}-N+1
  \end{equation}
  and the equality of the following conditioned distributions
  \begin{equation} \label{eq_cond_regeneration}
    \left(Y_{-n},Y_{-n+1},\cdots,Y_{-1}\right)|_{\tau = -n+1} \overset{\mathcal{D}}{=}\left(d_1(u_{-n+1})-1,d_1(u_{-n+2})-1,\cdots,d_1(u_0)-1\right)|_{N=n}.
  \end{equation}

  By the construction of \(EKT(\pi)\) and Construction 1, Eq. (\ref{eq_cond_regeneration}) follows.
  It remains to show Eq. (\ref{eq_tau}).
  We show it using the following inequality.
 For a realization $[T,o]$ of $EKT(\pi)$, and any $M$ such that $0<M \leq N-1$ ($N$ as above), let $u':=u'(M)$ denote the youngest child of $u_0$ such that $u \succ u'$ implies that $u \in \{u_0,u_{-1}, \cdots, u_{-M}\}$. 
  Then, we have 
  \begin{equation} \label{eqn:legall_sum}
    \sum_{i=0}^{M}(d_1(u_{-i})-1) \geq \# \{u \in D_1(u_0): u \prec u_{-N}\},
  \end{equation}
  with equality if and only if $M=N-1$.
  Indeed,
  \begin{align*}
    \sum_{i=0}^{M}(d_1(u_{-i})-1) =& d_1(u_0)-1 \\
    &+ \sum_{\{u \in D_1(u_0): u \succ u'\}} \sum_{v \in T_u} (d_1(v)-1) + \sum_{\{u \in T_{u'}: u \succeq u_{-M}\}}(d_1(u)-1),
  \end{align*}
  where $T_u$ in the above equation denotes the descendant tree of $u$ (with root as $u$).
  The above equation is obtained by rearranging the sum on the left-hand side into sums on the descendant trees of the children of $u_0$ whose descendants are contained in $\{u_0,u_{-1}, \cdots, u_{-M}\}$ and the remaining part.
  By applying part 1 of Lemma \ref{lemma:legall_sums} to the second sum and Remark \ref{remark:legall_sum} to the third sum of the above equation respectively, the right-hand side of the above equation is 
  \begin{align*}
    &\geq d_1(u_0)-1 + \sum_{\{u \in D_1(u_0):u \succ u'\}} (-1) + 0.
  \end{align*}
  Since \(u' \succeq u_{-N}\), the sum in the above equation is lower bounded by \(\#\{u \in D_1(u_0):u \succ u_{-N}\}\), so we obtain,
  \begin{align*}
    \sum_{i=0}^{M}(d_1(u_{-i})-1) &\geq d_1(u_0)-1 - \#\{u \in D_1(u_0):u \succ u_{-N}\}\\
    &= \# \{u \in D_1(u_0): u \prec u_{-N}\}.
  \end{align*}
  Equality in the above occurs if and only if $u' = u_{-N}$, i.e., $M=N-1$.

  Since the order of $\mathbf{o}_{-1}$ (\(=u_{-N}\)) is uniform among the children of $\mathbf{o}$ in $EKT(\pi)$, and $d_1(u_0)>0$ (because $d_1(u_0)\overset{\mathcal{D}}{=}\hat{\pi}$ and \(\hat{\pi}(\mathbb{N})=1\)), it follows that the distribution of $\# \{u \in D_1(u_0): u \prec u_{-N}\}$ conditioned on $\{d_1(u_0) = \ell\}$ is the uniform distribution on $\{0,1,2,\cdots,\ell-1\}$.
  Therefore, $\# \{u \in D_1(u_0): u \prec u_{-N}\} \overset{\mathcal{D}}{=}K$, where $K$ is the random variable in the first iteration of Construction 1.

  By the argument following Eq. (\ref{eqn:legall_sum}), it follows that $M=N-1$ is the smallest non-negative integer for which equality is attained in  Eq. (\ref{eqn:legall_sum}).
  Therefore, $\tau \overset{\mathcal{D}}{=}-(N-1)$.
  Thus, 
  \begin{align*}
    \left(Y_{-N},Y_{-N+1},\cdots,Y_{-1}\right) &\overset{\mathcal{D}}{=}(Z_{\tau},Z_{\tau+1},\cdots,Z_{-1},Y_{-1})\\
    &\overset{\mathcal{D}}{=} (d_1(u_{-N+1})-1,d_1(u_{-N+2})-1,\cdots,d_1(u_0)-1).
  \end{align*}

\end{proof}

\noindent \textbf{ Construction 2:} In the following steps, we provide an alternative construction of the sequence $Y=(Y_n)_{n \in \mathbb{Z}}$ (also illustrated in Figure \ref{fi:Rprobability}) from the sequence $X=(X_n)_{n \in \mathbb{Z}}$, where $X$ is as in Def. \ref{hyp:increments} with $\mathbb{E}[X_0]=0$.
This construction is used in Section \ref{subsection:r-foils} to compare the foil structures of \([\mathbb{Z},0,X]\) and \([\mathbb{Z},0,Y]\) under the record vertex-shift.
Let $\pi$ be the distribution of $X_0+1$, and $\hat{\pi}$ be the size-biased distribution of $\pi$.
\begin{enumerate}
  \item Take $(Y_n)_{n \geq 0} \stackrel{\mathcal{D}}{=}(X_n)_{n \geq 0}$.
  \item Let  $(Z_n)_{n \geq 1}$ be an i.i.d. sequence of random variables, independent of $(Y_n)_{n \geq 0}$, that take values in $\mathbb{Z}_{\geq -1}$ and whose common distribution is given by $\mathbb{P}[Z_1=b] = \hat{\pi}(b+1)$ for all $b \geq -1$. Note that $\mathbb{P}[Z_1=-1]=0$.
  \item Let $(K_n)_{n \geq 1}$ be an i.i.d. sequence of random variables such that $K_n$ is independent of $Z_m$ whenever $m \not = n$, and $\mathbb{P}[K_n = a| Z_n = b]=\frac{1}{b+1} \mathbf{1}_{0 \leq a\leq  b}$, for all $(a,b) \in \mathbb{Z}_{\geq 0} \times \mathbb{Z}_{\geq 0}$ and for all $n \geq 1$.
  \item Consider an array $(X_m^{(n)})_{m,n \in \mathbb{N}}$ of i.i.d. random variables, independent of all the above random variables, such that $X_1^{(1)} \stackrel{\mathcal{D}}{=}X_1$. 
  \item Define a sequence of random times \((\tau_n)_{n \geq 1}$ by $\tau_n:= \inf\{m \geq 1: -Z_n - \sum_{l=1}^{m}X_m^{(n)} \geq - K_n\}\) for all $n \geq 1$. Note that $\tau_n < \infty$ almost surely for all $n \geq 1$ by Chung-Fuchs theorem. Since $X_1^{(1)}$ cannot be smaller than $-1$, we have $-Z_n - \sum_{l=1}^{\tau_n}X_m^{(n)} = - K_n$ for all $n \geq 1$.
\end{enumerate}

Finally, consider the following independent sequence of random-sized tuples 
\begin{align*}
  (Y_{-1},Y_{-2},\cdots,Y_{-(\tau_1+1)}) &\stackrel{\mathcal{D}}{=}(Z_1,X_1^{(1)},X_2^{(1)},\cdots, X_{\tau_1}^{(1)}),\\
  (Y_{-(\tau_1+2)},\cdots,Y_{-(\tau_1+\tau_2+2)}) & \stackrel{\mathcal{D}}{=} (Z_2,X_1^{(2)},\cdots,X_{\tau_2}^{(2)}),\cdots,\\
 (Y_{-(\tau_1+\cdots+\tau_{n-1}+n)}+\cdots+ Y_{-(\tau_1+\cdots+\tau_{n-1}+\tau_n+n)}) &\stackrel{\mathcal{D}}{=} (Z_n,X_1^{n},\cdots,X_{\tau_n}^{(n)}), \cdots.
\end{align*}

We call the sequence $Y=(Y_n)_{n \in \mathbb{Z}}$ the { \bf $R$-probability sequence} associated to $X=(X_n)_{n \in \mathbb{Z}}$. 

\begin{remark}
  Note that $K_1$ is a proper random variable since $\mathbb{P}[K_1= a] = \sum_{l = a+1}^{\infty}\pi(l)$ for all non-negative integers $a \geq 0$ and $\mathbb{P}[K_n < \infty] = \sum_{a=0}^{\infty} \pi((a,\infty))= m(\pi)=1$.
Observe that $\mathbb{P}[K_1>0]>0$ since $\pi(0)<1$.
\end{remark}

From the construction, it is clear that $(K_n,\tau_n)\overset{\mathcal{D}}{=}(K,\tau)$, where $K, \tau$ are the random variables defined in Construction 1.

\section{Properties of the R-probability sequence for zero mean}\label{subsec_prop_r_prob_i_i}
In this section, we discuss the properties of the R-probability sequence $Y=(Y_n)_{n \in \mathbb{Z}}$ of $X=(X_n)_{n \in \mathbb{Z}}$, defined in the last section, when $\mathbb{E}[X_0]=0$.
Recall that $(Y_n)_{n \geq 0}$ is an i.i.d. sequence whose common distribution is given by $Y_0 \stackrel{\mathcal{D}}{=}X_0$.
So, we focus on the properties of $(Y_n)_{n<0}$.

For two integers $m<n$, let $w(m,n) := \sum_{k = m}^{n-1}Y_k$ and $W_n = -w(-n,0) =-\sum_{k = -n}^{-1}Y_k$, for $n \geq 1$.
We continue to use the notations of the last section in the following lemma and proposition.
From the definitions of \(\tau_n\) and \(K_n\), it follows that \(W_{\tau_1+\tau_2+\cdots+\tau_n+n}=-(K_1+K_2+\cdots+K_n)\) for all \(n\geq 1\).

\begin{lemma} \label{lemma:inequality_negative_sequence}
 Almost surely $W_m<-(K_1+\cdots+K_n+K_{n+1})$ for all positive integers $m,n$ such that $n \geq 1$ and $\tau_1+\tau_2+\cdots+\tau_{n}+n< m < \tau_1+\tau_2+\cdots+\tau_{n+1}+n+1$, where \((\tau_i)_{i \geq 1}\) are as in Construction 2.
\end{lemma}

\begin{proof}
The proof is by induction on $n$. 
For $n=1$ and $\tau_1+1<m<\tau_1+\tau_2+2$, we have 
\[W_m = W_{\tau_1+1}-w(-m,-(\tau_1+1))=-K_1-w(-m, -(\tau_1+1)).\]
But, by the definition of $\tau_2$,
\[-w(-m,-(\tau_1+1)) = -Y_{-(\tau_1+2)}-w(-m,-(\tau_1+2)) \stackrel{\mathcal{D}}{=}-Z_2- \sum_{l=1}^{m-(\tau_1+2)}X_{l}^{(2)}<-K_2.\]

Therefore, $W_m<-K_1-K_2$ for $n=1$. 

Assume that the statement is true for $n-1$. For $\tau_1+\cdots+\tau_n +n < m < \tau_1+\tau_2+\cdots+\tau_{n+1}+n+1$, we have

\begin{align*}
W_m &= W_{\tau_1+\cdots+\tau_n + n}- w(-m,-(\tau_1+\cdots+\tau_n + n))\\
& \stackrel{\mathcal{D}}{=}W_{\tau_1+\cdots+\tau_n + n} - Z_{n+1}- \sum_{l=1}^{m-(\tau_1+\cdots+\tau_n+n+1)}X_l^{(n+1)}\\
&<-(K_1+\cdots+K_n+K_{n+1}).
\end{align*}

Last step follows by induction, and by the definition of $\tau_{n+1}$.
\end{proof}

\begin{proposition}\label{prop:drift_negative}
  Almost surely $W_n \xrightarrow[]{} - \infty$ as $n \rightarrow \infty$.
\end{proposition}
\begin{proof}
  From Lemma \ref{lemma:inequality_negative_sequence}, and the fact that $K_n \geq 0, \forall n \geq 1$, it follows that $W_m<-(K_1+K_2+\cdots+K_n)$, for all $n \geq 1$ and $m> \tau_1+\cdots+\tau_n+n$.
  Therefore, the event \(\{W_m \xrightarrow[m \to \infty]{} - \infty\}\) contains the event $\{-(K_1+\cdots+K_n) \xrightarrow[n \to \infty]{} - \infty\}$.
So, it is enough to show that $\mathbb{P}[-(K_1+K_2+\cdots+K_n) \xrightarrow[n \to \infty]{} - \infty]=1$.
Since $K_n$ is non-negative for all $n \geq 1$, it is enough to show that 
  \(\mathbb{P}[K_n>0 \textit{  infinitely often }]=1.\)
  The events $\{K_n>0\}$ satisfy $\sum_{n=1}^{ \infty}\mathbb{P}[K_n>0] = \infty$ since, $\mathbb{P}[K_1>0]>0$ and $(K_n)_{n \geq 1}$ are i.i.d. random variables.
  Therefore, $\mathbb{P}[K_n>0 \ i.o.]=1$ by the Borel-Cantelli Lemma.
\end{proof}

The following corollary is used to describe the record foils of the \(R\)-probability (see Proposition \ref{proposition:minimal_element}).

\begin{corollary} \label{corollary:drift_negative}
  For any integer $j$, almost surely
  \[-w(j-n,j) = -\sum_{k=j-n}^{j-1}Y_k \xrightarrow[n \to \infty]{} - \infty.\]
\end{corollary}
\begin{proof}
  For $n>\max\{0,j\}$,
  \[ -w(j-n,j) = \begin{cases}
    W_n-w(0,j) \text{ if } j>0\\
    W_n + w(j,0) \text{ otherwise.}
  \end{cases}\]
  Therefore, by Proposition \ref{prop:drift_negative}, a.s. $-w(j-n,j) \rightarrow -\infty$ as $n \rightarrow \infty$.
\end{proof}

\section{Construction of R-probability for positive mean}

In the following steps, we construct a sequence \(\tilde{Y}=(\tilde{Y}_n)_{n \in \mathbb{Z}}\) and later show that \([\mathbb{Z},0,\tilde{Y}]\) is the \(R\)-probability of \([\mathbb{Z},0,X]\) when \(0<\mathbb{E}[X_0]<\infty\).
Recall the definitions of \(\tilde{\pi},\bar{\pi}\) and \(c\) from Eq. (\ref{eq:pi_defn}) in Section \ref{subsec_positive_mean}.

\begin{enumerate}
    \item The random variables \((\tilde{Y}_n)_{n\geq 0}\) are i.i.d. with the common distribution the same as that of \(X_0\). That is, \(\tilde{Y}_n \overset{\mathcal{D}}{=}X_0\) for all \(n \geq 0\).
   \item Let \(((K_n,Z_n))_{n \geq 1}\) be an i.i.d. sequence of random variables taking values in \(\mathbb{Z}_{\geq 0} \times \mathbb{Z}_{\geq 0}\) whose common joint-distribution is given by \(\mathbb{P}[K_1=m, Z_1 = l] = \mathbb{P}[X_0=m+l]c^l \) for all \(m \geq 0,\, l \geq 0\).
    \item Define an array \((X_m^{(n)})_{m,n \in \mathbb{N}}\) of i.i.d. sequence of random variables whose common distribution is given by \(\mathbb{P}[X_1^{(1)}=k] = \tilde{\pi}(k+1)\) for all \(k \geq 0\).
    \item Define the sequence of random times \((\tau_n)_{n \geq 1}\) by \(\tau_n = \inf\{m \geq 1: -Z_n -K_n-\sum_{l=1}^{m}X_m^{(n)} = -K_n\}=\inf\{m \geq 1: -Z_n -\sum_{l=1}^{m}X_m^{(n)} = 0\}\) if the infimum is attained and \(\tau_n = \infty\) otherwise, for all \(n \geq 1\). Note that \(\tau_1<\infty\) a.s. since the random walk \((\tilde{S}_m)_{m \geq 0}\) starting at \(0\) drifts to \(\infty\) by Lemma \ref{20230116135048}, where \(\tilde{S}_m = - \sum_{l=1}^{m}X_m^{(1)}\) for all \(m \geq 1\).
    \item Define \((\tilde{Y}_{-n})_{n \geq 0}\) by
        \begin{align*}
            (\tilde{Y}_{-1},\tilde{Y}_{-2},\cdots, \tilde{Y}_{-(\tau_1+1)}) &= (Z_1+K_1,X_1^{(1)},X_2^{(1)},\cdots,X_{\tau_1}^{(1)}),\\
            (\tilde{Y}_{-(\tau_1+2)},\tilde{Y}_{-(\tau_1+3)},\cdots, \tilde{Y}_{-(\tau_1+\tau_2+2)}) &= (Z_2+K_2,X_1^{(2)},X_2^{(2)},\cdots,X_{\tau_2}^{(2)}), \cdots,\\
            \cdots, (\tilde{Y}_{-(\tau_1+\cdots+\tau_{n-1}+n)},\cdots, \tilde{Y}_{-(\tau_1+\cdots+\tau_n+n)}) &= (Z_n+K_n,X_1^{(n)},\cdots,X_{\tau_n}^{(n)}), \cdots.
        \end{align*}
\end{enumerate}

Note that the \(R\)-probability of \([\mathbb{Z},0,X]\) exists by Theorem \ref{20230414180529}.
We now prove that the sequence constructed above is the \(R\)-probability of \([\mathbb{Z},0,X]\) when \(0<\mathbb{E}[X_0]<\infty\).
The proof of this statement together with Remark \ref{remark_rProbability_conditioned} implies that \([\mathbb{Z},0,\tilde{Y}]\) has the same distribution as that of the conditioned network \([\mathbb{Z},0,X]\) conditioned on the event that \(\{S_n \leq 0 \  \forall n \leq 0\}\), where \((S_n)_{n \in \mathbb{Z}}\) denotes the skip-free random walk associated to the increment sequence \(X=(X_n)_{n \in \mathbb{Z}}\).

For any Family Tree \(T\) that has a unique bi-infinite path, denote its bi-infinite path by \(\neswarrow(T)\).

\begin{proposition}\label{20230421185746}
    Let \(X=(X_n)_{n \in \mathbb{Z}}\) be the i.i.d. sequence of increments of a skip-free to the left random walk that saitsfies \(0<\mathbb{E}[X_0]<\infty\) and \([\mathbb{Z},0,X]\) be its associated network.
    Let \([\mathbb{Z},0,Y]\) be the \(R\)-probability of \([\mathbb{Z},0,X]\), where \(Y=(Y_n)_{n \in \mathbb{Z}}\).
    Then, \((Y_n)_{n \geq 0}\) is an i.i.d. sequence of random variables independent of \((Y_n)_{n<0}\) and their common distribution is the same as that of \(X_0\).
\end{proposition}

\begin{proof}
    Let \([\mathbf{T},0,t]\) be the record graph of \([\mathbb{Z},0,X]\) with the type function \(t\) as the mark function and \([\mathbf{T}',\mathbf{o}',t']\) be the random Family Tree obtained from \([\mathbf{T},0,t]\) by conditioning on the event \(0 \in \neswarrow\) (denoted as \([\mathbf{T},0,t]|0 \in \neswarrow\)).
    Note that \([\mathbf{T}',\mathbf{o}',t']\) is the \(F\)-probability of \([\mathbf{T},0,t]\) (see Proposition \ref{20230314185620} and Theorem \ref{20230226170801}).
    By Theorem \ref{r_graph_positive_drift_20230203174636} and Theorem \ref{20230226170801}, the distribution of \([\mathbf{T}',\mathbf{o}',t']\) is the \(MEKT(\bar{\pi},\tilde{\pi})\) where the marked bush of the root is distributed as in Eq. (\ref{eqn_joint_dist_type_offsprings}) and \(\bar{\pi},\, \tilde{\pi}\) are as in Eq. (\ref{eq:pi_defn}).

    Let \(U((\mathbf{T}',\mathbf{o}')) = (u_n)_{n \in \mathbb{Z}}\) be the succession line of \((\mathbf{T}',\mathbf{o}')\) with \(u_0 = \mathbf{o}'\), and \(U((\mathbf{T},0))=(w_n)_{n \in \mathbb{Z}}\) be the succession line of \((\mathbf{T},0)\) with \(w_0=0\).

    Since \([\mathbf{T}',\mathbf{o}',t'] \overset{\mathcal{D}}{=}[\mathbf{T},0,t]|0 \in \neswarrow\), we have the equality in distribution of the following two sequence of random variables:
    \begin{align} \label{eq_positive_sequence}
      ((d_1(u_n)-1) \mathbf{1}\{u_n \not \in \neswarrow\}&+(d_1(u_n)+t(u_n)-1) \mathbf{1}\{u_n \in \neswarrow\})_{n \geq 1} \\ \nonumber
      &\overset{\mathcal{D}}{=}\\ \nonumber
      ((d_1(w_n)-1) \mathbf{1}\{w_n \not \in \neswarrow\}&+(d_1(w_n)+t(w_n)-1) \mathbf{1}\{w_n \in \neswarrow\})_{n \geq 1}| 0 \in \neswarrow
    \end{align}
    \[\]
    Since \([\mathbb{Z},0,Y]\) is the push-forward of the \(MEKT(\bar{\pi},\tilde{\pi})\) (with \(t\) as the mark function) under the extended backward map \(\hat{\Phi}_R\), we have \((d_1(u_n)-1) \mathbf{1}\{u_n \not \in \neswarrow\}+(d_1(u_n)+t(u_n)-1) \mathbf{1}\{u_n \in \neswarrow\} = Y_{n-1}, \forall n \in \mathbb{{Z}}\).
    From the construction of the \(MEKT(\bar{\pi,\tilde{\pi}})\) with \(t\) as the mark function, the subtrees generated by \(\{u_n:n \geq 1\}\) and \(\{u_n: n \leq 0\}\) are independent of each other.
    This implies that the sequence \((Y_n)_{n \geq 0}\) and \((Y_n)_{n<0}\) are independent of each other.

    Recall that the extended backward map \(\hat{\Phi}_R\) is bijective, and we have \((d_1(w_n)-1) \mathbf{1}\{w_n \not \in \neswarrow\}+(d_1(w_n)+t(w_n)-1) \mathbf{1}\{w_n \in \neswarrow\}= X_{n-1}\) for all \(n \geq 1\).
    Observe that the event \(\{0 \in \neswarrow\} \) is equal to \(\{S_n \leq 0 \ \forall n<0\} \).

    Since \((X_n)_{n \geq 0}\) is independent of the event \(\{S_n \leq 0 \,\forall n<0\} \), the conditioned process \((X_n)_{n \geq 0}\) conditioned on this event has the same distribution as that of the unconditioned process \((X_n)_{n \geq  0}\).
    Therefore, by Eq. (\ref{eq_positive_sequence}), we have
    \[(Y_n)_{n \geq 0} \overset{\mathcal{D}}{=} (X_n)_{n \geq 0}|0 \in \neswarrow =(X_n)_{n \geq 0}|\{S_n \leq 0, \forall n<0\} \overset{\mathcal{D}}{=} (X_n)_{n \geq 0}.\]
    Thus,
    \[(Y_n)_{n \geq 0} \overset{\mathcal{D}}{=} (X_n)_{n \geq 0}.\]
\end{proof}

Consider the sequence \(\tilde{Y}=(\tilde{Y}_n)_{n \in \mathbb{Z}}\) generated by the above construction of this subsection.
Proposition \ref{20230421185746} proves that the non-negatively indexed part \((\tilde{Y}_n)_{n \geq 0}\) has the same distribution as that of the non-negatively indexed part \((Y)_{n \geq 0}\), where \(Y=(Y_n)_{n \in \mathbb{Z}}\) is the \(R\)-probability sequence associated to \(X\).
In the following proposition, we show that the negatively indexed part \((\tilde{Y}_n)_{n < 0}\) has the same distribution of the negatively indexed part \((Y_n)_{n<0}\).
These two propositions together show that the above construction generates the \(R\)-probability sequence of \(X\).

\begin{proposition}
    Let \([\mathbb{Z},0,Y]\) be the \(R\)-probability of \([\mathbb{Z},0,X]\), where \(Y=(Y_n)_{n \in \mathbb{Z}}\) where \(X\) is as in Proposition \ref{20230421185746}.
    Then, \((\tilde{Y}_n)_{n<0} \overset{\mathcal{D}}{=}(Y_n)_{n<0} \).
\end{proposition}
 \begin{proof}
    Let \([\mathbf{T},0,t]\) be the record graph of \([\mathbb{Z},0,X]\) with the type function \(t\) associated to \(X\) as the mark function, and \([\mathbf{T}',\mathbf{o}',t']\) be the \(F\)-probability of \([\mathbf{T},0,t]\).
    Note that the distribution of \([\mathbf{T}',\mathbf{o}',t']\) is \(MEKT(\bar{\pi},\tilde{\pi})\) whose bush of the root is distributed as in Eq. (\ref{eqn_joint_dist_type_offsprings}) (see Proposition \ref{20230314185620} and Theorem \ref{20230226170801}).
    Let \(U((\mathbf{T}',\mathbf{o}',t'))=(u_n)_{n \in  \mathbb{Z}}\) be the succession line of \((\mathbf{T}',\mathbf{o}',t')\) with \(u_0=\mathbf{o}'\).
    
    Let \(N>0\) be a positive integer such that \(u_{-N}\) is the smallest child of \(u_0\).
    By the construction of \([\mathbf{T}',\mathbf{o}',t']\), \(u_{-N}\) lies on  the bi-infinite \(F\)-path of \(\mathbf{T}'\).
    We know that \([\mathbf{T}',\mathbf{o}',t']\) is obtained as the joining of i.i.d. Family Trees \([\mathbf{T}_i,\mathbf{o}_i,t_i]_{i \in \mathbb{Z}}\), where \(t_i\) is the restriction of \(t'\) to \(V(\mathbf{T}_i)\) (see Subsection \ref{subsec_bi_variate_ekt}).
    So, we have \(\mathbf{o}_{-1}=u_{-N}\), \(\mathbf{o}_0=u_0\).
    Therefore, \(u_0, u_{-N}\) lie on the bi-infinite path and none of \(\{u_{-1}, \cdots, u_{-N+1}\}\) lie on the bi-infinite path.

    Since \([\mathbb{Z},0,Y] = \hat{\Phi}_R([\mathbf{T}',\mathbf{o}',t'])\) (see Theorem \ref{20230414180529}), we have 
    \[(Y_n)_{n<0} = ((d_1(u_{n+1})-1) \mathbf{1}\{u_{n+1} \not \in \neswarrow\}+(d_1(u_{n+1})+t(u_{n+1})-1) \mathbf{1}\{u_{n+1} \in \neswarrow\})_{n<0}.\]
   In particular, we have \((Y_{-1},Y_{-2},\cdots,Y_{-N}) \overset{\mathcal{D}}{=}(d_1(u_0)+t(u_0)-1,d_1(u_{-1})-1, \cdots, d_1(u_{-N+1})-1)\).

    Consider \(\tau_1\) of the construction of this subsection.
    It is defined as \(\tau_1 = \inf\{m \geq 1: -Z_1- \sum_{l=1}^m X_m^{(1)}=0\}\), where \(Z_1,X_m^{(1)}\) are as defined in the construction.
    In the construction, we generate the sequence \((\tilde{Y}_n)_{n<0}\) iteratively.
    The first iteration gives the sequence \((\tilde{Y}_{-1},\tilde{Y}_{-2},\tilde{Y}_{-3},\cdots, \tilde{Y}_{-(\tau_1+1)}) = (Z_1+K_1,X_1^{(1)}, X_2^{(1)},\cdots, X_{\tau_1}^{(1)})\).

    Because of the i.i.d. joining involved in the construction of \([\mathbf{T}',\mathbf{o}',t']\) and the iterative nature of the algorithm, it is sufficient to show that the joint distribution of the random variables associated to the subtree \([\mathbf{T}_0,\mathbf{o}_0,t_0]\) (i.e.,\((\tilde{Y}_{-1},\tilde{Y}_{-2},\tilde{Y}_{-3},\cdots, \tilde{Y}_{-(\tau_1+1)})\)) is the same as that of those that are generated in the first iteration of the construction.
    That is, it is sufficient to show that
    \begin{equation}\label{eq_rand_stop_20230423153855}
        N \overset{\mathcal{D}}{=} \tau_1+1
    \end{equation}
    and
    \begin{align} \label{eqn_iter_20230423153855}
        (Z_1+K_1,X_1^{(1)}, X_2^{(1)},\cdots, X_{n-1}^{(1)})|_{\tau_1=n-1} &\overset{\mathcal{D}}{=} \nonumber \\
        (d_1(u_0)+t(u_0)-1,d_1(u_{-1})-1,& \cdots, d_1(u_{-n+1})-1)|_{N=n}. 
    \end{align}

  The joint-distribution of \((K_1,Z_1)\) is chosen such that \(\mathbb{P}[K_1=m, Z_1=l] = \mathbb{P}[X_0=m+l]c^l = \mathbb{P}[t_0(\mathbf{o}_0)=m,d_1(\mathbf{o}_0)=l]\) (see Eq. (\ref{eqn_joint_dist_type_offsprings})).
  Since \(d_1(u_0)=d_1(\mathbf{o}_0)-1\),  we have \(Z_1+K_1 \overset{\mathcal{D}}{=}d_1(u_0)-1+t(u_0)\).
  On the event that \(N=n\), \(u_{-k}\) does not belong to the bi-infinite path of \(\mathbf{T}'\) and they have the same offspring distribution \(\tilde{\pi}\), for all \(1\leq k \leq n-1\).
  Since, \((X_i^{(1)}+1)_{i \geq 1}\) are i.i.d. random variables with the common distribution \(\tilde{\pi}\), Eq. (\ref{eqn_iter_20230423153855}) follows.

  We now show Eq. (\ref{eq_rand_stop_20230423153855}).
  Consider the finite ordered Family Tree \([\mathbf{T}_0,\mathbf{o}_0]\).
  The vertices of \(\mathbf{T}_0\) are \(u_{-N+1} \prec \cdots \prec u_{-1} \prec u_0\) listed according to the RLS order.
  For any \(0 \leq k \leq N-1\), by Remark \ref{remark:legall_sum} and Lemma \ref{lemma:legall_sums}, we have
  \[\sum_{i=-k}^{0}(d_1(u_i)-1) \geq 0\]
    and 
    \[\sum_{i=-N+1}^{0}(d_1(u_i)-1)=-1.\]
    Rewriting the above two equations, we obtain
    \[\sum_{i=-k}^{-1}(d_1(u_i)-1) +d_1(u_0) \geq 1\]
    and
    \[\sum_{i=-N+1}^{-1}(d_1(u_i)-1) +d_1(u_0)=0.\]

    Note that \(Z_1 \overset{\mathcal{D}}{=}d_1(u_0)\) (in \(\mathbf{T}_0\)) and \(d_1(u_i)-1 \overset{\mathcal{D}}{=}X_1^{(1)}\) for all \(-N+1 \leq i \leq -1\).
    Thus, we obtain 
    \[\sum_{i=-k}^{-1}X_{-i}^{(1)}+Z_1 \geq 1\]
    and
    \[\sum_{i=-N+1}^{-1}X_{-i}^{(1)}+Z_1 = 0.\]
    In particular, \(\tau_1 \overset{\mathcal{D}}{=}N-1\).
    This completes the proof.

 \end{proof}

\section{R-foils of R-probability for zero mean} \label{subsection:r-foils}

Let $Y=(Y_n)_{n \in \mathbb{Z}}$ be the $R$-probability sequence of $X=(X_n)_{n \in \mathbb{Z}}$, where $X$ is as in Def. \ref{hyp:increments} with $\mathbb{E}[X_0]=0$, and \([\mathbb{Z},0,Y], [\mathbb{Z},0,X]\) be their associated networks respectively.
We study the record vertex-shift on $[\mathbb{Z},0,Y]$ and compare it with that of $[\mathbb{Z},0,X]$.

\begin{remark}
  Observe that the record vertex-shift $R$ is continuous $[\mathbb{Z},0,Y]$-a.s..
  This implies that the map $\theta_R$ preserves the distribution of $[\mathbb{Z},0,Y]$ (see the proof of Corollary \ref{cor_f_prob_sing}).
\end{remark}

In contrast to Remark \ref{remark:foils}, we prove in this section that, almost surely, every $R$-foil of $[\mathbb{Z},0,Y]$ has a minimal element.
This fact can be guessed by comparing the $R$-graph of $[\mathbb{Z},0,X]$ with the $R$-graph of $[\mathbb{Z},0,Y]$.
Let us call the rooted tree $[\mathbf{T}',\mathbf{o}]$ obtained by trimming off the branches that lie above the bi-infinite \(F\)-path shown in Figure \ref{fi:EKT_succession_line} as the {\bf positive half of the Eternal Kesten Tree} (with the root same as before).
As shown in Figure \ref{fi:EKT_succession_line}, it is the subtree induced on the vertices that belong to the succession line of the Eternal Kesten Tree passing through the root.
Observe that the $R$-graph of $[\mathbb{Z},0,Y]$ is the positive half of the Eternal Kesten Tree.
From the construction of \(Y\), it is clear that a.s. \(R^n_Y(0)>R^{(n-1)}_Y(0)\) for all \(n \geq 1\).
This and the fact that \(-w(n,0) \to - \infty \) (see Corollary \ref{corollary:drift_negative}) as \(n \to \infty\) imply that the \(R\)-graph of \([\mathbb{Z},0,Y]\) is a.s. connected.
Every generation (i.e. \(F\)-foils) of $\mathbf{T}'$ has a unique vertex that belongs to the bi-infinite path.
These unique vertices are the smallest elements among their respective foils.

\begin{proposition}\label{proposition:minimal_element}
  Consider the $R$-foliation of $[\mathbb{Z},0,Y]$.
  For any integer $j \in \mathbb{Z}$, the $R$-foil of $j$  has a minimal element.
\end{proposition}
\begin{proof}
  By Corollary \ref{corollary:drift_negative}, a.s. the partial sums $-w(j-n,j) \rightarrow - \infty$ as $n  \rightarrow \infty$.
  So, there exists a (random) unique integer $k \leq  j$ such that $-w(n,j)<0$ for all $n<k$ and $w(k,j)\leq 0$.
  Indeed, \(k = \sup\{n \leq j:-w(l,j)<0\ \forall l<n\}\).
  This implies that $-w(n,k)=-w(n,j)+w(k,j)<0$ for all $n<k$.
  Therefore, by Lemma \ref{lemma:descendants}, every integer $n<k$ is a descendant of $k$.
  Since the $R$-graph of $[\mathbb{Z},0,Y]$ is connected  a.s., there exist smallest integers $a,b \geq 0$ such that $R^a(k)=R^b(j)$.
  
\noindent If $b=0$, then every integer smaller than $j$ is a descendant of $j$ by Lemma \ref{lemma:descendants} and from the discussion in the above paragraph.
  Hence, $j$ is the smallest element of its R-foil.

\noindent So, let us assume that $b \geq 1$.
In this case, we show that only finitely many descendants of \(k\) belong to the \(R\)-foil of \(j\).
This implies that the \(R\)-foil of \(j\) has a smallest element because all the integers smaller than \(k\) are descendants of \(k\) and there are only finitely many integers between \(k\) and \(j\).
If no descendant of $k$ belongs to the R-foil of $j$, then nothing to prove.
Suppose that there exists an integer $l<k$ that belongs to the R-foil of $j$.
Since $l$ is a descendant of $k$, we have $R^N(l)=R^N(j)$, for some smallest $N>a$.
Further, $N=b$, since $b>0$ is the smallest for which $R^b(j)=R^a(k)$.
Therefore, $l$ is a $(b-a)$-th descendant of $k$. 
Thus, every descendant of \(k\) that belongs to the \(R\)-foil of \(j\) is a \((b-a)\)-th descendant of \(k\).
But, $k$ has finitely many $(b-a)$-th descendants.
Hence, finitely many descendants of \(k\) belong to the \(R\)-foil of \(j\), and this completes the proof.
\end{proof}


\chapter{Vertex-shifts on processes with stationary increments}\label{chapter_stationary_seq}

In this chapter, we focus on record vertex-shift on stationary sequences and representation of unimodular EFTs as record graphs of a stationary sequence.
We use the following stationary framework taken from \cite{baccelliElementsQueueingTheory2003,bremaudProbabilityTheoryStochastic2020}.
The main results of this chapter are Theorem \ref{thm_phase_transition_stationary} and Theorem \ref{thm_representation}.
Theorem \ref{thm_phase_transition_stationary} shows that the component of \(0\) in the record graph of the network associated to a stationary sequence exhibits a phase transition at mean \(0\) and Theorem \ref{thm_representation} provides sufficient conditions under which a unimodular ordered EFT can be represented as the component of \(0\) in the record graph of the network associated to a stationary sequence.

{\bf (In French)} Dans ce chapitre, nous nous intéressons au décalage de sommet des records sur les suites stationnaires et à la représentation des EFT unimodulaires comme des graphes des records de suites stationnaires.
Nous utilisons le cadre stationnaire suivant, tiré de \cite{baccelliElementsQueueingTheory2003,bremaudProbabilityTheoryStochastic2020}.
Les principaux résultats de ce chapitre sont le théorème \ref{thm_phase_transition_stationary} et le théorème \ref{thm_representation}.
Le théorème \ref{thm_phase_transition_stationary} montre que la composante de \(0\) dans le graphe des records du réseau associé à une suite stationnaire présente une transition de phase à la moyenne \(0\) et le théorème \ref{thm_representation} fournit des conditions suffisantes sous lesquelles un EFT ordonné unimodulaire peut être représenté comme la composante de \(0\) dans le graphe des records du réseau associé à une suite stationnaire.

\section{Record vertex-shift on processes with stationary and ergodic increments }
In this section, we first show that if \(X=(X_n)_{n \in \mathbb{Z}}\) is a stationary and ergodic sequence, then the event \(\{\mathbb{Z}^R_X(0) \text{ is of class } \star\}\) is trivial (i.e., it occurs with probability either \(0\) or \(1\)), where \(\star \in \{\mathcal{I}/\mathcal{I}, \mathcal{I}/\mathcal{F}, \mathcal{F}/\mathcal{F}\}\) and \(\mathbb{Z}^R_X(0)\) is the component of \(0\) in the record graph of \([\mathbb{Z},0,X]\).
This is shown in Proposition \ref{prop_ergodic_trivial} and in Corollary \ref{cor_comp_f_f}.
In Theorem \ref{thm_phase_transition_stationary}, we show that \(\mathbb{Z}^R_X(0)\) exhibit a phase transition at \(0\) when the mean of \(X_0\) is varied between \(-\infty\) and \(+\infty\).
\begin{definition}\label{def_stationary_seq}
    Let \(X=(X_n)_{n \in \mathbb{Z}}\) be a stationary sequence, where \(X_n \in \mathbb{Z}\) for all \(n \in \mathbb{Z}\) and \(\mathbb{E}[|X_0|]<\infty\).
\end{definition}
Define the sequence of sums \((S_n)_{n \in \mathbb{Z}}\) by \(S_0=0\), \(S_n = \sum_{k=0}^{n-1}X_k\), for all \(n>0\), and \(S_n =- \sum_{k=n-1}^{-1}X_k\), for all \(n<0\).
For any sequence \(x=(x_n)_{n \in \mathbb{Z}}\) of integers, let \(\mathbb{Z}^R_x\) denotes the record graph of the network \((\mathbb{Z},x)\), and for any \(i \in \mathbb{Z}\), let \(\mathbb{Z}^R_x(i)\) denotes the component of \(i\) in \(\mathbb{Z}^R_x\).
Let \(T\) be the shift map that maps every sequence \(x=(x_n)_{n \in \mathbb{Z}}\) to \(Tx =(x_{n+1})_{n\in \mathbb{Z}}\).
The shift map naturally maps any network of the form \([\mathbb{Z},0,x]\) to \(T([\mathbb{Z},0,x]) = [\mathbb{Z},0,Tx]\).
For any integer \(i \in \mathbb{Z}\), consider the following events:
\begin{align*}
    E_i&:= \{[\mathbb{Z},0,x]:\mathbb{Z}^R_x(i) \text{ is of class } \mathcal{I}/\mathcal{I}\}\\
    F_i &:= \{[\mathbb{Z},0,x]:\mathbb{Z}^R_x(i) \text{ is of class } \mathcal{I}/\mathcal{F}\},
\end{align*}
and let \(T^{-1}E_{i} := \{[\mathbb{Z},0,x]:[\mathbb{Z},0,Tx]\in E_i\}\), \(T^{-1}F_{i} := \{[\mathbb{Z},0,x]:[\mathbb{Z},0,Tx]\in F_i\}\).
Then, for any \(i\in \mathbb{Z}\), we have
\begin{align}
    T^{-1}E_i  &= \{[\mathbb{Z},0,x]:\mathbb{Z}_{Tx}^R(i) \text{ is of class } \mathcal{I}/\mathcal{I}\} \nonumber\\
                      &=\{[\mathbb{Z},0,x]:\mathbb{Z}_x^R(i+1) \text{ is of class } \mathcal{I}/\mathcal{I}\} \label{eq_shift_I_I}\\
                      &=E_{i+1},\nonumber
\end{align}
where in Eq.(\ref{eq_shift_I_I}), we used the fact that \([\mathbb{Z},i,Tx] = [\mathbb{Z},i+1,x]\).
Similarly, we can show that \(T^{-1}F_i = F_{i+1}\).

Observe  that if \([\mathbb{Z},0,x] \in E_{i}\) (similarly \(F_{i}\)), then, by the interval property (Lemma \ref{lemma:record_descendants} and Remark \ref{remark_intervalProperty_integerValued}), \(i+1\) and \(i\) are in the same connected component, i.e., \(\mathbb{Z}^R_x(i)=\mathbb{Z}^R_x(i+1)\), which implies that  \([\mathbb{Z},0,x] \in E_{i+1}\) (similarly \(F_{i+1}\)).
Thus, we have \(E_i \subseteq T^{-1}(E_i)\).
Similarly, we have \(F_i \subseteq T^{-1}(F_i)\).
Using the fact that, if \((\mathbb{P},T)\) is ergodic and \(A\) is a measurable subset such that \(A \subseteq T^{-1}A\), then \(\mathbb{P}[A] \in \{0,1\}\) (see \cite[Theorem 16.1.9]{bremaudProbabilityTheoryStochastic2020}), we obtain the following proposition.

\begin{proposition}\label{prop_ergodic_trivial}
    Let \(X=(X_n)_{n \in \mathbb{Z}}\) be a stationary sequence as in Def. \ref{def_stationary_seq} and \emph{ergodic}. 
    Then, 
    \begin{align*}
        \mathbb{P}[\mathbb{Z}^R_X(0) \text{ is of class } \mathcal{I}/\mathcal{I}] \in \{0,1\}, \\
        \mathbb{P}[\mathbb{Z}^R_X(0) \text{ is of class } \mathcal{I}/\mathcal{F}] \in \{0,1\}.
    \end{align*}
\end{proposition}

\begin{corollary}\label{cor_comp_f_f}
    Under the assumptions of Proposition \ref{prop_ergodic_trivial}, the event \\ \(\{\mathbb{Z}^R_X(0) \text{ is of class \(\mathcal{F}/\mathcal{F}\)}\}\) has trivial probability measure.
\end{corollary}
\begin{proof}
    By the Foil classification Theorem \ref{thm_foil_classification}, \(\mathbb{Z}^R_X(0)\) is either of class \(\mathcal{I}/\mathcal{I}\) or \(\mathcal{I}/\mathcal{F}\) or \(\mathcal{F}/\mathcal{F}\).
    By Proposition \ref{prop_ergodic_trivial}, the first two events have trivial probability measure which implies that the third event has trivial probability measure.
\end{proof}

The proof of the following lemma is similar to that of Proposition \ref{proposition:r-graph_is_eft}.
The proof essentially relies on the interval property of record vertex-shift.
\begin{lemma}\label{lemma_stationary_one_component}
    Let \(X\) be a stationary sequence as in Def. \ref{def_stationary_seq}, \([\mathbb{Z},0,X]\) be its associated network, and \(\mathbb{Z}^R_X\) be the record graph of the network \((\mathbb{Z},X)\).
    If \(0\) has a.s. infinitely many ancestors in \(\mathbb{Z}^R_X\), then the record graph \(\mathbb{Z}^R_X\) is a.s. connected.
\end{lemma}
\begin{proof}
    For any \(i \in \mathbb{Z}\), let \(A_i\) be the event that \(i\) has infinitely many ancestors in \(\mathbb{Z}^R_X\).
    By the stationarity of \(X\), every event \(A_i\) occurs with probability \(1\) for each \(i\in \mathbb{Z}\) since \(A_0\) occurs with probability \(1\).
    Therefore, the event \(A = \cap_{i \in \mathbb{Z}}A_i\) occurs with probability \(1\).

    We now show that every integer \(i<0\) belongs to the component of \(0\) in  \(\mathbb{Z}^R\).
    Similarly, by interchanging the role of \(i\) and \(0\), we could show that every integer \(i>0\) belongs to the component of \(0\).

    Let $i<0$.
    Since \(A\) occurs with probability \(1\), the sequence $(R^n_X(i))_{n \in \mathbb{Z}}$ is strictly increasing a.s.
    So, there exists a smallest (random) integer $k>0$ such that $R_X^{k-1}(i)<0 \leq R^k_X(i)$ a.s..
    By Lemma \ref{lemma:record_descendants} (applied to $R_X^{k-1}(i)$) and by Remark \ref{remark_intervalProperty_integerValued}, $0$ is a descendant of $R^k_X(i)$ a.s..
    Thus, a.s., $0$ and $i$ are in the same connected component.
\end{proof}

\begin{lemma}\label{lemma_comp_0_fin_all_fin}
    Let \(X=(X_n)_{n \in \mathbb{Z}}\) be a stationary sequence of random variables as in Def. \ref{def_stationary_seq}, \([\mathbb{Z},0,X]\) be its associated network and for all \(n \in \mathbb{Z}\), \(\mathbb{Z}^R_X(n)\) be the component of \(n\) in the record graph of \((\mathbb{Z},X)\).
    Then, we have
    \begin{align}
        \mathbb{P}[\mathbb{Z}^R_X(0) \text{ is finite}] &= \mathbb{P}[\mathbb{Z}^R_X(n) \text{ is finite,} \forall n \geq 1]
        =\mathbb{P}[\mathbb{Z}_X^R(n) \text{ is finite,} \forall n \leq -1] \label{eq_F_F_probability}.
    \end{align}
\end{lemma}
\begin{proof}
    Let \(A\) be the event given by \(A = \{\mathbb{Z}^R_X(0) \text{ is finite}\}\).
    By Poincar\'e's recurrence lemma, \(\mathbb{P}[A] = \mathbb{P}[\{\mathbb{Z}_{T^nX}^R(0) \text{ is finite for infinitely many }n \geq 1\}]\).
    We show that for any \(\omega \in A\), if \(\mathbb{Z}_{T^nX(\omega)}^R(0)\) is finite for infinitely many \(n\geq 1\), then \(\mathbb{Z}_{X(\omega)}^R(n)\) is finite for all \(n \geq 1\).
    This implies the first equality in Eq. (\ref{eq_F_F_probability}).
    The second equality of Eq. (\ref{eq_F_F_probability}) follows by taking the shift  \(\widetilde{T} = T^{-1}\) instead of \(T\), and by using the equality of the events \(\{\mathbb{Z}_{\widetilde{T}^nX}^R(0) \text{ is finite}\} = \{\mathbb{Z}_X^R(-n) \text{ is finite}\}\) for all \(n \geq 1\).

    Since \(X\) is stationary, the network \([\mathbb{Z},0,X]\) is unimodular.
    So, we apply the classification theorem (Theorem \ref{thm_foil_classification}) to the record graph \(\mathbb{Z}_X^R\) to obtain the following result:
    for any integer \(n\), its component \(\mathbb{Z}_X^R(n)\) is finite if and only if there exists an integer \(i \in V(\mathbb{Z}_X^R(n))\) that has only finitely many ancestors.

    We \emph{claim} that  if \(\mathbb{Z}_{X(\omega)}^R(n_1)\) is finite for any positive integer \(n_1\), then \(\mathbb{Z}_{X(\omega)}^R(k)\) is finite for all \(k<n_1\).
    This follows because \(\mathbb{Z}_{X(\omega)}^R(n_1)\) is finite if and only if the sequence \((S_{n+n_1})_{n \geq 0}\) attains a maximum value \(a_{n_1}\) finitely many times, i.e., \(S_{n'+n_1}=a_{n_1}\) for some \(n' \geq 0\) and \(S_{n+n_1}<a_{n_1}\) for all \(n>n'\) (equivalently, from the discussion in the above paragraph, \(n_1\) has finitely many ancestors).
    So, if \(k \leq n'+n_1\), then the sequence \((S_{n+k})_{n \geq 0}\) attains the maximum value \(a_{n_1}\) at \(N=n'+n_1-k\) and never attains it for all \(n>N\).
    Thus, \(\mathbb{Z}_{X(\omega)}^R(k)\) is finite and the claim is proved.

    The above claim implies that if there exists a subsequence \((n_k)_{k \geq 1}\) of non-negative integers such that \(n_k \to \infty\) as \(k \to \infty\) and \(\mathbb{Z}^R_{X(\omega)}(n_k)\) is finite for all \(k \geq 1\), then \(\mathbb{Z}^R_{X(\omega)}(n)\) is finite for all \(n \in \mathbb{Z}\).
    Using this and the fact that \(\mathbb{Z}^R_{T^nX}(0)= \mathbb{Z}^R_{X}(n)\) for all \(n \geq 0\), we obtain that the events \(\{\mathbb{Z}_{T^nX}^R(0)\) is finite for infinitely many \(n \geq 1\}\) and \(\{\mathbb{Z}_{X}^R(n) \text{ is finite for all }n \geq 1\}\) are one and the same.
\end{proof}

\begin{remark}\label{rmk_f_f_all_f_f}
    By the above lemma, it follows that if \(\mathbb{Z}^R_X(0)\) is finite a.s., then a.s., \(\mathbb{Z}_X^R(n)\) is finite for all \(n \in \mathbb{Z}\).
\end{remark}

\begin{theorem}[Phase transition]\label{thm_phase_transition_stationary}
    Let \(X=(X_n)_{n \in \mathbb{Z}}\) be a stationary sequence as in Def. \ref{def_stationary_seq} and \emph{ergodic}.
    Let \([\mathbb{Z}^R_X(0),0]\) be the component of \(0\) in the record graph \(\mathbb{Z}_X^R\) of the network \((\mathbb{Z},X)\).
    \begin{enumerate}
        \item If \(\mathbb{E}[X_0]<0\), then a.s. every component of \(\mathbb{Z}^R_X\) is of class \(\mathcal{F}/\mathcal{F}\).
        \item If \(\mathbb{E}[X_0]>0\), then \(\mathbb{Z}^R_X\) is connected, and it is of class \(\mathcal{I}/\mathcal{F}\) a.s.
         \item If \(\mathbb{E}[X_0]=0\), then a.s., \(\mathbb{Z}^R_X\) is connected, and it is either of class \(\mathcal{I}/\mathcal{F}\) or of class \(\mathcal{I}/\mathcal{I}\).
    \end{enumerate}
\end{theorem}
\begin{proof}
    {\it (Proof of 1.)}
    Let \(\mathbb{E}[X_0]<0\).
    Since, by ergodicity, \(\frac{S_n}{n}\to \mathbb{E}[X_0]\) a.s., we have \(S_n \to -\infty\) a.s.
    Therefore, the sequence \((S_n)_{n \geq 0}\) attains a maximum value \(a_0\) at some (random) \(n'\geq 0\) and \(S_n<a_0\) for all \(n>n'\).
    This implies that \([\mathbb{Z}^R_X(0),0]\) is finite a.s. and hence it is of class \(\mathcal{F}/\mathcal{F}\).
    By Remark \ref{rmk_f_f_all_f_f}, every component of \(\mathbb{Z}^R_X\) is finite and hence it is of class \(\mathcal{F}/\mathcal{F}\).

    {\it (Proof of 2.)}
    Let \(\mathbb{E}[X_0]>0\).
    Since, by ergodicity, \(\frac{S_n}{n}\to \mathbb{E}[X_0]\) a.s., we have \(S_n \to \infty\) a.s. and  \(S_{-n} \to -\infty\) a.s.
    Therefore, \(0\) has infinitely many ancestors and by the interval property of record vertex-shift, some ancestor of \(0\) has infinitely many descendants.
    Hence, \(\mathbb{Z}^R_X\) has a single component (by Lemma \ref{lemma_stationary_one_component}) and it is of class \(\mathcal{I}/\mathcal{F}\).

    {\it (Proof of 3.)}
    Let \(\mathbb{E}[X_0]=0\) and \(A\) be the event
    \begin{equation*}
        A:=\{[\mathbb{Z}^R_X(0),0] \text{ is of class }\mathcal{F}/\mathcal{F}\}.
    \end{equation*}
    We show that \(\mathbb{P}[A]=0\).
    \emph{Suppose} that \(\mathbb{P}[A]>0\).
    Then, by Corollary \ref{cor_comp_f_f}, \(\mathbb{P}[A]=1\).
     Lemma \ref{lemma_comp_0_fin_all_fin} implies that \( \mathbb{P}[\mathbb{Z}^R_X(n) \text{ is finite } \forall n \in \mathbb{Z}]=1\).
     Note that \(\mathbb{Z}^R_{X(\omega)}(n)\) is finite for all \(n  \in \mathbb{Z}\) if and only if there exists a (random) subsequence \((n'_k)_{k \in \mathbb{Z}}\) of \(\mathbb{Z}\) such that \(n'_k \to +\infty\), \(n'_{-k} \to -\infty\) as \(k \to \infty\) and for every \(k \in \mathbb{Z}\), \(S_n(\omega) < S_{n'_k}(\omega)\) for all \(n>n'_{k}\).
    In particular, a.s., \(S_n \to - \infty\) and \(S_{-n} \to \infty\) as \(n \to \infty\).
    Consider the covariant subset \(Q\) consisting of peak points 
    \begin{equation*}
        Q := \left\{k \in \mathbb{Z}: \sum_{i=k}^{k+n} X_i \leq -1 \, \forall n\geq 0\right\}.
    \end{equation*}
    The covariant set \(Q\) is a subset of integers such that the sums starting from these integers are always negative.
    Let us enumerate the integers of the set \(Q\) as \(Q = (n_{i})_{i \in \mathbb{Z}}\), where \(n_i < n_{i+1}\) for all \(i \in \mathbb{Z}\) and \(n_0 > 0\) is the smallest positive integer that belongs to the set \(Q\).
    Such an enumeration is possible because \(S_n \to -\infty\) and thus, \(n_i \to \infty\) as \(i \to \infty\).
    Note that the intermediate sums between two consecutive integers of \(Q\) satisfy
    \begin{equation}\label{eq_inter_sums}
        \sum_{k=n_i}^{n_{i+1}-1}X_k \leq -1, \, \forall i \in \mathbb{Z}.
    \end{equation}
    Since \(\mathbb{P}[Q \text{ is non-empty}]=1\), the intensity \(\mathbb{P}[0 \in Q]\) of the set \(Q\) is positive (by Lemma \ref{lemma_non_empty_covariant_set}).
    By Birkhoff's pointwise ergodic theorem, we have a.s., \(\frac{S_n}{n} \to \mathbb{E}[X_0]\) as \(n \to \infty\).
    Since \(n_i \to \infty\) as \(i \to \infty\), we have a.s.,
    \begin{equation*}
        \lim_{i \to \infty} \frac{S_{n_i}}{n_i} = \lim_{n \to \infty} \frac{S_n}{n} = \mathbb{E}[X_0].
    \end{equation*}

    For all \(i\geq 1\), we have
    \begin{align}
        \frac{S_{n_i}}{n_i} &= \frac{X_0+\cdots+X_{n_0-1}+\sum_{k=0}^{i-1}(X_{n_k}+\cdots+X_{n_{k+1}-1})}{n_i} \nonumber \\
        &\leq \frac{X_0+\cdots+X_{n_0-1}+\sum_{k=0}^{i-1}(-1)}{n_i}\label{eq_upperbound},
    \end{align}
    where Eq. (\ref{eq_upperbound}) is obtained using Eq. (\ref{eq_inter_sums}).
    Taking the limit \(i \to \infty\) on both sides of Eq. (\ref{eq_upperbound}), we obtain a.s.,
    \begin{equation*}
        \mathbb{E}[X_0] \leq -\mathbb{P}[0 \in Q],
    \end{equation*}
    since \(\frac{-i}{n_i} = -\left(\frac{(\sum_{k=1}^{n_i}\mathbf{1}\{k \in Q\})-1}{n_i}\right)\) and the latter converges a.s. to \(-\mathbb{P}[0 \in Q]\) as \(i \to \infty\) by the cross-ergodic theorem for point processes (see \cite[Section 1.6.4]{baccelliElementsQueueingTheory2003}).
    Since, the intensity \(\mathbb{P}[0 \in Q]>0\) and \(\mathbb{E}[X_0]=0\), we obtain a contradiction.
    Hence,  \(\mathbb{P}[[\mathbb{Z}^R(0),0]\text{ is of class }\mathcal{F}/\mathcal{F}]=0\).
\end{proof}

We have already seen an instance when the component is of class \(\mathcal{I}/\mathcal{I}\).
When \(X=(X_n)_{n \in \mathbb{Z}}\) is i.i.d. with mean \(0\), we showed that the record graph is of class \(\mathcal{I}/\mathcal{I}\) (see Proposition \ref{proposition:r-graph_is_eft}).
We now give an example of class \(\mathcal{I}/\mathcal{F}\) when mean is \(0\).

\begin{example}[stationary, ergodic, mean \(0\) but the record graph is of class \(\mathcal{I}/\mathcal{F}\)]\label{example_i_f_mean_0}\normalfont
    See Figure \ref{fig_example_i_f_mean_0} for an illustration of this example.
    Consider an \(M/M/1/\infty\) queue with arrival rate \(\lambda\) and service rate \(\mu\) satisfying \(\lambda<\mu\).
    Let \((N_n)_{n \geq 1}\) be the number of customers in the queue observed at all changes of state.
    The Markov chain \((N_n)_{n \geq 1}\) has a unique stationary distribution \(\eta\) (see \cite{asmussenAppliedProbabilityQueues2003}).
    The Markov chain \((N_n)_{n \geq 1}\) starting with the stationary distribution \(\eta\) is ergodic and \(N_n = 0\) for infinitely many \(n \geq 1\).
    Note that \(N_n \geq 0\) and \(N_{n+1}-N_{n}\in \{-1,+1\}\), for all \(n \geq 1\).
    Consider the stationary version \((-N_n)_{n \in \mathbb{Z}}\) of \((-N_n)_{n\geq 1}\) (which can be done by extending the probability space, see \cite[Theorem 5.1.14]{bremaudProbabilityTheoryStochastic2020}).
    Let \(Z_n = -N_n\) and \(X_n = Z_{n+1}-Z_n\) for all \(n \in \mathbb{Z}\).
    The sequence \(X=(X_n)_{ \in \mathbb{Z}}\) is a stationary sequence of random variables taking values in \(\{-1,+1\}\) and with mean \(\mathbb{E}[X_0] = \mathbb{E}[Z_1-Z_0]=0\) (since the sequence \((Z_n)_{n \in \mathbb{Z}}\) is stationary).
    Consider the record vertex-shift on the network \([\mathbb{Z},0,X]\).
    Let \((S_n)_{n \in \mathbb{Z}}\) be the sums of \(X\) given by \(S_0=0\), \(S_n = \sum_{k=0}^{n-1}X_k\) for all \(n>0\), and \(S_n= \sum_{k=n}^{-1}-X_k\) for all \(n<0\).
    The relation between \(S_n\) and \(N_n\) is given by \(S_n=-N_n+N_0\), for all \(n \in \mathbb{Z}\).
    Since \(N_n\geq 0\) for all \(n\in \mathbb{Z}\) and \(N_n=0\) a.s. for infinitely many \(n \geq 1\), we have \(S_n \leq N_0\) for all \(n \in \mathbb{Z}\) and \(S_n = N_0\) for infinitely many \(n \geq 1\).
    Therefore, \(0\) has infinitely many ancestors (since some \(k\)-th record epoch of \(0\) satisfies \(S_{R^k(0)}=N_0\) and \(S_{R^{(k+i)}(0)}=N_0\), for all \(i \geq 1\)).
    Similarly, since \(S_n \leq N_0\) for all \(n< 0\), some ancestor of \(0\) has infinitely many descendants (by the interval property of the record vertex-shift).
    Thus, the component \([\mathbb{Z}^R(0),0]\) of the record graph of \([\mathbb{Z},0,X]\) is of class \(\mathcal{I}/\mathcal{F}\).
\end{example}

\begin{figure}[htbp]
    \centering 
    \includegraphics[scale=0.95]{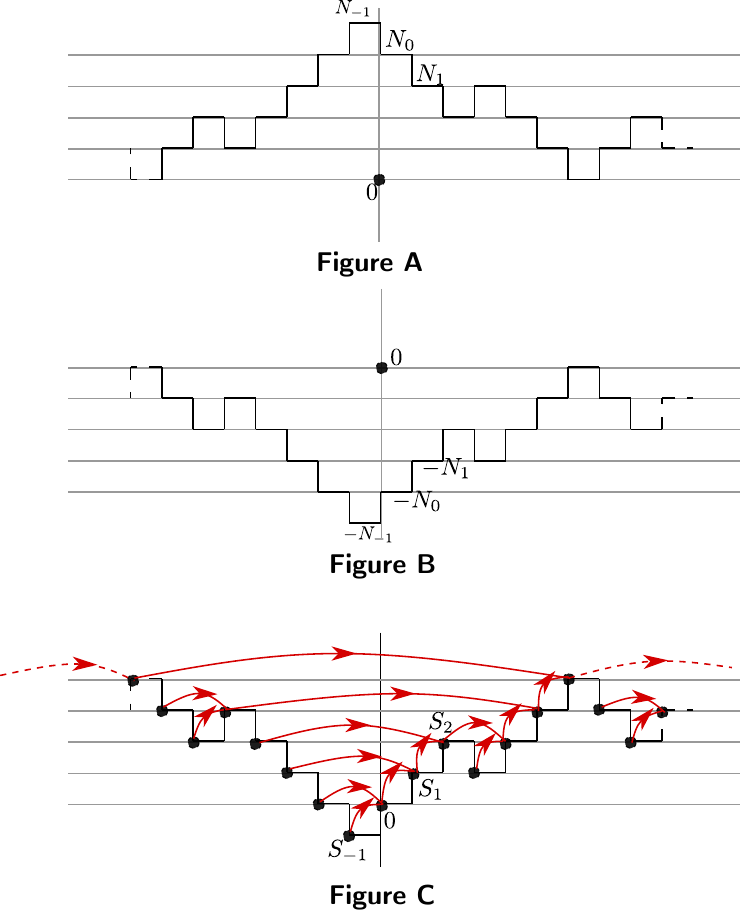}
    \caption{An Illustration of Example \ref{example_i_f_mean_0}, consisting in the construction of a stationary sequence \(X=(X_n)_{n \in \mathbb{Z}}\) with \(\mathbb{E}[X_0]=0\) such that the component of \(0\) of the record graph of the network \([\mathbb{Z},0,X]\) is of class \(\mathcal{I}/\mathcal{F}\).
    Figure A depicts the stationary sequence \((N_n)_{n \in \mathbb{Z}}\), Figure B gives the stationary sequence \((-N_n)_{n \in \mathbb{Z}}\), and Figure C is obtained from Figure B by shifting the x-axis to \(-N_0\). In Figure C, the increment sequence of \((S_n)_{n \in \mathbb{Z}}\) is \(X\), the record graph is drawn in red arrows and the sequence of top most arrows is the bi-infinite path.}
    \label{fig_example_i_f_mean_0}
  \end{figure}

\section{Record representable Eternal Family Trees}

In this section, we focus on the following question.
Given a unimodular ordered EFT \([\mathbf{T},\mathbf{o}]\), does there exist a stationary sequence \(Y=(Y_n)_{n \in \mathbb{Z}}\) of random variables taking values in \(\{-1,0,1,2,\cdots\}\) such that \([\mathbf{T},\mathbf{o}]\) is the component of \(0\) in the record graph of the network \((\mathbb{Z},X)\)?
If there exists such a sequence, then we call \([\mathbf{T},\mathbf{o}]\) a {\bf record representable EFT}.
In Theorem \ref{thm_representation}, we provide sufficient conditions for \([\mathbf{T},\mathbf{o}]\) to be record representable.
In particular, we show that every unimodular ordered EFT of class \(\mathcal{I}/\mathcal{I}\) is record representable.
As for the unimodular ordered EFTs of class \(\mathcal{I}/\mathcal{F}\), we show that those EFTs of this class that have a unique succession line are record representable.
We prove Theorem \ref{thm_representation} in the following way: we observe that on any unimodular ordered EFT, the functions \(a\) and \(b\) are well-defined.
The function \(a\) maps a vertex to its immediate predecessor, whereas the function \(b\) maps it to its immediate successor in an ordered EFT.
We then proceed to show that every unimodular ordered EFT has at most two succession lines.
The backward map \(\Phi_R\) is bijective on the set of ordered EFTs that have a unique succession line.
The stationary sequence \(Y\) is the push-forward of \([\mathbf{T},\mathbf{o}]\) under the backward map \(\Phi_R\).
The stationarity of \(Y\) follows by Mecke-Thorisson point stationarity \ref{prop_mecke}.

An immediate corollary of Theorem \ref{thm_representation} is that if \(X=(X_n)_{n \in \mathbb{Z}}\) is a stationary sequence of random variables taking values in \(\mathbb{Z}\) such that \(\mathbb{E}[|X_0|]<\infty\), \([\mathbb{Z},0,X]\) be its associated network, \(f\) is a vertex-shift on \([\mathbb{Z},0,X]\) and \([\mathbb{Z}^f_X(0),0]\) be the component of \(0\) in the \(f\)-graph of \((\mathbb{Z},X)\) such that \([\mathbb{Z}^f_X(0),0]\) has a unique succession line, then there exist a stationary sequence \(Y=(Y_n)_{n \in \mathbb{Z}}\) of random variables taking values in \(\{-1,0,1,2,\cdots\}\) such that the component \([\mathbb{Z}^R_Y(0),0]\) of $0$ in the record graph of the network \((\mathbb{Z},Y)\) has the same distribution as that of \([\mathbb{Z}^f_X(0),0]\).
In particular, if \([\mathbb{Z}^f_X(0),0]\) is of class \(\mathcal{I}/\mathcal{I}\), then such a stationary sequence \(Y\) exists.
The vertex-shift \(f\) on \([\mathbb{Z},0,X]\) is said to be {\bf record representable} if \([\mathbb{Z}^f_X(0),0]\) is record representable.

\begin{theorem}\label{thm_representation}
    Let \([\mathbf{T},\mathbf{o}]\) be a unimodular ordered EFT such that \(\mathbf{T}\) has a unique succession line.
    Then, there exists a stationary sequence \(Y=(Y_n)_{n \in \mathbb{Z}}\) of random variables taking values in \(\{-1,0,1,2,\cdots\}\) such that \\ \([\mathbf{T},\mathbf{o}] \overset{\mathcal{D}}{=}[\mathbb{Z}^R_Y(0),0]\), where \(R\) is the record vertex-shift and \(\mathbb{Z}^R_Y(0)\) is the component of \(0\) in the record graph \(\mathbb{Z}^R_Y\) of the network \((\mathbb{Z},Y)\).
    (In particular, if \([\mathbf{T},\mathbf{o}]\) is of class \(\mathcal{I}/\mathcal{I}\), then such a stationary sequence \(Y\) exists.)
\end{theorem}

\begin{corollary}
    Let \(X=(X_n)_{n \in \mathbb{Z}}\) be a stationary sequence that satisfies \(\mathbb{E}[|X_0|]< \infty\) and \([\mathbb{Z},0,X]\) be the network associated to \(X\).
    Let \(f\) be a vertex-shift on \([\mathbb{Z},0,X]\) such that \(0\) has a.s. infinitely many ancestors in the \(f\)-graph \(\mathbb{Z}^f_X\) and the component \(\mathbb{Z}^f_X(0)\) of \(0\) in \(\mathbb{Z}^f_X\) has a unique succession line.
    Then, there exists a stationary sequence \(Y=(Y_n)_{n \in \mathbb{Z}}\) of random variables taking values in \(\{-1,0,1,2,\cdots\}\) such that \([\mathbb{Z}^f_X(0),0] \overset{\mathcal{D}}{=}[\mathbb{Z}^R_Y(0),0]\), where \(R\) is the record vertex-shift and \(\mathbb{Z}^R_Y(0)\) is the component of \(0\) in the record graph \(\mathbb{Z}^R_Y\) of the network \((\mathbb{Z},Y)\).
\end{corollary}
\begin{proof}
    Since \(X\) is stationary, the network \([\mathbb{Z}^f_X(0),0]\) is unimodular.
    Apply Theorem \ref{thm_representation} to get the result.
\end{proof}

We first recall the notation of Section \ref{subsec_backward_map} and prove a few lemmas before proving Theorem \ref{thm_representation}.

Let \([\mathbf{T},\mathbf{o}]\) be a unimodular ordered EFT.
Denote the RLS (total) order obtained on the vertices of \(\mathbf{T}\) by \(\prec\) (see Section \ref{subsec_RLS}).
The set of minimal vertices of \(\mathbf{T}\) is a covariant subset.
It can be either empty or a singleton.
Since \(\mathbf{T}\) has infinitely many vertices, by No Infinite/ Finite Inclusion (Lemma \ref{lemma:no_infinite_finite}), a.s. \(\mathbf{T}\) does not have any minimal vertices.
Therefore, the map \(b:V(\mathbf{T})\to V(\mathbf{T})\) given by \(b(u) = \max\{v: v \prec u\}\) (if \(\max\) exists) is a.s. well-defined.
Since \(b\) is injective on \(V(\mathbf{T})\) and \([\mathbf{T},\mathbf{o}]\) is unimodular, \(b\) is bijective a.s. by Proposition \ref{prop:injective_bijective}.
But \(a(b(u)) = u\), for all \(u \in V(\mathbf{T})\), where \(a(v) = \min\{w:w \succ v\}\).
Since the range of \(b\) is \(V(\mathbf{T})\), the map \(a\) is well-defined on \(V(\mathbf{T})\).
Thus, we have the following lemma.

\begin{lemma}\label{lemma_a_b_wellDefined}
    Let \([\mathbf{T},\mathbf{o}]\) be a unimodular ordered EFT.
    Then, the maps \(a:V(\mathbf{T}) \to V(\mathbf{T})\), \(b:V(\mathbf{T}) \to V(\mathbf{T})\) are a.s. well-defined.
\end{lemma}

Let \(T\) be an ordered EFT.
For any vertex \(u \in V(T)\), let \(U((T,u)) = (u_n)_{n \in \mathbb{Z}}\) be the succession line passing through \(u\) (see Def. \ref{defn_succession_line}).
The set \(\{u_n:n \in \mathbb{Z}\}\) is called the succession line of \(u\).
It is easy to check that if two succession lines intersect then they are one and the same.

Let \([\mathbf{T},\mathbf{o}]\) be a unimodular ordered EFT.
Lemma \ref{lemma_a_b_wellDefined} implies that \(U((\mathbf{T},u))\) has distinct entries for all \(u \in V(\mathbf{T})\).
In Lemma \ref{lemma_unique_succession_line}, we show that a unimodular ordered EFT \([\mathbf{T},\mathbf{o}]\) has at most two succession lines.
Note that by the classification theorem, \([\mathbf{T},\mathbf{o}]\) is either of class \(\mathcal{I}/\mathcal{I}\) or of class \(\mathcal{I}/\mathcal{F}\).

\begin{lemma}\label{lemma_unique_succession_line}
    Let \([\mathbf{T},\mathbf{o}]\) be a unimodular ordered EFT.
    \begin{itemize}
        \item If \([\mathbf{T},\mathbf{o}]\) is of class \(\mathcal{I}/\mathcal{I}\), then \(\mathbf{T}\) has a unique succession line and it contains all the vertices of \(\mathbf{T}\).
        \item If \([\mathbf{T},\mathbf{o}]\) is of class \(\mathcal{I}/\mathcal{F}\), then \(\mathbf{T}\) has at most two succession lines; together they contain all the vertices of \(\mathbf{T}\).
    \end{itemize}
\end{lemma}

\begin{proof}

    If \([\mathbf{T},\mathbf{o}]\) is of class \(\mathcal{I}/\mathcal{I}\), then every vertex of \(\mathbf{T}\) has finitely many descendants.
    Let \(v \in V(\mathbf{T})\).
    Consider the smallest common ancestor \(w\) of \(v\) and \(\mathbf{o}\).
    The descendant tree of \(w\) is finite and it contains both \(v\) and \(\mathbf{o}\).
    Thus, the succession lines of \(v\) and \(o\) intersect, which implies that they are one and the same.

    Let us assume that \([\mathbf{T},\mathbf{o}]\) is of class \(\mathcal{I}/\mathcal{F}\).
    Let \(\neswarrow\) denote the unique bi-infinite path of \(\mathbf{T}\).
    Consider the set
  \begin{equation}\label{eq_W}
      W=\{u \in V(\mathbf{T}): u \prec v \ \forall v \in \neswarrow\}.
  \end{equation}
Denote its complement \(V(\mathbf{T})\backslash W\) by \(W'\).
We first show that \(W'\) has a unique succession line and all the vertices of \(W'\) belong to this succession line.
See Figure \ref{fig_i_f_succession_lines} for an illustration.

If \(u \in W'\), then there exists a vertex \(v \in \neswarrow\) such that \(u \succ v\) (we may safely assume it because if \(u \in \neswarrow\) then u has a child \(v' \in \neswarrow\) and \(u \succ v' \)).
The vertex \(a(u) = \min\{w: w \succ u\}\) also satisfies \(a(u) \succ v\) since \(a(u) \succ u\), implying that \(a(u) \in W'\).
The vertex \(b(u)= \max\{w: w \prec u\}\) satisfies \(b(u) \succeq v\) since \(v \in \{w:w \prec u\}\), implying that \(b(u) \in W'\).
Therefore, for every vertex \(u \in W'\), the succession line of \(u\) is contained in  \(W'\).
For any two vertices \(u,v\) in \(W'\), there are finitely many in between them, i.e., the set \(\{w:u \prec w \prec v\}\) is finite (assuming that \(u \prec v\)).
Therefore, the succession line of \(u\) and \(v\) are one and the same.
Thus, \(W'\) has a unique succession line and this succession line contains all the vertices of \(W'\).

We now show that if \(W\) is non-empty, then it has a unique succession line and all the vertices of \(W\) belong to this succession line.
Let \(u \in W\).
The vertex \(a(u)= \min\{w: w \succ u\}\) also belongs to \(W\) because \(\neswarrow \subset \{w:w \succ u\}\) and \(a(u) \not \in \neswarrow\).
The vertex \(b(u)= \max\{w: w \prec u\}\) also belongs to \(W\) since \(b(u)\prec u\) and \(u \in W\).
Therefore, the succession line of \(u\) is contained in \(W\).
For any two vertices \(u,v\) in \(W\), there are finitely many vertices between them.
Thus, \(W\) has a unique succession line (if \(W\) is non-empty).
This completes the proof.
\end{proof}

\begin{figure}[htbp]
       \includegraphics[scale=1]{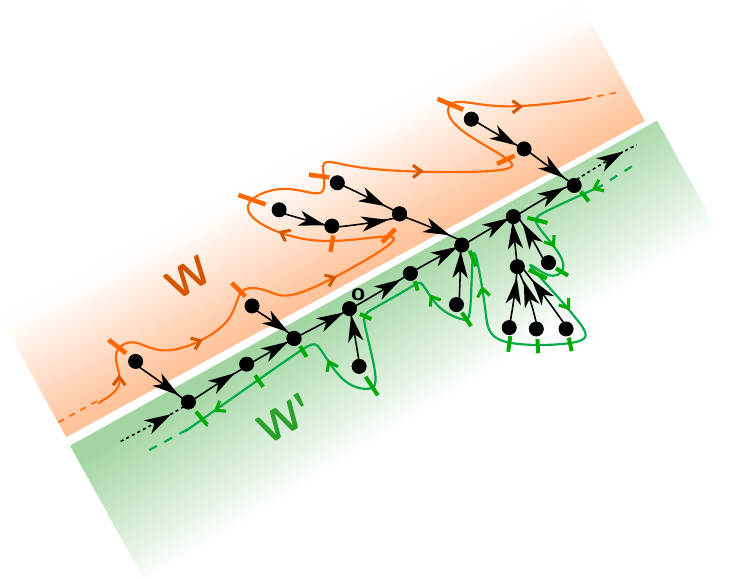}
       \caption{An Illustration of the \(\mathcal{I}/\mathcal{F}\) case in the proof of Lemma \ref{lemma_unique_succession_line}.
       The set \(W\) and its complement \(W'\) are highlighted. Each one of these set of vertices contains a unique succession line passing through all of its vertices. The arrow direction of the succession line indicates the images under the map \(b\), i.e., the immediate successor of the previous vertex. In this figure, the root \(\mathbf{o}\) belongs to the succession line containing the bi-infinite path.}
       \label{fig_i_f_succession_lines}
     \end{figure}

\begin{remark}
    If \([\mathbf{T},\mathbf{o}]\) is unimodular ordered EFT of class \(\mathcal{I}/\mathcal{F}\) that has a unique  succession line, then it has the ECS (Every Child Succeeding) order (see Subsection \ref{subsec_bi_variate_ekt} for the definition).
\end{remark}

Consider the following subset
\begin{equation}
    \widetilde{\mathcal{T}}_*:=\{[T,o]: T \text{ is an ordered EFT with a unique succession line}\}
\end{equation}
of EFTs.
It follows from Lemma \ref{lemma_a_b_wellDefined} and Lemma \ref{lemma_unique_succession_line} that for all \([T,o] \in \widetilde{\mathcal{T}}_*\), the succession line \(U((T,o)) = (u_n)_{n \in \mathbb{Z}}\) passing through \(o\) with \(u_0=o\) is a bijection from \(\mathbb{Z}\) to \(V(T)\) that maps \(n\) to \(u_n\), for all \(n \in \mathbb{Z}\).
Recall the backward map \(\Phi_R\) defined on the set \( \mathcal{T}'_*\) of Family Trees that have at most one bi-infinite \(F\)-path.
For all \([T,o] \in \mathcal{T}'_*\), we have \(\Phi_R([T,o]) = [\mathbb{Z},0,y]\), where \(y = (y_n)_{n \in \mathbb{Z}}\) is the sequence given by \(y_n = d_1(u_{n+1})-1\), for all \(n \in \mathbb{Z}\) and \((u_n)_{n \in \mathbb{Z}} = U((T,o))\) is the succession line passing through \(o\).
The sequence \(y=(y_n)_{n \in \mathbb{Z}}\) takes values in \(\{-1,0,1,2,\cdots\}^{\mathbb{Z}}\).
We show that \(\Psi_R \circ \Phi_R = I\) on \(\widetilde{\mathcal{T}}_*\), where \(\Psi_R([\mathbb{Z},0,y]) = [\mathbb{Z}^R_y(0),0]\) is the component of \(0\) in the record graph of the network \((\mathbb{Z},y)\).
The proof of the following lemma is the same as that of Proposition \ref{prop:reverse_R-graph_identity} except for a few modifications.
The proof relies on the property that every elder sibling of any vertex (if they exist) of a Family Tree with a unique succession line has finitely many descendants.

\begin{lemma}\label{lemma_encoded_sequence_same_tree}
    For all \([T,o] \in \widetilde{\mathcal{T}}_*\), \(\Psi_R \circ \Phi_R([T,o]) = [T,o]\).
\end{lemma}
\begin{proof}
    Let \((u_n)_{n \in \mathbb{Z}} = U((T,o))\) be the succession line of \((T,o)\) and $[\mathbb{Z},0,y] = \Phi_R([T,o])$, where $y=(y_n)_{n \in \mathbb{Z}}$.

    We prove that the bijective map \(\alpha:\mathbb{Z} \to V(T)\) defined by \(\alpha(n)= u_n\), for all \(n \in \mathbb{Z}\), induces a rooted network isomorphism \(\alpha\) from \([\mathbb{Z}^R_y(0),0]=\Psi_R([\mathbb{Z},0,y])\) to \((T,o)\).
    Note that, if \(u_j\) is a parent of \(u_i\) in \(T\) for some integers \(i\) and \(j\) then \(j>i\), which follows because \((u_n)_{ \in \mathbb{Z}}\) is the succession line of \((T,o)\).
    In view of this, to prove that \(\alpha\) induces a rooted network isomorphism, it is enough to show that for all $i \in \mathbb{Z}$, if \(u_j\) is the parent of \(u_i\) in \(T\) for some  \(j>i\), then \(j=R_y(i)\).
    It is enough to show the latter because (assuming the latter implication) if \(j = R_y(i)\) and \(u_k\) is a parent of \(u_i\) for some \(k>i\), then \(R_y(i)=k\), which implies that \(j=k\), and \(u_k = u_j\).
    
   Let \(i \in \mathbb{Z}\) and \(u_j\) be the parent of \(u_i\) in \(T\) for some integer $j>i$.
    If $j=i+1$, then $x_i = d_1(u_{i+1})-1 \geq 0$, because $u_i$ is a child of $u_{i+1}$.
    Hence, $R_y(i)=i+1$.
  
   So, let us assume that $j>i+1$.
   This implies that either $u_{i+1}$ is a leaf and a descendant of a sibling of $u_i$ or \(u_{i+1}\) is a sibling of \(u_i\).
   Every sibling \(v\) of \(u_i\) such that \(v \succ u_i\) has finitely many descendants.
   This follows because if there exists a sibling \(v \succ u_i\) of \(u_i\) that has infinitely many descendants, then \(u_i \in W\) (see Eq. \ref{eq_W}).
   But \(W\) is empty since \(T\) has a unique succession line.

    Let $u_{i_1} \prec u_{i_2} \prec \cdots  \prec u_{i_n}$ be the elder siblings of $u_i$ (i.e., \(u_i \prec u_{i_1}\)), and $T_1,T_2,\cdots,T_n$ be their descendant trees respectively.
   For $i<k<j$, let $l(k)$ be the smallest element of the set $\{1,2,\cdots,n\}$ such that $\{u_{i+1},u_{i+2},\cdots,u_k\} \subset V(T_1) \cup V(T_2) \cup \cdots V(T_{l(k)})$, and $u_{k_1}$ be the smallest vertex of $T_{l(k)}$.
  
   Then, the sum $z(i,k) = \sum_{m=i}^{k-1} y_m= \sum_{m=i+1}^{k}(d_1(u_m)-1)$ can be written as 
   \begin{align*}
     z(i,k) &= \left(\sum_{u \in V(T_1)}d_1(u)-1 \right)+ \cdots + \left(\sum_{u \in V(T_{l(k)-1})}d_1(u) - 1 \right) + \sum_{m=k_1}^k (d_1(u_m)-1)\\
     &< (-1) + \cdots + (-1) + \sum_{m=k_1}^k (d_1(u_m)-1)< 0,
   \end{align*}
   with the notation that \(V(T_0)\) is an empty set.
   The last steps follow from the first and the second statements of Lemma \ref{lemma:legall_sums}. 
  
    On the other hand, the sum $z(i,j) = \sum_{m=i+1}^j (d_1(u_m)-1)$ can be written as
    \begin{align*}
      z(i,j) &= \left(\sum_{u \in V(T_1)}d_1(u)-1 \right)+ \cdots + \left(\sum_{u \in V(T_n)}d_1(u) - 1 \right) + d_1(u_j)-1\\
      &= -n + d_1(u_j) - 1 \geq 0.
    \end{align*}
    The above steps follow from the first part of Lemma \ref{lemma:legall_sums}, and by the assumption that $u_j$ has at least $n+1$ children, namely $u_i,u_{i_1},\cdots,u_{i_n}$.
    Thus, $R_x(i)=j$, completing the proof.
\end{proof}

\begin{proof}[Proof of Theorem \ref{thm_representation}]
    Let \([\mathbb{Z},0,Y] = \Phi_R([\mathbf{T},\mathbf{o}])\).
    Since \(\mathbf{T}\) has a unique succession line, we have a.s., \(\Psi_R \circ \Phi_R([\mathbf{T},\mathbf{o}]) = [\mathbf{T},\mathbf{o}]\) by Lemma \ref{lemma_encoded_sequence_same_tree}.
    Therefore, a.s., \([\mathbb{Z}^R_Y(0),0] = [\mathbf{T},\mathbf{o}]\).

    We now show that the sequence \(Y\) is stationary.
    Let \((u_n)_{n \in \mathbb{Z}}=U((\mathbf{T},\mathbf{o}))\) be the succession line with \(u_0=\mathbf{o}\).
    For any \(i \in \mathbb{Z}\), the map \(u_n \mapsto u_{n+i}, \forall n \in \mathbb{Z}\)  is bijective.
    Therefore, the map \([\mathbf{T},\mathbf{o}] \mapsto [\mathbf{T},u_i]\) is measure-preserving, for any \(i \in \mathbb{Z}\), by Mecke-Thorisson point stationarity (Proposition\ref{prop_mecke}).
    Since, \([\mathbb{Z},0,T_iY] = \Phi_R([\mathbf{T},u_i])\), the map \((Y_n)_{n \in \mathbb{Z}} \mapsto (Y_{n+i})_{n \in \mathbb{Z}}\), for any \(i \in \mathbb{Z}\), is measure-preserving.
    Thus, \(Y\) is stationary.
\end{proof}

\subsection{Examples of record representable vertex-shifts}

\begin{example}[Strict record vertex-shift on an i.i.d. sequence (class \(\mathcal{I}/\mathcal{I}\))]\normalfont \label{example_strict_record}
    An illustration of this example can be seen in Figure \ref{fig_example_strict_record}.
    The strict record vertex-shift \(SR\) is defined on the set of networks of the form \([\mathbb{Z},0,x]\), where \(x=(x_n)_{n \in \mathbb{Z}}\) is a sequence of integers.
    It is defined by 
    \[SR(i) = \begin{cases*}
        \inf\{n>i:\sum_{k=i}^{n-1}x_k>0\} \text{ if the infimum exists,}\\
        i \text{ otherwise,}
    \end{cases*}\]
    for all \(i \in \mathbb{Z}\).
    Let \(X=(X_n)_{n \in \mathbb{Z}}\) be an i.i.d. sequence of random variables taking values in \(\mathbb{Z}\).
    One can show that the component of \(0\) in the \(SR\)-graph is of class \(\mathcal{I}/\mathcal{I}\) if \(\mathbb{E}[X_0]=0\) and \(\mathbb{E}[X_0^2]>0\), because \(\limsup_{n \to \infty}S_{-n} = \infty\), which implies that the descendant tree of \(0\) is a.s. finite.
    By Theorem \ref{thm_representation}, there exists a stationary sequence \(Y=(Y_n)_{n \in \mathbb{Z}}\) of random variables taking values in \(\{-1,0,1,2,\cdots\}\) such that \([\mathbb{Z}^{SR}_X(0),0] \overset{\mathcal{D}}{=} [\mathbb{Z}_Y^R(0),0]\).
    \begin{figure}[htbp]
        \includegraphics[scale=0.70]{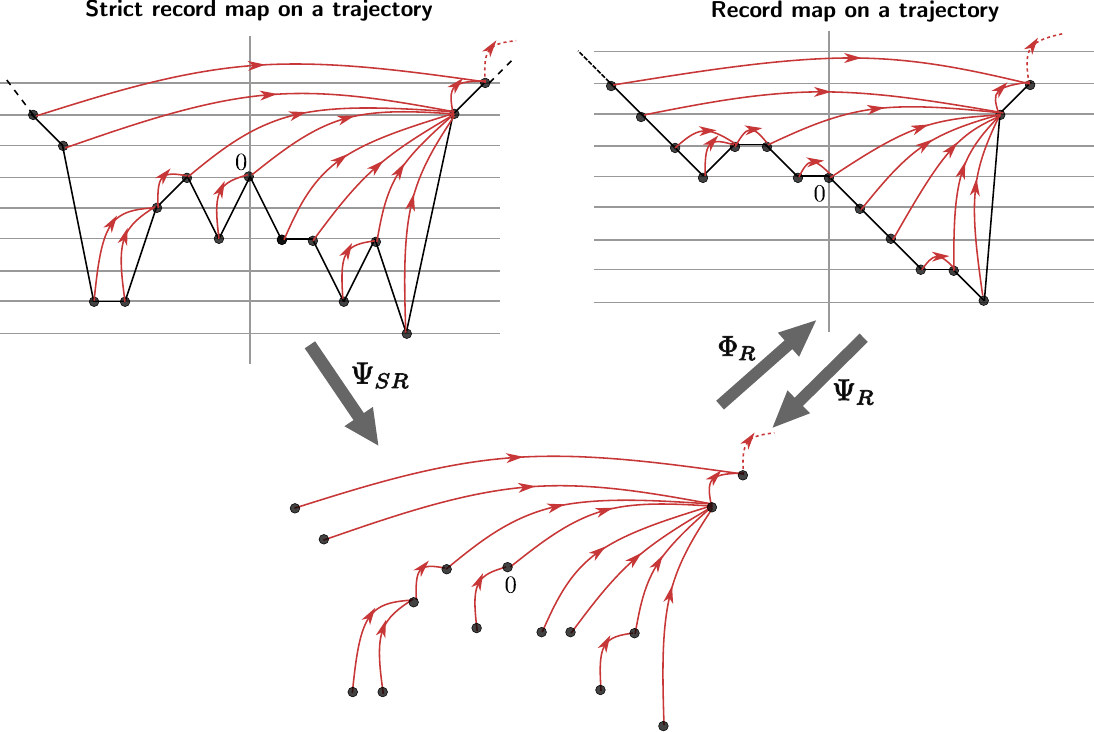}
        \caption{An Illustration of Example \ref{example_strict_record}.
        The top left figure shows the trajectory of a realisation of the i.i.d. sequence \(X=(X_n)_{n \in \mathbb{Z}}\), where \(X_n \in \mathbb{Z}, \forall n\).
        The  \(SR\)-graph (strict record graph) is drawn in red on the trajectory of the top left figure.
        The component of \(0\) in the \(SR\)-graph is drawn in the bottom figure.
        The top right figure shows the trajectory of the sequence \(Y=(Y_n)_{n \in \mathbb{Z}}\) which is obtained by applying the backward map \(\Phi_R\) to the rooted tree of the bottom figure.
        The component of \(0\) in the record graph of the trajectory in the top right figure is the same as the component of \(0\) in the \(SR\)-graph.}
        \label{fig_example_strict_record}
      \end{figure}
\end{example}

\begin{example}[Climbing point vertex-shift that depends on the future (class \(\mathcal{I}/\mathcal{F}\))]\normalfont \label{example_climbing_vs}
    An illustration of this example is shown in Figure \ref{fig_example_climbing_vertex-shift}.
    We define a vertex-shift $C$ (C for climbing point) based on an example taken from \cite[Example 1(a)]{fossStochasticSequencesRegenerative2013a}.
    Let \(X=(X_n)_{n \in \mathbb{Z}}\) be an i.i.d. sequence of random variables taking values in \(\{-1,0,1\}\) with probabilities given by: \(\mathbb{P}[X_0=1]=p, \mathbb{P}[X_0=-1]=q\) and \(\mathbb{P}[X_0 =  0]=1-p-q\), where \(0<q<p\) and \(p+q<1\).
    Define the vertex-shift on the networks of the form \([\mathbb{Z},0,x]\), where \(x=(x_n)_{n \in \mathbb{Z}}\) is a sequence with \(x_n \in \{-1,0,1\}, \forall n\) by
    \[C(i) = \begin{cases*}
        \inf\{n>i:\sum_{k=n}^{m}x_k \geq 0, \forall m>n\} \text{ if the infimum exists},\\
        i \text{ otherwise,}
    \end{cases*}\]
    for all \(i \in \mathbb{Z}\).
    Note that \(C(i)=k\) for some \(k>i\) if and only if \(k\) is the smallest integer larger than \(i \) such that \(S_{k} = \min\{S_n:n>i\}\).
    Since \(\mathbb{P}[X_0]=p-q>0\), the random walk \((S_n)_{n \geq 0}\) drifts to \(+\infty\), where \(S_0=0, S_n = \sum_{k=0}^{n-1}X_k\).
    Therefore, the event \(C(i)>i\) for all \(i \in \mathbb{Z}\) occurs with probability \(1\).
    Since, \(S_{-n}=\sum_{k=-n}^{-1}-X_k\) converges to \(-\infty\) as \( n \to \infty\), the random variable \(\max\{S_{-n}:n \geq 0\}\) is finite a.s.
    Thus, some ancestor of \(0\) in the \(C\)-graph has infinitely many descendants.
    So, the component \([\mathbb{Z}_X^C(0),0]\) of \(0\) in the \(C\)-graph is of class \(\mathcal{I}/\mathcal{F}\).
    One can show that the smallest child of every vertex which lies on the bi-infinite path of \(\mathbb{Z}^C_X(0)\) belongs to the bi-infinite path.
    Therefore, \([\mathbb{Z}_X^C(0),0]\) is of class \(\mathcal{I}/\mathcal{F}\) and has unique succession line (as it has ECS order).
    \begin{figure}[htbp]
           \includegraphics[scale=0.60]{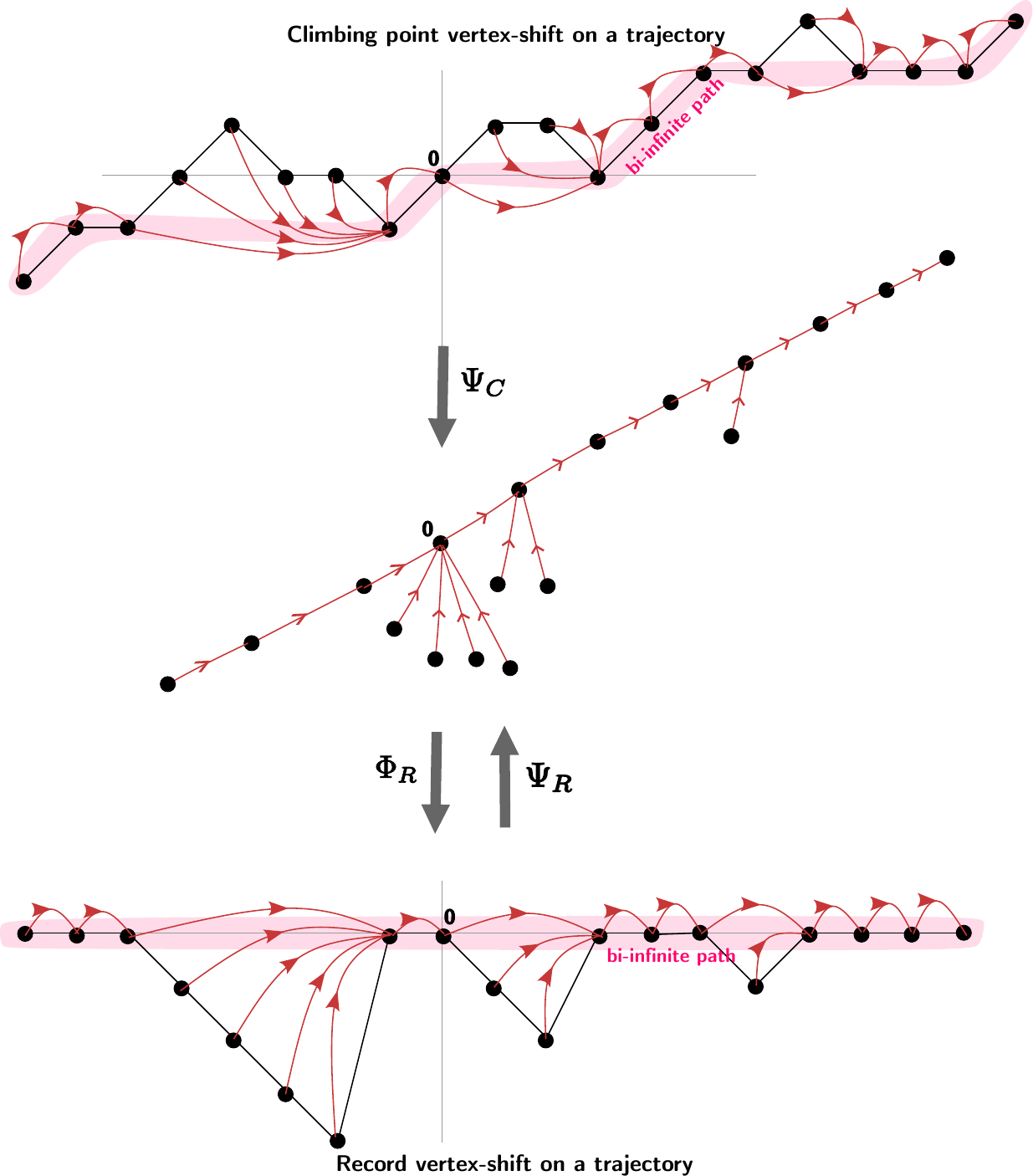}
           \caption{An Illustration of Example \ref{example_climbing_vs}.
           The top figure shows the trajectory of a realisation of the i.i.d. sequence \(X=(X_n)_{n \in \mathbb{Z}}\), where \(X_n \in \mathbb{Z}, \forall n\).
           The  \(C\)-graph (\(C\) for climbing point vertex-shift) is drawn in red on the trajectory of the top figure.
           The vertices on the bi-infinite path are highlighted.
           The component of \(0\) in the \(C\)-graph is drawn in the middle figure.
           The bottom figure shows the trajectory of the sequence \(Y=(Y_n)_{n \in \mathbb{Z}}\) which is obtained by applying the backward map \(\Phi_R\) to the rooted tree shown in the middle figure.
           The component of \(0\) in the record graph of the trajectory in the bottom figure is the same as the component of \(0\) in the \(C\)-graph.}
           \label{fig_example_climbing_vertex-shift}
         \end{figure}
\end{example}


\begin{thebibliography}{10}

    \bibitem{abrahamIntroductionGaltonWatsonTrees2015}
    Romain Abraham and Jean-Fran{\c c}ois Delmas.
    \newblock An introduction to {{Galton-Watson}} trees and their local limits.
    \newblock {\em arXiv:1506.05571 [math]}, June 2015.
    
    \bibitem{aldousAsymptoticFringeDistributions1991}
    David Aldous.
    \newblock Asymptotic {Fringe} {Distributions} for {General} {Families} of {Random} {Trees}.
    \newblock {\em Annals of Applied Probability}, 1(2):228--266, May 1991.
    
    \bibitem{aldousProcessesUnimodularRandom2007}
    David Aldous and Russell Lyons.
    \newblock Processes on {{Unimodular Random Networks}}.
    \newblock {\em Electronic Journal of Probability}, 12:1454--1508, 2007.
    
    \bibitem{asmussenAppliedProbabilityQueues2003}
    Søren Asmussen.
    \newblock {\em Applied probability and queues}.
    \newblock Springer, New York, 2003.
    \newblock OCLC: 51060198.
    
    \bibitem{baccelliElementsQueueingTheory2003}
    Fran{\c c}ois Baccelli and Pierre Br{\'e}maud.
    \newblock {\em Elements of {{Queueing Theory}}}.
    \newblock {Springer Berlin Heidelberg}, {Berlin, Heidelberg}, 2003.
    
    \bibitem{baccelliPointmapprobabilitiesPointProcess}
    Francois Baccelli and Mir-Omid Haji-Mirsadeghi.
    \newblock Point-map-probabilities of a point process and {Mecke}'s invariant measure equation.
    \newblock {\em The Annals of Probability}, 45(3):29.
    
    \bibitem{baccelliEternalFamilyTrees2018a}
    Fran{\c c}ois Baccelli, Mir-Omid Haji-Mirsadeghi, and Ali Khezeli.
    \newblock Eternal {Family} {Trees} and dynamics on unimodular random graphs.
    \newblock In Florian Sobieczky, editor, {\em Contemporary {Mathematics}}, volume 719, pages 85--127. American Mathematical Society, 2018.
    
    \bibitem{benniesRandomWalkApproach2000}
    J{\"u}rgen Bennies and G{\"o}tz Kersting.
    \newblock A {{Random Walk Approach}} to {{Galton}}\textendash{{Watson Trees}}.
    \newblock {\em Journal of Theoretical Probability}, 13(3):777--803, 2000.
    
    \bibitem{bhattacharyaRandomWalkBrownian2021}
    Rabi Bhattacharya and Edward~C. Waymire.
    \newblock {\em Random {{Walk}}, {{Brownian Motion}}, and {{Martingales}}}, volume 292 of {\em Graduate {{Texts}} in {{Mathematics}}}.
    \newblock {Springer International Publishing}, {Cham}, 2021.
    
    \bibitem{bremaudProbabilityTheoryStochastic2020}
    Pierre Brémaud.
    \newblock {\em Probability {Theory} and {Stochastic} {Processes}}.
    \newblock Universitext. Springer International Publishing, Cham, 2020.
    
    \bibitem{ferrariPoissonTreesSuccession2004a}
    Pablo Ferrari, Claudio Landim, and Hermann Thorisson.
    \newblock Poisson trees, succession lines and coalescing random walks.
    \newblock {\em Annales de l'Institut Henri Poincare (B) Probability and Statistics}, 40(2):141--152, April 2004.
    
    \bibitem{fossStochasticSequencesRegenerative2013a}
    Sergey Foss and Stan Zachary.
    \newblock Stochastic {Sequences} with a {Regenerative} {Structure} that {May} {Depend} {Both} on the {Future} and on the {Past}.
    \newblock {\em Advances in Applied Probability}, 45(4):1083--1110, December 2013.
    
    \bibitem{godrecheRecordStatisticsStrongly2017}
    Claude Godr{\`e}che, Satya~N Majumdar, and Gr{\'e}gory Schehr.
    \newblock Record statistics of a strongly correlated time series: Random walks and {{L\'evy}} flights.
    \newblock {\em Journal of Physics A: Mathematical and Theoretical}, 50(33):333001, August 2017.
    
    \bibitem{jimpitmanCombinatorialStochasticProcesses2006}
    {Jim Pitman}.
    \newblock {\em Combinatorial {{Stochastic Processes}}}, volume 1875 of {\em Lecture {{Notes}} in {{Mathematics}}}.
    \newblock {Springer-Verlag}, {Berlin/Heidelberg}, 2006.
    
    \bibitem{kallenbergFoundationsModernProbability2021}
    Olav Kallenberg.
    \newblock {\em Foundations of {{Modern Probability}}}, volume~99 of {\em Probability {{Theory}} and {{Stochastic Modelling}}}.
    \newblock {Springer International Publishing}, {Cham}, 2021.
    
    \bibitem{legallRandomTreesApplications2005b}
    Jean-Fran{\c c}ois Le~Gall.
    \newblock Random trees and applications.
    \newblock {\em Probability Surveys}, 2(0):245--311, 2005.
    
    \bibitem{nicolascurienRandomGraphs}
    {Nicolas Curien}.
    \newblock Lecture notes on random graphs.
    \newblock \url{https://www.math.u-psud.fr/~curien/cours/cours-RG.pdf}, February 2019.
    
    \bibitem{nicolascurienRandomWalksGraphs}
    {Nicolas Curien}.
    \newblock Lecture notes on random walks and graphs.
    \newblock \url{https://www.imo.universite-paris-saclay.fr/~nicolas.curien/enseignement.html}, October 2022.
    
    \bibitem{stepanovCharacterizationTheoremWeak1994}
    Alexei~Vladimirovich Stepanov.
    \newblock A {{Characterization Theorem}} for {{Weak Records}}.
    \newblock {\em Theory of Probability \& Its Applications}, 38(4):762--764, December 1994.
    
    \bibitem{thorissonCouplingStationarityRegeneration2000}
    Hermann Thorisson.
    \newblock {\em Coupling, stationarity, and regeneration}.
    \newblock Probability and its applications ({Springer}-{Verlag}). Springer, New York, 2000.
    
    \bibitem{wergenRecordsStochasticProcesses2013}
    Gregor Wergen.
    \newblock Records in stochastic processes\textemdash theory and applications.
    \newblock {\em Journal of Physics A: Mathematical and Theoretical}, 46(22):223001, June 2013.
    
    \end{thebibliography}
\end{document}